\documentclass[a4paper,12pt,
 leqno
]{article}

\usepackage{hyperref}
\hypersetup{linktocpage,
  colorlinks   = true, 
  urlcolor     = blue, 
  linkcolor    = blue,
  citecolor   = blue 
}\usepackage[all]{hypcap}  
\usepackage{breakurl}

\usepackage[utf8]{inputenc}
\usepackage[T1]{fontenc}
\usepackage{amsthm}
\usepackage{amssymb}
\usepackage{mathrsfs}
\usepackage{amsmath}
\usepackage{amsfonts}
\usepackage{mathtools}
\usepackage{graphicx}
\usepackage{float}
\usepackage{dsfont}
\usepackage{enumerate}
\usepackage[a4paper]{geometry}
\usepackage{csquotes}
\usepackage{IEEEtrantools}
\usepackage{appendix}
\usepackage{constants}
\usepackage[obeyDraft]{todonotes}
\usepackage{tikz-cd}

\newtheorem{The}{Theorem}[section]
\newtheorem{Lemme}[The]{Lemma}
\newtheorem{Prop}[The]{Proposition}
\newtheorem{Cor}[The]{Corollary}
\theoremstyle{definition}
\newtheorem{Def}[The]{Definition}

\theoremstyle{remark}
\newtheorem{Rk}[The]{Remark}

\title{\large 
\textbf{GEOMETRY OF  GAUSSIAN FREE FIELD SIGN CLUSTERS AND RANDOM INTERLACEMENTS}}
\author{}
\date{}
\makeatletter
\newcommand{\overbar}[1]{\mkern 1.5mu\overline%
{\mkern-1.5mu#1\mkern-1.5mu}\mkern 1.5mu}

\newcommand{\E}{\mathbb{E}}

\newcommand{\Q}{\mathbb{Q}}
\newcommand{\R}{\mathbb{R}}
\newcommand{\Z}{\mathbb{Z}}
\newcommand{\N}{\mathbb{N}}
\renewcommand{\P}{\mathbb{P}}
\newcommand{\eps}{\varepsilon}
\newcommand{\1}{\mathds{1}}

\newcommand{\I}{{\cal I}}
\newcommand{\V}{{\cal V}}
\renewcommand{\phi}{\varphi}
\renewcommand{\tilde}{\widetilde}
\renewcommand{\hat}{\widehat}
\renewcommand{\epsilon}{\varepsilon}
\renewcommand\theequation{\thesection.\arabic{equation}}

\@addtoreset{equation}{section}

\newcommand{\tend}[2]{\displaystyle\mathop{\longrightarrow}_{#1\rightarrow#2}}
\newconstantfamily{c}{symbol=c}
\makeatother

\usepackage{color}
\definecolor{Red}{rgb}{1,0,0} 
\definecolor{Blue}{rgb}{0,0,1}
\definecolor{Olive}{rgb}{0.41,0.55,0.13}
\definecolor{Yarok}{rgb}{0,0.5,0}
\definecolor{Green}{rgb}{0,1,0}
\definecolor{MGreen}{rgb}{0,0.8,0}
\definecolor{DGreen}{rgb}{0,0.55,0}
\definecolor{Yellow}{rgb}{1,1,0}
\definecolor{Cyan}{rgb}{0,1,1}
\definecolor{Magenta}{rgb}{1,0,1}
\definecolor{Orange}{rgb}{1,.5,0}
\definecolor{Violet}{rgb}{.5,0,.5}
\definecolor{Purple}{rgb}{.75,0,.25}
\definecolor{Brown}{rgb}{.75,.5,.25}
\definecolor{Grey}{rgb}{.7,.7,.7}
\definecolor{Black}{rgb}{0,0,0}

\usepackage{titletoc}

\dottedcontents{section}[4em]{}{2.9em}{0.7pc}
\dottedcontents{subsection}[0em]{}{3.3em}{1pc}

\usepackage{multicol}
\usepackage{parcolumns}

\begin{document}
\thispagestyle{empty}
\maketitle
\vspace{0.1cm}
\begin{center}
\vspace{-1.9cm}
Alexander Drewitz$^1$, Alexis Pr\'evost$^2$ and Pierre-Fran\c cois Rodriguez$^3$

\end{center}
\vspace{0.1cm}
\begin{abstract}
\centering
\begin{minipage}{0.85\textwidth}
For a large class of amenable transient weighted graphs $G$, we prove that the sign clusters of the Gaussian free field on $G$ fall into a regime of \textit{strong supercriticality}, 
in which two infinite sign clusters dominate
(one for each sign), and finite sign clusters are necessarily tiny, with overwhelming probability. Examples of graphs belonging to this class include regular lattices such as $\Z^d$, for $d\geq 3$, but also more intricate geometries, such as Cayley graphs of  
suitably growing (finitely generated) non-Abelian groups, and 
cases in which random walks exhibit anomalous diffusive behavior,
for instance various fractal graphs. 
As a consequence, we also show that the vacant set of random interlacements on these objects, introduced by Sznitman in
\textit{Ann.\@ Math., 171(3):2039--2087, 2010},
and which is intimately linked to the free field, contains an infinite connected component at small intensities.
In particular, this result settles
 an open problem from
\textit{Invent.\@ Math., 187(3):645--706, 2012.}

\end{minipage}
\end{abstract}

\vspace{0.4cm}
\begin{minipage}{0.89\textwidth}
{\small
\tableofcontents
}
\end{minipage}

\begin{flushleft}

\thispagestyle{empty}
\vspace{0.3cm}
\noindent\rule{5cm}{0.35pt} \hfill May 2024 \\
\vspace{0.13cm}

\begin{multicols}{2}
$^1$Universit\"at zu K\"oln \\  Department Mathematik/Informatik \\  50931 K\"oln, Germany \\\url{adrewitz@uni-koeln.de}\\  \vspace{0.15cm}

$^2$University of Geneva\\  Section of Mathematics\\ 1211 Geneva, Switzerland \\\url{alexis.prevost@unige.ch}  \columnbreak

$^3$Imperial College London\\ Department of Mathematics\\ London SW7 2AZ, United Kingdom\\
 \url{p.rodriguez@imperial.ac.uk}

\end{multicols}
\end{flushleft}

\newpage

\section{Introduction}
\label{section1}

This article rigorously investigates the phenomenon of
\textit{phase coexistence}
  which is associated to the geometry of certain random fields in their supercritical phase, characterized by the presence of strong, slowly decaying  
correlations. 
Our aim is to prove the existence of such a regime, and to describe the random geometry arising from the competing influences between two supercritical phases.
The leitmotiv of this work is to study the sign clusters of the Gaussian free field in 
``high dimensions'' (transient for the random walk), 
which offer a framework that is analytically tractable and has a rich algebraic structure, but questions of this flavor have emerged in various contexts, involving fields with similar large-scale behavior. One such instance is the model of random interlacements, introduced in \cite{MR2680403} and also studied in this article, which relates to the broad question of how random walks tend to create interfaces in high dimension, see e.g.\@ \cite{MR2520124}, \cite{MR2561432}, and also \cite{MR2838338}, \cite{MR3563197}. Another case in point (not studied in this article) is the nodal domain of a monochromatic random wave, e.g.\ a randomized Laplace eigenfunction on the $n$-sphere $\mathbb{S}^n$, at high frequency, which appears to display supercritical behavior when $n \geq 3$, see \cite{Sa1}.

As a snapshot of the first of our main results, Theorem \ref{T:mainresultGFF} below gives an essentially complete picture of the sign cluster geometry of the Gaussian free field $\Phi$ (see \eqref{intro:GFF}  for its definition) on a large class of transient graphs $G$. It can be informally summarized as follows. Under suitable assumptions on $G$, which hold e.g. when $G=\Z^d$, $d\geq 3$ --but see \eqref{intro:Gex1} below for further examples, which 
hopefully convey the breadth of our setup--,
\begin{equation}
\label{eq:MainGFFinformal}
\begin{split}
&\text{there exist exactly two infinite sign clusters of $\Phi$, one for}\\
&\text{each sign, which ``consume all the ambient space,'' up to}\\
&\text{(stretched) exponentially small finite islands of $+/-$ signs;}
\end{split}
\end{equation}
see Theorem \ref{T:mainresultGFF} for the corresponding precise statement. In fact, we will show that this regime of phase coexistence persists for level sets above small enough height $h=\varepsilon > 0$. It is worth emphasizing that \eqref{eq:MainGFFinformal} really comprises two distinct features, namely (i) the presence of unbounded sign clusters, which is an \textit{existence} result, and (ii) their ubiquity, which is \textit{structural} and forces bounded connected components to be very small. Our results further indicate a certain universality of this phenomenon, as the class of transient graphs $G$ for which we can establish \eqref{eq:MainGFFinformal} includes possibly fractal geometries, see the examples \eqref{intro:Gex1} below, where random walks typically experience slowdown due to the presence of ``traps at every scale,'' see e.g.\@ \cite{MR1802425}, \cite{MR1853353}, \cite{MR1938457} and the monograph \cite{MR3616731}. 

As it turns out, the phase coexistence regime for $\text{sign}(\Phi)$ described by \eqref{eq:MainGFFinformal} is also related to the existence of a supercritical phase for the vacant set of random interlacements; cf.\ \cite{MR2680403} and below \eqref{IP} for a precise definition. This is due to a certain algebraic relation linking $\Phi$ and the interlacements, see \cite{MR2892408}, \cite{MR3502602}, \cite{MR3492939}, whose origins can be traced back to early work in constructive field theory, see \cite{Sym}, and also \cite{brydges1982random}, \cite{dynkin1983markov}, and which will be a recurrent theme throughout this work. Interestingly, the arguments leading to the phase coexistence described in \eqref{eq:MainGFFinformal}, paired with the symmetry of $\Phi$, allow us to embed (in distribution) a large part of the interlacement set inside its complement, the vacant set, at small levels. As a consequence, we deduce the existence of a supercritical regime of the latter by appealing to the good connectivity properties of the former, for all graphs $G$ belonging to our class. We will soon return to these matters and explain them in due detail. For the time being, we note that these insights yield the answer to an important open question from \cite{MR2891880}, see the final Remark \nolinebreak5.6(2) therein and our second main result, Theorem \nolinebreak\ref{T:mainresultRI} below.

\bigskip

We now describe our results more precisely, and refer to Section \ref{2} for the details of our setup. We consider an infinite, connected, locally finite graph $G$ endowed with a positive and symmetric weight function $\lambda$ on the edges. 
To the data $(G,\lambda)$, we associate a canonical discrete-time random walk, which is the Markov chain with transition probabilities given by $p_{x,y} = \lambda_{x,y}/\lambda_x,$ where $\lambda_x =\sum_{y\in G} \lambda_{x,y}$.  
It is characterized by the generator
\begin{equation}
\label{Z:gen}
Lf(x)= \frac1{\lambda_x}\sum_{y\in G} \lambda_{x,y} (f(y)-f(x)) , \text{ for }x \in G,
\end{equation}
 for $f:G\to \mathbb R$ with finite support. We assume that the transition probabilities of this walk are uniformly bounded from below, see \nolinebreak\eqref{p0} in Section \nolinebreak \ref{2}, and writing $g(x,y),$ $x,y\in{G},$ for the corresponding Green density, see \eqref{eq:Greendef} below, that 
 \begin{equation}
\label{intro:Volume+Green}
\begin{split}
&\text{there exist parameters  
$\alpha$ and $\beta$ with $2 \leq \beta < \alpha$}\\ 
&\text{such that, for some distance function $d(\cdot,\cdot)$ on $G$, }\\
&    \lambda(B(x,L))\asymp L^{\alpha}\text{ and } g(x,y)\asymp (d(x,y)\vee 1)^{-(\alpha-\beta)}, \text{ for }x,y\in{G},
\end{split}
\end{equation}
where $\asymp$ means that the quotient is uniformly bounded from above and below by positive constants, $B(x,L)$ is the closed ball of radius $L$ in the metric $d(\cdot, \cdot)$ and $\lambda (A)= \sum_{x\in A} \lambda_x$ is the measure of $A \subset G$, see \eqref{Ahlfors} and \eqref{Green} in Section \nolinebreak\ref{2} for the precise formulation of \eqref{intro:Volume+Green}. The exponent $\beta$ in \eqref{intro:Volume+Green} reflects the diffusive (when $\beta=2$) or sub-diffusive (when $\beta>2$) behavior of the walk on $G$, cf.\ Proposition \ref{someproperties} below. Note that the condition on $g(\cdot ,\cdot)$ in \eqref{intro:Volume+Green} implies in particular that $G$ is transient for the walk. For more background on why condition \eqref{intro:Volume+Green} is natural, we refer to \cite{MR1853353}, \cite{MR1938457} as well as Remarks~\ref{T:HK} and  \nolinebreak\ref{remarkendofsection3} below regarding its relation to heat kernel estimates. As will further become apparent in Section \ref{secpre}, see in particular Proposition \ref{P:product} and Corollary \ref{C:examples}, choosing $d$ to be the graph distance on $G$ is not necessarily a canonical choice, for instance when $G$ has a product structure.

Apart from \eqref{p0}, \eqref{Ahlfors} and \eqref{Green}, we will often make one additional geometric assumption \eqref{weakSecIso} on $G$, introduced in Section \ref{2}. Roughly speaking, this hypothesis ensures a (weak) sectional isoperimetry of various large subsets of $G$, which allows for certain contour arguments. Rather than explaining this in more detail, we single out the following representative examples of graphs, which satisfy all four aforementioned assumptions  \eqref{p0}, \eqref{Ahlfors}, \eqref{Green} and \eqref{weakSecIso}, 
cf.\ Corollary \ref{C:examples} below: 
\begin{equation}
\label{intro:Gex1}
\begin{split}
&G_1 = \Z^d, \text{ with } d \geq 3, \\
&G_2 = G' \times \Z, \text{ with $G'$ the discrete skeleton of the Sierpinski gasket}, \\
&G_3 = \text{the standard $d$-dimensional graphical Sierpinski carpet for $d\geq3$,}\\
&G_4= 
\begin{array}{l}
\text{a Cayley graph of a finitely generated group $\Gamma = \langle S \rangle $ with $S=S^{-1}$}\\
\text{having 
polynomial volume growth of order $\alpha>2$}
\end{array}
\end{split}
\end{equation}
(see e.g.\@ \cite{MR1802425}, pp.6--7 for definitions of $G'$ and $G_3$, the latter corresponds to $V^{(d)}$ in the notation of \cite{MR1802425}), all endowed with unit weights and a suitable distance function $d$ (see Remark \ref{R:distance} and Section \ref{secpre}). The graph $G_2$ is a benchmark case for various aspects of \nolinebreak\cite{MR2891880}, to which we will return in Theorem \nolinebreak\ref{T:mainresultRI} below. The case $G_3$ underlines the fact that even in the fractal context a product structure is not necessarily required. The case $G_4$ subsumes $G_1$ of course, which is but a starting point for the current article, and it contains many interesting examples, for instance the $(2n+1)$-dimensional Heisenberg group $H_{2n+1}(\Z)$, for $n=1,2,\dots$

\begin{figure}[h]
\includegraphics[scale=0.8]{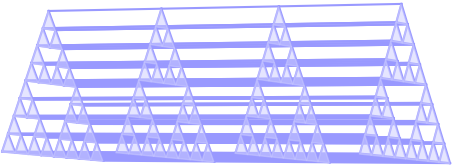}
\centering
\caption{\label{fig:G2} A graph of the form $G_2=G'\times \Z$, with $G'$ the discrete Sierpinski gasket.}
\end{figure} 

The fact that \eqref{weakSecIso} holds in cases $G_2$, $G_3$ and $G_4$ is not evident, and will follow by expanding on results of \cite{MR2996951}, see Section \ref{secpre}. In the case of $G_4$,  \eqref{weakSecIso} crucially relies on Gromov's deep structural result \cite{gromov1981groups}.
The reader may choose to focus on \eqref{intro:Gex1}, or even $G_1$, for the purpose of this introduction.

Our first main result deals with the Gaussian free field $\Phi$ on the weighted graph $(G,\lambda)$. Its canonical law $\P^G$ is the unique probability measure on $\R^G$ such that $(\Phi_x)_{x\in{G}}$ is a mean zero Gaussian field with covariance function 
\begin{equation}
\label{intro:GFF}
\E^G[\Phi_x\Phi_y]=g(x,y), \text{ for any $x,y\in{G}.$}
\end{equation}
On account of \eqref{intro:Volume+Green}, $\Phi$ exhibits (strong) algebraically decaying correlations with respect to the distance $d$, captured by the exponent 
\begin{equation}
\label{eq:nu}
\nu\stackrel{\text{def.}}{=}\alpha-\beta \ (>0).
\end{equation}
We  study the geometry of $\Phi$ in terms  
of its level sets
\begin{equation}
\label{E^h}
E^{\geq h}\stackrel{\text{def.}}{=}\{y\in{G};\,\Phi_x\geq h\},\  h\in{\R}.
\end{equation}
The random set $E^{\geq h}$ decomposes into connected components, also referred to as \textit{clusters}: two points belong to the same cluster of $E^{\geq h}$ if they can be joined by a path of edges whose endpoints all lie inside $E^{\geq h}$. Finite clusters are sometimes called \textit{islands}. 

As $h$ varies, the onset of a supercritical phase in $E^{\geq h}$ is characterized by a critical parameter $h_* = h_*(G)$, which records the emergence of infinite clusters,  \begin{equation}
\label{h_*}
    h_*\stackrel{\text{def.}}{=}\inf\left\{h\in{\R};\, \P^G\left(\text{there exists an infinite cluster in }E^{\geq h}\right)=0\right\}
\end{equation}
(with the convention $\inf \emptyset = \infty$).
The existence of a nontrivial phase transition, i.e., the statement $-\infty < h_* < \infty$, was initially investigated in \cite{MR914444}, and even in the case $G=G_1= \Z^d$ with $d\geq3,$ has only been completely resolved recently in \cite{MR3053773}. It was further shown in Corollary 2 of \cite{MR914444} that $h_*\geq0$ on $\Z^d$, and this proof can actually be adapted to any locally finite transient and connected weighted graph, see the Appendix of \cite{MR3765885}, or \cite{MR3502602} for a different proof.

Of particular interest are the connected components of $E^{\geq 0}$. The symmetry of $\Phi$ implies that $E^{\geq 0}$ and its complement in $G$ have the same distribution. The connected components of 
 $E^{\geq 0}$ and its complement are referred to as the positive and negative sign clusters of $\Phi,$ respectively.
It is an important problem to understand if these sign clusters fall into a \textit{supercritical} regime (below $h_*$), and, if so, what the resulting sign cluster geometry of $\Phi$ looks like. 
In order to formulate our results precisely, we introduce a critical parameter $\overbar{h}$ characterizing a regime of \textit{local uniqueness} for $E^{\geqslant h}$, whose distinctive features \eqref{eqhbar1} and \eqref{eqhbar2} below reflect (i) and (ii) in the discussion following \eqref{eq:MainGFFinformal}. Namely,
\begin{equation}
\label{hbardef}
	\overbar{h}=\sup\{h\in{\R};\,\Phi\text{ strongly percolates  above level }h'\text{ for all } h'<h\},
\end{equation}
with the convention $\sup\emptyset=-\infty$, where the Gaussian free field $\Phi$ is said to strongly percolate above level $h$ if there exist constants $c(h)>0$ and $C(h)<\infty$ such that for all $x\in{G}$ and $L\geqslant1,$
\begin{equation}
	\label{eqhbar1}
	\P^G\Big(E^{\geqslant h}\cap B(x,L)\text{ has no connected component with diameter at least }\frac{L}{5}\Big)\leq Ce^{-L^c}
\end{equation}
and
\begin{equation}
	\label{eqhbar2}
	\P^G\left(\begin{array}{c}\text{there exist connected components of }E^{\geqslant h}\cap B(x,L)\\\text{ with diameter at least }\frac{L}{10}\text{ which are not connected}\\ \text{in }E^{\geqslant h}\cap B(x,2\Cr{Cstrong}L)\end{array}\right)\leq Ce^{-L^c}
\end{equation}
(the constant $\Cr{Cstrong}$ is defined in \eqref{ballsalmostconnected} below). With the help of \eqref{eqhbar1}, \eqref{eqhbar2} and a Borel-Cantelli argument, one can easily patch up large clusters in $E^{\geq h} \cap B(x,2^k)$ for $k\geq 0$ and fixed $x\in{G}$ when $h< \overbar h$ to deduce that $\overbar{h}\leq h_*$. One also readily argues that for all $h<\overbar h$, there is a \textit{unique} infinite cluster in $E^{\geq h},$ as explained in \eqref{eq:unique} below.

 We will prove the following result, which makes \eqref{eq:MainGFFinformal} precise. For reference, conditions \eqref{p0}, \eqref{Ahlfors}, \eqref{Green} and \eqref{weakSecIso} appearing in \eqref{eq:mainGFF2} are defined in Section \ref{2}. All but \eqref{p0} depend on the choice of metric $d$ on $G$. 
 Following \eqref{intro:Volume+Green}, in assuming that conditions \eqref{Ahlfors}, \eqref{Green} and \eqref{weakSecIso} are met in various statements below, we understand \nolinebreak that 
 \begin{equation}
 \label{eq:condconvention}
 \begin{split}
&\text{\eqref{Ahlfors}, \eqref{Green} and \eqref{weakSecIso} hold with respect to \textit{some} distance function}\\
&\text{$d(\cdot,\cdot)$ on $G$, for \textit{some} values of $\alpha$ and $\beta$ satisfying $\alpha>2$ and $ \beta \in[2, \alpha).$}
\end{split}
 \end{equation}
\begin{The}
\label{T:mainresultGFF}
\begin{align}
&\text{If \eqref{p0}, \eqref{Ahlfors}, \eqref{Green} and \eqref{weakSecIso} hold, then $\overbar h>0$.} \label{eq:mainGFF2}
\end{align}
\end{The}
\smallskip
The proof of Theorem \ref{T:mainresultGFF} is given in Section \ref{denouement}. For a list of pertinent examples, see \eqref{intro:Gex1} and Section \ref{secpre}, notably Corollary \ref{C:examples} below, which implies that all conditions appearing in \eqref{eq:mainGFF2} hold true for the graphs listed in \eqref{intro:Gex1}. Some progress in the direction of Theorem \ref{T:mainresultGFF} was obtained in the recent work \cite{DrePreRod} by the authors, where it was shown that $h_*(\Z^d)>0$ for all $d\geq3$. Combining this with the sharpness result \cite{MR4568695}, which was initially released a couple of years after the first version of this article, one can deduce Theorem~\ref{T:mainresultGFF} on $\Z^d$, $d\geq3$, but not on other graphs satisfying \eqref{eq:condconvention}. The sole existence of an infinite sign cluster without proof of \eqref{eqhbar2} at small enough $h\geq 0$ can be obtained under slightly weaker assumptions, see condition \eqref{weakSecIso'} in Remark \ref{R:wsi'1} and Theorem \ref{thmconclusion} below. As an immediate consequence of \eqref{eqhbar1}, \eqref{eqhbar2} and \eqref{eq:mainGFF2}, we note that for all $h< \bar h$, and in particular when $h=0$, denoting by $\mathscr{C}^h(x)$ the cluster of $x$ in $E^{\geq h}$, 
\begin{equation}
\label{eq:diambound}
	\P^G\big( L \leq \, \text{diam}\big(\mathscr{C}^h(x)\big) < \infty \big)\leq Ce^{-L^c}.
\end{equation}
The parameter $\overbar{h},$ or a slight modification of it, see Remark \ref{lastremark}, \ref{lastremark1}) below, has already appeared when $G= \Z^d$ in \cite{MR3390739}, \cite{MR3417515}, \cite{MR3568036}, \cite{MR3650417}, \cite{MR3774433} and \cite{chiarini2018entropic} to test various geometric properties of the percolation cluster in $E^{\geqslant h}$ in the regime $h<\overbar{h};$ note that $\overbar{h}>-\infty$ is known to hold on $\Z^d$ as a consequence of Theorem 2.7 in \cite{MR3390739}, thus making these results not vacuously true, but little is known about $\overbar h$ otherwise. These findings can now be combined with Theorem \ref{T:mainresultGFF}.  
 For instance, as a consequence of \eqref{eq:mainGFF2} and Theorem 1.1 in \cite{MR3568036}, when $G=\Z^d$, denoting by $\mathscr{C}_{\infty}^{+}$ the infinite $+$-sign cluster,
\begin{equation}
\label{IP}
\begin{split}
&\text{$\P^G$-a.s., conditionally on starting in $\mathscr{C}_{\infty}^{+}$, the random walk on $\mathscr{C}_{\infty}^{+}$}\\
&\text{(see below (1.2) in \cite{MR3568036} for its definition) converges weakly to a}\\
&\text{non-degenerate Brownian motion under diffusive rescaling of space and time.}
\end{split}
\end{equation}
We refer to the above references for further results exhibiting, akin to \eqref{IP}, the ``well-behavedness'' of the phase $h<\overbar h$, to which the sign clusters belong.

\bigskip

We now introduce and state our results regarding \textit{random interlacements}, leading to Theorem \ref{T:mainresultRI} below, and explain their significance. As alluded to above, cf.\ also the discussion following Theorem \ref{T:mainresultRI} for further details, the interlacements, which constitute a Poisson cloud $\omega^u$ of bi-infinite random walk trajectories as in \eqref{Z:gen} modulo time-shift, were introduced on $\Z^d$ in \cite{MR2680403}, see also \cite{MR2525105} and Section \ref{2}, and naturally emerge due to their deep ties to $\Phi$. The parameter $u>0$ appears multiplicatively in the intensity measure of $\omega^u$ and hence governs how many trajectories enter the picture -- the larger $u$, the more trajectories. The law of the interlacement process $(\omega^u)_{u>0}$ is denoted by $\P^I$ and the random set $\I^u\subset G$, the interlacement set at level $u$, is the subset of vertices of $G$ visited by at least one trajectory in the support of $\omega^u$. Its complement $\V^u=G\setminus\I^u$ is called the {\em vacant set (at level $u$).} The process $\omega^u$ is also related to the loop-soup construction of \cite{zbMATH05911373}, if one ``closes the bi-infinite trajectories at infinity,'' as in \cite{MR2932978}.

Originally, $\omega^u$ was introduced in order to investigate the local limit of the trace left by simple random walk on large, locally 
 transient graphs $\{G_N ; \, N \geq 1\}$ with $G_N \nearrow G$ as $N \to \infty$, when run up to suitable timescales of the form $u \,  t_N$ with $u>0$ and  $t_N= t_N(G_N)$, see \cite{benjamini2008giant}, \cite{sznitman2008universal}, \cite{MR2520124}, \cite{MR2561432}, \cite{MR2838338}, as well as \cite{MR2386070} and \cite{MR3563197}.
  The trajectories in the support of $\omega^u$ can roughly be thought of as corresponding to successive excursions of the random walk in suitably chosen sets, and the timescale $t_N$ defines a Poissonian limiting regime for the occurrence of these excursions (note that this limit is hard to establish due to the long-range dependence between the excursions of the walk). Of particular interest in this context are the percolative properties of $\mathcal{V}^u$, as described by the  critical parameter (note that $\mathcal{V}^u$ is decreasing in $u$)
\begin{equation}
\label{u_*}
    u_* \stackrel{\text{def.}}{ =}\inf \big\{u\geq0;\, \P^I(\text{there exists an infinite connected component in }\mathcal{V}^u)=0 \big \}.
\end{equation}
This corresponds to a drastic change in the behavior of the complement of the trace of the walk on $G_N$, as the parameter $u$ appearing multiplicatively in front of $t_N$ varies across $u_*$, provided this threshold is non-trivial; see for instance \cite{MR2838338} for simulations when $G_N=(\Z/N\Z)^d$ with $t_N \asymp N^d$. The finiteness of $u_*$, i.e.\ the existence of a subcritical phase for $\mathcal{V}^u$, and even a phase of stretched exponential decay for the connectivity function of $\mathcal{V}^u$ at large values of $u$, can be obtained by adapting classical techniques, once certain decoupling inequalities are available. As a consequence of Theorem \ref{decoup} below, see Remark \nolinebreak\ref{R:applicor}, \ref{R:subcrit}) and Corollary \nolinebreak\ref{applidecoup}, such a phase is
 exhibited for any graph $G$ satisfying \eqref{p0}, \eqref{Ahlfors} and \eqref{Green} as in \eqref{eq:condconvention}. 

On the contrary, the existence of a supercritical phase is much less clear in general. It was proved in \cite{MR2891880} that $u_* >0$ for graphs of the type $G=G'\times\Z,$ endowed with some distance $d$ such that \eqref{intro:Volume+Green} holds, see (1.8) and (1.11) in \cite{MR2891880} (note that the exponent $\alpha$ from \eqref{intro:Volume+Green} actually corresponds to $\alpha+\beta/2$ in \cite{MR2891880}, see (1.9) therein). However, in this source the condition $\nu \geq1$ was required, cf.\ \eqref{eq:nu}, excluding for instance the case $G=G_2$ in which $\nu=\frac{\log9-\log5}{\log4}<1,$ see \cite{MR1378848} and \nolinebreak\cite{MR1668115}, as well as the case $G=G_3$ in dimension three (which was anyway not of the type $G'\times\Z$), see Remark~\ref{R:extensions}, \ref{R:boundsierpinskicarpet}).  As a consequence of the following result, we settle the question about the positivity of $u_*$ affirmatively under our assumptions. This solves a principal open problem from \nolinebreak\cite{MR2891880}, see Remark 5.6(2) therein, and implies the existence of a phase transition for the percolation of the vacant set $\mathcal V^u$ of random interlacements on such graphs. We remind the reader of the convention \eqref{eq:condconvention} regarding conditions \eqref{Ahlfors}, \eqref{Green} and \eqref{weakSecIso}, which is in force in the following:

\begin{The}
\label{T:mainresultRI}
Suppose $G$ satisfies \eqref{p0}, \eqref{Ahlfors}, \eqref{Green} and \eqref{weakSecIso}. Then there exists $\tilde{u} >0$ and for every $u \in (0, \tilde u]$, a probability space $(\Omega^u, \mathcal{F}^u, Q^u)$ governing three random subsets $\mathcal I$, $\mathcal V$ and $\mathcal K$ of $G$ with the following properties:
\begin{equation}
\begin{split}
\label{eq:mainresultRI}
&\text{i) $\mathcal I$, resp. $\mathcal V$, have the law of $\mathcal{I}^u$, resp. $\mathcal{V}^u$, under $\mathbb{P}^I$}.\\
&\text{ii) $\mathcal{K}$ is independent of $\mathcal{I}$.}\\
&\text{iii) $Q^u$-a.s., $\mathcal I \cap \mathcal{K} $ contains an infinite cluster, and $(\mathcal I \cap \mathcal{K}) \subset \mathcal V $.}
\end{split}
\end{equation}
A fortiori, $u_* \geq \tilde u (>0)$.
\end{The}

Thus, our construction of an infinite cluster of $\mathcal{V}^u$ for small $u>0$, and hence our resolution of the conjecture in \cite{MR2891880}, proceeds by stochastically embedding a large part of its complement, $\mathcal{I}^u  \cap \mathcal{K} $ inside $\mathcal{V}^u$. The law of the set $\mathcal{K}$ can be given explicitly, see Remark \ref{lastremark}, \ref{remarkchoicemathcalK}), and $\mathcal{K}$ could also be chosen independent of $\V$ instead of $\I,$ see Remark \ref{lastremark}, \ref{complement}). 

Let us now elaborate shortly on the important case $G=G'\times \Z$ considered in \cite{MR2891880}. In this setting, the conclusions of Theorem \ref{T:mainresultRI} hold under the mere assumptions that \eqref{p0} holds and $G'$ satisfies the upper and lower heat kernel estimates \eqref{UE} and \eqref{LE}, see Remark \ref{T:HK}, with respect to $d=d_{G'}$, the graph distance on $G'$, for some $\alpha > 1$ and $ \beta \in[2, 1+\alpha);$ for instance, if $G=G_2$ from \eqref{intro:Gex1}, then $\alpha=\frac{\log 3}{\log 2}$ and $\beta=\frac{\log 5}{\log 2}$, see \cite{MR2177164, MR1378848}. This (and more) will follow from Propositions \ref{P:product} and \ref{generalstarconnected} below; see also Remark \nolinebreak\ref{R:extensions} for further examples. Incidentally, let us note that Theorem \ref{T:mainresultRI} is also expected to provide further insights into the disconnection of cylinders $G_N \times \Z$ by a simple random walk trace, for $G_N$ a large finite graph, for instance when $G_N$ is a ball of radius $N$ in the discrete skeleton of the Sierpinski gasket (corresponding to  $G_2$ of \eqref{intro:Gex1}), cf. Remark \nolinebreak5.1 in \cite{sznitman2008universal}. 

Theorem~\ref{T:mainresultRI} is true on more general product graphs $G=G'\times G''$, see Section~\ref{secpre}, but in fact also on graphs which are not product graphs, such as the graphs $G_3$ and $G_4$ from \eqref{intro:Gex1}. Our proof still requires certain useful geometric features which are trivially true on $G'\times \Z$ and were crucial in \cite{MR2891880} also. We show that in the case of graphs which are not product graphs, these geometric features can still be derived from the general assumptions \eqref{p0}, \eqref{Ahlfors}, \eqref{Green} and \eqref{weakSecIso}. For instance, we show that there are no large bottlenecks in the graph, see the proof of Lemma~\ref{Lemmaharnack}, which is useful to obtain certain elliptic Harnack inequalities on annuli reminiscent of Lemma~2.3 in \cite{MR2891880}.

Since Theorem \ref{T:mainresultRI} builds on the arguments leading to Theorem \nolinebreak\ref{T:mainresultGFF}, we delay further remarks concerning \eqref{eq:mainresultRI} for a few lines, and first provide an overview of the proof of Theorem \nolinebreak\ref{T:mainresultGFF}.

As hinted at above, a key ingredient and the starting point of the proof of Theorem \nolinebreak\ref{T:mainresultGFF} is a certain isomorphism theorem, see \cite{MR2892408}, \cite{MR3502602}, \cite{MR3492939} and \eqref{Isomorphism} and Corollary \ref{phiisagff} below, which links the free field $\Phi$ to the interlacement  $\omega^u$. The argument unfolds by first studying the random set $\mathcal I^u$, which has remarkable connectivity properties: even though its density tends to $0$ as $u \downarrow 0$, $\mathcal I^u$ is an \textit{unbounded connected} set for \textit{every} $u>0$.
 Much more is in fact true, see Section \ref{secconnec}, in particular Proposition \ref{connectivity} below, the set $\mathcal I^u$ is actually \textit{locally} well-connected. These features of $\mathcal I^u$, especially for $u$ close to $0$, will figure prominently in our construction of various large random sets, and ultimately serve as an indispensable tool to build percolating sign clusters. 
Indeed, as a consequence of the aforementioned correspondence between $\Phi$ and $\omega^u$, see also \eqref{Iuincluded} below, one can use $\I^u$ in a first step as a system of ``highways'' to produce connections inside $E^{\geq -h}$, for every $h =\sqrt{2u} >0$. 

A substantial part of these connections persists to exist in $\tilde{E}^{\geq - h}$ $(h>0)$, the level sets of the free field $\tilde{\varphi}$ on a continuous extension $\tilde G$ of the graph, the associated \textit{cable system}. This object, to which all above processes can naturally be extended, goes back at least to \cite{MR737386} and is obtained by replacing the edges between vertices
by one-dimensional cables. 
This result, which quantifies and strengthens the early insight $h_*(\Z^d) \geq 0$ of \cite{MR914444} 
-- deduced therein by a soft but indirect and general argument -- is in fact sharp on the cables, see Theorem \ref{T:cablessharp} below. Importantly, the recent result of \cite{MR3492939}, which can be applied in our framework, see Corollary \ref{phiisagff}, further allows to formulate a condition in terms of an (auxiliary) Gaussian free field $\tilde\gamma$ appearing in the isomorphism and $\tilde{\I}^u$, the continuous interlacement set, for points in $\tilde{E}^{\geq - h}$ to ``rapidly'' (i.e.\ at scale $L_0$ in the renormalization argument detailed in the next paragraph) connect to the interlacement $\tilde{\I}^{u=h^2/2}$. Following ideas from our precursor work \cite{DrePreRod}, we can then rely on a certain robustness property exhibited on the cables to pass from $\tilde{E}^{\geq - h}$ to ${E}^{\geq + h}$ by means of a suitable coupling, which operates independently at any given vertex when certain favorable conditions are met. These conditions in turn become typical as $u \to 0^+$, see Lemma \ref{lemmaflip} and Proposition \ref{iuincluvu}.

The previous observations can be combined into a set of \textit{good features}, assembled in Definition \ref{defgood} below, which are both increasingly likely as $L_0 \to \infty$ and entirely local, in that all properties constituting a good vertex $x\in G$ are phrased in terms of the various fields inside balls of radius $\approx L_0$ in the distance $d$ around $x$. This notion can then be used as the starting point of a renormalization argument, presented in Sections \ref{S:renorm} and \ref{sectionpercsigncluster}, to show that good regions form large connected components. Importantly, with a view towards \eqref{eqhbar1} and \eqref{eqhbar2}, good regions need not only to \textit{form} but do so \textit{everywhere} inside of $ G$. This comes under the proviso of \eqref{weakSecIso} as a feature of the renormalization scheme, which ensures that subsets of $G$ having large diameter are typically connected by paths of good vertices, see Lemmas \ref{largecompareconnbygood} and \ref{proofofhbartilde>0} below. Using additional randomness, the connection by paths of good vertices is turned into a connection by paths in $E^{\geq h},$ and this completes the proof of Theorem \ref{T:mainresultGFF}, see Section \ref{denouement}. 

A renormalization of the parameters involved in the scheme is necessary due to the presence of the strong correlations, and it relies on suitable decoupling inequalities, see Theorem \ref{decoup} below. At the level of generality considered here, namely assuming only \eqref{p0}, \eqref{Ahlfors}, \eqref{Green}, and particularly in the case of ${\I}^u$, see \eqref{decouplingRI}, these inequalities generalize results of \cite{MR2891880} and are interesting in their own right. At the technical level, they are eventually obtained from the soft local time technique introduced in \cite{MR3420516} and developed therein on $\Z^d$. The  difficulty stems from having to control the resulting error term, which is key in obtaining  \eqref{decouplingRI}. This control ultimately rests on chaining arguments and a suitable elliptic Harnack inequality, see in particular Lemmas \ref{Lemmaharnack} and \ref{beforedecoupRI}, which provides good bounds if certain sets of interest do not get too close (note that, due to their Euclidean nature, the arguments leading to the precise controls of \cite{MR3420516} valid even at short distances seem out of reach within the current setup). Fortunately, this is good enough for the purposes we have in mind.

The proof of Theorem \ref{T:mainresultRI} then proceeds by using the results leading to Theorem \ref{T:mainresultGFF} and adding one more application of the coupling provided in Corollary  \ref{phiisagff}. Indeed, the above steps essentially allow to roughly translate the probabilities in \eqref{eqhbar1} and \eqref{eqhbar2} regarding $E^{\geq h}$, for $h >0$ in terms of the interlacement $\I^u$, for $u=h^2/2$ and some ``noise'', see Lemma \ref{percolfora} and (the proof of) Lemma  \ref{proofofhbartilde>0}, but $E^{\geq h}$ is in turn naturally embedded into $\mathcal{V}^u$, see \eqref{Iuincluded}. Following how the percolative regime for $\mathcal{V}^u$ is obtained, one thus starts with its complement $\I^u$, first passes to $\Phi$ and proves the phase coexistence regime around $h=0$ asserted in Theorem \ref{T:mainresultGFF}, and then translates back to $\mathcal{V}^u$. The existence of the phase coexistence regime along with the symmetry of $\Phi$ is then ultimately responsible for producing the inclusion $iii)$ in \eqref{eq:mainresultRI}. The set $\mathcal{K}$ appearing there morally corresponds to all the undesired noise produced by bad regions in the argument leading to Theorem \nolinebreak\ref{T:mainresultGFF}. It would be interesting to devise a direct argument for $u_*>0$ which by-passes the use of $\Phi$. We are currently unable to do so, except when $\nu > 1$, in which case the reasoning of \cite{MR2891880} can be adapted, see Remark \ref{R:applicor}, \ref{supernu>1}). We refer to Remark \ref{lastremark}, \ref{percoplanes})--\ref{betterunderstanding}) for further open questions.

\medskip
We now describe how this article is organized. Section \ref{2} introduces the precise framework, the processes of interest and, importantly, the conditions \eqref{p0}, \eqref{Ahlfors}, \eqref{Green} and \eqref{weakSecIso} appearing in our main results. We then collect some first consequences of this setup. The decoupling inequalities mentioned above are stated in Theorem \ref{decoup} at the end of that section. 

Section \ref{secpre} has two main purposes. After gathering some preliminary tools from harmonic analysis (for the operator $L$ in \eqref{Z:gen}), which are used throughout, we first discuss in Proposition \nolinebreak\ref{P:product}
how \eqref{Ahlfors}, \eqref{Green} are obtained for product graphs of the form $G= G'\times G''$, when the factors satisfy suitable heat kernel estimates. This has important applications, notably to the graph $G=G_2$ in \eqref{intro:Gex1}, and requires that we work with general distances $d$ in conditions  \eqref{Ahlfors}, \eqref{Green}. For this reason, we have also included a proof of the classical (in case $d=d_G$, the graph distance) estimates of Proposition \nolinebreak\ref{someproperties} in the appendix. The second main result of Section \nolinebreak\ref{secpre} is to deduce in Corollary \nolinebreak\ref{C:examples} that the relevant conditions \eqref{p0}, \eqref{Ahlfors}, \eqref{Green} and \eqref{weakSecIso} appearing in Theorems \nolinebreak\ref{T:mainresultGFF} and \nolinebreak\ref{T:mainresultRI} apply in all cases of \eqref{intro:Gex1}. In addition to Proposition \nolinebreak\ref{P:product}, this requires proving \eqref{weakSecIso} and dealing with boundary connectivity properties of connected sets, which is the object of Proposition \nolinebreak\ref{generalstarconnected}. 

Section \ref{secconnec} collects the local connectivity properties of the continuous interlacement set $\tilde\I^u$, see Proposition \ref{connectivity} and Corollary \ref{RIconnected}. The overall strategy is similar to what was done in \cite{MR2819660} on $\Z^d$, see also \cite{DrePreRod}, to which we frequently refer. The proof of Proposition \nolinebreak\ref{connectivity} could be omitted on first reading.

Section \ref{seclevelsettilde} is centered around the isomorphism on the cables. The main takeaway for later purposes is Corollary \ref{phiisagff}, see also Remark \ref{R:iso}, which asserts that the coupling of Theorem 2.4 in \cite{MR3492939} can be constructed in our framework. This requires that certain conditions be met, which are shown in Lemma \ref{unicityinfinitecluster} and Proposition \ref{nonpercolationsignclusterscable}. The latter also yields the desired inclusion \eqref{Iuincluded}. The generic absence of ergodicity makes the verification of these properties somewhat cumbersome. Lemma \ref{lemmaflip} contains the adaptation of the sign-flipping argument from \cite{DrePreRod}, from which certain desirable couplings needed later on in the renormalization are derived in Proposition \ref{iuincluvu}. Section \ref{seclevelsettilde} closes with a more detailed overview over the last four sections, leading to the proofs of our main results.

Section \ref{secdecoup} is devoted to the proof of Theorem \ref{decoup}, which contains the decoupling inequalities. While the free field can readily be dispensed with by adapting results of \cite{MR3325312}, the interlacements are more difficult to deal with. We apply the soft local times technique from \cite{MR3420516}. All the work lies in controlling a corresponding error term, see Lemma \ref{boundonsoftlocaltimes}. The regularity estimates for hitting probabilities needed in this context, see the proof of Lemma \ref{beforedecoupRI}, rely on Harnack's inequality, see Lemma \ref{Lemmaharnack} for a tailored version.

Section \ref{S:renorm} introduces the renormalization scheme needed to put together the ingredients of the proof, which uses the decoupling inequalities of Theorem \ref{decoup}. The important Definition \ref{defgood} of good vertices appears at the end of that section, and Lemma \ref{whygood} collects the features of good long paths, which are later relied upon. The good properties appearing in this context are expressed in terms of (an extension of) the coupling from Corollary \ref{phiisagff}.

Section \ref{sectionpercsigncluster} takes advantage of the renormalization scheme introduced in Section \ref{S:renorm} to create a giant and ubiquitous cluster of good vertices, and of random interlacements with suitable properties. Proposition \ref{Rpathofbad} first yields the desired estimate that long paths of bad vertices are very unlikely, for suitable choices of the parameters. Lemmas \ref{percolfora} and \ref{proofofhbartilde>0} provide precursor estimates to \eqref{eqhbar1} and \eqref{eqhbar2}, which are naturally associated to our notion of goodness. In particular, Lemma \ref{proofofhbartilde>0} directly implies that $\overline{h}\geq0$ as a first step toward Theorem \ref{T:mainresultGFF}, see Corollary \ref{cor:hbargeq0}. An important technical step with regards to Lemma \ref{proofofhbartilde>0} is Lemma \ref{largecompareconnbygood}, which asserts that large sets in diameter are typically connected by a path of good vertices.

The pieces are put together in Section \ref{denouement}, and the proofs of Theorems \ref{T:mainresultGFF} and \ref{T:mainresultRI} appear towards the end of this last section. Proposition \ref{iuincluvu} exhibits the coupling transforming (for instance) giant good regions from Lemma \ref{proofofhbartilde>0} into giant subsets of ${E}^{\geq h},$ $h>0,$ see Lemma \ref{giant}, from which \eqref{eqhbar1} and \eqref{eqhbar2} are eventually inferred. Finally, Section \ref{denouement} also contains the simpler existence result, Theorem \ref{thmconclusion}, alluded to above, which can be obtained under a slightly weaker condition \eqref{weakSecIso'}, introduced in Remark \ref{R:wsi'1}.

\medskip

We conclude this introduction with our convention regarding constants. In the rest of this article, we denote by $c,c',\dots$ and $C,C',\dots$ positive constants changing from place to place. Numbered constants $c_0,$ $C_0,$ $c_1,$ $C_1,\dots$ are fixed when they first appear and do so in increasing numerical order. All constants may depend implicitly ``on the graph $G$'' through conditions \eqref{p0}, \eqref{Ahlfors} and \eqref{Green} below, in particular they may depend on $\alpha$ and $\beta$. Their dependence on any other quantity will be made explicit. 

\medskip
For the reader's orientation, we emphasize that the conditions \eqref{p0}, \eqref{Ahlfors}, \eqref{Green} and \eqref{weakSecIso}, which will be frequently referred to, are all introduced in Section \ref{2}. We seize this opportunity to highlight the set of assumptions \eqref{eq:Ass} on $(G,\lambda)$ appearing at the beginning of Section \ref{secpre}, which will be in force from then on until the end.

\section{Basic setup and first properties}
\label{2}
In this section, we introduce the precise framework alluded to in the introduction, formulate the assumptions appearing in Theorems \ref{T:mainresultGFF} and \ref{T:mainresultRI}, and collect some of the basic geometric features of our setup. We also recall the definitions and several useful facts concerning the two protagonists, random interlacements and the Gaussian free field on \nolinebreak$G$, as well as their counterparts on the cable system. We then state in Theorem \ref{decoup} the relevant decoupling inequalities for both interlacements and the free field, which will be proved in Section \ref{secdecoup}. 

Let $(G,E)$ be a countably infinite, locally finite and connected graph with vertex set $G$ and (unoriented) edge set $E \subset G \times G$. We will often tacitly identify the graph $(G,E)$ with its vertex set $G$. We write $x\sim y,$ or $y\sim x,$ if $\{x,y\} \in E$, i.e., if $x$ and $y$ are connected by an edge in $G$. Such vertices $x$ and $y$ will be called \textit{neighbors}. We also say that two edges in $E$ are neighbors if they have a common vertex. A \textit{path} is a sequence of neighboring vertices in $G$, finite or infinite. For $A\subset G,$ we set $A^{\mathsf{c}} = G \setminus A$, we write $\partial A = \{y\in A;\,\exists\,z\in{A^{\mathsf{c}}},\,z\sim y\}$ for its inner boundary, and define the external boundary  of $A$ by 
\begin{equation}
\label{eq:bext}
\partial_{ext} A\stackrel{\mathrm{def.}}{=}
\{y\in{A^{\mathsf{c}}};\,\exists\text{ an unbounded path in }A^{\mathsf{c}}\text{ beginning in $y$ and }\exists\,z\in{A},\,z\sim y\}
\end{equation}
We write $x\leftrightarrow y$ in $A$ (or $x\stackrel{A}{\longleftrightarrow}y$ in short) if there exists a nearest-neighbor path in $A$ containing $x$ and $y$, 
 and we say that $A$ is \textit{connected} if $x\stackrel{A}{\longleftrightarrow}y$ for any $x,y\in{A}.$ For all $A_1\subset A_2\subset G,$ we write $A_1\subset\subset A_2$ to express that $A_1$ is a finite subset of $A_2.$ We endow $G$ with a non-negative and symmetric weight function $\lambda=(\lambda_{x,y})_{x, y\in{G}},$ such that $\lambda_{x,y} \geq 0$ for all $x,y \in G$ and $\lambda_{x,y} > 0$ if and only if $\{x,y\} \in E$. We define the weight of a vertex $x\in{G}$ and of a set $A\subset G$ by $   \lambda_x\stackrel{\mathrm{}}{=}\sum_{y\sim x}\lambda_{x,y}$  and $\lambda(A)\stackrel{\mathrm{}}{=}\sum_{x\in{A}}\lambda_x$. We often regard $\{ \lambda_x: \, x \in G\}$ as a positive measure on $G$ endowed with its power set $\sigma$-algebra in the sequel.

To the weighted graph $(G,\lambda)$, we associate the discrete-time Markov chain with transition probabilities
\begin{equation}\label{transitionprobability}
    p_{x,y}\stackrel{\mathrm{def.}}{=}\frac{\lambda_{x,y}}{\lambda_x}, \quad \text{for }x,y\in G.
\end{equation}
 We write $ P_x$, $x\in{G}$, for the canonical law of this chain started at $x$, and $Z=(Z_n)_{n\geq0}$ for the corresponding canonical coordinates. For a finite measure $\mu$ on $G,$ we also set
\begin{equation}\label{Pmu}
    P_{\mu}\stackrel{\mathrm{def.}}{=}\sum_{x\in{G}}\mu(x) P_x.
\end{equation}
Our assumptions, see in particular \eqref{Green} below, will ensure that $Z$ is in fact transient. We assume that $G$ has controlled weights, i.e., there exists a constant $c_0$ such that
\begin{equation}
\tag{$p_0$}
\label{p0}
    p_{x,y}\geq c_0\text{ for all }x\sim y\in{G}.
\end{equation}
Note that \eqref{p0} implies that each $x\in{G}$ has at most $\left\lfloor1/c_0\right\rfloor$ neighbors, so $G$ has uniformly bounded degree. 

We introduce the symmetric Green function associated to $Z$,
\begin{equation}\label{eq:Greendef}
    g(x,y)\stackrel{\mathrm{def.}}{=}\frac{1}{\lambda_y}E_x\Big[\sum_{k=0}^\infty 1_{\{Z_k=y\}}\Big]\text{ for all } x,y\in{G}.
\end{equation}
For $A\subset G,$ we let $T_A\stackrel{\mathrm{def.}}{=}\inf\{k\geq0;\,Z_k\notin A\},$ the first exit time of $A$ and $H_A\stackrel{\mathrm{def.}}{=} T_{A^{\mathsf c}} = \inf\{k\geq0;\,Z_k\in{A}\}$ the first entrance time in $A,$ and introduce the killed Green function
\begin{equation}\label{Greenstopped}
    g_A(x,y)\stackrel{\mathrm{def.}}{=}\frac{1}{\lambda_y}E_x\Big[\sum_{k=0}^{T_A-1}1_{\{Z_k=y\}}\Big]\text{ for all } x,y\in{A}.
\end{equation}
Applying the strong Markov property at time $T_A$ for $A\subset\subset G,$ we obtain the relation
\begin{equation}\label{Greenmarkov}
   E_x[g(Z_{T_A},y)] + g_A(x,y)=g(x,y), \text{ for all $x,y\in{A}$}.
\end{equation}
Finally, the heat kernel of $Z$ is defined as
\begin{equation}
\label{heatkernel}
    p_n(x,y)\stackrel{\mathrm{}}{=}\lambda_y^{-1} P_x(Z_n=y)\text{ for all } x,y\in{G}\text{ and }n\in{\N}.
\end{equation}
We further assume that $G$ is endowed with a distance function $d.$ 

\begin{Rk}
\label{R:distance} A natural choice is $d=d_G$, the graph distance on $G$, but this does not always fit our needs. We will return to this point in the next section. Roughly speaking, some care is needed due to our interest in product graphs such as $G_2$ in \eqref{intro:Gex1}, and more generally graphs of the type $G=G'\times \Z$ as in \cite{MR2891880}. This is related to the way by which conditions \eqref{Ahlfors} and \eqref{Green} below propagate to a product graph, especially in cases where the factors have different diffusive scalings, see Proposition \ref{P:product} and in particular \eqref{prod11} below. These choices of general graph and distance generate a few technical difficulties that we will solve along the way. For instance, balls might not be connected, but every two points in the ball can be connected within some neighborhood of the ball in view of \eqref{ballsalmostconnected}. Furthermore, there could be bottlenecks in the graph which make it hard to create local connections between points, but they actually disappear when considering large enough regions, see the proof of Lemma~\ref{Lemmaharnack}. 
\end{Rk}
\bigskip
We denote by $B(x,L){=}\{y\in{G}:\,d(x,y)\leq L\}$ the closed ball of center $x$ and radius $L$ for the distance $d$ 
and by $B_E(x,L)$ the set of edges for which both endpoints are in $B(x,L)$. For all $A\subset G$ we write $d(A,x)=\inf_{y\in{A}}d(y,x)$ for the distance between $A\subset G$ and $x\in{G},$ $B(A,L)\stackrel{\mathrm{def.}}{=}\{y\in{G}:\,d(A,y)\leq L\}$ is the closed $L$-neighborhood of $A$, and if $A \neq \emptyset$ we write $\delta(A)\stackrel{\mathrm{def.}}{=}\sup_{x,y\in{A}}d(x,y)\in{[0,\infty]}$ for the diameter of $A.$ Note that unless $d=d_G$, balls in the distance $d$ are not necessarily connected in the sense defined below \eqref{eq:bext}.
 
We now introduce two -- natural, see Remark \ref{T:HK} below -- assumptions on $(G,\lambda)$, one geometric and the other analytic. We suppose that $G$ has regular volume growth of degree $\alpha$ with respect to $d$, that is, there exists $\alpha>2$ and constants $0<\Cl[c]{cAhlfors}\leq\Cl{CAhlfors}<\infty$ such that 
\begin{equation}
\label{Ahlfors}
\tag{$V_{\alpha}$}
    \Cr{cAhlfors}L^{\alpha}\leq \lambda\big(B(x,L)\big)\leq \Cr{CAhlfors} L^{\alpha},\text{ for all }x\in{G}\text{ and }L\geq1.
\end{equation} 
We also assume that the Green function $g$ has the following decay: there exist constants $0<\Cl[c]{cGreen}\leq\Cl{CGreen}<\infty$ such that, with $\alpha$ as in \eqref{Ahlfors}, for some 
$\beta \in [2, \alpha)$, $g$ satisfies
\begin{equation}
\label{Green}
\tag{$G_{\beta}$}
\begin{split}
&\Cr{cGreen}\leq g(x,x)\leq \Cr{CGreen}\text{ for all }x \in{G} \text{ and }  \\
&  \Cr{cGreen}d(x,y)^{-\nu}\leq g(x,y)\leq \Cr{CGreen}d(x,y)^{-\nu}\text{ for all }x \neq y\in{G},
\end{split}
\end{equation}
where we recall that $\nu=\alpha-\beta$ from \eqref{eq:nu}.
The parameter $\beta \geq 2$ in \eqref{eq:nu} can be thought of as characterizing the order of the mean exit time from balls (of radius $L$), which grows like $L^{\beta}$ as $L \to \infty$, see Lemma \ref{expectedexittime}.

\begin{Rk}[Equivalence to heat kernel bounds]  \label{T:HK}
The above assumptions are very natural. Indeed, in case $d(\cdot,\cdot)$ is the graph distance -- but see Remark \ref{R:distance} above -- the results of \cite{MR1853353}, see in particular Theorem 2.1 therein, assert that, assuming $\eqref{p0}$, the conditions \eqref{Ahlfors} and \eqref{Green} are equivalent to the following sub-Gaussian estimates on the heat kernel: for all $x,y \in G$ and $n \geq 1$,
\begin{equation}
\label{UE}
\tag{UHK$(\alpha,\beta)$}
    p_n(x,y)\leq Cn^{-\frac{\alpha}{\beta}}\exp\bigg\{-\bigg(\frac{d(x,y)^{\beta}}{Cn}\bigg)^{\frac{1}{\beta-1}}\bigg\}
\end{equation}
and, if $n\geq d_G(x,y)\vee1$, 
\begin{equation}
\label{LE}
\tag{LHK$(\alpha,\beta)$}
        p_n(x,y)+p_{n+1}(x,y)\geq cn^{-\frac{\alpha}{\beta}}\exp\bigg\{-\bigg(\frac{d(x,y)^{\beta}}{cn}\bigg)^{\frac{1}{\beta-1}}\bigg\}.
\end{equation}
Many examples of graphs $G$ for which \eqref{UE} and \eqref{LE} hold for the graph distance are given in \cite{MR1378848}, \cite{MR1701339} and \cite{MR2112125}, and further characterizations of these estimates can be found in \cite{MR1938457}, \cite{MR2076770}, \cite{MR2177164} and \cite{MR3616731}.
We will return to the consequences of \eqref{Ahlfors}, \eqref{Green}, and their relation to estimates of the above kind within our framework, i.e., for general distance function $d$, in Section \nolinebreak\ref{secpre}, cf. Proposition \ref{someproperties} and Remark \ref{remarkendofsection3} below.
 \end{Rk}
  \bigskip
We now collect some simple geometric consequences of the above setup. 
We seize the opportunity to recall our convention regarding constants at the end of Section \ref{section1}.

\begin{Lemme} Assume $\eqref{p0}$, \eqref{Ahlfors}, and \eqref{Green} to be fulfilled. Then:
    \begin{align}
    \label{conditiondistance}
  &  d(x,y)\leq  \Cl{Cdistance} d_G(x,y)\text{ for all }x,y\in{G},\\[0.2em]
&    \label{conditiondistance2}
        d(x,y)\geq\Cl[c]{cdistance}\text{ for all }x\neq y\in{G},\\[0.2em]
&\label{lambda}
    \Cl[c]{cweight}\leq\lambda_{x,y}\leq\lambda_x\leq\Cl{Cweight}\text{ for all }x\sim y\in{G}.
\end{align}
\end{Lemme}
\begin{proof} We first show \eqref{conditiondistance}. Using \eqref{p0}, \eqref{Green}, and the strong Markov property at time $H_y\stackrel{\mathrm{def.}}{=}H_{\{y\}}$, for all $x\sim y\in{G}
$ we have
    \begin{equation*}
        g(x,y)= P_x(H_y<\infty)g(y,y)\geq p_{x,y}g(y,y)\geq c_0\Cr{cGreen},
         \end{equation*}
    where $p_{x,y}$ is the transition probability between $x$ and $y$ for the random walk $Z,$ see \eqref{transitionprobability}. Thus, one can find $\Cr{Cdistance}$ such that
    \begin{equation}
    \label{nnest1}
       d(x,y)\stackrel{\eqref{Green}}{\leq} \left(\frac{g(x,y)}{\Cr{CGreen}}\right)^{-\frac1\nu}\leq \Cr{Cdistance}\text{ for all }x\sim y\in{G}.
    \end{equation}
 For arbitrary $x$ and $y$ in $G,$ we then consider a geodesic for the graph distance between \nolinebreak$x$ and \nolinebreak$y$, apply the triangle inequality (for $d$) and use \eqref{nnest1} repeatedly to deduce \eqref{conditiondistance}. Similarly, for all $x\neq y\in{G},$
 \begin{equation*}
 d(x,y)\stackrel{\eqref{Green}}{\geq} \left(\frac{g(x,y)}{\Cr{cGreen}}\right)^{-\frac1\nu}\stackrel{\eqref{Green}}{\geq}\left(\frac{\Cr{CGreen}}{\Cr{cGreen}}\right)^{-\frac1\nu}\stackrel{\mathrm{def.}}{=}\Cr{cdistance}.
 \end{equation*}
 
We now turn to \eqref{lambda}. For $x\sim y\in{G},$ we have $x\in{B(x,1)}$ and thus, by \eqref{Ahlfors}, $\lambda_{x,y}\leq\lambda_x\leq\Cr{CAhlfors}\stackrel{\mathrm{def.}}{=}\Cr{Cweight}.$ Moreover, $g(x,x)\geq\lambda_x^{-1}$ by definition, and thus by \eqref{p0} and \eqref{Green},
        \begin{equation*}
            \lambda_{x,y}\geq c_0\lambda_x\geq\frac{c_0}{g(x,x)}\geq\frac{c_0}{\Cr{CGreen}}\stackrel{\mathrm{def.}}{=}{\Cr{cweight}}.
        \end{equation*}
      \end{proof}

We now define the weak sectional isoperimetric condition alluded to in Section \ref{section1}. This is an additional condition on the geometry of $G$ 
that will enter in Section \nolinebreak\ref{sectionpercsigncluster} to guarantee that certain ``bad'' regions are sizeable and thus costly in terms of probability, cf.\@ the proofs of Lemma \ref{percolfora} and Lemma \ref{largecompareconnbygood}. We say that $(x_1,\dots,x_n)$ is an \textit{$R$-path from $x$ to $B(x,N)^{\mathsf{c}}$} if $x_1=x,$  $x_n\in{B(x,N)^{\mathsf{c}}},$ and $d(x_i,x_{i+1})\leq R$ for all $i\in{\{1,\dots,n-1\}},$ with the additional convention that $(x_1)$ is an $R$-path from $x$ to $B(x,N)^{\mathsf{c}}$ if $N\leqslant R.$ The weak sectional isoperimetric condition is a condition on the existence of a long $R$-path in the boundary of sets, and similar conditions have already been used to study Bernoulli percolation, see \cite{MR2054174}. More precisely, this weak sectional isoperimetric condition states that there exists $R_0\geq1$ and $\Cl[c]{cwsi2}\in{(0,1)}$ such that \vphantom{$\Cl{Cwsi2}$}
\begin{equation}
\tag{WSI}\label{weakSecIso}
\begin{split}
&\text{for each finite connected subset $A$ of $G$ and all }x\in{\partial_{ext}{A}},
\\
&\text{there exists an $R_0$-path from $x$ to $ B(x,\Cr{cwsi2}\delta(A))^{\mathsf{c}}$ in $\partial_{ext}A$.}
\end{split}
\end{equation}

\medskip 
We now introduce the processes of interest. For each $x\in{G},$ we denote by $\Phi_x$ the coordinate map on $\R^{G}$ endowed with its canonical $\sigma$-algebra, $\Phi_x(\omega)=\omega_x$ for all $\omega\in{\R^{G}},$ and $\P^G$ is the probability measure 
defined in \eqref{intro:GFF}. Any process $(\phi_x)_{x\in{G}}$ with law $\P^G$ will be called a {\em Gaussian free field on $G;$} see \cite{MR2322706} as well as the references therein for a rigorous introduction to the relevance of this process. Recalling the definition of the level sets $E^{\geq h}$ of $\Phi$ in \eqref{E^h} and of the parameter $\overbar h$ from \eqref{hbardef}, we now provide a simple argument that
\begin{equation}
\label{eq:unique}
\text{for each $h< \overbar h$, $\P^G$-a.s., $E^{\geq h}$ contains a unique infinite cluster.}
\end{equation}
Indeed, on the event $A_L^h=\{B(x,L/2)\text{ intersects at least two infinite clusters of }E^{\geq h}\},$ for some fixed $x\in{G}$, if $L$ is large enough, there are at least two clusters of $E^{\geq h}\cap B(x,L)$ with diameter at least $L/10$ which are not connected in $G,$ and thus the event in \eqref{eqhbar2} occurs. The events $A_L^h$ are increasing toward $\{E^{\geq h}\text{ has at least two infinite clusters}\}$ as $L$ goes to infinity, and thus by \eqref{eqhbar2} $E^{\geq h}$ contains $\P^G$-a.s.\ at most one infinite cluster for all $h<\overbar{h},$ and \eqref{eq:unique} follows since $\overbar{h}\leqslant h_*$ as explained below \eqref{eqhbar2}.

On the other hand, random interlacements on a graph $G$ as above are defined under a probability measure $\P^I$ as a Poisson point process $\omega$ on the product space of doubly infinite trajectories on $G$ modulo time-shift, whose forward and backward parts escape all compact sets in finite time, times the label space $[0,\infty),$ see \cite{MR2525105}. For $u>0,$ we denote by $\omega^u$ the random interlacement process at level $u,$ which consists of all the trajectories in $\omega$ with label at most $u$. By $\I^u$ we denote the random interlacement set associated to $\omega^u$, which is the set of vertices visited by at least one trajectory in the support of $\omega^u,$ by $\V^u \stackrel{\mathrm{def.}}{=} G\setminus\I^u$ the vacant set of random interlacements, and by  $(\ell_{x,u})_{x\in{G}}$ the field of occupation times associated to $\omega^u,$ see (1.8) in \cite{MR2892408}, which collects the total time spent in each vertex of $G$ by the trajectories in the support of $\omega^u$, with additional independent exponential holding times at each vertex with parameter $\lambda_{x}$, $x\in{G}$. As stated in Corollary \ref{RIconnected} below, if \eqref{p0}, \eqref{Ahlfors} and \eqref{Green} hold,
\begin{equation}
\label{Iconnected}
    \text{for all }u>0,\ \I^u\text{ is $\P^I$-a.s.\ an infinite connected subset of } G.
\end{equation}
For vertex-transitive $G$, \eqref{Iconnected} is in fact a consequence of Theorem 3.3 of \cite{MR3076674}, since all graphs considered in the present paper are amenable on account of \eqref{due} below as well as display (14) and thereafter in \cite{MR3076674} (their spectral radius is equal to one).

Recall the definitions of the critical parameters $h_*$ and $u_*$ from \eqref{h_*} and \eqref{u_*}, which describe the phase transition of $E^{\geq h}$, the level sets of $\Phi$ (as $h$ varies), and that of $\V^u$ (as $u$ varies). Note that \eqref{Iconnected} indicates a very different geometry of $\I^u$ and $\V^u$ as $u \to 0$ in comparison with independent Bernoulli percolation on $G.$ Indeed, it is proved in \cite{MR3520023} that for all the graphs from \eqref{intro:Gex1}, both the set of open vertices and its complement undergo a non-trivial phase transition.

In order to derive an alternative representation of the critical parameters $u_*$ and $h_*$, we recall that the FKG inequality was proved in Theorem 3.1 of \cite{MR2525105} for random interlacements, and that it also holds for the Gaussian free field on $G.$ Indeed, it is shown in \cite{MR665603} for any centered Gaussian field with non-negative covariance function on a finite space, and by conditioning on a finite set and using a martingale convergence theorem this result can be extended to an infinite space, see for instance the proof of Theorem 2.8 in \cite{MR1707339}. As a consequence, for any $x\in{G},$ we have that
\begin{equation}
\label{eq:u_*alternative}
    u_*=\inf \big\{u\geq0;\, \P^I(\text{the connected component of $\mathcal{V}^u$ containing $x$ is infinite})=0 \big\},
\end{equation}
and similarly for $h_*$.

\medskip
The proofs of Theorems \ref{T:mainresultGFF} and \ref{T:mainresultRI} involve a continuous version of the graph $G$, its cable system $\tilde{G},$ and of the various processes associated to it. We attach to each edge $e=\{x,y\}$ of $G$ a segment $I_e$ of length $\rho_{x,y}=1/(2\lambda_{x,y}),$ and $\tilde{G}$ is obtained by glueing these intervals to $G$ through their respective endpoints. In other words, $\tilde{G}$ is the metric graph where every edge $e$ has been replaced by an interval of length $\rho_e$. We regard $G$ as a subset of $\tilde{G},$ and the elements of $G$ will still be called vertices. One can define on $\tilde{G}$ a continuous diffusion $\tilde{X},$ via probabilities $\tilde{P}_z,$ $z\in{\tilde{G}},$ such that for all $x\in{G},$ the projection on $G$ of the trajectory of $\tilde{X}$ under $\tilde{P}_x$ has the same law as the discrete random walk $Z$ on the weighted graph $G$ under $ P_x,$ and we will often identify $Z$ with this projection. This diffusion can be defined from its Dirichlet form or directly constructed from the random walk $Z$ by adding independent Brownian excursions on the edges beginning at a vertex. We refer to Section 2 of \cite{MR3502602} or Section 2 of \cite{MR3152724} for a precise definition and construction of the cable system $\tilde{G}$ and the diffusion $\tilde{X};$
see also Section 2 of \cite{DrePreRod} for a detailed description in the case $G=\Z^d.$ For all $x,y\in{\tilde{G}}$ we denote by $\tilde{g}(x,y),$ $x,y\in{\tilde{G}},$ the Green function associated to $\tilde{X},$ i.e., the density relative to the Lebesgue measure on $\tilde{G}$ of the $0$-potential of $\tilde{X},$ which agrees with $g$ on $G,$ as well as $\tilde{g}_U$ for an open set $U\subset \tilde{G}$,  the Green function associated to the process $\tilde{X}$ killed on exiting $U.$ We also denote by $H_A$ and $T_A$ the first hitting and exit time of a set $A\subset\tilde{{G}}$ for $X$, which exactly correspond to the notions introduced above \eqref{Greenstopped} when $A\subset G$.

We define for $\tilde{A}\subset\tilde{G}$ the set $\tilde{A}^* \subset G$ as the minimal set with respect to inclusion such that 
$\tilde{A}^* \supset G \cap \tilde{A},$ 
and such that for all $z\in{\tilde{A}}\setminus G,$ there exist $x,y\in{\tilde{A}^*}$ such that $z\in{I_{\{x,y\}}}$.
  For all $x\in{G}$ and $L>0,$ we write $\tilde{B}(x,L)$ for the largest subset $\tilde{B}$ of $\tilde{G}$ such that $\tilde{B}^*=B(x,L),$ and for all $\tilde{A}\subset{\tilde{G}}$ and $L>0,$ we let $\tilde{B}(\tilde{A},L)$ denote the largest subset $\tilde{B}$ of $\tilde{G}$ such that $\tilde{B}^*=B(\tilde{A}^*,L).$ Moreover, for $\tilde{A}\subset\tilde{G},$ we write 
\begin{equation} \label{simConn}
z\stackrel{\sim}{\longleftrightarrow} z' \text{ in } \tilde{A},
\end{equation}
 if there exists a continuous path between $z$ and $z'$ in $\tilde{A}.$ We say that $\tilde{A}$ is connected in $\tilde{G}$ if $z\stackrel{\sim}{\longleftrightarrow} z'$ in $\tilde{A}$ for all $z,z'\in{\tilde{A}}.$ Similarly, for $\tilde{A}_1\subset\tilde{A}$ and $\tilde{A}_2\subset\tilde{A},$ we write $\tilde{A}_1\stackrel{\sim}{\longleftrightarrow}\tilde{A}_2$ in $\tilde{A}$ if there exists a continuous path between $\tilde{A}_1$ and $\tilde{A}_2$ in $\tilde{A}.$ 

The Gaussian free field naturally extends to the metric graph $\tilde{G}$: Let $\tilde{\Phi}_z,$ $z\in{\tilde{G}},$ be the coordinate functions on the space of continuous real-valued functions $C(\tilde{G},\R),$ the latter endowed with the $\sigma$-algebra generated by the maps $\tilde{\Phi}_z,$ $z \in \widetilde G.$ Let $\tilde{\P}^G$ be the probability measure on $C(\tilde{G},\R)$ such that, under $\tilde{\P}^G,$ $(\tilde{\Phi}_z)_{z\in{\tilde{G}}}$ is a centered Gaussian field with covariance function 
\begin{equation}
\label{defphitilde}
   \tilde{\E}^G\big[\tilde{\Phi}_{z_1}\tilde{\Phi}_{z_2} \big]=\tilde{g}(z_1,z_2)\text{ for all }z_1,z_2\in{\tilde{G}}.
\end{equation}
The existence of such a continuous process was shown in \cite{MR3502602}. Any random variable $\tilde{\phi}$ with values in $C(\tilde{G},\R)$ and with law $\tilde{\P}^G$ will be called a Gaussian free field on $\tilde{G}.$ Moreover, if $\tilde{\phi}$ is a Gaussian free field on $\tilde{G},$ then it is plain that $(\tilde{\phi}_x)_{x\in{G}}$ is a Gaussian free field on $G.$ With a slight abuse of notation, we will henceforth write $\phi_x$ instead of $\tilde{\phi}_x$ when $x\in{G}$ for emphasis. 
We now recall the spatial Markov property for the Gaussian free field on $\tilde{G},$ see Section 1 of \cite{MR3492939}. Let $K\subset\tilde{G}$ be a compact subset with finitely many connected components, and let $U=\tilde{G}\setminus K$ be its complement. We can decompose any Gaussian free field $\tilde{\phi}$ on $\tilde{G}$ as
\begin{equation}
\label{markovproperty}
\tilde{\phi}=\tilde{\phi}^U+\tilde{\beta}^U\text{ with }\tilde{\beta}^U_z=\tilde{E}_z \big[\tilde{\phi}_{\tilde{X}_{T_U}}\1_{\{T_U<\infty\}}\big]\text{ for all }z\in{\tilde{G}},
\end{equation}
$\tilde{\phi}^U$ is a Gaussian free field independent of $\sigma(\tilde{\phi}_z,\,z\in{K})$ and with covariance function $\tilde{g}_U,$ and in particular $\tilde{\phi}^U$ vanishes on $K.$

One can also adapt the usual definition of random interlacements on $G,$ see \cite{MR2525105}, to the cable system $\tilde{G}$ as in \cite{MR3502602}, \cite{MR3492939} and \cite{DrePreRod}. For each $u>0,$ one thus introduces under a probability measure $\tilde{\P}^I$ the random interlacement process $\tilde{\omega}^u$ on $\tilde{G}$ at level $u,$ whose restriction to the trajectories hitting $K\subset\subset G$, and started after their first hitting time of $K$, can be described by a Poisson point process with intensity $u\tilde{P}_{e_K}$ where $e_K$ is the usual equilibrium measure of $K\subset\subset{G},$ see \eqref{defequilibrium} below, and $\tilde{P}_{e_K}$ is defined similarly as in \eqref{Pmu} but replacing $P$ by $\tilde{P}$. One then defines a continuous field of local times $(\tilde{\ell}_{z,u})_{z\in{\tilde{G}}}$ relative to the Lebesgue measure on $\tilde{G}$ associated to the random interlacement process on $\tilde{G}$ at level $u,$ i.e., $\tilde{\ell}_{z,u}$ corresponds for all $z\in{\tilde{G}}$ to the density with respect to the Lebesgue measure on $\tilde{G}$ of the total time spent by the random interlacement process around $z.$ For all $u>0,$ the restriction $(\tilde{\ell}_{x,u})_{x\in{{G}}}$ of the local times to $G$ coincides with the field of occupation times $(\ell_{x,u})_{x\in{G}}$ associated with the discrete random interlacement process $\omega^u$ defined above \eqref{Iconnected}, and just like for the free field, we will write $\ell_{x,u}$ instead of $\tilde{\ell}_{x,u}$ when $x\in{G}$. We also define for each measurable subset $\tilde{B}$ of $\tilde{G}$ and $u>0$ the family 
\begin{equation}
\label{eq:l_B}
   \tilde{\ell}_{\tilde{B},u}\stackrel{\mathrm{def.}}{=}(\tilde{\ell}_{z,u})_{z\in{\tilde{B}}}\in{C({\tilde{B}},\R)},
   \end{equation}
and the random interlacement set at level $u$ by
\begin{equation}
\label{defItilde}
    \tilde{\I}^u=\{z\in{\tilde{G}};\,\tilde{\ell}_{z,u}>0\}.
\end{equation}
The connectivity properties of $ \tilde{\I}^u$ will be studied in Section \ref{secconnec}. In particular, as stated in Corollary \ref{RIconnected}, $\widetilde {\mathcal I}^u$ is $\widetilde {\P}^I$-a.s.\ an unbounded and connected subset of $\widetilde G$, and the same is true of $ {\mathcal I}^u$ (as a subset of $G$).
We will elaborate on an important link between the fields $\tilde{\ell}_{\tilde{G},u}$ and $\tilde{\phi}$ from \eqref{defphitilde} and \eqref{eq:l_B} in Section \ref{seclevelsettilde}.

Finally, one of the main tools in the study of the percolative properties of the vacant set of random interlacements and of the level sets of the Gaussian free field, and the driving force behind the renormalization arguments of Section \ref{sectionpercsigncluster} are a certain family of correlation inequalities on $\tilde{G}$, which we now state. Their common feature is a small sprinkling for the parameters $u$ and $h$, respectively, which partially compensates the absence of a BK-inequality (after van den Berg and Kesten, see for instance \cite{MR1707339}) caused by the presence of long-range correlations in these models.
 The results below, in particular \eqref{decouplingRI} below, are of independent interest. We recall the notation from the paragraph preceding \nolinebreak\eqref{defphitilde} and \eqref{eq:l_B} and use $C(A,\R)$ to denote the space of continuous functions from $A$ to the reals, where the topology on $A$ is generally clear from the context. We moreover endow $C(A,\R)$ with the partial order $f\leq g$  if and only if $f(x)\leq g(x)$ for all $x\in{A}$.
 
\begin{The}
\label{decoup} Suppose $G$ is infinite, connected and $(G,\lambda)$ such that \eqref{p0}, \eqref{Ahlfors}, \eqref{Green} hold. Let $\tilde{A}_1$ and $\tilde{A}_2$ be two Borel-measurable subsets of $\tilde{G}$, at least one of which is bounded. Let $s=d(\tilde{A}_1^*,\tilde{A}_2^*)$ and $r = \delta(\tilde{A}_1^*) \wedge \delta(\tilde{A}_2^*)$ (note that $r< \infty$). There exist $\Cl{CdecouGFF}$ and $\Cl[c]{cdecouGFF}$ 
such that for all $\eps\in{(0,1)},$  and all measurable functions $f_i:C(\tilde{A}_i,\R)\rightarrow[0,1]$, $i=1,2$, which are either both increasing or both decreasing, if $s>0$,
\begin{equation}
\label{decouplingGFF}
\begin{split}
    &\tilde{\E}^G\left[f_1\big(\tilde{\Phi}_{|\tilde{A}_1}\big)f_2\big(\tilde{\Phi}_{|\tilde{A}_2}\big)\right]\\
   & \qquad \quad\leq\tilde{\E}^G\left[f_1\big(\tilde{\Phi}_{|\tilde{A}_1}\pm\eps\big)\right]\tilde{\E}^G\left[f_2\big(\tilde{\Phi}_{|\tilde{A}_2}\pm\eps\big)\right]+\Cr{CdecouGFF}(r+s)^{\alpha}\exp\left\{-\Cr{cdecouGFF}\eps^2s^{\nu}\right\},
   \end{split}
\end{equation}
\vphantom{$\Cl[c]{ceta}$}and there exist $\Cl{Ceta},$ $\Cl{CdecouRI}$  and $\Cl[c]{cdecouRI}$ 
such that for all $u>0$, $\eps\in{(0,1)}$ and $f_i$ as above, if $s\geq\Cr{Ceta}(r\vee1),$
\begin{equation}
\label{decouplingRI}
\begin{split}
   & \tilde{\E}^I\left[f_1\big(\tilde{\ell}_{\tilde{A}_1,u}\big)f_2\big(\tilde{\ell}_{\tilde{A}_2,u}\big)\right]\\
   &\qquad \quad \leq\tilde{\E}^I\left[f_1\big(\tilde{\ell}_{\tilde{A}_1,u(1\pm\eps)}\big)\right]\tilde{\E}^I\left[f_2\big(\tilde{\ell}_{\tilde{A}_2,u(1\pm\eps)}\big)\right]+\Cr{CdecouRI}(r+s)^{\alpha}\exp\left\{-\Cr{cdecouRI}\eps^2us^{\nu}\right\},
   \end{split}
\end{equation}
where the plus sign corresponds in both equations to the case where the functions $f_i$ are increasing and the minus sign to the case where the functions $f_i$ are decreasing.
\end{The}
The proof of Theorem \ref{decoup} is deferred to Section \ref{secdecoup}. While \eqref{decouplingGFF} follows rather straightforwardly from the decoupling inequality from \cite{MR3325312} for the Gaussian free field (see also Theorem \ref{decoupGFFstrong} for a strengthening of \eqref{decouplingGFF}), the proof of \eqref{decouplingRI} is considerably more involved. It uses the
soft local times technique introduced in \cite{MR3420516} on $\Z^d$ for random interlacements, but a generalization to the present setup requires some effort (note also that for graphs of the type $G = G' \times \Z$, one could also use the inequalities of \nolinebreak \cite{MR2891880}, which are proved by different means).

\section{Preliminaries and examples}
\label{secpre}

We now gather several aspects of potential theory for random walks on the weighted graphs introduced in the last section. These include estimates on killed Green functions, see Lemma \nolinebreak\ref{L:killGreen} below, a resulting (elliptic) Harnack inequality, bounds on the capacities of various sets, see Lemma \ref{Lemmacapline}, and on the heat kernel, see Proposition \ref{someproperties}, which will be used throughout. We then proceed to discuss product graphs in Proposition \ref{P:product} and, with a view towards \eqref{weakSecIso}, connectivity properties of external boundaries in Proposition \nolinebreak\ref{generalstarconnected}. These results are helpful in showing how the examples from \eqref{intro:Gex1}, which constitute an important class, fit within the framework of the previous section. We conclude this section by deducing in Corollary \ref{C:examples} that our main results, Theorems \ref{T:mainresultGFF} and \ref{T:mainresultRI}, apply in all cases of \eqref{intro:Gex1}.

\bigskip
From now on,
\begin{equation}
\label{eq:Ass}
\begin{split}
&\text{we assume that $(G,\lambda)$ is an infinite, connected, weighted graph endowed with }\\&\text{a distance function $d$ that satisfies \eqref{p0}, \eqref{Ahlfors} and \eqref{Green}}
\end{split}
\end{equation}
(see Section \ref{2}). Throughout the remainder of this article, we \textit{always} tacitly work under the assumptions \eqref{eq:Ass}. Any additional assumption will be mentioned explicitly.

The following lemma collects an estimate similar to \eqref{Green} for the stopped Green function \eqref{Greenstopped}.

\begin{Lemme}
\label{L:killGreen}
 There exists a constant $\Cl{CHarnack}>1$ such that, if $U_1\subset U_2\subset\subset G$ with $d(U_1,U_2^{\mathsf{c}})\geq \Cr{CHarnack}(\delta (U_1)\vee1),$ then
\begin{equation}
\label{Greenstoppedbound}
\begin{split}
    \frac{\Cr{cGreen}}{2}d(x,y)^{-\nu}&\leq g_{U_2}(x,y)\leq \Cr{CGreen}d(x,y)^{-\nu}\text{ for all }x\neq y\in{U_1}, \text{ and } \\
    \frac{\Cr{cGreen}}{2}&\leq g_{U_2}(x,x)\leq \Cr{CGreen}\text{ for all }x\in{U_1}.
    \end{split}
\end{equation}
\end{Lemme}
\begin{proof}
  Let $U_1\subset U_2\subset\subset G.$ The upper bound in \eqref{Greenstoppedbound} follows immediately from \eqref{Green} since $g_{U_2}(x,y)\leq g(x,y)$ for all $x,y\in{G}$ by definition. For the lower bound, using \eqref{Greenmarkov} and \eqref{Green}, we obtain that for all $x\neq y\in{U_1},$
        \begin{equation*}
            g_{U_2}(x,y)\geq\Cr{cGreen}d(x,y)^{-\nu}-\Cr{CGreen}E_x\big[d(Z_{T_{U_2}},y)^{-\nu}\big]\geq\Cr{cGreen}d(x,y)^{-\nu}-\Cr{CGreen}d(U_1,U_2^{\mathsf{c}})^{-\nu}.
        \end{equation*}
        Thus, choosing $\Cr{CHarnack}$ large enough such that $\frac{\Cr{cGreen}}{2} \geq \frac{\Cr{CGreen}}{\Cr{CHarnack}^\nu} $, it follows that if $d(U_1,U_2^{\mathsf{c}})\geq\Cr{CHarnack}\delta(U_1)\, (\geq\Cr{CHarnack} d(x,y))$, then
        \begin{equation*}
            g_{U_2}(x,y)\geq\frac{\Cr{cGreen}}{2}d(x,y)^{-\nu}\text{ for all }x\neq y\in{U_1}.
        \end{equation*}
The lower bound for $g_{U_2}(x,x),$ $x\in{U_1},$ is obtained similarly. 
\end{proof}

Using Lemma A.2 in \cite{MR2778797}, which is an adaptation of Lemma 10.2 in \cite{MR1853353}, an important consequence of $\eqref{Greenstoppedbound}$ is the elliptic Harnack inequality in \eqref{EHI} below. For this purpose, recall that a function $f$ defined on
  $\overbar{U_2}\stackrel{\mathrm{def.}}{=}B_G(U_2,1),$ the closed $1$-neighborhood of $U_2$ for the graph distance, 
  is called $L$-\nolinebreak harmonic (or simply harmonic) in $U_2$ if $E_x[f(Z_1)] = f(x)$, or equivalently $Lf(x)=0$ (see \eqref{Z:gen}), for all $x \in U_2$. The bounds of \eqref{Greenstoppedbound} imply that there exists a constant $\Cl[c]{cHarnack}\in(0,1)$ such that for all $U_1\subset U_2\subset\subset G$ with $\delta(U_1)\geq2\Cr{Cdistance}$ and $d(U_1,U_2^{\mathsf{c}})\geq \Cr{CHarnack}(2\delta (U_1)\vee1),$ and any non-negative function $f$ on $\overbar{U_2}$ which is harmonic in \nolinebreak$U_2$, 
\begin{equation}
\label{EHI}
    \inf_{y\in{U_1}}{f(y)}\geq\Cr{cHarnack}\sup_{y\in{U_1}}{f(y)}.
\end{equation}
Another important consequence of $\eqref{Greenstoppedbound}$ is that the balls for the distance $d$ are almost connected in the following sense: 
\begin{equation}
\label{ballsalmostconnected}
\forall x\in{G}, \, R \geq 1 \text{ and }   y,y'\in{B(x,R)},\,y\leftrightarrow y'\text{ in }B(x,\Cl{Cstrong}R),\text{ with }\Cr{Cstrong}=2\Cr{CHarnack}+1.
\end{equation}
Indeed, for all $U\subset\subset G$ and $y,y'\in{G},$ $y\stackrel{U}{\longleftrightarrow} y'$ is equivalent to $g_U(y,y')>0,$ and by definition,\phantom{$\Cl[c]{cstrong}$}
\begin{equation}
\label{dforCstrong}
    d\big(B(x,R),B(x,\Cr{Cstrong}R)^{\mathsf{c}}\big)\geq2\Cr{CHarnack}R\geq\Cr{CHarnack}\delta\big(B(x,R)\big).
\end{equation}
As a consequence, \eqref{Greenstoppedbound} implies that $g_{B(x,\Cr{Cstrong}R)}(y,y')>0$ for all $y,y'\in{B(x,R)}.$

We now recall some facts about the equilibrium measure and capacity of various sets. For $A\subset\subset U\subset G,$ the equilibrium measure of $A$ relative to $U$ is defined as
\begin{equation}
\label{defequilibrium}
    e_{A,U}(x)\stackrel{\mathrm{def.}}{=}\lambda_x P_x(\tilde{H}_{A}>T_{U})\1_{A}(x) \text{ for all }x\in{G},
\end{equation}
where $\tilde{H}_{A}\stackrel{\mathrm{def.}}{=}\inf\{n\geq1,\ Z_n\in{A}\}$ is the first return time in $A$ for the random walk on \nolinebreak$G,$ and the capacity of $A$ relative to $U$ as the total mass of the equilibrium measure,
\begin{equation}
\label{defcap}
    \mathrm{cap}_U(A)\stackrel{\mathrm{def.}}{=}\sum_{x\in{A}}e_{A,U}(x).
\end{equation}
By \cite[(1.57)]{MR2932978} for the graph with infinite killing on $U^c$, for all $A\subset\subset U\subset G,$ the following last-exit decomposition relates the entrance time $H_{A}$ of $Z$ in $A,$ the exit time $T_{U}$ of $U,$ the stopped Green function and the equilibrium measure:
\begin{equation}
\label{entrancegreenequi}
     P_x(H_{A}<T_{U})=\sum_{y\in{A}}g_{U}(x,y)e_{A,U}(y)\text{ for all }x\in{U}.
\end{equation}
For $\emptyset\neq A\subset\subset G$ and $x\in{G},$ we introduce the equilibrium measure, capacity and harmonic measure as 
\begin{equation}
\label{normalized}
    e_{A}(x)\stackrel{\mathrm{def.}}{=}e_{A,G}(x),\qquad\mathrm{cap}(A)\stackrel{\mathrm{def.}}{=}\mathrm{cap}_G(A)\qquad\text{and}\qquad\overline{e}_A(x)\stackrel{\mathrm{def.}}{=}\frac{e_A(x)}{\mathrm{cap}(A)},
\end{equation}
respectively. The capacity is a central notion for random interlacements, since we have the following characterization for the random interlacement set $\I^u$
\begin{equation}
\label{defIu}
\P^I(\I^u\cap A=\emptyset)=\exp\{-u\cdot \mathrm{cap}(A)\}\text{ for all }A\subset\subset G;
\end{equation}
see Remark 2.3 in \cite{MR2525105}.
With these definitions, it then follows using \eqref{entrancegreenequi} and \eqref{conditiondistance} that for all $R\geq\Cr{Cdistance}$ and $x_0\in{G},$
\begin{equation*}
\begin{split}
    \Cr{cGreen}R^{-\nu}\mathrm{cap}\left(B(x_0,R)\right)\leq 1&=\sum_{y\in{\partial B(x_0,R)}}g(x_0,y)e_{B(x_0,R)}(y)
    \\&\leq\Cr{CGreen} (R-\Cr{Cdistance})^{-\nu}\mathrm{cap}\left(B(x_0,R)\right),
\end{split}
\end{equation*}
and hence there exist constants $0<\Cl[c]{ccapacity}\leq\Cl{Ccapacity}<\infty$ only depending on $G$ such that for all $R\geq1$ and $x\in{G},$ 
\begin{equation}
    \label{ballcapacity}
    \Cr{ccapacity}R^{\nu}\leq\mathrm{cap}\left(B(x,R)\right)\leq \Cr{Ccapacity}R^{\nu}.
\end{equation}
A useful characterization of capacity in terms of a variational problem is given by
\begin{equation}
\label{variational}
    \mathrm{cap}(A)=\Big(\inf_{\mu}\sum_{x,y\in{A}}g(x,y)\mu(x)\mu(y)\Big)^{-1}, \quad \text{ for } A \subset \subset G,
\end{equation}
where the infimum is over  probability measures $\mu$ on $A,$ see e.g.\ Proposition 1.9 in \cite{MR2932978} for the case of a finite graph with non-vanishing killing measure (the proof can be extended to the present setup). In particular, since every probability measure $\mu$ on $A$ is also a probability measure on any set containing $A,$ the capacity is increasing, so for $A,B \subset G,$
\begin{equation}
\label{capincrease}
A\subset B\qquad \text{ implies }\qquad\mathrm{cap}(A)\leqslant\mathrm{cap}(B).
\end{equation}
Another consequence of the representation \eqref{variational} is the following lower bound on the capacity of a set.
\begin{Lemme}
	\label{Lemmacapline}
	There exists a constant $c$ depending only on $G$ such that for all $L\geqslant1$ and $A \subset G$ connected with diameter at least $L,$  
	\begin{align}
	\begin{split}
	\label{eqcapline}
	\mathrm{cap}(A)\geqslant \left\{\begin{array}{ll} cL, &\text{ if }\nu>1,	
	\\\frac{cL}{\log(L+1)}, &\text{ if }\nu=1,
	\\cL^{\nu}, &\text{ if }\nu<1.\end{array}\right.
	\end{split}
	\end{align}
	Moreover, if $A\subset G$ is infinite and connected, then for all $x_0\in{G}$
	\begin{equation}
	\label{capatinfinity}
	\mathrm{cap}(A\cap B(x_0,L))\to \infty \quad \text{ as } L \to \infty,
	\end{equation}
	and thus $A\cap\I^u\neq\emptyset$ $\P^I$-a.s.
\end{Lemme}
\begin{proof}
	Let us fix some $L\geqslant1,$ $A$ connected subset of $G$ with diameter at least $L,$ and $x_0\in{A}.$ We introduce $L'=\lceil L/(2\Cr{Cdistance})\rceil$ and for each $k\in{\{1,\dots,L'\}}$ the set $A_k=A\cap(B(x_0,\Cr{Cdistance}k)\setminus B(x_0,\Cr{Cdistance}(k-1))),$ which is non-empty by \eqref{conditiondistance}. Then for all $k\in{\{1,\dots,L'\}}$ and $x\in{A_k},$ we have by \eqref{Green} that
	\begin{equation*}
	\sum_{p=1}^{L'}\sup_{y\in{A_p}}g(x,y)\leqslant\Cr{CGreen}\Big(2+\Cr{Cdistance}^{-\nu}\sum_{p=1}^{L'} 1_{p\neq k-1} (k-1-p)^{-\nu}\Big)\leqslant2\Cr{CGreen}\Big(2+\Cr{Cdistance}^{-\nu}\sum_{p=1}^{L'}p^{-\nu}\Big).
	\end{equation*}
 Now let $\mu$ be the probability measure on $A$ defined by $\mu(x)=(L'|A_k|)^{-1}$ if $x\in{A_k}$ for some $k\in{\{1,\dots,L'\}},$ and $\mu(x)=0$ otherwise, we have
	\begin{equation*}
	\sum_{x,y\in{A}}g(x,y)\mu(x)\mu(y)
	\leqslant\frac{2\Cr{CGreen}}{L'}\Big(2+\Cr{Cdistance}^{-\nu}\sum_{p=1}^{L'}p^{-\nu}\Big).
	\end{equation*}
	Combining this bound with \eqref{variational}, the inequality \eqref{eqcapline} follows. If $A$ is now an infinite and connected subset of $G,$ then for each $x_0\in{G}$ there exists $L_0>0$ such that for all $L\geqslant L_0,$ the set $A\cap B_G(x_0,L/\Cr{Cdistance})$ has diameter at least $\frac{L}{2\Cr{Cdistance}},$ and thus by \eqref{conditiondistance} $A\cap B(x_0,L)$ contains at least a connected component of diameter $\frac{L}{2\Cr{Cdistance}},$ and \eqref{capatinfinity} then follows directly from \eqref{eqcapline}. Finally, by \eqref{defIu},
	\begin{equation*}
	\P^I(A\cap\I^u=\emptyset)\leqslant\P^I(A\cap\I^u\cap B(x_0,L)=\emptyset)\leqslant\exp\big\{-u\cdot \mathrm{cap}(A\cap B(x_0,L))\big\}\tend{L}{\infty}0.
	\end{equation*}
\end{proof}

Next, we collect an upper bound on the heat kernel \eqref{heatkernel} and an estimate on the distribution of the exit time of a ball $T_{B(x,R)}$.

\begin{Prop}
	\label{someproperties}$\quad$ 
	
	\begin{enumerate}[i)]
		\item There exists a constant $C$ such that for all $x,y\in{G}$ and $n>0,$
		\begin{equation}
		\label{due}
		p_n(x,y)\leq Cn^{-\frac{\alpha}{\beta}}.
		\end{equation}
		\item There exist constants $c$ and $C$ such that for all $x\in{G},$ $R>0$ and positive integer \nolinebreak$n$,
		\begin{equation}
		\label{exittime}
		P_x\big(T_{B(x,R)}\leq n\big)\leq C\exp\Big\{-\Big(\frac{cR^{\beta}}{n}\Big)^{\frac{1}{\beta-1}}\Big\}.
		\end{equation}
	\end{enumerate}
\end{Prop}

Proposition \ref{someproperties} is essentially known, for instance if $d$ is the graph distance $d_G$ then these results (as well as \eqref{UE} and \eqref{LE}) are proved in \cite{MR1853353}. For a general distance $d,$ some estimates similar to \eqref{due} and \eqref{exittime} (as well as \eqref{UE} and \eqref{LE}) are also proved in \cite{MR3194164} and \cite{MR2430977} in the more general setting of metric spaces, and we could apply them to the variable rate continuous time Markov chain on $G.$ However, there does not seem to be any proof in the literature that exactly fits our needs (general distance $d,$ discrete time random walk \nolinebreak$Z$), and so, for the reader's convenience, we have included a proof of Proposition \ref{someproperties} in the Appendix. 

\begin{Rk}
	\leavevmode
	\label{remarkendofsection3}
	\begin{enumerate}[1)]
		\item \label{remarkendofsection31} With Proposition \ref{someproperties} at our disposal, following up on Remark \ref{T:HK}, we briefly discuss the relation of the above assumptions \eqref{eq:Ass} to heat kernel bounds within our setup. A consequence of \eqref{due} and \eqref{exittime} is that, under condition \eqref{p0},  
		\begin{equation}
		\label{V+GimpliesUE}
		\eqref{Ahlfors}+\eqref{Green}\Rightarrow\eqref{UE};
		\end{equation}
		note that in contrast to the results of Remark 2.2, this holds true even when $d$ is not the graph distance, where \eqref{UE} is defined in Remark \ref{T:HK}. Indeed, for $d=d_G$ this implication is part of Proposition 8.1 in \cite{MR1853353}, but the proof remains valid for any distance $d.$ However, the corresponding lower bound \eqref{LE} on the heat kernel does not always hold. To see this, take for example $G$ a graph such that \eqref{p0}, \eqref{Ahlfors} and \eqref{Green} hold when $d$ is the graph distance, and let $d'=d^{\frac{1}{\kappa}}$ for some $\kappa>1$ (cf. Proposition \ref{P:product} and \eqref{prod11} below for a situation where this is relevant). Then for the graph $G$ endowed with the distance $d',$ the conditions \eqref{p0}, $(V_{\alpha'})$ and $(G_{\beta'})$ hold with $\alpha'=\alpha\kappa$ and $\beta'=\beta\kappa.$ Moreover, using \eqref{UE} for the distance $d$, one obtains that
$p_n(x,y)+p_{n+1}(x,y)\leq 2Cn^{-\frac{\alpha'}{\beta'}}\exp\{-(\frac{d'(x,y)^{\beta'}}{Cn})^{\frac{1}{\beta-1}}\}.$		Taking $n=\lfloor d'(x,y)\rfloor$ for instance, it follows that for any $c>0,$ since $\beta' > \beta$,
		\begin{equation*}
			\begin{split}
				&  \big(p_n(x,y)+p_{n+1}(x,y)\big) n^{\frac{\alpha'}{\beta'}}\exp\bigg\{\Big(\frac{d'(x,y)^{\beta'}}{cn}\Big)^{\frac{1}{\beta'-1}}\bigg\}\\
				&\qquad \qquad \leq2C\exp\bigg\{-\Big(\frac{n^{\beta'-1}}{C}\Big)^{\frac{1}{\beta-1}}+\Big(\frac{n^{\beta'-1}}{c}\Big)^{\frac{1}{\beta'-1}}\bigg\}\tend{n}{\infty}0,
			\end{split}
		\end{equation*}
		thus (LHK$(\alpha',\beta')$) cannot hold for $G$ endowed with the distance $d'.$
		
		\item \label{RkLHK'} Even in cases where \eqref{LE} does not hold, it is still possible to obtain some slightly worse lower bounds for a general distance $d.$ We will not need these results in the rest of the article, and therefore we only sketch the proofs. We introduce the following near-diagonal lower estimate
		\begin{equation}
		\label{NLHK}
		\tag{NLHK$(\alpha,\beta)$}
		p_n(x,y)+p_{n+1}(x,y)\geqslant cn^{-\frac{\alpha}{\beta}}\quad\text{for all }x,y\in{G}\text{ and }n\geqslant cd(x,y)^{\beta}.
		\end{equation}Let us assume that the condition \eqref{p0} is fulfilled, we then have the following equivalences for all $\alpha>2$ and $\beta\in{[2,\alpha)}$
		\begin{equation}
		\label{equivHK}
		\eqref{Ahlfors}+\eqref{Green}\Leftrightarrow\eqref{UE}+\eqref{NLHK}.
		\end{equation}
		The first implication follows from (13.3) in \cite{MR1853353}, whose proof remains valid for a general distance $d$, given \eqref{V+GimpliesUE}, \eqref{due}, \eqref{boundexittime} and \eqref{EHI}, and the proof of its converse follows from a careful inspection of the proof of Proposition 15.1 in \cite{MR1853353} or Lemma 4.22 and Theorem 4.26 in \cite{MR3616731}. Estimates similar to \eqref{UE} and \eqref{NLHK} for the continuous time Markov chain on $G$ with jump rates $(\lambda_x)_{x\in{G}}$ and transition probabilities $(p_{x,y})_{x,y\in{G}},$ see \eqref{transitionprobability}, are also equivalent to \eqref{equivHK}, see Theorem 3.14 in \cite{MR3194164}. Let us now also assume that there exist constants $c>0$ and $\zeta\in{[1,\beta)}$ such that 
		\begin{equation}
		\label{distancezeta}
		\tag{$D_{\zeta}$}
		\begin{split}
		&\text{for all } r>0,\, k\in{\N}\text{ and }x,y\in{G}\text{ such that }d(x,y)\leqslant ck^{\frac{1}{\zeta}}r,\,\text{there exists}
		\\&\text{ a sequence }x_1=x,x_2,\dots,x_k=y\text{ with }d(x_{i-1},x_{i})\leqslant r\text{ for all }i\in{\{2,\dots,k\}},
		\end{split}
		\end{equation}
		then the conditions in \eqref{equivHK} are also equivalent to \eqref{UE} plus the following lower estimate
		\begin{equation}
		\label{LHK'}
		\tag{LHK$(\alpha,\beta,\zeta)$}
		p_n(x,y)+p_{n+1}(x,y)\geqslant cn^{-\frac{\alpha}{\beta}}\exp\Big\{-\Big(\frac{d(x,y)^{\beta}}{cn}\Big)^{\frac{\zeta}{\beta-\zeta}}\Big\}\quad\text{for all }n\geqslant d_G(x,y).
		\end{equation}
		Indeed, under condition \eqref{distancezeta}, the proof that \eqref{equivHK} implies \eqref{LHK'} is similar to the proof of Proposition 13.2 in \cite{MR1853353} or Proposition 4.38 in \cite{MR3616731}, modulo some slight modifications when $d$ is a general distance, and its converse is trivial. Note that if $d=d_G,$ it is clear that ($D_1$) holds and that the lower estimate (LHK$(\alpha,\beta,1)$) is the same as \eqref{LE}, and thus we recover the results from Remark \ref{T:HK}. If $d'=d_G^{\frac{1}{\kappa}}$ for some $\kappa\geqslant1$ as in the counter-example of Remark \ref{remarkendofsection3}, \ref{remarkendofsection31}), and \eqref{Ahlfors} and \eqref{Green} hold with the distance $d_G,$ then $(D_{\kappa})$ hold for the distance $d'$ and thus also (LHK$(\alpha',\beta',\kappa)$) for the distance $d'$, where $\beta'=\beta\kappa$ and $\alpha'=\alpha\kappa,$ which is exactly the same as \eqref{LE} for the distance $d_G.$ 
	\end{enumerate}
\end{Rk}

We now discuss product graphs. Let $G_1$ and $G_2$ be two graphs as in the previous section (countably infinite, connected and with bounded degree), endowed with weight functions $\lambda^1$ and $\lambda^2$. The graph $G=G_1\times G_2$ is defined such that $x=(x_1,x_2)\sim y=(y_1,y_2)$ if and only there exists $i\neq j\in{\{1,2\}}$ such that $x_i\sim y_i$ and $x_j=y_j.$ One naturally associates with $G$ the weight function $\lambda$ such that for all $x=(x_1,x_2)\sim y=(y_1,y_2),$ one has
\begin{equation}
\label{prodweight}
\text{$\lambda_{x,y}=\lambda^i_{x_i,y_i},$ where $i\in{\{1,2\}}$ is such that $x_i\neq y_i.$}
\end{equation}

\begin{Prop}
\label{P:product} 
Suppose that $(G_i,\lambda^i)$ satisfy (UHK$(\alpha_i,\beta_i)$) and (LHK$(\alpha_i,\beta_i)$) with respect to the graph distance $d_{G_i}$, for $i=1,2$, as well as \eqref{p0}. Assume that 
\begin{equation}
\label{prod10}
\text{$\alpha_i \geq 1$ and $2 \leq \beta_i \leq 1+ \alpha_i$, for $i=1,2,$ and $ \exists\,j \in \{1,2\}$
 s.t. $\alpha_j > 1$ or $\beta_j >2$. }
\end{equation}
Then, if $\beta_1 \leq \beta_2,$  
the graph
$G_1 \times G_2$ endowed with the weights \eqref{prodweight} satisfies $\eqref{Ahlfors}$, $\eqref{Green}$ with
\begin{equation}
\label{prod11}
\begin{split}
&\alpha = \alpha_1 \frac{\beta_2}{\beta_1}+ \alpha_2, \ \beta=\beta_2 \text{ and } d(x,y)=\max\big(d_{G_1}(x_1,y_1)^{\frac{\beta_1}{\beta_2}},d_{G_2}(x_2,y_2)\big).
\end{split}
\end{equation}
\end{Prop}

\begin{proof}
We first argue that $(G,\lambda)$ satisfies \eqref{Ahlfors}. By Remark \ref{T:HK}, $(G_i,\lambda^i),$ $i=1,2,$ satisfy $(V_{\alpha_i}).$ On account of \eqref{prodweight}, one readily infers that $\lambda(A\times B)=\lambda^1(A)\cdot|B|+ |A|\cdot \lambda^2(B)$ for all $A\subset G_1$, $B\subset G_2$. Applying this to  $A= B_{d_{G_1}}(x_1, R^{\beta_2/\beta_1})$, $B=B_{d_{G_2}}(x_2, R)$, observing that $B_d((x_1,x_2),R)=A \times B$ by definition of $d(\cdot,\cdot)$ and noting that $\Cr{cweight}|A|\leq \lambda^1(A) \leq \Cr{Cweight}|A|$ (and similarly for $B$), see \eqref{lambda}, it  follows that uniformly in $(x_1,x_2)\in G$, $\lambda (B_d((x_1,x_2),R))$ is of order $R^{\alpha}$  with $\alpha$ given by \eqref{prod11}, whence \eqref{Ahlfors} is fulfilled.

It remains to show that \eqref{Green} holds. Let $(\overline{X}_t^i)_{t\geq 0}$, $i=1,2,$ denote the continuous time walk on $G_i$ with jump rates $\lambda_x^i = \sum_{y: d_{G_i}(x,y)=1}\lambda_{x,y}^i$, and suppose $\overline{X}_{\cdot}^1$, $\overline{X}_{\cdot}^2$ are independent.  Let $\overline{X}_{\cdot}$ be the corresponding walk on $G$ (with jump rates $\lambda_x$, cf.\@  \eqref{prodweight}). Then $\overline{X}_{\cdot}$ has the same law as $(\overline{X}^1_{\cdot}, \overline{X}^2_{\cdot})$ and in view of \eqref{eq:Greendef},
\begin{equation}
\label{Greenproduct}
    g(x,y)=\int_{0}^{\infty} P_x(\overline{X}_t=y)\, \mathrm{d}t=\int_0^\infty P_{x_1}(\overline{X}_t^1=y_1)P_{x_2}(\overline{X}_t^2=y_2) \, \mathrm{d}t,
\end{equation}
with $x=(x_1,x_2)$ and $y=(y_1,y_2)$. We introduce for $i = 1,2,$ the additive functionals
\begin{equation}
\label{eq:tc1}
    A_t^i = \int_0^t \lambda^i_{\overline{X}_s^i}\, \mathrm{d}s, \quad \text{ for } t\geq 0,\, i=1,2,
    \end{equation}
along with $\tau_t^i = \inf \{ s \geq 0 ; \, A_s^i \geq t\} $ and the corresponding time-changed processes
\begin{equation*}
\label{eq:tc2}
   ( Y_{t}^i)_{t\geq 0} \stackrel{\text{def.}}{=}  (\overline{X}_{\tau_{t}^i}^i)_{t\geq 0}. 
    \end{equation*} 
By the above assumptions, the discrete skeletons of $ Y_{\cdot}^i$, $i=1,2$, satisfy the respective heat kernel bounds HK$(\alpha_i, \beta_i)$ in the notation of \cite{MR3616731}, and thus by Theorem 5.25 in \cite{MR3616731} (the process $ Y_{\cdot}^i$ has unit jump rate), for all $x=(x_1,x_2)$ and $y=(y_1,y_2)$ in $G$, abbreviating $d_i = d_{G_i}(x_i,y_i)$ and $d = d(x,y)$, so that $d^{\beta_2}= d_1^{\beta_1} \vee d_2^{\beta_2}$,
\begin{equation}
\label{Pxiwalk}
    ct^{-\frac{\alpha_i}{\beta_i}}\exp\Big\{-\Big(\frac{d^{\beta_2}}{ct}\Big)^{\frac{1}{\beta_i-1}}\Big\} \leq P_{x_i}(Y_t^i=y_i)\leq c't^{-\frac{\alpha_i}{\beta_i}}\exp\Big\{-\Big(\frac{d_{i}^{\beta_i}}{c't}\Big)^{\frac{1}{\beta_i-1}}\Big\},
\end{equation}
where the lower bound holds for all $t\geq d_i\vee1$ and the upper bound for all $t\geq d_i.$ Going back to \eqref{Greenproduct}, noting that $\overline{X}_{t}^i= Y_{A_t^i}^i$ and that $\Cr{cweight}t \leq A_t^i \leq \Cr{Cweight}t$ for all $t \geq 0$ by \eqref{lambda} and \eqref{eq:tc1}, observe that
 \begin{equation} \label{eq:infsupBd}
\inf_{i\in{\{1,2\}}}\sup_{t\leq \Cr{Cweight}\Cr{cweight}^{-1}(d_1 \vee d_2)} P_{x_i}(Y_t^i=y_i)\leq Ce^{-c(d_1\vee d_2)},
\end{equation}
which follows for instance from Theorem 5.17 in \cite{MR3616731}. We obtain for all $x$ and $y$, with constants possibly depending on $\alpha_i$ and $\beta_i$, keeping in mind that $d^{\beta_2}=d_i^{\beta_i}$ for some $i$ in the third line below,
\begin{equation}
\label{Greenproduct2}
\begin{split}
&g(x,y) 
\leq \int_{0}^{\infty} \sup_{\Cr{cweight}t \leq s \leq \Cr{Cweight}t}  \left\{P_{x_1}(Y_s^1=y_1)P_{x_2}(Y_s^2=y_2)\right\}\, \mathrm{d}t \\
&\hspace{-1.6em}\stackrel{\eqref{Pxiwalk}, \eqref{eq:infsupBd}}{\leq} C(d_1\vee d_2)e^{-c(d_1\vee d_2)} + C' \int_{\Cr{cweight}^{-1}(d_1 \vee d_2)}^{\infty} t^{-(\frac{\alpha_1}{\beta_1}+ \frac{\alpha_2}{\beta_2})}\exp\Big\{-\sum_{i=1}^2\Big(  \frac{d_{i}^{\beta_i}}{c''t}\Big)^{\frac{1}{\beta_i-1}}\Big\} \,  \mathrm{d}t\\
&\stackrel{\mathclap{u=d^{-\beta_2}t}}{\leq}\hspace{0.2cm}C e^{-cd} + C' \int_{0}^{\infty} d^{-(\frac{\beta_2\alpha_1}{\beta_1}+\alpha_2)}u^{-(\frac{\alpha_1}{\beta_1}+ \frac{\alpha_2}{\beta_2})}\exp\Big\{-(c''u)^{-\frac{1}{\beta_i-1}}\Big\} d^{\beta_2} \, \mathrm{d}u \\&\leq C''d^{-(\alpha-\beta)},
\end{split}
\end{equation}
recalling the definition of $\alpha$ and $\beta$ from \eqref{prod11} in the last step; we also note that the integral over $u$ in the last but one line is finite since $\alpha_i \geq 1 $ and $\beta_i\leqslant1+\alpha_i,$ so that $\frac{\alpha_i}{\beta_i} \geq \frac{\alpha_i}{1+\alpha_i} \geq \frac12$ with strict inequality for at least one of the $i$'s due to \eqref{prod10}, whence $\frac{\alpha_1}{\beta_1}+ \frac{\alpha_2}{\beta_2} >1$. In view of \eqref{eq:nu}, \eqref{Greenproduct2} yields the desired upper bound. For the corresponding lower bound, one proceeds similarly, starting from \eqref{Greenproduct}, discarding the integral over $0 \leq t \leq  \Cr{cweight}^{-1}(d_1 \vee d_2\vee 1)$, and applying the lower bound from \eqref{Pxiwalk}. Thus, \eqref{Green} holds, which completes the proof.
\end{proof}

\begin{Rk}
	\leavevmode
\begin{enumerate}[1)]
\item Proposition \ref{P:product} is sufficient for our purposes but one could extend it to graphs $(G_i,\lambda_i)$ which satisfy \eqref{p0}, $(\text{UHK}(\alpha_i,\beta_i))$ and $(\text{NLHK}(\alpha_i,\beta_i))$ under a general distance $d_i$ for $i=1,2$. 

\item Under the hypotheses of Proposition \ref{P:product}, one can show that there exist constants $c>0$ and $C<\infty$ such that for all $n\in{\N},$ $x_1\in{G_1}$ and $x_2,y_2\in{G_2},$ the upper bound \eqref{UE} and the lower bound  \eqref{LE} for $p_n((x_1,x_2),(x_1,y_2))$ hold, and for all $n\in{\N},$ $x_1,y_1\in{G_1}$ and $x_2\in{G_2},$ the upper bound
	\begin{equation}
	\label{UHK'}
	p_n((x_1,x_2),(y_1,x_2))\leqslant Cn^{-\frac{\alpha}{\beta}}\exp\Big\{-\Big(\frac{d(x,y)^{\beta}}{Cn}\Big)^{\frac{1}{\beta_1-1}}\Big\}
	\end{equation}
	and the corresponding lower bound  (LHK$(\alpha,\beta,\beta_2/\beta_1)$) for $p_n((x_1,x_2),(y_1,x_2))$ hold. In particular, the estimates \eqref{UE} and (LHK$(\alpha,\beta,\beta_2/\beta_1)$) are the best estimates one can obtain for all $x,y\in{G}.$ We only sketch the proofs since these results will not be needed in the rest of the paper. Between vertices of the type $x=(x_1,x_2)$ and $y=(x_1,y_2),$ one can show that the condition $(D_1)$ holds, and \eqref{LE}=(LHK$(\alpha,\beta,1)$) is then proved as in Remark \ref{remarkendofsection3}, \ref{RkLHK'}), and the upper bound \eqref{UE} is a consequence of \eqref{V+GimpliesUE}. Between the vertices $x=(x_1,x_2)$ and $y=(y_1,x_2),$ one can prove a result similar to \eqref{boundexittime} but for the expected exit time of the cylinder $B'(x,R)=B_{G_1}\big(x_1,R^{\frac{\beta_2}{\beta_1}}\big)\times B_{G_2}\big(x_2,R^{\frac{\beta_2}{\beta_1}}\big),$ and the proof of \eqref{UHK'} is then similar to the proof of \eqref{V+GimpliesUE}, and (LHK$(\alpha,\beta,\beta_2/\beta_1)$) is proved in Remark \ref{remarkendofsection3}, \ref{RkLHK'}) since  $(D_{\frac{\beta_2}{\beta_1}})$ always holds on $G.$ 
	\end{enumerate}
\end{Rk}

We now turn to the proof of \eqref{weakSecIso} for product graphs and the standard $d$-dimensional Sierpinski carpet, $d\geqslant3.$ If $G=G_1\times G_2,$ we say that two vertices $x=(x_1,x_2)$ and $y=(y_1,y_2)$ are $*$-neighbors if and only if both the graph distance in $G_1$ between $x_1$ and $y_1$ and the graph distance in $G_2$ between $x_2$ and $y_2$ are at most $1.$ If $G$ is the standard $d$-dimensional Sierpinski carpet, we say that $x=(x_1,\dots,x_d)$ and $y=(y_1,\dots,y_d)$ in $G$ are $*$-neighbors if and only if there exist $i,j\in{\{1,\dots,d\}}$ such that $|x_i-y_i|\leqslant1,$ $|x_j-y_j|\leqslant1,$ and $x_k=y_k$ for all $k\neq i,j.$ Moreover, we say in both cases that $A\subset G$ is $*$-connected if every two vertices of $A$ are connected by a path of $*$-neighbor vertices. We are going to prove that in these two examples--assuming in the product graph case that $G_1$ and $G_2$ are both connected--the external boundary of any finite and connected subset $A$ of $G$ is $*$-connected. In order to do this, we are first going to prove a property which generalizes Lemma 2 in \cite{MR2996951}, and then apply it to our graphs. In Proposition \ref{generalstarconnected}, we say that $C$ is a cycle of edges if it is a finite set of edges such that every vertex has even degree in $C,$ that $P$ is a path of edges between $x$ and $y$ in $G$ if $x$ and $y$ are the only vertices with odd degree in $P,$ and \textit{we always understand the addition of sets of edges modulo $2.$} We also define for all $x\in{G}$ the set $\partial_{ext}^xA=\{y\in{\partial_{ext}A};\,y\stackrel{A^\mathsf{c}}{\longleftrightarrow} x\}.$

\begin{Prop}
	\label{generalstarconnected}
	Let $\mathcal{C}$ be a set of cycles of edges such that for all finite sets of edges $\mathcal{S}\subset E$ and all cycles of edges $Q,$ 
	\begin{equation}
	\label{conditionstarconnect}
	\text{there exists }\mathcal{C}_0=\mathcal{C}_0(\mathcal{S},Q)\subset\mathcal{C}\text{ with }\mathcal{S}\cap\Big(Q+\sum_{C\in{\mathcal{C}_0}}C\Big)=\emptyset.
	\end{equation}
	Then for all finite and connected sets $A\subset G$ and for all $x\in{A^{\mathsf{c}}},$ the set $\partial_{ext}^xA$ is connected in $G^+,$ the graph with the same vertices as $G$ and where $\{y,z\}$ is an edge of $G^+$ if and only if $y$ and $z$ are both traversed by some $C\in{\mathcal{C}}.$ 
	
	In particular, if $A$ is either a finite and connected subset of $G_1\times G_2$ for two infinite and locally finite connected graphs $G_1$ and $G_2,$ or of the standard $d$-dimensional Sierpinski carpet for $d\geqslant3,$ then $\partial_{ext}A$ is $*$-connected. 
\end{Prop}
\begin{proof}
	Let $A$ be a finite and connected subset of $G,$ and let us fix some $x_0\in{A},$ $x_1\in{A^{\mathsf{c}}},$ and $S_1$ and $S_2$ two arbitrary non-empty disjoint subsets of $G$ such that $\partial_{ext}^{x_1}A=S_1\cup S_2.$ Define $\mathcal{S}_i=\{(x,y)\in{E};\,x\in{A}\text{ and }y\in{S_i}\}$ for each $i\in{\{1,2\}}.$ We will prove that there exists $C\in{\mathcal{C}}$ which contains at least one edge of $\mathcal{S}_1$ and one edge of $\mathcal{S}_2;$ thus by contraposition $\partial_{ext}^{x_1}A$ will be connected in $G^+$ since $S_1$ and $S_2$ were chosen arbitrary. Since $A$ is finite and connected and $S_1$ and $S_2$ are non-empty, there exist two paths $P_1$ and $P_2$ of edges between $x_0$ and $x_1$ such that $P_i \cap \mathcal{S}_i \neq \emptyset$ but $P_i \cap \mathcal{S}_{3-i} = \emptyset$ for all $i\in{\{1,2\}},$ and then $Q=P_1+P_2$ is a cycle of edges. By \eqref{conditionstarconnect}, there exists $\mathcal{C}_0\subset\mathcal{C}$ such that
	\begin{equation*}
	Q'=Q+\sum_{C\in{\mathcal{C}_0}}C
	\end{equation*}
	does not intersect $\mathcal{S}_2.$ Let us define $\mathcal{C}_1=\{C\in{\mathcal{C}_0};\,C\cap\mathcal{S}_1\neq\emptyset\}$ and $\mathcal{C}_2=\mathcal{C}_0\setminus\mathcal{C}_1,$ then
	\begin{equation}
	\label{P_2+sum}
	P_2+\sum_{C\in{\mathcal{C}_2}}C=Q'+P_1+\sum_{C\in{\mathcal{C}_1}}C.
	\end{equation}
	The left-hand side of \eqref{P_2+sum} is a path of edges between $x_0$ and $x_1$ which does not intersect $\mathcal{S}_1$ by definition, and thus it intersects $\mathcal{S}_2.$ Therefore, the right-hand side of \eqref{P_2+sum} intersects $\mathcal{S}_2$ as well, i.e., there exists $C\in{\mathcal{C}_1}$ which intersects $\mathcal{S}_2,$ and also $\mathcal{S}_1$ by definition. 
	\\
	
	We now prove that $\partial_{ext}A$ is $*$-connected when $G=G_1\times G_2,$ for $G_1$ and $G_2$ two infinite and locally finite connected graphs. We start with considering the case that $G_2$ is a tree, i.e., it does not contain any cycle.  We define $\mathcal{C}$ by saying that $C\in{\mathcal{C}}$ if and only if $C$ contains exactly every edge between $(x_1,x_2),$ $(x_1,y_2),$ $(y_1,y_2)$ and $(y_1,x_2)$ for some $x_1\sim y_1\in{G_1}$ and $x_2\sim y_2\in{G_2}.$ Hence, with this choice of $\mathcal{C}$, a set is connected in $G^+$ if and only if it is $*$-connected. Note that since $G_1$ and $G_2$ are infinite, $\partial_{ext}A=\partial_{ext}^xA$ for all $x\in{A^{\mathsf{c}}},$ and thus we only need to prove \eqref{conditionstarconnect}. 
	
	Let $\mathcal{S}$ be a finite set of edges and $Q_0$ be a cycle of edges. We fix a nearest-neighbor path of vertices $\pi=(y_0,y_1,\dots,y_p)\in{G_2^{p+1}}$ such that all the vertices visited by the edges in $Q_0$ are contained in $G_1\times\{\pi\},$ $y_p\notin\{y_0,\dots,y_{p-1}\},$  and $\mathcal{S}\cap(G_1\times\{y_p\})=\emptyset.$ For all $n\in{\{0,\dots,p-1\}}$ and all edges $e=\{(x_1,y_n),(x_2,y_n)\}$, writing $e_1\stackrel{\mathrm{def.}}{=}\{x_1,x_2\}\in{E_1}$ with  $E_1$ denoting the edges of $G_1$,
	we define $C_e^n$ as the unique cycle in $\mathcal{C}$ containing the edges $e$ and $(e_1,y_{n+1})\stackrel{\mathrm{def.}}{=} \{(x_1,y_{n+1}),(x_2,y_{n+1})\}$.
	 Next, we recursively define a sequence $(Q_n)_{n\in{\{0,\dots,p\}}}$ of sets of edges by
	\begin{equation*}
	Q_{n+1}=Q_n+\sum_{e\in{Q_n}\cap(G_1\times\{y_n\})}C_e^n\quad\text{for all }n\in{\{0,\dots,p-1\}}.
	\end{equation*}
	By construction, for all $n\in{\{0,\dots,p-1\}},$ $Q_p$ does not contain any edge in $G_1\times\{y_n\}$ and thus if $e$ is an edge in $Q_p$ of the form $(e_1,y)$ for some $e_1\in{E_1}$ and $y\in{G_2},$ then necessarily $y=y_p.$ Since $Q_p$ is a cycle of edges and since $G_2$ does not have any cycle, $Q_p\subset G_1\times\{y_p\},$ and thus $Q_p\cap\mathcal{S}=\emptyset,$ which gives us \eqref{conditionstarconnect}.
	
	Let us now assume that $G_2$ contains exactly one cycle of edges, and let $\{x_2,y_2\}$ and $\{x_2,z_2\}$ be two different edges of this cycle. Let $A$ be a finite and connected subset of $G,$ then the exterior boundary of $A$ in $G_1\times(G_2\setminus\{x_2,y_2\})$ and the exterior boundary of $A$ in $G_1\times(G_2\setminus\{x_2,z_2\})$ are $*$-connected in $G$ since $G_2\setminus\{x_2,y_2\}$ and $G_2\setminus\{x_2,z_2\}$ do not contain any cycle. First assume that there exists $x_1\in{G_1}$ such that $(x_1,x_2)\in{A},$ $(x_1,y_2)\in{\partial_{ext}A}$ and $(x_1,z_2)\in{\partial_{ext}A},$ then $(x_1,z_2)$ is $*$-connected in $G$ to any vertex of the external boundary of $A$ in $G_1\times(G_2\setminus\{x_2,y_2\})$ and $(x_1,y_2)$ is $*$-connected in $G$ to any vertex of the external boundary of $A$ in $G_1\times(G_2\setminus\{x_2,z_2\}),$ that is $(x_1,y_2)$ and $(x_1,z_2)$ are $*$-connected in $G.$ The other cases are similar, and we obtain that the exterior boundary of $A$ in $G$ is $*$-connected. We can thus prove by induction on the number of cycles that if $G_2$ has a finite number of cycles of edges, then the external boundary of any finite and connected subset $A$ of $G$ is $*$-connected. Otherwise, let $x$ and $y$ be any two vertices in $\partial_{ext}A,$ and let $\pi^x$ be an infinite nearest-neighbor path in $A^\mathsf{c},$ without loops, beginning in $x,$ such that the projection of $\pi^x$ on $G_1$ is a finite path on $G_1,$ i.e.\ constant after some time, and $\pi^y$ be a finite nearest-neighbor path in $A^{\mathsf{c}},$ without loops, beginning in $y$ and ending in $\pi^x.$ Let $G'_2$ be the graph with vertices the projection on $G_2$ of $A\cup\partial_{ext}A\cup\{\pi^x\}\cup\{\pi^y\},$ and with the same edges between two vertices of $G'_2$ as in $G_2.$ By definition $G'_2$ is infinite and only contains a finite number of cycles of edges, so the exterior boundary of $A$ in $G_1\times G'_2$ is $*$-connected in $G_1\times G'_2,$ and thus $x$ and $y$ are $*$-connected in $G.$
	\\
	
	Let us now take $G$ to be the standard $d$-dimensional Sierpinski carpet, $d\geqslant3,$ that we consider as a subset of $\N^d,$ and $A$ a finite and connected subset of $G.$ We define $\mathcal{C}$ as the set of cycles with exactly $4$ edges, and then a set is connected in $G^+$ if and only if it is $*$-connected, thus we only need to prove \eqref{conditionstarconnect}. Let $\mathcal{S}$ be a finite set of edges, $Q_0$ be a cycle of edges, and $p\in{\N}$ such that $Q_0\subset G\cap(\{0,\dots,p-1\}\times\N^{d-1})$ and $\mathcal{S}\cap(\{p\}\times\N^{d-1})=\emptyset.$  We also define $\mathcal{V}_n$ as the set of $d-1$-dimensional squares $V=\{n_2,\dots,n_2+m\}\times\dots\times\{n_{d},\dots,n_{d}+m\}$ such that $\{n\}\times\overline{V}\subset G$ and $(\{n+1\}\times \overline{V})\cap G=\{n+1\}\times(\overline{V}\setminus V),$ where $\overline{V}=\{n_2-1,\dots,n_2+m+1\}\times\dots\times\{n_{d}-1,\dots,n_d+m+1\}.$  Let us now define recursively two sequences $(Q_n)_{n\in{\{0,\dots,p\}}}$ and $(R_n)_{n\in{\{1,\dots,p\}}}$ of cycles of edges such that $Q_n\subset\{n,\dots,p\}\times\Z^{d-1}$ for all $n\in{\{0,\dots,p\}}.$ For each square $V\in{\mathcal{V}_n},$ all the vertices of $\{n\}\times V$ have an even degree in $Q_n\cap(\{n\}\times\overline{V})$ since $Q_n\cap(\{n-1\}\times V)=Q_n\cap(\{n+1\}\times V)=\emptyset$ and $Q_n$ is a cycle of edges. Moreover, since $d\geqslant3,$ every cycle of edges in $\{n\}\times \overline{V}$ is a sum of cycles with exactly $4$ edges in $\{n\}\times\overline{V},$ and thus one can find a set $\mathcal{C}_V\subset\mathcal{C}$ (with $\mathcal{C}_V=\emptyset$ if $(\{n\}\times V)\cap Q_n=\emptyset$) of cycle of edges included in $\{n\}\times\overline{V}$ such that
	\begin{equation*}
	(\{n\}\times \overline{V})\cap\Big(Q_n+\sum_{C\in{\mathcal{C}_V}}C\Big)\subset\{n\}\times(\overline{V}\setminus V).
	\end{equation*}
	We first define $R_{n+1}$ by
	\begin{equation*}
	R_{n+1}=Q_{n}+\sum_{V\in{\mathcal{V}_n}}\sum_{C\in{\mathcal{C}_V}}C.
	\end{equation*}
	By construction, every edge $e=(n,e_1)\in{R_{n+1}\cap(\{n\}\times\Z^{d-1})}$ is such that $(n+1,e_1)\in{G},$ and we then define $C_e^n$ as the unique cycle in $\mathcal{C}$ containing the edges $e$ and $(n+1,e_1),$ and we take
	\begin{equation*}
	Q_{n+1}=R_{n+1}+\sum_{e\in{R_{n+1}\cap(\{n\}\times\Z^{d-1})}}C_e^n.
	\end{equation*}
	By construction, $Q_{n+1}\cap(\{0,\dots,n\}\times\Z^{d-1})=\emptyset$ and since $Q_{n+1}$ is a cycle of edges, we have $Q_{n+1}\subset\{n+1,\dots,p\}\times\Z^{d-1}.$ Therefore, we have $Q_p\cap\mathcal{S}=\emptyset$ by our choice of $p,$ which gives us \eqref{conditionstarconnect}.
\end{proof}

\begin{Rk}
	\leavevmode
	\label{remarkconnectivity}
	\begin{enumerate}[1)]
		\item One can extend Proposition \ref{generalstarconnected} similarly to Theorem 3 in \cite{MR2996951}. Let us assume that there exists $\mathcal{C}$ such that \eqref{conditionstarconnect} holds, and that for each edge $e$ of $E^+\setminus E,$ where $E^+$ is the set of edges of $G^+,$ there exists a cycle $O_e$ of edges of $G^+$ such that $O_e\setminus\{e\}\subset E.$  Then for each finite set $A$ connected in $G^+$ and for all $x\in{A^{\mathsf{c}}},$ the set $\partial^x_{ext}A$ is connected in $G^{++},$ the graph with the same vertices and edges as $G^+$ plus every edge of the type $\{x,y\}$ for $x,y$ both crossed by $O_e$ for some edge $e\in{E^+\setminus E}.$ Indeed let $G^+_A$ be the graph with the same vertices as $G,$ and edge set $E^+_A$ which consists of $E$ plus the edges in $E^+\setminus E$ with both endpoints in $A,$ and let $\mathcal{C}^+_A=\mathcal{C}\cup\{O_e,\,e\text{ edge of }E^+_A\setminus E\}.$ For each cycle $Q$ of edges in $E^+_A$ we then have that
		\begin{equation*}
		Q+\sum_{e\in{Q\setminus G}}O_e
		\end{equation*} 
		is a cycle of edges in $E,$ and thus by \eqref{conditionstarconnect} for $G$ with the set of cycles of edges $\mathcal{C},$ one can easily show that \eqref{conditionstarconnect} also hold for $G^+_A$ with the set of cycles of edges $\mathcal{C}_A^+.$ Since $A$ is connected in $G^+_A,$ by Proposition \ref{generalstarconnected}, $\partial_{ext}^xA$ is connected in $G^{++}.$
		
		In particular, if $G$ is either a product of infinite graphs $G_1\times G_2$ or the $d$-dimensional Sierpinski carpet, $d\geqslant3,$ taking $O_e$ such that $O_e\setminus\{e\}$ only contains two connected edges of $E$ for each $e\in{E^+\setminus E},$ we get that the external boundary of every finite and $*$-connected subset $A$ of $G$ is $*$-connected since $G^{++}=G^+$.
		
		\item Proposition \ref{generalstarconnected} provides us with a stronger result than Lemma 2 in \cite{MR2996951} even when $G=\Z^d,$ $d\geqslant3.$ Indeed, $\Z^d=\Z^{d-1}\times\Z$ and thus the external boundary of every finite and connected (or even $*$-connected) subset of $\Z^d$ is $*$-connected in the sense of product graphs previously defined, i.e., it is connected in $\Z^d\cup\{\{(x,n),(y,n+1)\};\,n\in{\Z},x\sim y\in{\Z}^{d-1}\}.$
		
		\item \label{Menger}An example of a graph $G$ for which we cannot apply Proposition \ref{generalstarconnected}, and in fact where we can find a finite and connected set whose boundary is not $*$-connected, and where \eqref{weakSecIso} does not hold, but where \eqref{Green} and \eqref{Ahlfors} hold, is the Menger sponge. It is defined as the graph associated to the following generalized $3$-dimensional Sierpinski carpet, see Section 2 of \cite{MR1802425}: split $[0,1]^3$ into $27$ cubes of size length $1/3,$ remove the central cube of each face as well as the central cube of $[0,1]^3,$ and iterate this process for each remaining cube. It is easy to show that $G$ endowed with the graph distance verifies \eqref{Ahlfors} with $\alpha=\frac{\log(20)}{\log(3)},$ and \eqref{Green} follows from Theorem 5.3 in \cite{MR1802425} since the random walk on the Menger sponge is transient, see p.741 of \cite{MR1701339}. One can then easily check that taking $A_n=(3^n/2,5\times3^n/2)^3\cap G,$ where we see $G$ as a subset of $\R^3,$ then $\partial_{ext}A$ is not $*$-connected. In fact for each $p<n,$ there exists $x\in{\partial_{ext}A_n}$ such that there is no $3^p$-path between $x$ and $B(x,2\times3^p)^{\mathsf{c}},$ and thus \eqref{weakSecIso} does not hold. 
	\end{enumerate}
\end{Rk}
	
	We can now conclude that our main results apply to the examples mentioned in the introduction.

\begin{Cor}

\label{C:examples}
The graphs in \eqref{intro:Gex1} (endowed with unit weights) satisfy \eqref{p0}, \eqref{Ahlfors}, \eqref{Green}, for some $\alpha > 2$, $\beta \in [2,\alpha)$ and \eqref{weakSecIso}, with respect to a suitable distance function $d(\cdot,\cdot)$. In particular, the conclusions of Theorems \nolinebreak\ref{T:mainresultGFF} and \nolinebreak\ref{T:mainresultRI} hold for these graphs.
\end{Cor}
\begin{proof}
Condition \eqref{p0} holds plainly in all cases since all graphs in \eqref{intro:Gex1} have unit weights and uniformly bounded degree. For $G_1$, we classically have $\alpha=d$, $\beta=2$ and \eqref{weakSecIso} follows e.g.\ from Proposition \ref{generalstarconnected} with $d=d_G$ (or even the $\ell^{\infty}$-norm) since $\Z^d=\Z^{d-1}\times\Z.$ The case of $G_2$ is an application of Propositions \ref{P:product} and \ref{generalstarconnected}: it is known \cite{MR2177164, MR1378848} that $G'$, the discrete skeleton of the Sierpinski gasket, satisfies ($V_{\alpha_2}$) and ($G_{\beta_2}$) with $\alpha_2= \frac{\log 3}{\log 2}$ and $\beta_2 =\frac{\log 5}{\log 2}$, whence \eqref{Ahlfors}, \eqref{Green}, hold for $G_2$ with respect to $d$ in \eqref{prod11}, for $\alpha=\frac{\log45}{2\log2}$ and $\beta=\frac{\log 5}{\log 2}$ as given by \eqref{prod11} with $\alpha_1=1$, $\beta_1=2$ (note that $\alpha_2>1$ so \eqref{prod10} holds), and it is easy to see that any $*$-connected path is also a $1$-path for $d$ in \eqref{prod11}, hence \eqref{weakSecIso} holds. Regarding $G_3$, the standard $d$ dimensional graphical Sierpinski carpet endowed with the graph distance, with $d\geq3$ (cf.\ p.6 of \cite{MR1802425}), $\alpha=\log(3^d-1)/\log(3)$ (with $d=d_G$) and \eqref{Green} then follows from Theorem 5.3 in \nolinebreak\cite{MR1802425} since the random walk on $G_3$ is transient for $d\geq3$, see p.741 of \cite{MR1701339}. We refer to Remark~\ref{R:extensions}, \ref{R:boundsierpinskicarpet}) below for bounds on the value of $\beta$. Moreover, \eqref{weakSecIso} on $G_3$ follows from Proposition \ref{generalstarconnected} since any $*$-connected path in $G_3$ is also a $2$-path. 

Finally, $G_4$ endowed with the graph distance $d=d_{G_4}$ satisfies \eqref{Ahlfors} for some $\alpha > 2$ by assumption and \eqref{Green} holds with $\beta=2$ by Theorem 5.1 in \cite{MR1217561}. To see that \eqref{weakSecIso} holds, we first observe that the group $\Gamma =\langle S \rangle$ which has $G_4$ as a Cayley graph is finitely presented. Indeed, by a classical theorem of Gromov \cite{gromov1981groups}, $\Gamma$ is virtually nilpotent, i.e., it has a normal subgroup $H$ of finite index which is nilpotent. Furthermore, $H$ is finitely generated (this is because $\Gamma/H$ is finite, so writing $gH$, $g \in C$ with $|C|< \infty$ and $1\in{C}$ for all the cosets, one readily sees that $H=\langle \{h \in H ; \, h = g^{-1}sg' \text{ for some $g,g' \in C$ and some $s \in S$} \}\rangle$).

Since $H$ is nilpotent and finitely generated, it is in fact finitely presented, see for instance 2.2.4 (and thereafter) and 5.2.18 in \cite{MR1357169}, and so is $\Gamma/H$, being finite. Together with the normality of $H$ one straightforwardly deduces from this that $\Gamma$ is finitely presented, see again 2.2.4 in \cite{MR1357169}. As a consequence $\Gamma=\langle S| R\rangle$ for a suitable finite set of relators $R$. This yields a generating set of cycles for $G_4$ of maximal cycle length $t< \infty$, where $t$ is the largest length of any relator in $R$, and Theorem 5.1 of \cite{MR2286517} (alternatively, one could also apply Proposition \ref{generalstarconnected}) readily yields that, for all $x\in{\partial_{ext}A},$ every two vertices of $\partial_{ext}^xA$ are linked via an $R_0$ path in $\partial^x_{ext}A,$ with $R_0=t/2.$ Moreover, since $G$ has sub-exponential growth, $\{\partial_{ext}^xA,\,x\in{\partial_{ext}A}\}$ contains at most two elements, see for instance Theorem 10.10 and 12.2, (g), in \cite{MR1743100} and, since $G$ does not have linear growth, in fact only $1,$ see for instance Lemma 5.4, (a), and Theorem 5.12 in \cite{MR753005}. We also prove this fact for any graph satisfying \eqref{eq:Ass} in the course of proving Lemma \ref{Lemmaharnack}, see for instance below \eqref{eq:2.11}.

In order to prove \eqref{weakSecIso}, we thus only need to show that there exists $c>0$ such that $\delta(\partial_{ext}A)\geqslant c\delta(A)$ for all finite and connected subgraphs $A$ of $G,$ and we are actually going to show this inequality in the general setting of vertex-transitive graphs $G$. Write $m\stackrel{\mathrm{def.}}{=}\delta(\partial_{ext}A),$ let us fix some $x_0\in{\partial_{ext}A},$ and for $x\in{G}$ introduce
\begin{equation*}
\overline{B}(x,m)=\{y\in{G};\,\text{every unbounded path beginning in }y\text{ intersects }B(x,m)\}.
\end{equation*}
 Let us assume that there exists $x_1\in{\overline{B}(x_0,m)}$ such that $B(x_1,m)\cap B(x_0,m)=\emptyset,$ and then we have $\overline{B}(x_1,m)\subset\overline{B}(x_0,m)\setminus B(x_0,m).$ Since $G$ is vertex-transitive, there exists $x_2\in{\overline{B}(x_1,m)}$ such that $B(x_2,m)\cap B(x_1,m)=\emptyset.$ Moreover, by definition, $\overline{B}(x_2,m)\subset\overline{B}(x_1,m)\setminus B(x_1,m),$  and $x_1\leftrightarrow x_2$ in $\overline{B}(x_1,m).$  Iterating this reasoning, we can thus construct recursively a sequence $(x_n)_{n\in{\N}}$ of vertices such that $\overline{B}(x_{n+1},m)\subset\overline{B}(x_{n},m)\setminus B(x_{n},m),$ and  $x_n\leftrightarrow x_{n+1}$ in $\overline{B}(x_{n},m)$ for all $n\in{\N}.$ Therefore, there exists an unbounded path beginning in $x_1$ in $\overline{B}(x_0,m)\setminus B(x_0,m),$ which is a contradiction by definition of $\overline{B}(x_0,m).$ Hence, $\delta(\overline{B}(x_0,m))\leqslant4m,$ and so $\delta(A)\leqslant4\delta(\partial_{ext}A).$
\end{proof}

\begin{Rk}
\leavevmode
\label{R:extensions}
\begin{enumerate}[1)]
\item\label{R:boundsierpinskicarpet} To the best of our knowledge, the explicit value of the constant $\beta$ for the $d$-dimensional Sierpinski carpet $G_3$, $d\geq 3$, is not known. However, Remark~5.4 2.\ in \cite{MR1701339} with the choices $l_F=a=3$, $b=1$ and $m_F=3^d-1$, provides lower and upper bounds for the so-called resistance scale factor $\rho_F.$ These bounds can be plugged into (5.3) of \cite{MR1701339} to obtain bounds on the constant $d_w$ appearing therein, which is in fact equal to $\beta$ in view of Theorem~5.3 in \cite{MR1802425}, and supply us with
\begin{equation*}
\nu=\alpha-\beta\in{\frac{\log(3^{d-1}(3^{d-1}-1))}{\log3}+\left[-d,-\frac{\log(3^d-2)}{\log3}\right]},
\end{equation*}
and in particular $\nu\in{(d-3,d-2)}$ for all $d\geq3$.
\item The conclusions of Theorems \ref{T:mainresultGFF} and \ref{T:mainresultRI} do not only hold for $G_2$ in \eqref{intro:Gex1}, but also for any product graph $G_1\times G_2$ under the same hypotheses as in Proposition \nolinebreak\ref{P:product}. Further interesting examples can be generated involving graphs $G$ endowed with a distance $d\neq d_G$ which is not of the form of a product of graph distances as in \eqref{prod11}. For instance, in Corollary 4.12 of \cite{MR2112125}, estimates similar to (UHK$(\alpha',\alpha'+1)$) and (LHK$(\alpha',\alpha'+1,\zeta)$) for some $\alpha'>1$ and $\zeta\in{[1,\alpha'+1)}$ are proved for different recurrent fractal graphs $G'$ when the distance $d'$ on $G'$ is the effective resistance as defined in (2.4) of \cite{MR2112125}. By Lemma 3.2 in \cite{MR2112125}, $(V_{\alpha'})$ holds on $G'$ endowed with the distance $d',$ and thus one can then prove similarly as in the proof of Proposition \ref{P:product} that $G=G'\times\Z$ (or some other product with an infinite graph satisfying \eqref{UE} and \eqref{NLHK}) satisfy \eqref{Ahlfors} and \eqref{Green} with $\alpha=\frac{3\alpha'+1}{2}$ and $\beta=\alpha'+1$ for the distance 
\begin{equation*}
	d((x',n),(y',m))=d'(x',y')\vee |n-m|^{\frac{2}{\beta}}\quad\text{for all }x',y'\in{G'}\text{ and }n,m\in{\Z}.
\end{equation*}
Moreover, \eqref{weakSecIso} is also verified on $G$ by Proposition \ref{generalstarconnected}, and thus the conclusions of Theorems \ref{T:mainresultGFF} and \ref{T:mainresultRI} hold for $G.$ It should be noted that $d'$ is not always equivalent to the graph distance on $G',$ see for instance the graph $G'$ considered in Corollary 4.16 of \cite{MR2112125}. This graph is also another example of a graph where \eqref{distancezeta} hold for some $\zeta>1$ but not $\zeta=1,$ and where the estimates \eqref{UE} and \eqref{LHK'} are optimal at this level of generality.

\end{enumerate}
\end{Rk}

\section{Strong connectivity of the interlacement set}
\label{secconnec}
We now prove a strong connectivity result for the random interlacement set on the cable system, Proposition \ref{connectivity} below; see also Proposition 1 in \cite{MR2819660} and Lemma 3.2 in \cite{DrePreRod} for similar findings in the case $G = \Z^d.$ We recall our standing assumption \eqref{eq:Ass}. The availability of controls on the heat kernel and exit times provided by Proposition \ref{someproperties} will figure prominently in obtaining the desired estimates; see also Remark \ref{R:condconn} below. The connectivity result will play a crucial role in Section \nolinebreak\ref{denouement}, where $\tilde{\I}^u$ will be used as a random network to construct certain continuous level-set paths for the free field. We recall the notation introduced in \eqref{simConn} and \eqref{ballsalmostconnected}, and our standing assumptions \eqref{eq:Ass}.
\begin{Prop}
	\label{connectivity}
	For each $u_0>0,$ there exist constants $\Cl[c]{cStrong}>0,$ $c>0$ and $C<\infty$ all depending on $u_0$ such that, for all $x_0\in{G},$ $u\in{(0,u_0]}$ and $L\geq1,$\phantom{$\Cl{CStrong}$}
	\begin{equation}
	\label{strong}
	\tilde{\P}^I\Big(\bigcap_{z,z'\in\tilde{\I}^u\cap \tilde{B}(x_0,L)}\big\{z\stackrel{\sim}{\longleftrightarrow}z'\ in\ \tilde{\I}^{u}\cap \tilde{B}\left(x_0,2\Cr{Cstrong}L\right) \big\}\Big) 
	\geq 1-C\exp\left\{-cL^{\Cr{cStrong}}u\right\}.
	\end{equation}
\end{Prop}

The proof of Proposition \ref{connectivity} requires some auxiliary lemmas and appears at the end of the section. In the rest of the paper, we will not use directly Proposition \ref{connectivity} because the event in \eqref{connectivity} is neither increasing nor decreasing, see above \eqref{defevents}, and therefore cannot be used in the decoupling inequalities, see Theorem \ref{decoup}. We will however use two auxiliary lemmas which together readily imply Proposition \ref{connectivity}, namely Lemmas \ref{12} and \ref{10}. Another interest of Proposition \ref{connectivity} is the following corollary, which is a generalization of Corollary 2.3 of \cite{MR2680403} from $\Z^d$ to $G$ as in \eqref{eq:Ass}.
\begin{Cor} \label{RIconnected}
	Let $u>0$. Then
	$\widetilde {\P}^I$-a.s., the subset $\widetilde {\mathcal I}^u$  of $\widetilde G$ is unbounded and connected.
	Analogously, ${\P}^I$-a.s., the subset $ {\mathcal I}^u$  of $G$ is infinite and connected.
\end{Cor}
\begin{proof}[Proof of Corollary \ref{RIconnected}]
	Fix any vertex $x_0 \in G$. Let $A_L$ denote the event appearing on the left-hand side of \eqref{strong}, and $A_L'= \{\widetilde {\mathcal I}^u \cap  \tilde{B}\left(x_0,L\right) \neq \emptyset\}$. Note that $\{ \widetilde {\mathcal I}^u \text{ is unbounded, connected} \} \supset ( \bigcup_L A_L' )\cap \liminf_L A_L$. The events $A_L'$ are increasing with $\lim_L \widetilde {\P}^I(A_L')=1$ by \eqref{ballcapacity}, and by \eqref{strong} and a Borel-Cantelli argument, $\widetilde {\P}^I(\liminf_L A_L)=1$. The same reasoning applies also to $ {\mathcal I}^u$ (with \eqref{strongonedges} below in place of \nolinebreak\eqref{strong}).
\end{proof}

Let us denote for each $u>0$ by $\hat{\I}^u$ the set of edges of $G$ traversed by at least one of the trajectories in the trace of the random interlacement process $\omega^u.$ From the construction of the random interlacement process on the cable system $\tilde{G}$ from the corresponding process on $G$ by adding Brownian excursions on the edges, it follows that the inequality
\begin{equation}
\label{strongonedges}
\P^I\Big(\bigcap_{x,y\in{\I}^u\cap B(x_0,L)} \big\{x\stackrel{\wedge}{\longleftrightarrow} y\ \text{ in } \hat{\I}^{u}\cap B_E\left(x_0,2\Cr{Cstrong}L\right)\big\}\Big) \geq 1-C(u_0)\exp\left\{-L^{c(u_0)}u\right\}
\end{equation}
for all $u\leq u_0$, will entail \eqref{strong},
where for $x,y\in{G}$ and $A\subset E,$ $\{x\stackrel{\wedge}{\longleftrightarrow} y \text{ in }A\}$ means that there exists a nearest neighbor path from $x$ to $y$ crossing only edges contained in \nolinebreak$A.$ We refer to the discussion at the beginning of the Appendix of \cite{DrePreRod} for a similar argument on why \eqref{strongonedges} implies \eqref{strong}. In order to prove \eqref{strongonedges}, we will apply a strategy inspired by the proof of Proposition 1 in \cite{MR2819660} for the case $G=\Z^d.$ 

For $U\subset\subset G$ let $N_U^u$ be the number of trajectories in supp$(\omega^u)$ which enter $U.$ By definition, $N_U^u$ is a Poisson variable with parameter $u\mathrm{cap}(U),$ and thus there exist constants $c,C \in (0,\infty)$ such
that uniformly in $u \in (0,\infty)$,
\begin{equation}
\label{poisson}
\P^I\big(c u \cdot \mathrm{cap}(U)\leq N_U^u\leq Cu \cdot  \mathrm{cap}(U)\big)\geq1-C\exp\left\{-cu \cdot \mathrm{cap}(U)\right\},
\end{equation}
cf.\@ display (2.11) in \cite{MR2819660}.
We now state a lemma which gives an estimate in terms of capacity for the probability to link two subsets of $B(x,L)$ through edges in $\hat{\I}^u\cap B(x,\Cr{Cstrong}L).$

\begin{Lemme}
	\label{12}
	There exist constants $c \in (0,1)$ and $C\in [1,\infty)$ such that for all $L\geq1$, $u>0$ and  all subsets $U$ and $V$ of $B(x,L),$ 
	\begin{equation}
	\label{12.0}
	\P^I\big(U\stackrel{\wedge}\longleftrightarrow V\text{ in }\hat{\I}^u\cap B_E(x,\Cr{Cstrong}L)\big)\geq1-C\exp\left\{-cL^{-\nu}u\mathrm{cap}(U)\mathrm{cap}(V)\right\},
	\end{equation}
	with $\nu$ as in \eqref{eq:nu}.
\end{Lemme}

\begin{proof}
	For $U$ not to be connected to $V$ through edges in $\hat{\I}^u\cap B_E(x,\Cr{Cstrong}L),$ all of the $N_U^u$ trajectories hitting $U$ must not hit $V$ after hitting $U$ and before leaving $B(x,\Cr{Cstrong}L)$, so 
	\begin{align}
		\label{12.1}
		\begin{split}
			&\P^I\left(U\stackrel{\wedge}{\longleftrightarrow} V\text{ in }\hat{\I}^u\cap B(x,\Cr{Cstrong}L)\right)\\
			&\quad \geq 1-\P^I(N_U^u<cu\mathrm{cap}(U))-\left(P_{\overline{e}_U}(H_V>T_{B(x,\Cr{Cstrong}L)})\right)^{cu\mathrm{cap}(U)}
		\end{split}
	\end{align}
	(recall \eqref{Pmu} and \eqref{normalized} for notation). For all $y\in{B(x,L)},$ by \eqref{entrancegreenequi}, \eqref{dforCstrong} and \eqref{Greenstoppedbound},
	\begin{equation}
	\label{12.2}
	P_y(H_V>T_{B(x,\Cr{Cstrong}L)})\leq1-\sum_{z\in{B(x,L)}}g_{B(x,\Cr{Cstrong}L)}(y,z)e_V(z)\leq 1-\frac{\Cr{cGreen}}{2}(2L)^{-\nu}\mathrm{cap}(V),
	\end{equation}
	where we also used $e_V\le e_{V,B(x,\Cr{Cstrong}L)}$ in the first inequality.
	Since $\mathrm{cap}(V)\leq \Cr{Ccapacity}L^{\nu}$ by \eqref{ballcapacity}, we can combine \eqref{12.1}, \eqref{poisson} and \eqref{12.2} to get \eqref{12.0}.
\end{proof}

For each $x\in{G}$ and $L\geqslant1,$ if $x\in{\I^u},$ we denote by $C^u(x,L)$ the set of vertices in $G$ connected to $x$ by a path of edges in $\hat{\I}^u\cap B_E(x,L),$ and we take $C^u(x,L)=\emptyset$ otherwise. On our way to establishing \eqref{strongonedges} we introduce the following thinned processes. For each $i\in{\{1,2,3\}},$ let $\omega^{u/3}_i$ be the Poisson point process which consists of those trajectories in $\omega^u$ which have label between $(i-1)u/3$ and $iu/3.$ I.e., $\omega^{u/3}_i,$ $i\in{\{1,2,3\}},$ have the same law as three independent random interlacement processes at level $u/3$ on $G.$ For each $i\in{\{1,2,3\}},$ let $\I^{u/3}_i$ and $\hat{\I}^{u/3}_i,$ respectively, be the set of vertices and edges, respectively, crossed by at least one trajectory in supp($\omega^{u/3}_i),$ and for each $x\in{G}$ and $L>0,$ let $C_i^{u/3}(x,L)$ be the set of vertices connected to $x$ by a path of edges in $\hat{\I}^{u/3}_i\cap B_E(x,L).$ Note that $\P^I$-a.s.\@ we have $\I^u=\cup_{i=1}^3{\I}^{u/3}_i$ and $\hat{\I}^u=\cup_{i=1}^3\hat{\I}^{u/3}_i.$ Now fix some $x_0\in{G}$ and $L>0,$ and assume there exist $x,y\in{\I^u}\cap B\left(x_0,L\right)$ such that $x$ is not connected to $y$ through edges in $\hat{\I}^{u}\cap B_E\left(x_0,2\Cr{Cstrong}L\right).$ Let $i,j\in{\{1,2,3\}}$ be such that $x\in{\I^{u/3}_i}$ and $y\in{\I^{u/3}_j}$, and let $k=k(i,j)$ be the smallest number in ${\{1,2,3\}}$ different from $i$ and $j$. By definition, $C_i^{u/3}(x,L)$ is not connected to $C_j^{u/3}(y,L)$ through edges in $\hat{\I}_k^{u/3}\cap B_E(x_0,2\Cr{Cstrong}L)$, and so
\begin{align}
	\label{decompositioninter}
	&\P^I\left(x,y\in\I^u,\left\{x\stackrel{\wedge}\longleftrightarrow y\text{ in }\hat{\I}^{u}\cap B_E(x_0,2\Cr{Cstrong}L)\right\}^{\mathsf{c}}\right)\nonumber
	\\&\leq\sum_{i,j=1}^3{\P^I\left(
	\begin{array}{c}
	x\in\I_{i}^{u/3},y\in\I_{j}^{u/3},
	\\\left\{C_i^{u/3}(x,L)\stackrel{\wedge}{\longleftrightarrow} C_j^{u/3}(y,L)\text{ in }\hat{\I}_k^{u/3}\cap B_E(x_0,2\Cr{Cstrong}L)\right\}^{\mathsf{c}}
	\end{array}
	\right)}.
\end{align}
Since $\hat{\I}_k^{u/3}$ is independent from $\hat{\I}_i^{u/3}$ and $\hat{\I}_j^{u/3}$ and $C_i^{u/3}(x,L)\subset B(x_0,2L),$ we can use Lemma \ref{12} to upper bound the last probability in \eqref{decompositioninter}. In order to obtain \eqref{strongonedges}, we now need a lower bound on the capacity of $C_i^{u/3}(x,L),$ and for this purpose we begin with a lower bound on the capacity of the range of $N$ random walks. For each $N\in{\N}$ and $S_N=(x_1,\dots,x_N)\in{G^N}$ we define a sequence $(Z^i)_{i\in{\{1,\dots,N\}}}$ of independent random walks on $G$ with fixed initial point $Z^i_0=x_i$ under some probability measure $P^{S_N},$ i.e., for each $i\in{\{1,\dots,N\}},$ $Z^i$ has the same law under $P^{S_N}$ as $Z$ under $P_{x_i}.$ For all positive integers $M$ and $N$ we define the trace $T(N,M)$ 
on $G$ of the $N$ first random walks up to time $M$ by
\begin{equation*}
	T(N,M)\stackrel{\text{def.}}{=}\bigcup_{i=1}^N\bigcup_{p=0}^{M-1}\{Z^i_p\}.
\end{equation*}
For ease of notation, we also set 
\begin{equation}
\label{F(n,gamma)}
\gamma=\frac{\alpha}{\beta}>1\qquad\text{and}\qquad F_{\gamma}(M)=\left\{
\begin{array}{ll}
M^{2-\gamma}&\text{ if }\gamma<2,\\
\log(M)&\text{ if }\gamma=2,\\
1&\text{ otherwise,}
\end{array}\right.
\end{equation}
with $\alpha$ and $\beta$ from \eqref{Ahlfors} and \eqref{Green}. The function $F_{\gamma}$ reflects the fact that the ``size'' of $\{Z_n ; \, n \ge 0\}$ (as captured by $
\beta,$ see Lemma \ref{expectedexittime}) becomes increasingly small relative to the overall geometry of $G$ (controlled by $\alpha$) as $\gamma$ grows. As a consequence,
intersections between independent walks in $\mathcal{I}^u$ are harder to produce for larger $\gamma$. This is implicit in the estimates below.
\begin{Lemme}
	\label{firstboundcap}
	There exists $C<\infty$ such that for all $t>0,$ positive integers $N$ and $M,$ and starting points $S_N\in{G^N},$
	\begin{equation}
	\label{firstboundcap0}
	P^{S_N}\left(\mathrm{cap}\big(T(N,M)\big)\leq t\min\left(\frac{NM}{F_\gamma(M)},M^{\gamma-1}\right)\right)\leq Ct.
	\end{equation}
	
\end{Lemme}
\begin{proof}
	Consider positive integers $N$ and $M,$ and $S_N\in{G^N}.$ By Markov's inequality,
	\begin{align}
		\label{doublesum}
		\begin{split}
			&P^{S_N}\left(\mathrm{cap}\big(T(N,M)\big)\leq t\min\left(\frac{NM}{F_\gamma(M)},M^{\gamma-1}\right)\right)\\
			&\quad \leq t\min\left(\frac{NM}{F_\gamma(M)},M^{\gamma-1}\right)E^{S_N}\left[\mathrm{cap}\big(T(N,M)\big)^{-1}\right].
		\end{split}
	\end{align}
	Applying \eqref{variational} with the probability measure $\mu=\frac{1}{(M-\lceil M/2\rceil+1)N}\sum_{i=1}^N\sum_{p=\lceil M/2\rceil}^M\delta_{Z_p^i}$, which has support in $T(N,M),$ yields
	\begin{equation}
	\label{firstboundcap1}
	E^{S_N}\big[\mathrm{cap}(T(N,M))^{-1}\big]\leq E^{S_N}\Big[\frac{C}{N^2M^2}\sum_{i,j=1}^N\sum_{p,q=\lceil M/2\rceil}^{M-1}g\big(Z^i_p,Z^j_q\big)\Big].
	\end{equation}
	Moreover, using the heat kernel bound \eqref{due} and the Markov property at time $p$, we have uniformly in all $p\in{\N}$ and $x,y\in{G},$
	\begin{equation}
	\label{f_p(y)}
	f_p^x(y)\stackrel{\text{def.}}{=}E_x\left[g(Z_p,y)\right]=\sum_{n=p}^{\infty}p_n(x,y)\leq C\sum_{n=p}^{\infty}n^{-\gamma}\leq Cp^{1-\gamma},
	\end{equation}
	and, thus, for $p<q$ and every $i\in{\{1,\dots,M\}}$, with $\hat{P}_{\cdot}$ an independent copy of $P_{\cdot}$ governing the process $\hat{Z}$, using symmetry of $g(\cdot,\cdot)$,
	\begin{equation}
	\label{firstboundcap2}
	E_{x_i}\big[g(Z_p^i,Z_q^i)\big]= E_{x_i}\Big[ \hat{E}_{Z_p^i}[g(\hat{Z}_0, \hat{Z}_{q-p})\Big] =    E_{x_i}\big[ f_{q-p}^{Z_p^i}(Z_p^i) \big] \stackrel{\eqref{f_p(y)}}{\leq} C(q-p)^{1-\gamma},
	\end{equation}
	and the same upper bound applies to $E_{x_i}\left[g(Z_q^i, Z_p^i)\right]$, again by symmetry of $g$.
	Considering the on-diagonal terms in the first sum on the right-hand side of \eqref{firstboundcap1}, we obtain
	\begin{equation}
	\label{firstboundcap3}
	\begin{split}
	E^{S_N}\bigg[\sum_{i=1}^N\sum_{p,q=\lceil M/2\rceil}^{M-1}g(Z^i_p,Z^i_q)\bigg]&\leq 2N\max_{i\in\{1,\dots,N\}}E^{S_N}\bigg[\sum_{\stackrel{p,q=\lceil M/2\rceil}{p\leq q}}^{M-1}g(Z_p^i,Z_q^i)\bigg]
	\\&\stackrel{\eqref{firstboundcap2}}{\leq} CNM\Big(1+\sum_{k=1}^{\lceil M/2\rceil}k^{1-\gamma}\Big) \stackrel{\eqref{F(n,gamma)}}{\leq} CNMF_{\gamma}(M).
	\end{split}
	\end{equation}
	For $i\neq j$ on the other hand, \eqref{f_p(y)} implies
	\begin{equation*}
	\begin{split}
		E^{S_N}\bigg[\sum_{p,q=\lceil M/2\rceil}^{M-1}g\left(Z^i_p,Z^j_q\right)\bigg]= \sum_{p,q=\lceil M/2\rceil}^{M-1}E^{S_N}\left[f_p^{x_i}\left(Z^j_q\right)\right]&\leq CM\sum_{p=\lceil M/2\rceil}^{M-1}p^{1-\gamma}
		\\&\leq CM^{3-\gamma}.
	\end{split}
	\end{equation*}
	Combining this with \eqref{doublesum}, \eqref{firstboundcap1} and \eqref{firstboundcap3} yields \eqref{firstboundcap0}.
\end{proof}

We now iterate the bound from Lemma \ref{firstboundcap} over the different parts of the random walks $(Z^i)_{i\in{\{1,\dots, N\}}}$ in order to improve it. 
\begin{Lemme}
	\label{Lemma6}
	For each $\eps\in{(0,1)},$ there exist constants $c(\varepsilon)>0$ and $C(\varepsilon)\in [1, \infty)$ such that for all positive integers $N$ and $M,$ and $S_N\in{G^N}$,
	\begin{equation}
	\label{Lemma6.0}
	P^{S_N}\left(\mathrm{cap}\big(T(N,M)\big)\leq c \kappa\right)\leq C\exp\{-cM^{\eps}\},
	\end{equation}
	where 
	\begin{equation}
	\label{eq:Ftilde}
	\kappa= \kappa (N,M, \gamma, \varepsilon) = \min\left(\frac{NM^{1-\eps}}{F_\gamma(M^{1-\eps})},M^{(\gamma-1)(1-\eps)}\right).
	\end{equation}
\end{Lemme}
\begin{proof}
	For $\eps\in{(0,1)},$ all positive integers $N,$ $M$ and $k$, we define
	\begin{equation*}
		T_k(N,M)=\bigcup_{i=1}^N\bigcup_{p=(k-1)M}^{k M-1}\{Z^i_p\}.
	\end{equation*}
	By the Markov property and Lemma \ref{firstboundcap}, for all $t>0,$ 
	$\eps\in{(0,1)}$ and $S_N\in{G^N}$, with $\mathcal{F}_k^{N,M}=\sigma(Z_p^i,1\leq i\leq N,1\leq p\leq(k-1)\lceil M^{1-\varepsilon}\rceil)$, 
	\begin{equation}
	\label{Lemma6.1}
	P^{S_N}\left(\mathrm{cap}\big(T_k(N,\lceil M^{1-\eps}\rceil)\big)\leq t \kappa \,\Big|\,\mathcal{F}_k^{N,M}\right)\leq Ct.
	\end{equation}
	Moreover,
	\begin{equation*}
		\bigcup_{k=1}^{\lfloor M^{\eps}/2\rfloor}T_k(N,\lceil M^{1-\eps}\rceil)\subset T(N,M),
	\end{equation*}
	whence $\mathrm{cap}\big(T(N,M)\big)\leq L$ implies $\mathrm{cap}\big(T_k(N,\lceil M^{1-\eps}\rceil)\big)\leq L$ for all $1\leq k \leq \lfloor M^{\eps}/2\rfloor$ by the monotonicity property \eqref{capincrease}. Thus, applying the Markov property and using \eqref{Lemma6.1} inductively we obtain 
	\begin{equation*}
		P^{S_N}\left(\mathrm{cap}\big(T(N,M)\big)\leq t\kappa\right)\leq \left(Ct\right)^{\lfloor M^{\eps}/2\rfloor}\leq \exp\{-cM^{\eps}\}
	\end{equation*}
	for all $t$ small enough and $M\geq 2^{1/\eps}$. This yields \eqref{Lemma6.0}.
\end{proof}

The next step is to transfer the bound in Lemma \ref{Lemma6} from the trace on $G$ of $N$ independent random walks to a subset of the random interlacement set. For all $u>0$
 and $A \subset\subset G,$
conditionally on the number  $N_A^u$ of trajectories in supp($\omega^u$) which hit $A,$
 let $S_A^u\in{G^{N_A^u}}$ be the family of entrance points in $A$ by trajectories in the support of the random interlacement process $\omega^u$ on $G.$  With a slight abuse of notation, we identify $Z^1,\dots,Z^{N_A^u}$ under $P^{S_A^u}$ with the forward (seen from the first hitting time of $A$) parts of the trajectories in supp($\omega^u$) which hit $A$ under $\P^I(\cdot\,|\,S_A^u).$  We define $\Psi(u,A,M)=T(N_A^u,M)$ for all positive integers $M.$

\begin{Lemme}
	\label{firstboundtry}
	For each $u_0>0$ and $\eps\in{(0,1)},$ there exist constants $c'=c'(\eps)>0$ independent of $u_0,$ $c(u_0,\eps)>0$ and $C(u_0,\eps)<\infty$ such that for all $u\in{(0,u_0]}$,  $A\subset\subset G,$ $x\in{G},$ and positive integers $M,$ with $\tilde{\kappa}_{u,A}\stackrel{\text{def.}}{=} \kappa (u\mathrm{cap}(A), M, \gamma,\varepsilon)$ (cf.\@ \eqref{eq:Ftilde}),
	\begin{align}
		\label{boundoncappsi}
		\begin{split}
			&\P^I\left(\mathrm{cap}\big(\Psi(u,A,M)\big)\leq c' \tilde{\kappa}_{u,A}  \right) \leq C\exp\big\{-c\left(u\mathrm{cap}(A)\wedge M^{\eps}\right)\big\},
		\end{split}
	\end{align}
	and for all positive integers $k,$ if $A\subset B\big(x,kM^{\frac{1+\eps}{\beta}}\big)$ (with $\beta$ as in \eqref{Green}),
	\begin{equation}
	\label{boundonpsi}
	\P^I\left(\Psi(u,A,M)\not\subset B\big(x,(k+1)M^{\frac{1+\eps}{\beta}}\big)\right)\leq Ck^{\nu}\exp\big\{-cM^{\frac{\eps(\nu\wedge1)}{\beta}}u\big\}.
	\end{equation}
\end{Lemme}

\begin{proof}
	Writing, with $N= \lceil cu\mathrm{cap}(A)\rceil $, 
	\begin{align*}
		& \P^I\left(\mathrm{cap}\big(\Psi(u,A,M)\big) \leq c' \tilde{\kappa}_{u,A} \right)  \leq \P^I\big(N_A^u< N\big) + \sup_{S_N }P^{S_N}\left(\mathrm{cap}\big(T(N,M)\big)\leq c' \tilde{\kappa}_{u,A} \right),
	\end{align*}
	the inequality \eqref{boundoncappsi} easily follows from the Poisson bound \eqref{poisson} and Lemma \ref{Lemma6}. We turn to the proof of (\ref{boundonpsi}), and we fix $x\in{G},$ $\eps\in{(0,1\wedge(\gamma-1))}$ as well as positive integers $k$ and \nolinebreak$M.$ Let us write $A_k=B\big(x,kM^{\frac{1+\eps}{\beta}}\big)$ to simplify notation. If $\Psi(u,A_k,M)\not\subset A_{k+1},$ then for at least one trajectory $Z^i$ among the forward trajectories $Z^1,\dots,Z^{N_{A_k}^u}$ in supp($\omega^u$) which hit $A_k,$ the walk $Z^i$ will leave $B\big(Z^i_0,M^{\frac{1+\eps}{\beta}}\big)$ before time $M,$ which is atypically short on account of Proposition \ref{someproperties} ii). Therefore, since $N_A^u\leqslant N_{A_k}^u$ and $\Psi(u,A,M)\subset \Psi(u,A_k,M)$,
	\begin{equation*}
		\begin{split}
			&\P^I\left(\Psi(u,A,M)\not\subset A_{k+1}\right)\\
			&\qquad   \leq\P^I\left(N_{A_k}^u\geq Cu \cdot \mathrm{cap}(A_k)\right)+Cu \cdot \mathrm{cap}(A_k)\sup_{y\in{A_k}}P_{y}\Big(T_{B(y,M^{(1+\eps)/\beta})}\leq M\Big).
		\end{split}
	\end{equation*}
	Using \eqref{poisson}, \eqref{ballcapacity} and \eqref{exittime}, we get
	\begin{equation*}
		\P^I\big(\Psi(u,A,M)\not\subset A_{k+1}\big)\leq C\exp\big\{-cuk^{\nu}M^{\frac{\nu(1+\eps)}{\beta}}\big\}+Cuk^{\nu}M^{\frac{\nu(1+\eps)}{\beta}}\exp\big\{-cM^{\frac{\eps}{\beta-1}}\big\},
	\end{equation*}
	and \eqref{boundonpsi} follows.
\end{proof}

With Lemma \ref{firstboundtry} at hand, we can finally produce the desired bound on the capacity of $C^u(x,L)$ (see after Lemma \ref{12} for the definition).
\begin{Prop}
	\label{10}
	For each $u_0>0$ there exist $\Cl[c]{ccapbig}>0$ and $\Cl{Ccapbig}<\infty$ independent of $u_0,$ $c=c(u_0)>0$ and $C=C(u_0)\in [1,\infty)$ such that for every $u\in{(0,u_0]},$ $x\in{G}$ and $L\geq1$,
	\begin{equation}
	\label{10.1}
	\P^I\left(x\in{\I^u},\ \mathrm{cap}(C^u(x,L))\leq \Cr{ccapbig}L^{3\nu/4}u^{\lfloor\gamma-1\rfloor}\right)\leq C\exp\big\{-cuL^{\Cr{Ccapbig}}\big\}.
	\end{equation}
\end{Prop}
\begin{proof}
	We focus on the case $\gamma<2$. Let $u_0>0,$ $x\in{G},$ and $u\in{(0,u_0)}$ as above and consider a positive integer $M$ and $\delta\in{(0,1)}$ to be chosen suitably. Since $\gamma<2,$ we have $F_{\gamma}(M)=M^{2-\gamma}$ by \eqref{F(n,gamma)}. Thus, by Lemma \ref{Lemma6},
	\begin{align*}
	&\P^I\left(x\in{\I^u},\mathrm{cap}\big(\Psi(u,\{x\},M)\big)\leq c'M^{(1-\delta)(\gamma-1)}\right)
	\\&\qquad \leq \E^I\left[\1_{x\in{\I^u}}P_x\left(\mathrm{cap}\big(T(1,M)\big)\leq c'M^{(1-\delta)(\gamma-1)}\right)\right]
	\leq C\exp\{-cM^{\delta}\},
	\end{align*}
	and with \eqref{boundonpsi},
	\begin{equation*}
	\P^I\Big(\Psi(u,\{x\},M)\not\subset B\big(x,2M^{\frac{1+\delta}{\beta}}\Big)\big)\leq C\exp\big\{-cM^{\frac{\delta(\nu\wedge1)}{\beta}}u\big\}.
	\end{equation*}
	Note that if $\Psi(u,\{x\},M)\subset B\big(x,2M^{\frac{1+\delta}{\beta}}\big),$ then $\Psi(u,\{x\},M)\subset C^u\big(x,2M^{\frac{1+\delta}{\beta}}\big)$ by definition. Thus, combining the previous two estimates,
	\begin{equation*}
	\P^I\Big(x\in{\I^u},\ \mathrm{cap}\big(C^u\big(x,2M^{\frac{1+\delta}{\beta}}\big)\Big)\leq c'M^{(1-\delta)(\gamma-1)}\Big)\leq C\exp\big\{-cM^{\frac{\delta}{\beta-1}}u\big\}
	\end{equation*}
	and \eqref{10.1} follows by taking $M=\big\lfloor(L/2)^{\frac{7\beta}{8}}\big\rfloor$($\neq0$ upon assuming w.l.o.g.\ that $L$ is large enough), and $\delta=\frac{1}{7}$ since $\beta(\gamma-1)=\nu.$
	
	For $\gamma\geq2,$ stronger bounds are required than the one provided by Lemma \ref{firstboundtry} to deduce \eqref{10.1}. The idea is to apply recursively Lemma \ref{firstboundtry} to a sequence of $\lfloor\gamma\rfloor$ independent random interlacement processes at level $u/\lfloor\gamma\rfloor$ as in Lemma 8, 9 and 10 of \cite{MR2819660} or Lemma A.3 and Corollary A.4 in \cite{DrePreRod} for $G=\Z^d.$ We refer the reader to these references for details.
\end{proof}
We conclude with the proof of Proposition \ref{connectivity}.

\begin{proof}[Proof of Proposition \ref{connectivity}] Fix some $u_0>0.$ Recall the notation below Lemma \ref{12}, and write for all $x_0\in{G}$, $L\geq 1$, $u\in{(0,u_0]}$ and $x,y\in{B(x_0,L)}$,
	\begin{equation*}
	\begin{gathered}
	E_1=\left\{\mathrm{cap}\Big(C_i^{u/3}(x,L)\Big)\geq \Cr{ccapbig}L^{3\nu/4}u^{\lfloor\gamma-1\rfloor}\right\},\, \\E_2=\left\{\mathrm{cap}\Big(C_j^{u/3}(y,L)\Big)\geq \Cr{ccapbig}L^{3\nu/4}u^{\lfloor\gamma-1\rfloor}\right\}.
	\end{gathered}
	\end{equation*}
	The probability in the second line of \eqref{decompositioninter} is upper bounded by
	\begin{align}
	\label{proofofLemma13}
	\begin{split}
	&\P^I\left(E_1\cap E_2\setminus\left\{C_i^{u/3}(x,L)\stackrel{\wedge}\longleftrightarrow C_j^{u/3}(y,L)\text{ in }\hat{\I}_k^{u/3}\cap B_E(x_0,2\Cr{Cstrong}L)\right\}\right)\\
	&\quad +\P^I\big(\{x\in{\I_{i}^{u/3}}\}\setminus E_1\big)+\P^I\big(\{y\in{\I_{j}^{u/3}}\}\setminus E_2\big).
	\end{split}
	\end{align}
	For the first term in \eqref{proofofLemma13}, we fix the constant $\Cr{cStrong}= \Cr{cStrong}(\eps)\in{\big(0,\Cr{Ccapbig}/2\big]}$ small enough so that, using Lemma \ref{12} and the capacity estimates on the event $E_1 \cap E_2$, for all $x,y\in{B(x_0,L)}$, whenever $uL^{2\Cr{cStrong}}\geq 1$,
	\begin{align}
	\label{4.1.2}
	\begin{split}
	&\P^I\left(E_1\cap E_2\setminus\left\{C_i^{u/3}(x,L)\stackrel{\wedge}\longleftrightarrow C_j^{u/3}(y,L)\text{ in }\hat{\I}_k^{u/3}\cap B_E(x_0,2\Cr{Cstrong}L)\right\}\right)\\
	&\leq C\exp\left\{-cL^{-\nu}u\times L^{3\nu/2}u^{2\lfloor\gamma-1\rfloor}\right\}
	\leq C
	\exp\big\{-cL^{2\Cr{cStrong}}u\big\}.
	\end{split}
	\end{align}
	Note  that when $uL^{2\Cr{cStrong}}\leq 1,$ it is easy to see that \eqref{4.1.2} still holds upon increasing the constant $C.$ To bound the probabilities in the second line of \eqref{proofofLemma13}, we apply Proposition \ref{10}. Combining the resulting estimate with \eqref{decompositioninter}, \eqref{proofofLemma13}, \eqref{4.1.2}, we get for all $u\leq u_0,$ $L\geq1$ and $x,y\in{B(x_0,L)},$
	\begin{align*}
	&\P^I\left(x,y\in{\I^u},\left\{x\stackrel{\wedge}\longleftrightarrow y\text{ in }\hat{\I}^{u}\cap B_E(x_0,2\Cr{Cstrong}L)\right\}^{\mathsf{c}}\right) \leq C\exp\{-cL^{2\Cr{cStrong}}u\},
	\end{align*}
	and \eqref{strongonedges} follows from a union bound on $x,y\in{B(x,L)}$, \eqref{Ahlfors} and \eqref{lambda}.
\end{proof}

\begin{Rk} \label{R:condconn}
	The resulting connectivity estimate \eqref{strong} is not optimal, see for instance \eqref{4.1.2}. Notwithstanding, its salient feature for later purposes (see Section \ref{sectionpercsigncluster}) is that it imposes a polynomial condition on $u$ and $L$ of the type $u^aL^b \geq C$, for some $a,b >0$, in order for the complement of the probability in \eqref{strong} to fall below any given deterministic threshold (later denoted $\Cr{ccor}l_0^{-4\alpha}$, see Proposition \ref{Cordecoup}). 
\end{Rk}

\section{Isomorphism, cable system and sign flipping}
\label{seclevelsettilde}
In the first part of this section we explore some connections between the interlacement $\widetilde{\mathcal{I}}^u$ and the (continuous) level sets
\begin{equation}
\label{Etildedef}
    \tilde{E}^{>h}\stackrel{\mathrm{def.}}{=}\{z\in{\tilde{G}};\,\tilde{\Phi}_z>h\}
\end{equation} 
of the Gaussian free field on the cable system defined in \eqref{defphitilde}. Among other things, we aim to eventually apply a recent strengthening of the Ray-Knight type isomorphism from \cite{MR2892408}, see Theorem 2.4 in \cite{MR3492939} and Corollary \ref{phiisagff}
below. This improvement will be crucial in our understanding that certain level sets tend to \textit{locally} (i.e.\ at the smallest scale $L_0$ of our renormalization scheme -- see Section \ref{S:renorm}) connect to $\widetilde{\mathcal{I}}^u$ and that the latter can be used to build connections of desired type, but it requires that certain conditions be met within our framework \eqref{eq:Ass}. We will in fact prove that the critical parameter for the percolation of the (continuous) level sets \eqref{Etildedef} is zero, and that $\tilde{E}^{>-h}$ contains $\tilde{\P}^G$-a.s.\ a unique unbounded connected component for all $h>0.$ In the second part of this section, we use a ``sign-flipping'' device which we introduced in \cite{DrePreRod}, see Lemma \ref{lemmaflip}, but improve it in view of the isomorphism from  Corollary \ref{phiisagff}, which leads to certain desirable couplings gathered in Proposition \ref{iuincluvu} as a first step in proving Theorem \ref{T:mainresultGFF} and \ref{T:mainresultRI}.

Our starting point is the following observation from \nolinebreak\cite{MR3502602}, see also (1.27)--(1.30) in \cite{MR3492939} (N.B.: \eqref{Isomorphism} below is in fact true on any transient weighted graph $(G,\lambda)$). For each $u>0,$ there exists a coupling $\tilde{\P}^u$ between two Gaussian free fields $\tilde{\phi}$ and $\tilde{\gamma}$ on $\tilde{G},$ and local times $\tilde{\ell}_{\tilde{G},u}$ of a random interlacement process on $\tilde{G}$ at level $u$ such that,
 \begin{equation}
\begin{split}
\label{Isomorphism}
& \text{$\tilde{\P}^u$-a.s.,
 $\tilde{\ell}_{\tilde{G},u}$ and $\tilde{\gamma}$ are independent and}\\ 
&    \frac{1}{2}\big(\tilde{\phi}_z+\sqrt{2u}\big)^2=\tilde{\ell}_{z,u}+\frac{1}{2}\tilde{\gamma}_z^2,\quad\text{ for all }z\in{\tilde{G}}.
    \end{split}
\end{equation}
The isomorphism \eqref{Isomorphism} has the following immediate consequence: $\tilde{\P}^u$-a.s.,
\begin{equation}
\label{Iuincluded2}
    \tilde{\I}^u\subset\{z\in{\tilde{G}};\,|\tilde{\phi}_z+\sqrt{2u}|>0\}.
\end{equation}
In particular, by continuity, 
$\tilde{\I}^u$ is either included in $\{z\in{\tilde{G}};\,\tilde{\phi}_z>-\sqrt{2u}\}$ or $\{z\in{\tilde{G}};\,\tilde{\phi}_z<-\sqrt{2u}\}.$ This result will be improved with the help of Corollary \ref{RIconnected} in Proposition \ref{nonpercolationsignclusterscable}. We begin with the following lemma about the connected components of $\{z\in{\tilde{G}};\,|\tilde{\Phi}_z+h|>0\}.$ 

\begin{Lemme}
\label{unicityinfinitecluster}
For each $h\ne 0,$ $\tilde{\P}^G$-a.s.\ the set
\begin{equation*}
    \{z\in{\tilde{G}};\,|\tilde{\Phi}_z+h|>0\}
\end{equation*}
contains a unique unbounded connected component.  
\end{Lemme}
\begin{proof}
By symmetry of $\tilde{\Phi}$ it is sufficient to consider the case \nolinebreak$h>0.$ For convenience, we write $h = \sqrt{ 2u}$ for suitable $u>0$ and consider the field $\tilde{\phi}$ with law $\tilde{\P}^G$ under $\tilde{\P}^u$ instead of $\tilde{\Phi}.$ The existence of an unbounded connected component of $ \{z\in{\tilde{G}};\,|\tilde{\phi}_z+h|>0\}$ follows from \eqref{Iuincluded2} in combination with Corollary \ref{RIconnected}. Thus, it remains to show uniqueness. Assume on the contrary that the set $\{z\in{\tilde{G}};\,|\tilde{\phi}_z+\sqrt{2u}|>0\}$ contains at least two unbounded connected components. Then by connectivity of $\tilde{\I}^u,$ see Corollary \ref{RIconnected}, and by the inclusion \eqref{Iuincluded2}, at least one of these unbounded connected components does not intersect $\tilde{\I}^u$. Call it $\mathcal{C}^u.$ Since $\mathcal{C}^u\subset\tilde{\V}^u,$ the isomorphism \eqref{Isomorphism} and continuity imply that $\mathcal{C}^u$ is an infinite cluster of $\{z\in{\tilde{G}};\,|\tilde{\gamma}_z|>0\}.$ But since $\tilde{\gamma}$ and $\tilde{\I}^u$ are independent, it follows from Lemma \ref{Lemmacapline} that  $\tilde{\P}^u$-a.s.\ all the unbounded connected components of $\{z\in{G};\,|\tilde{\gamma}_z|>0\},$ and thus $\mathcal{C}^u,$ intersect $\tilde{\I}^u,$  which is a contradiction.
\end{proof}

The uniqueness and existence of the unbounded component of \nolinebreak$\{z\in{\tilde{G}};\,|\tilde{\Phi}_z+h|>0\}$ for $h>0$ ensured by Lemma \ref{unicityinfinitecluster} implies that $\tilde{\P}^G$-a.s.\ either $\tilde{E}^{>-h}$ or $\tilde{G}\setminus \tilde{E}^{>-h}$ contains an unbounded connected component, and we are about to show that it is always $\tilde{E}^{>-h}.$ For graphs $G$ having a suitable action by a group of translations (for instance graphs of the form $G=G'\times\Z$), this result is clear by ergodicity and symmetry of the Gaussian free field. Due to the lack of ergodicity, we use a different argument here. The measure $\tilde{\P}^u$ refers to the coupling in \eqref{Isomorphism}.

\begin{Prop}
\label{nonpercolationsignclusterscable}
For all $h>0,$ $\tilde{\P}^G$-a.s., the set $\tilde{E}^{>h}$ only contains bounded connected components whereas the set $\tilde{E}^{>-h}$ contains a unique unbounded connected component.
Moreover, for all $u>0,$ $\tilde{\P}^u$-a.s.,
\begin{equation}
\label{Iuincluded}
    \tilde{\I}^u\subset\{z\in{\tilde{G}};\,\tilde{\phi}_z>-\sqrt{2u}\}.
\end{equation}
\end{Prop}
\begin{proof}
We only need to show that for all $h>0$
\begin{equation}
\label{resttoshow}
    \tilde{\P}^{u=\frac{h^2}{2}}\big(\{z\in{\tilde{G}};\,\tilde{\phi}_z<-h\}\text{ contains an unbounded connected component}\big)=0.
\end{equation}
Indeed, if \eqref{resttoshow} holds then by symmetry $\tilde{E}^{>h}$ only contains bounded connected components, by Lemma \ref{unicityinfinitecluster} $\tilde{E}^{>-h}$ contains $\tilde{\P}^G$-a.s.\ a unique unbounded component and \eqref{Iuincluded} follows from \eqref{Iuincluded2} and Corollary \ref{RIconnected}.

Assume that \eqref{resttoshow} does not hold for some height $h>0,$ which is henceforth fixed, and set $u=\frac{h^2}{2}.$ Let $\mathscr{C}^h \subset \tilde{G}$ be the set of points belonging to the infinite connected component of $\{z\in{\tilde{G}};\,\tilde{\phi}_z<-h\}$ whenever it exists ($\mathscr{C}^h=\emptyset$ if there is no such component). By a union bound there exists $x_0\in{G}$ such that
\begin{equation}
\label{Ch>0}
    \tilde{\P}^u\big(x_0\in{\mathscr{C}^h}\big)>0.
\end{equation}
For all $n\in{\N},$ we define the random variable
\begin{equation}
\label{eq:defY_n}
    Y_n=\frac{|\I^u\cap B(x_0,n)|}{|B(x_0,n)|}, \quad \text{(where $u=h^2/2.$)} 
\end{equation}
All constants from here on until the end of this proof may depend implicitly on $u$ (or $h$).
By definition of random interlacements and since $\mathrm{cap}(\{x\})=g(x,x)^{-1}$, $\P^I(x\in{\I^u})=1 - e^{-\frac{u}{g(x,x)}}$, whence for all $x \in G$, $c\leq \P^I(x\in{\I^u})\leq C$ due to \eqref{Green} and thus, in view of \eqref{eq:defY_n},
\begin{equation}
\label{EYN}
    c\leq\tilde{\E}^u[Y_n]=\frac{1}{|B(x_0,n)|}\sum_{x\in{B(x_0,n)}}\tilde{\P}^u(x\in{\I^u})\leq C.
\end{equation}
Following the lines of the proof of (1.38) in \cite{MR2891880} one finds with the help of \eqref{Green} that there exists a constant $C$ such that for all $x,x'\in{G},$
\begin{equation}
\label{cov}
    \text{Cov}_{\tilde{\P}^u}\big(\1_{x\in{\I^u}},\1_{x'\in{\I^u}}\big)=\text{Cov}_{\P^I}\big(\1_{x\in{\V^u}},\1_{x'\in{\V^u}}\big)\leq Cg(x,x').
\end{equation}
Moreover, by \eqref{lambda} and Lemma \ref{expectedexittime}, there exists a constant $C<\infty$ such that for all $x\in{G}$ and $n\in{\N},$
\begin{equation}
    \label{sumg}
    \sum_{y\in{B(x,n)}}g(x,y)\leq Cn^{\beta}.
\end{equation}
Combining \eqref{cov}, \eqref{sumg}, \eqref{lambda} and \eqref{Ahlfors} yields that for all $n\in{\N}$
\begin{equation}
\label{VARYN}
    \text{Var}_{\tilde{\P}^u}(Y_n)=\frac{1}{|B(x_0,n)|^2}\sum_{x,x'\in{B(x_0,n)}}\text{Cov}_{\tilde{\P}^u}\big(\1_{x\in{\I^u}},\1_{x'\in{\I^u}}\big)\leq Cn^{\beta-\alpha}=Cn^{-\nu}.
\end{equation}
With  \eqref{EYN}, \eqref{VARYN} and Chebyshev's inequality, one then finds $N_0>0$ large enough such that for all $n\geq N_0,$
\begin{equation}
\label{YnsmallerE}
    \tilde{\P}^u\Big(Y_n\leq\frac{\tilde{\E}^u[Y_n]}{2}\Big)\leq\frac{4\text{Var}_{\tilde{\P}^u}(Y_n)}{\tilde{\E}^u[Y_n]^2}\leq Cn^{-\nu}\leq \frac{\tilde{\P}^u(x_0\in{\mathscr{C}^h})}{2},
\end{equation}
where the last step follows from the assumption \eqref{Ch>0}. Using \eqref{YnsmallerE} and \eqref{EYN}, we get that for all $n\geq N_0,$
\begin{equation}
\label{chvsIu}
    \tilde{\E}^u[Y_n \cdot \1_{x_0\in{\mathscr{C}^h}}]\geq\frac{\tilde{\E}^u[Y_n]}{2} \cdot \tilde{\P}^u\Big(Y_n\geq\frac{\tilde{\E}^u[Y_n]}{2},x_0\in{\mathscr{C}^h}\Big)\geq c\tilde{\P}^u(x_0\in{\mathscr{C}^h}).
\end{equation}
If $x_0\in{\mathscr{C}^h},$ then $\mathscr{C}^h$ is the unique unbounded connected component of $\{z\in{\tilde{G}};\,|\tilde{\phi}_z+h|>0\}$ by Lemma \ref{unicityinfinitecluster}, and thus by \eqref{Iuincluded2}, \eqref{chvsIu}, \eqref{Ahlfors} and \eqref{lambda}, for all $n\geq N_0$ the lower bound 
\begin{equation}
\label{lowerboundCh}
    \tilde{\E}^u\left[\big|\mathscr{C}^h\cap B(x_0,n)\big| \cdot \1_{x_0\in{\mathscr{C}^h}}\right]\geq\tilde{\E}^u\Big[\big|\I^u\cap B(x_0,n)\big|\cdot \1_{x_0\in{\mathscr{C}^h}}\Big]\geq cn^{\alpha}\tilde{\P}^u(x_0\in{\mathscr{C}^h})
\end{equation}
follows. 
On the other hand, 
\begin{equation}
\label{chsum}
    \tilde{\E}^u\left[\big|\mathscr{C}^h\cap B(x_0,n)\big|\cdot \1_{x_0\in{\mathscr{C}^h}}\right]=\sum_{x\in{B(x_0,n)}}\tilde{\P}^u(x\in{\mathscr{C}^h},x_0\in{\mathscr{C}^h}),
\end{equation}
and, according to the proof of Proposition 5.2 in \cite{MR3502602}, for all $x\in{G},$
\begin{equation}
\begin{split}
\label{chxx0}
    \tilde{\P}^u(x\in{\mathscr{C}^h},x_0\in{\mathscr{C}^h})&\leq\tilde{\P}^u(x\stackrel{\sim}{\longleftrightarrow} x_0\text{ in  }\{z\in{\tilde{G}};\,|\tilde{\phi}_z|>0\})
    \\&\leq\text{arcsin}\bigg(\frac{g(x_0,x)}{\sqrt{g(x_0,x_0)g(x,x)}}\bigg)
   \stackrel{\eqref{Green}}{\leq}Cg(x_0,x).
\end{split}
\end{equation}
Combining \eqref{chsum}, \eqref{chxx0} and \eqref{sumg} then yields the upper bound 
\begin{equation}
    \label{upperboundCh}
    \tilde{\E}^u\left[\big|\mathscr{C}^h\cap B(x_0,n)\big|\cdot \1_{x_0\in{\mathscr{C}^h}}\right]\leq Cn^{\beta}.
\end{equation}
Finally, by \eqref{lowerboundCh} and \eqref{upperboundCh} one obtains, for all $n\geq N_0,$
$    \tilde{\P}^u(x_0\in{\mathscr{C}^h})\leq Cn^{\beta-\alpha}\leq Cn^{-\nu} $, which contradicts \eqref{Ch>0} as $n\rightarrow\infty.$
\end{proof}

Having shown Proposition \ref{nonpercolationsignclusterscable}, taking complements in \eqref{Iuincluded}, we know that for all $u>0$,
\begin{equation}
\label{incluvufaible}
    \{z\in{\tilde{G}};\,\tilde{\phi}_z<-\sqrt{2u}\}\subset \tilde{\V}^u
\end{equation}
(and in particular $h_*\leq\sqrt{2u_*})$ for all graphs $G$ satisying our assumptions \eqref{eq:Ass}. Moreover, as will become clear in the proof of Corollary \ref{phiisagff} below, Proposition \ref{nonpercolationsignclusterscable} provides us with a very explicit way to construct a coupling $\tilde{\P}^u$ as in \eqref{Isomorphism} with the help of \cite{MR3492939}. With a slight abuse of notation (which will soon be justified), for all $u>0$, we consider a (canonical) coupling $\widetilde{\P}^{u}$ between a  Gaussian free field $\tilde{\gamma}$ on $\tilde{G}$ (with law $\tilde{\P}^G$) and an independent family of local times $(\tilde{\ell}_{z,u})_{z\in{\tilde{G}}}$ continuous in $z \in \tilde{G}$ of a random interlacement process with the same law as under $\tilde{\P}^I$, cf.\@ \eqref{eq:l_B}. Note that this defines the set $\tilde{\I}^u$ by means of \eqref{defItilde}. We then define
\begin{equation}
\begin{split}
\label{eq:Cu}
\mathcal{C}_u^\infty
&\text{ as the union of the connected components }\\
&\text{of }\{ z\in{\tilde{G}};\,2\tilde{\ell}_{z,u}+\tilde{\gamma}^2_z>0\}\text{ intersecting }\tilde{\I}^u.
\end{split}
\end{equation} 
The following is essentially an application of Theorem 2.4 in \cite{MR3492939}.
\begin{Cor}
	\label{phiisagff}
	The process $(\tilde{\phi}_z)_{z\in{\tilde{G}}}$ defined by 
	\begin{equation}
	\label{couplingbetweenGFFandRI}
	\tilde{\phi}_z=\left\{\begin{array}{ll}
	-\sqrt{2u}+\tilde{\gamma}_z&\text{ if }z\notin{\mathcal{C}_u^{\infty}},
	\\-\sqrt{2u}+\sqrt{2\tilde{\ell}_{z,u}+\tilde{\gamma}_z^2}&\text{ if }z\in{\mathcal{C}_u^{\infty}}.
	\end{array}\right.
	\end{equation}
	for all $z\in{\tilde{G}},$ is a Gaussian free field, i.e., its law is $\tilde{\P}^G,$ and the joint field $(\tilde{\gamma}_{\cdot}, \tilde{\ell}_{\cdot,u}, \tilde{\phi}_{\cdot})$ thereby defined constitutes a coupling such that \eqref{Isomorphism} holds. Moreover, $\mathcal{C}_u^{\infty}$ is the unique unbounded connected component of $\{z\in{\tilde{G}};\,\tilde{\phi}_z>-\sqrt{2u}\}.$
\end{Cor}

\begin{proof}
We aim at invoking Theorem 2.4 in \cite{MR3492939} in order to deduce that the field $\tilde{\phi}$ defined in \eqref{couplingbetweenGFFandRI} is indeed a Gaussian free field. The conditions to apply this result are that
\begin{equation}
\label{conditiontildephi>0}
    \tilde{\P}^G\text{-a.s., }\{z\in{\tilde{G}};\,|\tilde{\Phi}_z|>0\} \text{ only contains bounded connected components},
\end{equation}
and $g(x,x)$ is uniformly bounded. The latter is clear by \eqref{Green}, but it is not obvious that \eqref{conditiontildephi>0} holds. However, by direct inspection of the proof of Theorem 2.4 in \cite{MR3492939}, we see that \eqref{conditiontildephi>0} is only used to prove (1.33) and (2.48) in \cite{MR3492939}, and that it can be replaced by the following (weaker) conditions:
\begin{align}
&\text{for all }u>0,\ \tilde{\P}^u\text{-a.s., }\tilde{\I}^u\subset\{z\in{\tilde{G}};\,\tilde{\phi}_z>-\sqrt{2u}\}\text{ and } \label{eq:isonewcond1}
\\
&\text{all the unbounded connected components of }\{z\in{\tilde{G}};\,|\tilde{\gamma}_z|>0\}\text{ intersect }\tilde{\I}^u, \label{eq:isonewcond2}
\end{align}
and the proof of Theorem 2.4 in \cite{MR3492939} continues to hold. For the class of graphs \eqref{eq:Ass} considered here the condition \eqref{eq:isonewcond1} has been shown in \eqref{Iuincluded} and the condition \eqref{eq:isonewcond2} follows from Lemma \ref{Lemmacapline} and the independence of $\tilde{\gamma}$ and $\tilde{\I}^u.$  Thus, Theorem 2.4 in \cite{MR3492939} applies and yields that $\tilde{\phi}$ defined in \eqref{couplingbetweenGFFandRI} has law $\tilde{\P}^G$. 

By \eqref{eq:Cu}, $\tilde{\ell}_{z,u}=0$ for $z\notin{\mathcal{C}_u^{\infty}}$ and it then follows plainly from \eqref{couplingbetweenGFFandRI} that \eqref{Isomorphism} holds. Finally, the fact that $\mathcal{C}_u^\infty$ is the unique unbounded cluster of $\{z\in{\tilde{G}};\,\tilde{\phi}_z>-\sqrt{2u}\}$ is a consequence of Proposition \ref{nonpercolationsignclusterscable} and the definitions of $\mathcal{C}_u^\infty$ and $\tilde{\phi}$, recalling that $\widetilde{\mathcal I}^u = \{ z \in \widetilde G; \, \widetilde \ell_{z,u}>0\}$ is an a.s.\ unbounded connected set due to Corollary \ref{RIconnected} and \eqref{defItilde}. 
\end{proof}

\begin{Rk} 
	\leavevmode
\label{R:iso} \begin{enumerate}[1)]
\item \label{R:iso1} An interesting consequence of Corollary \ref{phiisagff} is that for all graphs satisfying our assumptions \eqref{eq:Ass}, the inclusion \eqref{incluvufaible} can be strengthened to
\begin{equation}\label{vuincluphistrong}
\text{for all }A\subset(-\infty,0),\,\{z\in{\tilde{G}};\,\tilde{\phi}_z\in{-\sqrt{2u}+A}\}\subset \tilde{\V}^u\cap\{z\in{\tilde{G}};\,\tilde{\gamma}_z\in{A}\},
\end{equation}
see Corollary 2.5 in \cite{MR3492939}. 
\item \label{R:iso2} For the remainder of this article, with a slight abuse of notation, we will \textit{solely} refer to $\tilde{\P}^u$ as the coupling between $(\tilde{\gamma}_{\cdot}, \tilde{\ell}_{\cdot,u}, \tilde{\phi}_{\cdot})$ constructed around \eqref{eq:Cu} and  \eqref{couplingbetweenGFFandRI}. Thus, the conclusions of Corollary \ref{phiisagff} hold, and in particular $\tilde{\P}^u$ satisfies \eqref{Isomorphism}. 
\end{enumerate}
\end{Rk}
We now adapt a result from Section 5 in \cite{DrePreRod} which roughly shows that, under $\tilde{\P}^u,$ for each $x\in{G}$ and with $u=h^2/2$ for a suitable $h>0$, except on an event with small probability, a suitable conditional probability that $\tilde{\phi}_z\geq-h$ for all $z$ on the first half of an edge starting in $x$ is smaller than the respective conditional probability that $\phi_x\geq h$ at the vertex $x$ whenever $h$ (or $u$) is small enough. 

For each $x\sim y\in{G},$ we denote by $U^{x,y}$ the compact subset of $\tilde{G}$ which consists of the points on the closed half of the edge $I_{\{x,y\}}$
beginning in $x,$ and for $x\in{G}$ let
$U^x=\bigcup_{y\sim x}U^{x,y}$ and $\mathcal{K}^x=\partial U^x$, i.e., $\mathcal{K}^x$ is the finite set consisting of all midpoints on any edge incident on $x.$ For all $U\subset\tilde{G},$ we denote by $\mathcal{A}_U$ the $\sigma$-algebra $\sigma(\tilde{\phi}_z,\,z\in{U}).$ For all $x\in{G}$, $u>0$ and $K >0$, we also define the events
\begin{equation}
\label{eq:eventsRSbar}
\begin{split}
&R_u^x=\big\{\exists\,y \in G;\, y\sim x \text{ and }\tilde{\phi}_z\geq-\sqrt{2u}\text{ for all }z\in{U^{x,y}}\big\}, \\
&\overbar{S}_K^x=\big\{\tilde{\phi}_z\geq -K\text{ for all }z\in{\mathcal{K}^x}\big\}.
\end{split}
\end{equation}
For all $z\in{\mathcal{K}^x},$ let $y_z$ be the unique $y\sim x$ such that $z\in{U^{x,y}}.$ Recall that by the Markov property \eqref{markovproperty} of the free field, one can write, for all $x \in G$,
\begin{equation}
\label{eq:phicross}
\phi_x=\beta_x^{U^x}+\phi_x^{U^x},
\text{ where }\beta_x^{U^x}=\sum_{z\in{\mathcal{K}^x}} P_x\big(\tilde{X}_{T_{U^x}}=z\big)\tilde{\phi}_z=\frac{1}{\lambda_x}\sum_{z\in{\mathcal{K}^x}}\lambda_{x,y_z}\tilde{\phi}_z
\end{equation}
is $\mathcal{A}_{\mathcal{K}^x}$-measurable and $\phi_x^{U^x}$ is a centered Gaussian variable independent of $\mathcal{A}_{\mathcal{K}^x}$ and with variance $g_{U^x}(x,x)=\frac{2}{\sum_{y\sim x}(\rho_{x,y}/2)^{-1}}=\frac{1}{2\lambda_x}$, where we recall 
$\rho_{x,y}=1/(2\lambda_{x,y})$ and refer to p.2123 of \cite{MR3502602} for an analogous calculation.
\begin{Lemme}
	\label{lemmaflip}
	There exists $\Cl[c]{Kflip}>0$ such that for all $u>0,$ $x\in{G}$ and $K>\sqrt{2u}$ satisfying 
	\begin{equation} \label{eq:smallnessCond}
	K\lambda_x\sqrt{2u}\leq\Cr{Kflip},
	\end{equation}
	we have
	\begin{equation}
	\label{lemmaflipfirst}
	\1_{\overbar{S}_K^x}\tilde{\P}^u\big(R_u^x,\phi_x\leq 2\sqrt{2u}\,|\,\mathcal{A}_{\mathcal{K}^x}\big)\leq\frac12 \tilde{\P}^u\big(\sqrt{2u}\leq \phi_x\leq2\sqrt{2u}\,|\,\mathcal{A}_{\mathcal{K}^x}\big)\quad \text{ on } \{\beta_x^{U^x}\leq K\},
	\end{equation}
	and, denoting by $F$ the cumulative distribution function of a standard normal variable,
	\begin{equation}
	\label{lemmaflipsecond}
	\tilde{\P}^u\big(\phi_x\geq\sqrt{2u}\,|\,\mathcal{A}_{\mathcal{K}^x}\big)\geq F\big(\sqrt{2\lambda_x}(K-\sqrt{2u})\big)\quad \text{ on } \{\beta_x^{U^x}> K\}.
	\end{equation}
\end{Lemme}
\begin{proof}
	We first consider the event $\{\beta_x^{U^x}\leq K\}.$  For any $u>0$ and $K>\sqrt{2u},$ on the event $\{\beta_x^{U^x}\leq K\}\cap\overbar{S}_K^x,$ we have $|\beta_x^{U^x}|\leq K$ by \eqref{eq:eventsRSbar} and \eqref{eq:phicross} and thus
	\begin{align}
	\label{eq:cross1}
	\begin{split}
	&\tilde{\P}^u\big(-\sqrt{2u}\leq\phi_x\leq2\sqrt{2u}\,\big|\,\mathcal{A}_{\mathcal{K}^x}\big)\\&=\sqrt{\frac{\lambda_x}{\pi}}\int_{-\sqrt{2u}}^{2\sqrt{2u}}\exp\left\{-\lambda_x(y-\beta^{U^x}_x)^2\right\}\mathrm{d}y 
	\\&\leq\sqrt{\frac{2u\lambda_x}{\pi}}\exp\big\{-\lambda_x(\beta^{U^x}_x)^2\big\}\times3\exp\big\{4\sqrt{2u}\lambda_xK\big\}.
	\end{split}
	\end{align}
	Similarly, still on the event $\{\beta_x^{U^x}\leq K\}\cap\overbar{S}_{K}^x,$
	\begin{equation}
	\label{eq:cross2}
	\tilde{\P}^u\big(\sqrt{2u}\leq\phi_x\leq2\sqrt{2u}\,\big|\,\mathcal{A}_{\mathcal{K}^x}\big)\geq\sqrt{\frac{2u\lambda_x}{\pi}}\exp\left\{-\lambda_x(\beta^{U^x}_x)^2\right\}\exp\big\{-8\sqrt{2u}\lambda_xK\big\}.
	\end{equation}
	For any $x\in{G}$ and $z\in{\mathcal{K}^x},$ by the Markov property \eqref{markovproperty}, the law of the Gaussian free field $\tilde{\phi}$ on $U^{x,y_z}$ conditionally on $\mathcal{A}_{\mathcal{K}^x\cup\{x\}}$ is that of a Brownian bridge of length $\rho_{x,y_z}/2=(4\lambda_{x,y_z})^{-1}$ between $\phi_x$ and $\tilde{\phi}_z$ of a Brownian motion with variance $2$ at time $1$. Furthermore, still conditionally  on $\mathcal{A}_{\mathcal{K}^x\cup\{x\}},$  these bridges form an independent family in  $z\in{\mathcal{K}^x}.$ Therefore, on the event $\{-\sqrt{2u}\leq\phi_x\leq2\sqrt{2u}\}\cap\{\beta_x^{U^x}\leq K\}\cap \overbar{S}_{K}^x,$ using an exact formula for the distribution of the maximum of a Brownian bridge, see for instance \cite{MR1912205}, Chapter IV.26, we obtain
	\begin{equation}
	\begin{split}
	\label{eq:cross3}
	\tilde{\P}^u\big(R_u^x\,|\,\mathcal{A}_{\mathcal{K}^x\cup\{x\}}\big)&=1-\prod_{y\sim x}\tilde{\P}^u\left(\exists\,z\in{U^{x,y}};\,\tilde{\phi}_z<-\sqrt{2u}\,\middle|\,\mathcal{A}_{\mathcal{K}^x\cup\{x\}}\right)
	\\&=1-\prod_{\underset{\tilde{\phi}_z\geq-\sqrt{2u}}{z\in{\mathcal{K}^x};}}\exp\big\{-4\lambda_{x,y_z}(\tilde{\phi}_z+\sqrt{2u})(\phi_x+\sqrt{2u})\big\}
	\\&\leq 1-\exp\big\{-24\sqrt{2u}\lambda_xK\big\}
	\leq 24\sqrt{2u}\lambda_xK.
	\end{split}
	\end{equation}
	Together, \eqref{eq:cross1}, \eqref{eq:cross2} and \eqref{eq:cross3} imply that for all $u>0$ and $K>\sqrt{2u},$ on the event $\{\beta_x^{U^x}\leq K\}\cap\overbar{S}_K^x,$
	\begin{equation}
	\label{Fudivided}
	\frac{\tilde{\P}^u\left(R_u^x\cap\{\phi_x\leq2\sqrt{2u}\}\,\middle|\,\mathcal{A}_{\mathcal{K}^x}\right)}{\tilde{\P}^u\left(\sqrt{2u}\leq\phi_x\leq2\sqrt{2u}\,\middle|\,\mathcal{A}_{\mathcal{K}^x}\right)}\leq 72\sqrt{2u}\lambda_xK\exp\big\{12\sqrt{2u}\lambda_xK\big\}.
	\end{equation}
	We now choose the constant $\Cr{Kflip}$ such that the right-hand side of \eqref{Fudivided} is smaller than \nolinebreak$1/2$ if $\sqrt{2u}\lambda_xK\leq\Cr{Kflip},$ and \eqref{lemmaflipfirst} then readily follows from \eqref{Fudivided}. 
	The inequality \eqref{lemmaflipsecond} follows simply from \eqref{eq:phicross}: for all $u>0,$ $K>\sqrt{2u}$ and $x\in{G},$ on the event $\{\beta_x^{U^x}> K\},$
	\begin{align*}
	\tilde{\P}^u\big(\phi_x\geq\sqrt{2u}\,|\,\mathcal{A}_{\mathcal{K}^x}\big)
	\geq\tilde{\P}^u\big(\phi_x^{U^x}\geq\sqrt{2u}-K\,|\,\mathcal{A}_{\mathcal{K}^x}\big)=F\big(\sqrt{2\lambda_x}(K-\sqrt{2u})\big).
	\end{align*}
	This completes the proof of Lemma \ref{lemmaflip}.
\end{proof}

For all parameters $u>0$ and $p\in (0,1)$, we consider a probability measure $\widetilde{\Q}^{u,p}$, extension of the coupling $\widetilde{\P}^u$ introduced above \eqref{eq:Cu}, see also Remark \ref{R:iso}, \ref{R:iso2}), governing the fields $((\tilde{\gamma}_z)_{z \in \tilde{G}}, (\tilde{\ell}_{z,u})_{z\in{\tilde{G}}}, (\mathcal{B}_{x}^p)_{x \in G})$ such that, under $\widetilde{\Q}^{u,p}$, 
\begin{equation}
\label{eq:Qup}
\begin{gathered}
\text{the fields $\tilde{\gamma}_{\cdot}$, $\tilde{\ell}_{\cdot,u}$ 
	are those from above \eqref{eq:Cu} (and thus Corollary \ref{phiisagff}}\\
\text{applies), $\mathcal{B}_{x}^p$, $x \in G$ are i.i.d.\@ $\{ 0,1\}$-valued random variables with}\\\widetilde{\Q}^{u,p}(\mathcal{B}_{x}^p=1)= 
p,
\text{ the three fields $\mathcal{B}_{\cdot}^p$, $\tilde{\gamma}_{\cdot}$, $\tilde{\ell}_{\cdot,u}$ are independent}.
\end{gathered}
\end{equation}
We introduce for $u>0,$ $K>\sqrt{2u}$ and $p\in{(0,1)}$ the condition
\begin{equation} \label{eq:pSmallnesCond}
\frac12\leq p< \inf _{x \in G} F\big(\sqrt{2\lambda_x}(K-\sqrt{2u})\big).
\end{equation}
Recalling the definition of the $\sigma$-algebra $\mathcal{A}_{\mathcal K^x}$, $x \in G$, we consider a family $(X^x_{u,K,p})_{x\in{G}}\in{\{0,1\}^G}$ of random variables defined on the same underlying probability $\tilde{\Q}^{u,p}$ from \eqref{eq:Qup} and the property that, for $K > \sqrt{2u}$ and all $x \in G$,
\begin{equation}
\label{conditiononXKu}
\1_{\beta_x^{U^x}\geq K} \cdot \tilde{\Q}^{u,p}\left(X^x_{u,K,p}=1\,|\,\mathcal{A}_{\mathcal K^x}\right)\leq p.
\end{equation}
We will consider the following two natural choices for $X_{u,K,p}$, either
\begin{equation}
\label{pxsmaller}
\text{$X^x_{u,K,p}=\mathcal{B}^p_x$,}\quad x\in{G},
\end{equation}
or
\begin{equation}
\label{Xisphismall}
X^x_{u,K,p}=\1_{\{\phi_x\leq K\}}, \quad x\in{G},
\end{equation}
and we will allow for both. The reason for this twofold choice is explained below in Remark \ref{lastremark}, \ref{remarkchoicemathcalK}). In  case \eqref{pxsmaller}, inequality \eqref{conditiononXKu} follows directly from the definition \eqref{eq:Qup}, whereas in the case \eqref{Xisphismall} it is a consequence of the decomposition \eqref{eq:phicross} and the fact that $\tilde{\Q}^{u,p}(\varphi_x^{U_x} \le 0\,|\,\mathcal{A}_{\mathcal{K}^x})=1/2\leq p.$ We introduce the event
\begin{equation}
\label{eq:S}
S_K^x\stackrel{\mathrm{def.}}{=}\big\{\tilde{\gamma}_y\geq -K+\sqrt{2u}\text{ for all }y\in{\mathcal{K}^x}\big\}
\end{equation}
and the following random subsets of $G$, cf.\ \eqref{eq:eventsRSbar} for the definitions of $R_u^x$ and $\overbar{S}_K^x$:
\begin{align}
\label{eq:allrandomsets}
\begin{split}
R_u&\stackrel{\mathrm{def.}}{=}\{x\in{G};\,R_u^x\text{ occurs}\}, \\
S_K&\stackrel{\mathrm{def.}}{=}\{x\in{G};\,S_K^x\text{ occurs}\},\\
\overbar{S}_K&\stackrel{\mathrm{def.}}{=}\{x\in{G};\,\overbar{S}_K^x\text{ occurs}\}, \text{ and }\\
X_{u,K,p}&\stackrel{\mathrm{def.}}{=}\{x\in{G};\,X^x_{u,K,p}=1\}.\\
\end{split}
\end{align}
By \eqref{couplingbetweenGFFandRI}, under $\tilde{\Q}^{u,p},$ if $\tilde{\phi}_z< -K,$ then $\tilde{\gamma}_z< -K+\sqrt{2u}$ for all $z\in{\tilde{G}},$ and thus for all $x\in{G},$ in view of \eqref{eq:eventsRSbar} and \eqref{eq:S},
\begin{equation}
\label{phi_ylarger2u}
(\overbar{S}_K^x)^{\mathsf{c}}\subset(S_K^x)^{\mathsf{c}},\text{ and therefore } S_K\subset\overbar{S}_K.
\end{equation}
We now take advantage of Lemma \ref{lemmaflip} to obtain the following coupling.

\begin{Prop}
	\label{iuincluvu}
	For all $u>0$, $K>\sqrt{2u}$ and $p\in{(0,1)}$ such that 
	\eqref{eq:smallnessCond}
	and 
	\eqref{eq:pSmallnesCond}
	hold true
	for all $x\in{G},$ with $(X^x_{u,K,p})_{x\in{G}}$ as in \eqref{pxsmaller} or \eqref{Xisphismall}, one can find an extension $(\Omega^{u,K,p},\mathcal{F}^{u,K,p},\mathcal{Q}^{u,K,p})$ of the probability space underlying $\tilde{\Q}^{u,p}$ on which one can define for each $0\leq v\leq u$ two random subsets $H=H_{u,v,K,p}$ and $\overline{E}^{\geq\sqrt{2v}}$ of $G$  such that
	\begin{equation}
	\label{eq:couplinglaw1}
	\overline{E}^{\geq\sqrt{2v}}\text{ has the same law under $\mathcal{Q}^{u,K,p}$ as }E^{\geq\sqrt{2v}}\text{ under }\P^G,
	\end{equation}
	the family $\{x\in{H}\}_{x\in{G}}$ is i.i.d., for each $x \in G$ we have that  $\{x\in{H}\}$ is independent of $\{y\in{\overline{E}^{\geq\sqrt{2v}}}\}_{y\in{G\setminus\{x\}}},$ $\tilde{\I}^u,$ $\tilde{\gamma}$ and $(\mathcal{B}^p_x)_{x\in{G}};$ moreover $\mathcal{Q}^{u,K,p}(x\in{H})>0$ and the following inclusion holds true:
	\begin{equation}
	\label{eq:couplinglaw3}
	(R_u\cup H)\cap {S}_K\cap X_{u,K,p}\subset\overline{E}^{\geq\sqrt{2v}}.
	\end{equation}
\end{Prop}

\begin{proof}
	For fixed values of $u$,  $K$ and $p$ satisfying the above assumptions, we consider an extension $(\Omega^{u,K,p},\mathcal{F}^{u,K,p},\mathcal{Q}^{u,K,p})$ of the probability space underlying $\tilde{\Q}^{u,p},$ on which we also can define an i.i.d.\ family $(V_x)_{x\in{G}}$ of uniform random variables on $[0,1],$ independent of $\tilde{\I}^u,$ $\tilde{\gamma}$ and $(\mathcal{B}^p_x)_{x\in{G}}.$ For each $x\in{G}$ and $0\leq v\leq u$, there exists a measurable function $f^x_{u,v}:\R^{\mathcal{K}^x}\rightarrow(-\infty,1]$ such that, with $\mathcal{K}=\bigcup_{x\in{G}}\mathcal{K}^x$ and ${\mathcal{A}}_{\mathcal{K}}= \sigma(\widetilde{\varphi}_x,  x\in \mathcal{K})$,
	\begin{equation}
	\label{eq:def_f_coupling}
	f^x_{u,v}(\tilde{\phi}_{|\mathcal{K}^x}) = \frac{\tilde{\Q}^{u,p}(\phi_x\geq\sqrt{2v}\,|\,\mathcal{A}_{\mathcal{K}})-\tilde{\Q}^{u,p}(R_u^x\cap \overbar{S}_K^x\cap \{X_{u,K,p}^x=1\}\,|\,\mathcal{A}_{\mathcal{K}})}{1-\tilde{\Q}^{u,p}(R_u^x\cap \overbar{S}_K^x\cap \{X_{u,K,p}^x=1\}\,|\,\mathcal{A}_{\mathcal{K}})}
	\end{equation}
	(in particular, the right-hand side depends on $\tilde{\phi}_{|\mathcal{K}}$ only through $\tilde{\phi}_{|\mathcal{K}^x}$). 
In order to lower bound $f^x_{u,v}(\psi)$ for 	$\psi \in{\R^{\mathcal{K}^x}}$ with $\psi_y\geq -K$ for all $y \in \mathcal K^x$, we distinguish cases recalling the decomposition \eqref{eq:phicross}: on the event $\{\beta_x^{U^x}> K\}$ we can use the lower bound \eqref{lemmaflipsecond} in combination with  \eqref{conditiononXKu} to obtain 
$f^x_{u,v}(\psi) \ge F\big(\sqrt{2\lambda_x}(K-\sqrt{2u})\big)-p.$
On the complementary event $\{\beta_x^{U^x}\le  K\}$ we take advantage of  \eqref{lemmaflipfirst}, using that $v\leq u$ and the inequality $\P(\mathcal{N}(0,1)\in{[a,b]})\geq c(b-a)\exp(-C\max(a^2,b^2))$ for all $a\leq b$ in ${\R}$; choosing $a=\sqrt{2 \lambda_x}(\sqrt{2u}-\beta_x^{U^x})$ and $b=\sqrt{2 \lambda_x}(2\sqrt{2u} -\beta_x^{U^x})$ in combination with the fact that $|\beta_x^{U^x}| \leq K $, we get that $f^x_{u,v}(\psi) \ge c\sqrt{u\lambda_x}\exp\big(-C{\lambda_x}(K+2\sqrt{2u})^2\big).$
Combining these lower bounds, 
we have for all $\psi \in{\R^{\mathcal{K}^x}}$ with $\psi_y\geq -K$ for all $y \in \mathcal{K}^x$,
	\begin{equation*}
	f^x_{u,v}({\psi})\geq \big(F\big(\sqrt{2\lambda_x}(K-\sqrt{2u})\big)-p\big)\wedge \big(c\sqrt{\lambda_xu}\exp\big(-C{\lambda_x}(K+2\sqrt{2u})^2\big)\big).
	\end{equation*}
	By \eqref{lambda} and \eqref{eq:pSmallnesCond}, we thus have
	\begin{equation}
	\label{infpos}
	{f}_{\text{min}}  \stackrel{\text{def.}}{=}   \inf_{x\in{G}}\inf_{\substack{\psi\in{\R^{\mathcal{K}^x}}:\\\ \psi_y\geq -K, y \in \mathcal{K}^x} }f^x_{u,v}({\psi})>0.
	\end{equation}
	For all $0\leq v\leq u$, let
	\begin{equation}
	\label{defebar}
	\overline{E}^{\geq\sqrt{2v}}\stackrel{\text{def.}}{=}\{x\in{G};\,V_x\leq f^x_{u,v}(\tilde{\phi}_{|\mathcal{K}^x})\}\cup \big(R_u\cap \overbar{S}_K\cap X_{u,K,p}\big)
	\end{equation}
	and
	\begin{equation}
	\label{eq:defHcoupling}
	H=H_{u,v,K,p}\stackrel{\text{def.}}{=}\big\{x\in{G};\,V_x\leq {f}_{\text{min}} \big\}.
	\end{equation}
	It is clear that the family $\{x\in{H}\}_{x\in{G}}$ is i.i.d., that $\{x\in{H}\}$ is independent of $\{y\in{\overline{E}^{\geq\sqrt{2v}}}\}_{y\in{G\setminus\{x\}}},$ $\tilde{\I}^u,$ $\tilde{\gamma}$ and $(\mathcal{B}^p_x)_{x\in{G}},$ and that $\mathcal{Q}^{u,K,p}(x\in{H})>0$ due to \eqref{infpos}. We proceed to verify \eqref{eq:couplinglaw3} with $\overbar{S}_K$ replacing $S_K$, which is sufficient due to  \eqref{phi_ylarger2u}. We have
	\begin{equation*}
	\begin{split}
	 H \cap \overbar{S}_K\cap X_{u,K,p} &\hspace{-4.3mm}\stackrel{\eqref{eq:eventsRSbar}, \eqref{infpos}}{\subset} H \cap \overbar{S}_K \cap \{ x\in G: \, f^x_{u,v}(\tilde{\phi}_{|\mathcal{K}^x}) \geq {f}_{\text{min}} \}   \\&\stackrel{\eqref{eq:defHcoupling}}{\subset}  \{x\in{G};\,V_x\leq f^x_{u,v}(\tilde{\phi}_{|\mathcal{K}^x})\}	   \stackrel{\eqref{defebar}}{\subset} \overline{E}^{\geq\sqrt{2v}},
	\end{split}
	\end{equation*}
	from which \eqref{eq:couplinglaw3} (with $\overbar{S}_K$ in place of $S_K$) immediately follows, since $(R_u\cap \overbar{S}_K\cap X_{u,K,p}) \subset \overline{E}^{\geq\sqrt{2v}}$. 
	
	It remains to check that \eqref{eq:couplinglaw1} holds. Abbreviating $q=  \mathcal{Q}^{u,K,p}\big( \, x\in \big(R_u\cap \overbar{S}_K\cap X_{u,K,p}\big)
	\, \big| {\mathcal{A}}_{\mathcal{K}} \big)$, we have
	\begin{equation}
	\begin{split}
	\label{eq:cond_distr_coupling1}
	& \mathcal{Q}^{u,K,p}\big( \,x\in \overline{E}^{\geq\sqrt{2v}} \, \big| {\mathcal{A}}_{\mathcal{K}} \big) \\
	&\stackrel{\eqref{defebar}}{=} q+  \mathcal{Q}^{u,K,p}\big(\, V_x\leq f^x_{u,v}(\tilde{\phi}_{|\mathcal{K}^x}), \, x\in \big(R_u\cap \overbar{S}_K\cap X_{u,K,p}\big)^{\mathsf{c}}
	\, \big| {\mathcal{A}}_{\mathcal{K}} \big)\\
	& \ \,= \ \, q+ f^x_{u,v}(\tilde{\phi}_{|\mathcal{K}^x}) (1-q)\\
	&\stackrel{\eqref{eq:def_f_coupling}}{=} \tilde{\Q}^{u,p}( \phi_x\geq\sqrt{2v}\,|\,\mathcal{A}_{\mathcal{K}}). 
	\end{split}
	\end{equation}
	By \eqref{defebar} and the definition of $R_u^x$, $\overbar{S}_K^x$ and $X_{u,K,p}^x$ see \eqref{eq:eventsRSbar} and \eqref{pxsmaller} or \eqref{Xisphismall}, conditionally on $\mathcal{A}_{\mathcal{K}},$ the events $\{x\in{\overline{E}^{\geq \sqrt{2v}}}\},$ ${x\in{G}},$ respectively $\{\phi_x\geq\sqrt{2v}\}$, ${x\in{G}},$ are independent and so by \eqref{eq:cond_distr_coupling1} the sets $\overline{E}^{\geq\sqrt{2v}}$ and $\{x\in{G}:\phi_x\geq\sqrt{2v}\}$ have the same conditional law. Integrating, we obtain \eqref{eq:couplinglaw1}.
\end{proof}

\begin{Rk}
	Lemma \ref{lemmaflip} is stated in terms of the field $\tilde{\varphi}$ under the measure $\widetilde{\P}^{u}$ with $u>0,$ or equivalently under the measure $\tilde{\Q}^{u,p},$ to which it will eventually be applied. Nevertheless, let us note here that it could in fact be stated for the Gaussian free field $\tilde{\Phi}$ under $\tilde \P^G$ for any weighted graph $(G,\lambda)$ since the assumptions \eqref{eq:Ass} are not required for its proof. Proposition \ref{iuincluvu} is valid on any transient weighted graph $(G,\lambda)$ such that \eqref{lambda} and Corollary \ref{phiisagff} hold, i.e.\ on any graph such $g(x,x)$ is uniformly bounded and such that the conditions \eqref{lambda}, \eqref{eq:isonewcond1} and \eqref{eq:isonewcond2} hold. In particular, the assumptions \eqref{eq:Ass} are not necessarily required.
	\end{Rk}
\bigskip

We close this section with an outlook of the remaining sections. Under $\mathcal{Q}^{u,K,p}$ from Proposition \ref{iuincluvu} with $X_{u,K,p}$ from \eqref{pxsmaller}, we have that $S_K\cap X_{u,K,p}$ and $\I^u$ are independent, and by \eqref{Iuincluded} that $\I^u\cap S_K\cap X_{u,K,p}\subset R_u\cap S_K\cap X_{u,K,p}\subset \overline{E}^{\geq\sqrt{2u}}.$ Moreover by \eqref{eq:couplinglaw1} and \eqref{vuincluphistrong}, we have that $\overline{E}^{\geq\sqrt{2u}}$ is stochastically dominated by $\V^u.$ In order to prove Theorem \ref{T:mainresultRI} (but not Theorem \ref{T:mainresultGFF}), we thus only need to show that $\I^u\cap S_K\cap X_{u,K,p}$ percolates for a suitable choice of $u,$ $K$ and $p$ with $K\lambda_x\sqrt{2u}\leq\Cr{Kflip}$ and $p< F\big(\sqrt{2\lambda_x}(K-\sqrt{2u})\big)$ for all $x\in{G}.$ A promising strategy to prove that the intersection of $\I^u$ and a large set percolates on $G$ is to apply the decoupling inequalities of Theorem \ref{decoup} to a suitable renormalization scheme, similarly to \cite{MR3024098} and \cite{DrePreRod}. This requires roughly the same amount of work as obtaining an estimate like \eqref{eqhbar1} for small $h>0$ (both are ``existence''-type results), and they will follow as a by-product of the renormalization argument developed in the course of the next three sections. The actual renormalization scheme will be considerably more involved than the arguments presented in \cite{MR3024098} and \cite{DrePreRod} in order to produce an estimate like \eqref{eqhbar2} for small $h>0$ and thereby allow us to deduce Theorem \ref{T:mainresultGFF}.

\section{Proof of decoupling inequalities}
\label{secdecoup}
The coupling $\widetilde{\Q}^{u,p}$ of \eqref{eq:Qup} will eventually feature within a certain renormalization scheme that will lead to the proof of our main results, Theorems \nolinebreak\ref{T:mainresultGFF} and \ref{T:mainresultRI}. This is the content of Sections \ref{S:renorm} and \ref{sectionpercsigncluster}. The successful deployment of these multi-scale techniques hinges on the availability of suitable decoupling inequalities, which were stated in Theorem \ref{decoup} and which we now prove. In essence, both inequalities \eqref{decouplingGFF} (for the free field) and \eqref{decouplingRI} (for interlacements) constituting Theorem \ref{decoup} will follow from two corresponding results in \cite{MR3325312} and \cite{MR3420516}, see also \eqref{decoupstrongGFFineq} and \eqref{decoupRI2} below (these results are stated in \cite{MR3325312}, \cite{MR3420516}, for $\Z^d$ but can be extended to $\widetilde{G}$, the cable system of any graph satisfying \eqref{eq:Ass}), once certain error terms are shown to be suitably small. In the free field case, see Lemma  \ref{L:decouGFFlemme}, the respective estimate is straightforward and we give the short argument, along with the proof of \eqref{decouplingGFF}, first.

The issue of controlling the error term is considerably more delicate for the interlacement. The key control comes in Lemma \ref{boundonsoftlocaltimes} below. Following arguments in \cite{MR3420516}, it essentially boils down to estimates on the second moment and on the tail of the so-called \textit{soft local times} attached to the relevant excursion process (for one random walk trajectory), see \eqref{defF1} below, which are given in Lemma \ref{beforedecoupRI}. For $G=\Z^d$, these bounds follow from the strong estimates of Proposition \nolinebreak6.1 in \cite{MR3420516}, but its proof is no longer valid at the level of generality considered here (the details of the argument are very Euclidean; see for instance Section 8 in \cite{MR3420516}). We bypass this issue by presenting a way to obtain the desired bounds in Lemma \ref{beforedecoupRI} and along with it, the decoupling inequality \eqref{decouplingRI}, without relying on (strong) estimates akin to Proposition \nolinebreak6.1 of \cite{MR3420516}. This approach is shorter even when $G=\Z^d$ but comes at the price of  requiring an additional assumption on the distance between the sets. An essential ingredient is a certain consequence of the Harnack inequality \eqref{EHI}, see Lemma \nolinebreak\ref{Lemmaharnack} below.

The following lemma will be useful to find ``approximate lattices'' at all scales inside \nolinebreak$G$.
It will be applied in the context of certain chaining arguments below. These lattices will also be essential in setting up an appropriate renormalization scheme in Section \nolinebreak\ref{S:renorm}.  

\begin{Lemme} \label{Lambda}
Assume $\eqref{p0}$, \eqref{Ahlfors}, and \eqref{Green} to be fulfilled. Then there exists a constant $\Cl{CLambda}$ such that for each $L\geq1,$ one can find a set of vertices $\Lambda(L)\subset G$ with\vphantom{$\Cl[c]{cLambda}$}
\begin{equation}
\label{Lambda1}
    \bigcup_{y\in{\Lambda(L)}}B(y,L)=G,
\end{equation}
and for all $x\in{G}$ and $N\geq1,$
\begin{equation}
\label{Lambda2}
    |\Lambda(L)\cap B(x,LN)|\leq \Cr{CLambda}N^{\alpha}.
\end{equation}
\end{Lemme}
\begin{proof}
For a given $L\geq1$, let $\Lambda(L)\subset G$ have the following two properties: i) for all $y\neq y'\in{\Lambda(L)},$ $d(y,y')>L$, and ii) for all $x\in{G},$ there exists $y\in{\Lambda(L)}$ such that $d(x,y)\leq L.$ Indeed, one can easily construct such a set $\Lambda(L)=\{ y_0, y_1, \dots\},$ e.g.\@ by labeling all the vertices in $G=\{x_0,x_1,\dots \}$ and then ``exploring'' $G$, starting at $y_0 = x_0 \in G$, then defining $y_1$ as the point with smallest label 
in the complement
 of $B(x_0,L)$, idem for $y_2$
  in the complement
   of $ B(y_0,L)\cup B(y_1,L)$, etc.
   
 By ii), for each $x\in{G},$ there exists $y\in{\Lambda(L)}$ such that $d(x,y)\leq L,$ and so in particular $\bigcup_{y\in{\Lambda(L)}} B(y,L)=G$. Moreover, for all $x\in{G}$ and $N\geq1,$ 
\begin{equation*}
    \bigcup_{y\in{\Lambda(L)\cap B(x,NL)}}B\Big(y,\frac{L}{2}\Big)\subset B(x,L(N+1)),
\end{equation*}
and the balls $B\big(y,\frac{L}{2}\big),$ $y\in{\Lambda(L)},$ are disjoint by i). Combining this with \eqref{Ahlfors} we infer that for $L\geq2$,
\begin{equation*}
    |\Lambda(L)\cap B(x,NL)|\leq\frac{\Cr{CAhlfors}(L(N+1))^{\alpha}}{\Cr{cAhlfors}(L/2)^{\alpha}}\leq \frac{4^{\alpha} \Cr{CAhlfors}}{\Cr{cAhlfors}}N^{\alpha},
\end{equation*}
and the proof of \eqref{Lambda2} for $1\leq L<2$ is trivial by \eqref{Ahlfors} and \eqref{lambda} (choose $\Lambda(L) =G$).
\end{proof}

We start with some preparation towards \eqref{decouplingGFF}. Let $\tilde{A}_1$ and $\tilde{A}_2$ be two disjoints measurable subsets of $\tilde{G}$ such that $\tilde{A}_1$ is compact with finitely many connected components, and let $\tilde{U}_1=\tilde{A}_1^{\mathsf{c}}.$ We recall the definition of the harmonic extension $\tilde{\beta}^{\tilde{U}_1}$ of the Gaussian free field $\tilde{\Phi}$ from \eqref{markovproperty}, and for each $\eps>0$ define the event
\begin{equation}
\label{Heps}
    H_\eps=\Big\{\sup_{z\in{\tilde{A}_2}}\big|\tilde{\beta}^{\tilde{U}_1}_z\big|\leq\frac{\eps}{2}\Big\}.
\end{equation}
The following result is stated on $\Z^d$ in \cite{MR3325312}  but its proof is actually valid on $\tilde{G}$, for any $G$ as in \eqref{eq:Ass}, using the Markov property of the free field on $\tilde{G},$ cf.\ \eqref{markovproperty}, instead of the Markov property on $\Z^d.$
\begin{The}[{\cite[Theorem 1.2]{MR3325312}}]
\label{decoupGFFstrong}
Let $\tilde{A}_1$ and $\tilde{A}_2$ be two disjoints measurable subsets of $\tilde{G}$ such that $\tilde{A}_1$ is compact with finitely many connected components, and let $f_2:C({\tilde{A}_2},\R)\rightarrow[0,1]$ be a measurable and increasing or decreasing function. Then for all $\eps>0,$ $\tilde{\P}^G$-a.s.,
\begin{equation}
\label{decoupstrongGFFineq}
\begin{split}
    &\left\{\tilde{\E}^G\left[f_2(\tilde{\Phi}_{|\tilde{A}_2}-\sigma\eps)\right]-\tilde{\P}^G\left(H_\eps^{\mathsf{c}}\right)\right\}\1_{H_\eps} \\
 &\qquad \leq\tilde{\E}^G\left[f_2(\tilde{\Phi}_{|\tilde{A}_2})\,\big|\,\tilde{\phi}_{|\tilde{A}_1}\right]\1_{H_\eps}\leq \left\{\tilde{\E}^G\left[f_2(\tilde{\Phi}_{|\tilde{A}_2}+\sigma\eps)\right]+\tilde{\P}^G\left(H_\eps^{\mathsf{c}}\right)\right\}\1_{H_\eps}
    \end{split}
\end{equation}
where $\sigma=1$ if $f_2$ is increasing and $\sigma=-1$ if $f_2$ is decreasing.  
\end{The}

\begin{Rk} \label{decoupGFFext}We note in passing that conditions \eqref{p0}, \eqref{Ahlfors} and \eqref{Green} are  not even necessary here: Theorem \ref{decoupGFFstrong} holds on any locally finite, transient, connected weighted graph \nolinebreak$(G,\lambda)$.  
\end{Rk}
Assume now that $\tilde{A}_1$ is no longer compact, but only bounded (and measurable) and let $\tilde{A}_1'$ be the largest subset $\tilde{B}$ of $\tilde{G}$ such that $\tilde{B}^*=\tilde{A}_1^*$ (see before display \eqref{simConn} for a definition of $\tilde{B}^*$), i.e., $\tilde{A}_1'$ is the closure of the set where one adds to $\tilde{A}_1$ all the edges $I_e$ such that $\tilde{A}_1\cap I_e\neq\emptyset,$ and 
$ \tilde{A}_1'^* = \tilde{A}_1^* \subset G$ is the ``print'' of $\tilde{A}_1'$ in ${G}$. Note that every continuous path started in $\widetilde{G} \setminus \tilde{A}_1'$ and entering $\tilde{A}_1'$ will do so by traversing one of the vertices in $\tilde{A}_1^*$. The set $\tilde{A}_1'$ is a compact subset of $\tilde{G}$ with finitely many connected components. We can thus define $H'_{\eps}$ as in \eqref{Heps} but with $\tilde{U}_1' \stackrel{\text{def.}}{=} (\tilde{A}_1')^{\mathsf{c}}$ in place of $\tilde{U}_1$, for any bounded measurable set $\tilde{A}_1 \subset \widetilde{G}$. The inequality \eqref{decouplingGFF} will readily follow from Theorem \ref{decoupGFFstrong} once we have the following lemma, which is similar to Proposition 1.4 in \nolinebreak\cite{MR3325312}.
\begin{Lemme}
\label{L:decouGFFlemme}
Let $\tilde{A}_1$ and $\tilde{A}_2$ be two Borel-measurable subsets of $\tilde{G}$, $s=d(\tilde{A}_1^*,\tilde{A}_2^*)$ and $r=\delta(\tilde{A}_1^*)$. Assume that $s>0$ and $r < \infty$. There exist constants $\Cr{cdecouGFF}>0$ and $\Cr{CdecouGFF}<\infty$ such that for all such $\tilde{A}_1$, $\tilde{A}_2$ and all $\eps>0,$
\begin{equation}
\label{decouGFFlemme}
\tilde{\P}^G\left({H_\eps'}^{\mathsf{c}}\right)\leq\frac{\Cr{CdecouGFF}}{2}(r+s)^{\alpha}\exp\left\{-\Cr{cdecouGFF}\eps^2s^{\nu}\right\}.
\end{equation}
\end{Lemme}
\begin{proof}
Let $K=\partial B(\tilde{A}_1^*,s).$ By assumption, every connected path on $\tilde{G}$ from $\tilde{A}_2$ to $\tilde{A}_1$ must enter $K$ prior to $\tilde{A}_1^*$. By the strong Markov property of $\tilde{X}$, we have $\tilde{\beta}^{\tilde{U}'_1}_z = \sum_{x\in K} \widetilde{P}_z(H_K< \infty, \tilde{X}_{H_K}=x)\tilde{\beta}^{\tilde{U}'_1}_x$ for all $z\in \tilde{A}_2$ and therefore, in view of \eqref{Heps}, we obtain the bound
\begin{equation} \label{hepspcomp}
    \tilde{\P}^G\left({H_\eps'}^{\mathsf{c}}\right)\leq \tilde{\P}^G\left(\sup_{x\in{K}}\big|\tilde{\beta}^{\tilde{U}'_1}_x\big|>\frac{\eps}{2}\right)=\P^G\left(\sup_{x\in{K}}\big|\beta_x^{\tilde{A}_1^*}\big|>\frac{\eps}{2}\right),
\end{equation}
with $\beta_x^{\tilde{A}_1^*}=E_x\big[\Phi_{Z_{H_{\tilde{A}_1^*}}}\1_{{H_{\tilde{A}_1^*}}<\infty}\big].$ Here, the equality follows from the fact that under $\tilde{P}_x$ for $x \in K$, $ \tilde{X}_{{T}_{\tilde{U}_1'}} =\tilde{X}_{{H}_{\tilde{A}_1'}}$ is always on $\tilde{A}_1^*$ (cf.\ the discussion below Remark \ref{decoupGFFext}), that the law of $\tilde{\Phi}_{|G}$ under $\tilde{\P}^G$ is $\P^G,$ and that the law of $\tilde{X}_{|G}$ under $\tilde{P}_x$ is $ P_x$ for each $x\in{G}$. Following the proof of Proposition 1.4 in \cite{MR3325312} (see the computation of $\text{Var}( h_x)$ therein), if $s>2\Cr{Cdistance},$ then for each $x\in{K},$ $\beta_x^{\tilde{A}_1^*}$ is a centered Gaussian variable with variance upper bounded by 
\begin{equation}
\label{eq:gvarbound}
    \sup_{y\in{\tilde{A}_1^*}}g(x,y)\stackrel{\eqref{Green}}{\leq}\Cr{CGreen} \sup_{y\in{\tilde{A}_1^*}}d(x,y)^{-\nu}\stackrel{\eqref{conditiondistance}}{\leq}\Cr{CGreen}(s-\Cr{Cdistance})^{-\nu}\leq Cs^{-\nu},
\end{equation}
noting that $d(K, \tilde{A}_1^*) \geq s- \Cr{Cdistance}$ by \eqref{conditiondistance}. By possibly adjusting the constant $C$, we see that \eqref{eq:gvarbound} continues to hold if $s\leq 2\Cr{Cdistance},$ for then $s^{-\nu} \geq c$ and 
 $\sup_{x \in K, y\in{\tilde{A}_1^*}}g(x,y)\leq \sup_{x \in G} g(x,x)\leq \Cr{CGreen}$ by \eqref{Green} and using that $g(x,y)= P_x(H_y < \infty)g(y,y)\leq g(y,y)$.
By a union bound, using \eqref{Ahlfors} and \eqref{lambda}, we finally get with
\eqref{eq:gvarbound} and \eqref{hepspcomp},
\begin{equation*}
    \tilde{\P}^G\left({H_\eps'}^{\mathsf{c}}\right)\leq2\Cr{CAhlfors}\Cr{cweight}^{-1}(r+s)^{\alpha}\exp\big\{-cs^{\nu}\eps^2\big\},
\end{equation*}
for all $s > 0$ and $r< \infty$, which completes the proof.
\end{proof}

\begin{proof}[Proof of \eqref{decouplingGFF}] We may assume without loss of generality that $\tilde{A}_1$ is bounded and $r= \delta(\tilde{A}_1)$. Applying Theorem \ref{decoupGFFstrong} with $\tilde{A}_1'$ and $\tilde{A}_2,$ multiplying the upper bound in \eqref{decoupstrongGFFineq} by $f_1(\tilde{\phi}_{|\tilde{A}_1})$ for some monotone function $f_1:C({\tilde{A}_1},\R)\rightarrow[0,1]$ and integrating yields
\begin{equation}
\label{eq:decalmost}
    \tilde{\E}^G\left[f_1\big(\tilde{\Phi}_{|\tilde{A}_1}\big)f_2\big(\tilde{\Phi}_{|\tilde{A}_2}\big)\right]\leq\tilde{\E}^G\left[f_1\big(\tilde{\Phi}_{|\tilde{A}_1}\pm\eps\big)\right]\tilde{\E}^G\left[f_2\big(\tilde{\Phi}_{|\tilde{A}_2}\pm\eps\big)\right]+2\tilde{\P}^G\left({H'_\eps}^{\mathsf{c}}\right).
\end{equation}
The inequality \eqref{decouplingGFF} then follows from \eqref{eq:decalmost} and \eqref{decouGFFlemme}.
\end{proof}

We now turn to \eqref{decouplingRI}, the decoupling inequality for random interlacements. We will eventually use the soft local times technique which has been introduced in \cite{MR3420516} to prove a similar (stronger) inequality on $\Z^d,$ for $d\geq3.$ 
In anticipation of arising difficulties when estimating the error term which naturally appears within this method, we first show a certain Harnack-type inequality, see \eqref{harnack} below, which will be our main tool to deal with this issue. Let
\begin{equation}
\label{eq:constK}
K \geq 5 \vee (2\Cr{Cdistance})^2
\end{equation} 
be a parameter to be fixed later (the choice of $K$ will correspond to the constant $\Cr{Ceta}$ appearing above \eqref{decouplingRI}, see \eqref{defC9} below).  
We consider $\tilde{A}_1$ and $\tilde{A}_2$ two measurable subsets of $\tilde{G}$ and we assume that the diameter $r$ of $\tilde{A}_1^*$ is finite and smaller than the diameter of $\tilde{A}_2^*$ (recall the definition of $\tilde{A}^*\subset G$ for $\tilde{A}\subset\tilde{G}$ from Section \ref{2}), and that $s=d(\tilde{A}_1^*,\tilde{A}_2^*)\geq K( r\vee 1)$.  We then define
\begin{equation}
\label{eq:decsetup}
   A_1=\tilde{A}_1^*,\ A_2=B\Big(A_1,\frac{s}{2}\Big)^{\mathsf{c}}\text{ and }V=\partial B\Big(A_1,\frac{s}
   {\sqrt{K}}\Big).
\end{equation}
These assumptions imply that $s\geq K \stackrel{\eqref{eq:constK}}{\geq}2\Cr{Cdistance}\sqrt{K},$ so that by \eqref{conditiondistance}, the sets $A_1,$ $A_2$ and $V$ are disjoint subsets of $G$, $A_2 \supset \tilde{A}_2^*$ and any nearest neighbor path from $A_1$ to $A_2$ crosses $V.$ The following lemma will follow from \eqref{EHI} and a chaining argument.

\begin{Lemme}
\label{Lemmaharnack}
There exists $c\geq 5 \vee (2\Cr{Cdistance})^2$ such that the following holds true: for all $K \geq c$, there exists $\Cl{Charnack}= \Cr{Charnack}(K)\geq1$ such that for any $A_1$, $A_2$, $V$ as above, $B \in\{ A_1,A_2, A_1\cup A_2\} $, $v$ a non-negative function on $G,$ $L$-harmonic on $B^{\mathsf{c}}$,
\begin{equation}
\label{harnack}
    \sup_{y\in{V}}v(y)\leq \Cr{Charnack}\inf_{y\in{V}}v(y).
\end{equation}
\end{Lemme}

\begin{proof}
Set $\varepsilon(K)=\frac{1}{\sqrt{K}}$ and
\begin{equation*}
    U_0=B\big(A_1,\varepsilon^2(2\Cr{CHarnack}+1)s\big),\  U_1=B\big(A_1,\varepsilon s\big),\  U_2=B\big(A_2
    ,\varepsilon ^2(2\Cr{CHarnack}+1)s\big)^{\mathsf{c}},
\end{equation*}
\begin{figure}[h]
\includegraphics[scale=0.4]{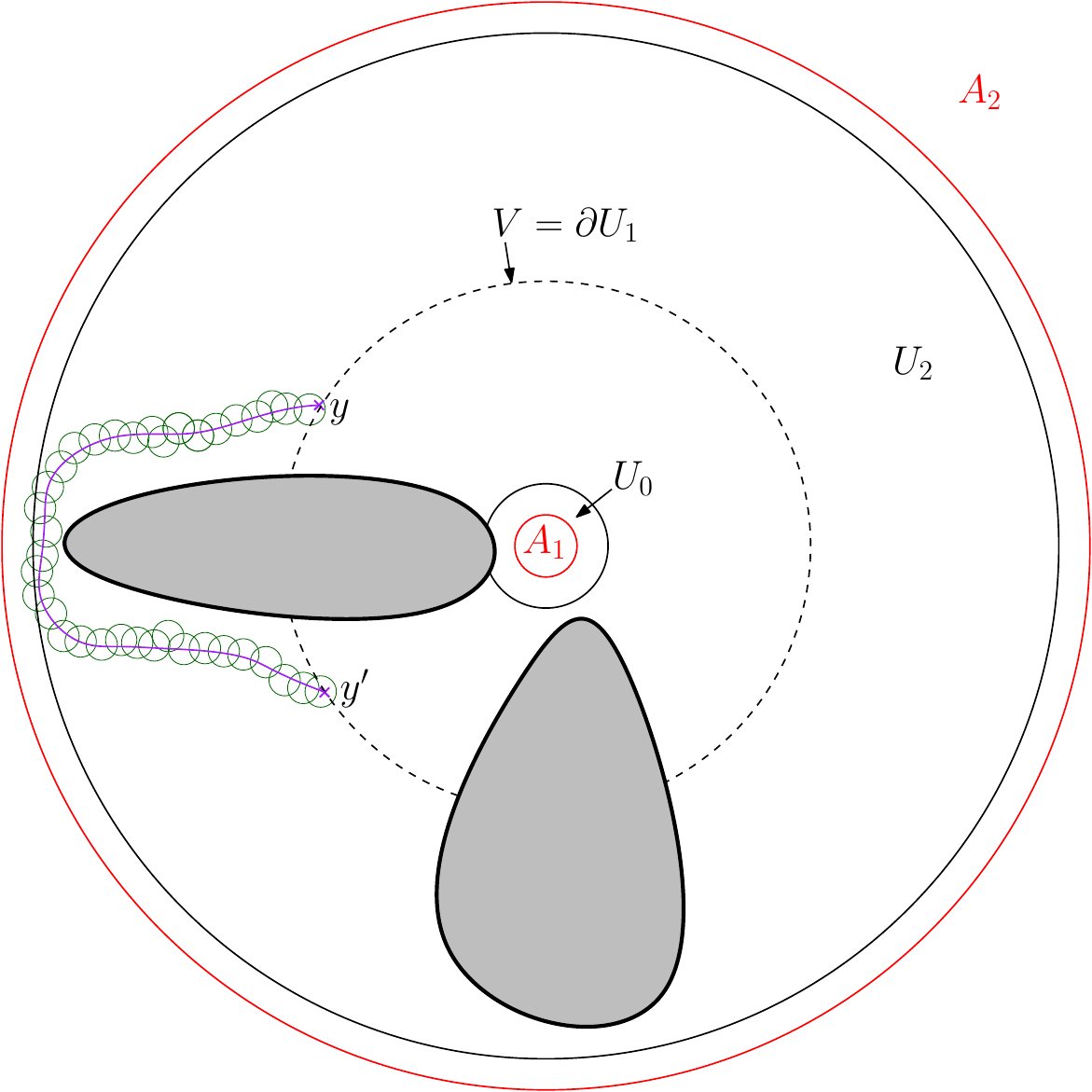}
\centering
\caption{\label{fig:Lemma6.5} The grey regions are not part of the graph, but there is always a (purple) path between $y$ and $y'$ in $U_0^{\mathsf{c}} \cap U_2$. In green, the balls $B(z_i,\eps^2s)$ are included in $(A_1\cup A_2)^{\mathsf{c}}\subset B^{\mathsf{c}}$.}
\end{figure}
where $\Cr{CHarnack}$ corresponds to the constant in the elliptic Harnack inequality, see above \eqref{EHI} and Lemma \nolinebreak\ref{L:killGreen}. We first prove that if  $K \geq c$ (so that $\varepsilon $ is small enough) then 
$V(=\partial U_1)$ is connected in $U_0^{\mathsf{c}} \cap U_2$, that is for every vertices $y,y'\in{V}$, there exists a path in $U_0^{\mathsf{c}} \cap U_2$ from $y$ to $y'$.  We first assume that $K \geq c$ so that $U_0\subset U_1\subset U_2.$ Then by \eqref{Greenmarkov} for all $y,y'\in{V}$ we have
\begin{equation}
\label{eq:2.100}
\begin{split}
g_{U_0^{\mathsf{c}} \cap U_2}(y,y')&\geq g(y,y')-P_y(H_{U_0}<T_{U_2})\sup_{z\in{U_0}}g(z,y')-\sup_{z\in{U_2^{\mathsf{c}}}}g(z,y')
\\&\geq  \Cr{cGreen}(2\eps s)^{-\nu}-P_y(H_{U_0}<T_{U_2})\Cr{CGreen}(\eps s-\eps^2s-\Cr{Cdistance})^{-\nu}- \Cr{CGreen}(s-\eps s)^{-\nu},
\end{split}
\end{equation}
where we used \eqref{Green} and \eqref{conditiondistance} in the last inequality. 
Recall the relative equilibrium measure $e_{U_0,U_2}(\cdot)$ and capacity $\mathrm{cap}_{U_2}(U_0)$ from \eqref{defequilibrium} and \eqref{defcap}. Using that $s \geq K r$, it follows that for $K \geq c',$ $d(U_1,U_2^{\mathsf{c}})\geq \Cr{CHarnack}\delta(U_1)$ so that, by \eqref{Greenstoppedbound} and \eqref{entrancegreenequi}, one obtains for all $x\in{A_1}\subset U_0,$
\begin{equation}
\label{eq:2.10}
\begin{split}
    1=\sum_{x'\in{U_0}}g_{U_2}(x,x')e_{U_0,U_2}(x')
    &\geq\frac{\Cr{cGreen}}{2}\big(2r+\varepsilon ^2(2\Cr{CHarnack}+1)s\big)^{-\nu}\mathrm{cap}_{U_2}(U_0) \\
    &\stackrel{r \leq \varepsilon ^2s}{\geq}\frac{\Cr{cGreen}}{2}\big(\varepsilon ^2(2\Cr{CHarnack}+3)s\big)^{-\nu}\mathrm{cap}_{U_2}(U_0).
\end{split}
\end{equation}
We further assume that $K \geq c$ and $\varepsilon $ is small enough so that $d(U_0,V)\geq\frac{\varepsilon s}{4},$ 
and then, using again \eqref{Greenstoppedbound} and \eqref{entrancegreenequi}, for all $y \in V$,
\begin{equation}
\label{eq:2.11}
   P_y(H_{U_0}<T_{U_2})=\sum_{x\in{U_0}}g_{U_2}(y,x)e_{U_0,U_2}(x)\leq \Cr{CGreen}d(U_0,V)^{-\nu}\mathrm{cap}_{U_2}(U_0)\stackrel{\eqref{eq:2.10}}{\leq} C\times\varepsilon ^{\nu} .
\end{equation}
We stress that $C$ is uniform in $K$ (and $\varepsilon $) in \eqref{eq:2.11}. 
Combining \eqref{eq:2.100} with \eqref{eq:2.11} we get that $g_{U_0^{\mathsf{c}} \cap U_2}(y,y')>0$ for $K \geq c$, and hence $y$ is connected to $y'$ in $U_0^{\mathsf{c}} \cap U_2$.
\vspace{3mm}

For all $x\in{B(A_1,2\varepsilon ^2\Cr{CHarnack}s)^{\mathsf{c}}\cap B(A_2,2\varepsilon ^2\Cr{CHarnack}s)^{\mathsf{c}}},$  $v$ is harmonic on $B(x,2\varepsilon ^2\Cr{CHarnack}s)$ by assumption and thus \eqref{EHI} gives
\begin{equation}
\label{harnackonaball}
    \inf_{z\in{B(x,\varepsilon ^2s)}}v(z)\geq \Cr{cHarnack}\sup_{z\in{B(x,\varepsilon ^2s)}}v(z).
\end{equation}
By connectivity of $V$ in $U_0^{\mathsf{c}} \cap U_2$ and \eqref{Lambda1}, for all $y,y'\in{V},$ one can find $N\in{\N},$ a sequence $z_0,\dots,z_{N}$ in $\Lambda\big(\varepsilon ^2s/2\big)\cap B(A_1,2\varepsilon ^2\Cr{CHarnack}s)^{\mathsf{c}}\cap B(A_2,2\varepsilon ^2\Cr{CHarnack}s)^{\mathsf{c}},$ with $\Lambda\big(\varepsilon ^2s/2\big)$ as in Lemma \ref{Lambda}, such that $z_i\neq z_j$ for $i\neq j,$ $y\in{B(z_0,\varepsilon ^2s)},$ $y'\in{B(z_N,\varepsilon ^2s)}$ and for all $i\in{\{1,\dots,N\}},$ there exists $y_i\in{B(z_{i-1},\varepsilon ^2s)\cap B(z_i,\varepsilon ^2s)}.$ Note that with the help of \eqref{Lambda2}, we can choose $N$ uniformly in $s$ and $y,y' \in V$ (but still as a function of $K$). We then apply \eqref{harnackonaball} recursively on each of the balls $B(z_i,\varepsilon ^2s),$ $i\in{\{0,\dots,N\}},$ to find
\begin{equation*}
    v(y)\geq \Cr{cHarnack}^{N+1}v(y'),
\end{equation*}
and \eqref{harnack} follows. 
\end{proof}

We now recall some facts about soft local times from \cite{MR3420516}. We continue with the setup of \eqref{eq:decsetup} and introduce the excursion process between $B\in{\{A_1,A_2,A_1\cup A_2\}}$ and $V$ for the Markov chain $Z_{\cdot}$ on $G$ as follows. Let $\theta_n:G^\N\rightarrow G^\N$ denote the canonical time shifts on $G^\N,$ that is for all $n,\,p\in{\N}$ and $\omega\in{G^\N},$ $(\theta_n(\omega))_p=\omega_{n+p}.$ The successive return times to $B$ and $V$ are recursively defined by $D_0=0$ and for all $k\geq1$,
\begin{equation}
\label{eq:stoppingtimesexcursions}
    R_k=H_B\circ\theta_{D_{k-1}}+D_{k-1}\qquad\qquad\qquad D_k=H_V\circ\theta_{R_k}+R_k,
\end{equation}
where $H_B$ is the first hitting of $B$ by $Z_{\cdot}$, cf.\ below \eqref{eq:Greendef}. Let $N^B=\inf\{k\geq0:\,R_k=\infty\},$ and note that $N^B<\infty$ a.s.\@ since $Z_{\cdot}$ is transient. 
For $k\in\{1,\dots,N^B-1\},$ a trajectory $\Sigma_k\stackrel{\text{def.}}{=}(Z_n)_{n\in{\{R_k,\dots,D_k\}}}$ is called an excursion between $B$ and $V$. It takes values in $\Xi_B$, the set of trajectories starting in $\partial B$ and either ending the first time $V$ is hit or never visiting $V.$ We add a cemetery point $\Delta$ to $\Xi_B$ and, with a slight abuse of notation, introduce a new point $\Delta'$ in $G$ such that for any random variable $H\in{\N}\cup\{\infty\},$ $Z_H=\Delta'$ if $H=\infty.$ For each $x\in{\partial B},$ let $\Xi_B(x)$ be the set of trajectories in $\Xi_B\setminus\{\Delta\}$ starting in $x.$ Set $\Xi_B(\Delta')=\{\Delta\}$ and for all $\sigma\in{\Xi_B},$ let $\sigma^e\in{V}$ be the last point visited by $\sigma$ if $\sigma$ is a finite trajectory of $\Xi_B\setminus\{\Delta\}$, and $\sigma^e=\Delta'$ otherwise. Upon defining $\Sigma_k=\Delta$ for $k\geq N^B,$ the sequence $(\Sigma_k)_{k\geq1}$ can be viewed as a Markov process on $\Xi_B$, called the excursion process between $B$ and $V.$ 

We now sample the Markov chain $(\Sigma_k)_{k\geq1}$ using a Poisson point process as described in Section 4 of \cite{MR3420516}. Let $\mu_B$ be the measure on $\Xi_B$ given by
\begin{equation}\label{eq:mu_B}
    \mu_B(S)=\sum_{x\in\partial B} P_x(\Sigma_1\in{S})+\delta_{\Delta}(S)
\end{equation}
for all $S$ in the $\sigma$-algebra generated by the canonical coordinates, where $\delta_{\Delta}$ denotes a Dirac mass at $\Delta$, and let $p_B:\Xi_B\times \Xi_B\rightarrow[0,\infty)$ be defined (see also (5.18) of \cite{MR3420516}) by
\begin{equation}
\label{eq:exctransition}
    p_B(\sigma,\sigma')=P_{\sigma^e}(H_B=x)\text{ for all }\sigma\in{\Xi_B}\text{ and }\sigma'\in{\Xi_B(x)},\ x\in{\partial B}\cup\{\Delta'\},
\end{equation}
with the convention $P_{\Delta'}(H_B=\Delta')=1.$ Let $\eta$ be a Poisson random measure on some probability space $(\Omega,\mathcal{F},\P)$  with intensity $\mu_B\otimes\lambda,$ where $\lambda$ is the Lebesgue measure on $[0,\infty).$ Let $\sigma_0$ be a random variable taking values in $\Xi_B$, independent of $\eta$, such that
\begin{equation*}
    \P(\sigma_0^e=y)=\overline{e}_V(y)\text{ for all }y\in{V}
\end{equation*}
(see \eqref{normalized} for notation). Moreover, set $\Gamma_0:\Xi_B\rightarrow\R_+$ with $\Gamma_0(\sigma)=0$ for all $\sigma\in{\Xi_B}.$ We now define recursively the random variables $\xi_n,$ $\sigma_n,$ $v_n$ and $\Gamma_n$: for all $n\geq1,$ $(\sigma_n,v_n)$ is the $\P$-a.s.\ unique point in $\Xi_B\times[0,\infty)$ such that
\begin{equation}
\label{eq:defxi}
    \xi_n\stackrel{\text{def.}}{=}\inf_{(\sigma,v)}\frac{v-\Gamma_{n-1}(\sigma)}{p_B(\sigma_{n-1},\sigma)}
\end{equation}
is reached in $(\sigma_n,v_n),$ where the infimum is taken among all the possible  pairs $(\sigma,v)$ in $\text{supp}(\eta)\setminus{\{(\sigma_1,v_1),\dots,(\sigma_{n-1},v_{n-1})\}},$ and define
\begin{equation}
\label{eq:defGn}
    \Gamma_n(\sigma)=\Gamma_{n-1}(\sigma)+\xi_np_B(\sigma_{n-1},\sigma)\text{ for all }\sigma\in{\Xi_B}.
\end{equation}
Note that, for all $n\geq 1$ and $(\sigma,v)\in{\text{supp}(\eta)},$ as follows from \eqref{eq:defxi} and \eqref{eq:defGn}, $\P$-a.s.	,
\begin{equation}
\label{vsmallerthanGn}
    v\leq \Gamma_n(\sigma)\ \Longrightarrow\ (\sigma,v)\in{\{(\sigma_1,v_1),\dots,(\sigma_n,v_n)\}}.
\end{equation}
According to Proposition 4.3 in \cite{MR3420516}, $(\sigma_n)_{n\geq 1}$ has the same law under $\P$ as $(\Sigma_n)_{n\geq 1}$ under $P_{\overline{e}_V}$ (recall the notation from \eqref{Pmu}). By definition, see \eqref{eq:exctransition}, for all $\sigma,\sigma'\in{\Xi_B},$ $p_B(\sigma,\sigma')$ only depend on the last vertex visited by $\sigma$ and on the first vertex visited by $\sigma'$ and thus, on account of \eqref{eq:defGn}, for all $x\in{\partial B\cup\{\Delta'\}}$ and $\sigma, \sigma'\in\Xi_B(x),$  $\Gamma_n(\sigma)=\Gamma_n(\sigma').$ In particular, we can define the \textit{soft local time} up to time $T^B \stackrel{\text{def.}}{=} \inf\{n;\,\sigma_n=\Delta\}$ of the excursion process between $B$ and $V$ by
\begin{equation}
\label{defF1}
    F_1^B(x)=\Gamma_{T^B}(\sigma_x)\text{ for all } x\in{\partial B\cup{\{\Delta'\}}},
\end{equation}
where $\sigma_x$ is any trajectory in $\Xi_B(x).$ By definition, see \eqref{eq:defGn}, we can also write
\begin{equation}
\label{seconddefF1}
     F_1^B(x)=\sum_{k=1}^{T^B}\xi_kp_B(\sigma_{k-1},\sigma_x), \quad \text{for all $x\in{\partial B\cup{\{\Delta'\}}}$}.
\end{equation}
Assume that $(\Omega,\mathcal{F},\P)$ is suitably enlarged as to carry a family $ F= \{ F_k^B ; \, k=1,2,\dots\}$ of i.i.d.\ random variables with the same law as $F_1^B,$ and, for each $u>0,$ a random variable $\Theta_u^V$ with law  Poisson($u \cdot \mathrm{cap}(V)$) independent of $F$. The variables $F_k^B$, $1\leq k \leq \Theta_u^V$ correspond to the soft local times attached to each of the trajectories in the support of $\omega^u$, the interlacement point process, which visit the set $V$ (by \eqref{eq:decsetup} these are the trajectories causing correlations between $\tilde{\ell}_{\tilde{A}_1,u}$ and $\tilde{\ell}_{\tilde{A}_2,u}$). For all $u>0$ and $x\in{\partial B}$, we then set
\begin{equation}
\label{GubFkB}
    G_u^B(x)=\sum_{k=1}^{\Theta_u^V}F_k^B(x),
\end{equation}
which has the same law as the accumulated soft local time of the excursion process between $B$ and $V$ up to level $u$ defined in (5.22) of \cite{MR3420516} (note that Section 5 in \cite{MR3420516} can be adapted, mutatis mutandis, to any transient graph).

The proof of Proposition 5.3 in \cite{MR3420516} then asserts that there exists a coupling $\Q$ between three random interlacements processes $\omega,$ $\omega_1$ and $\omega_2$ such that $\omega_1$ and $\omega_2$ are independent and, for all $u>0$ and $\eps\in{(0,1)},$
\begin{equation}
\label{decoupRI1}
\begin{split}
    &\Q\left[(\omega^{u(1-\eps)}_i)_{|A_i}\leq(\omega^u)_{|A_i}\leq(\omega^{u(1+\eps)}_i)_{|A_i},\ i=1,2\right]
    \\&\qquad\qquad\geq 1-\sum_{\substack{(v,B)=(u(1\pm\eps),A_1),\\(u(1\pm\eps),A_2),(u,A_1\cup A_2)}}\sum_{x\in{\partial B}}\P\left(\left|G_v^B(x)-\E[G_v^B(x)]\right|\geq\frac{\eps}{3}\E[G_v^B(x)]\right),
\end{split}
\end{equation}
where $(\omega^u)_{|A_i}$ is the point process consisting of the restriction to $A_i$ of the trajectories in $ \omega^u$ hitting $A_i$ and we write $\mu\leq\nu$ if and only if $\nu-\mu$ is a non-negative measure. Adding independent Brownian excursions on the cable system $\widetilde G$ as in the proof of Theorem 3.6 in \cite{DrePreRod}, one then easily infers that \eqref{decoupRI1} can be extended to the local times on the cable system, and thus, in the framework of \eqref{eq:decsetup}, since $A_1=\tilde{A}_1^*$ and $\tilde{A}^*_2\subset A_2$, that there exists a coupling $\tilde{\Q}$ such that
\begin{equation}
\label{decoupRI2}
\begin{split}
    &\tilde{\Q}\left[\tilde{\ell}^{\,i}_{x,u(1-\eps)}\leq\tilde{\ell}_{x,u}\leq\tilde{\ell}^{\,i}_{x,u(1+\eps)},\ x\in{\tilde{A}_i},\ i=1,2\right]
    \\&\qquad\qquad\geq 1-\sum_{\substack{(v,B)=(u(1\pm\eps),A_1),\\(u(1\pm\eps),A_2),(u,A_1\cup A_2)}}\sum_{x\in{\partial B}}\P\left(\left|G_v^B(x)-\E[G_v^B(x)]\right|\geq\frac{\eps}{3}\E[G_v^B(x)]\right),
    \end{split}
\end{equation}
where $(\tilde{\ell}_{x,u})_{x\in{\tilde{G}}},$ $(\tilde{\ell}^{\,1}_{x,u})_{x\in{\tilde{G}}}$ and $(\tilde{\ell}^{\,2}_{x,u})_{x\in{\tilde{G}}}$ have the law under $\tilde{\Q}$ of local times of random interlacements on the cable system $\tilde{G},$ cf.\ around \eqref{eq:l_B}, with $\tilde{\ell}^{\,1}$ independent from $\tilde{\ell}^{\,2}.$ The decoupling inequality \eqref{decouplingRI} will follow at once from \eqref{decoupRI2}, see the end of this section, once the following large deviation inequality on the error term is shown. We continue with the setup leading to \eqref{eq:decsetup}. Recall the multiplicative parameter $K$ in \eqref{eq:constK} controlling the distance $d(\tilde{A}_1^*,\tilde{A}_2^*)$.

\begin{Lemme}
\label{boundonsoftlocaltimes}
There exists 
$K_0 \geq 5 \vee (2\Cr{Cdistance})^2$ such that for all $u>0,$ $\eps\in{(0,1)}$ and $B\in{\{A_1,A_2,A_1\cup A_2\}}$ as in \eqref{eq:decsetup} with $K \geq K_0$ and $x \in \partial B$,
\begin{equation*}
\P\left(\left|G_u^B(x)-\E[G_u^B(x)]\right|\geq\frac{\eps}{3}\E[G_u^B(x)]\right)\leq C(K)\exp\left\{-c(K)\eps^2us^{\nu}\right\}.
\end{equation*}
\end{Lemme}
In order to prove Lemma \ref{boundonsoftlocaltimes}, cf.\ \eqref{GubFkB}, we need some estimates on the law of $F_1^B(x)$, which deals with one excursion process between $B$ and $V.$ Let us define\begin{equation}
\label{pibyx}
    \pi^B(y,x)=E_y\bigg[\sum_{k=1}^{N^B-1}\delta_{Z_{R_k},x}\bigg], \quad \text{for $x\in{B}$ and $y\in{V}$},
\end{equation}
the average number of times an excursion starts in $x$ for the excursion process beginning in $y$ (here, $\delta_{x,y}=1$ if $x=y$ and $0$ otherwise; recall $N^B$ from below \eqref{eq:stoppingtimesexcursions}). It follows from (5.24) in \cite{MR3420516} that
\begin{equation}
\label{defpib}
    \pi^B(x)\stackrel{\text{def.}}{=}\E[F_1^B(x)]=\sum_{y\in{V}}\overline{e}_V(y)\pi^B(y,x).
\end{equation}
The following estimates will be useful to prove Lemma \ref{boundonsoftlocaltimes}. 
\begin{Lemme}
\label{beforedecoupRI}
For $K \geq K_0$, there exist $\Cl[c]{clemmadecou}(K)>0$ and $\Cl{Clemmadecou}(K)<\infty$ such that, for all $B\in{\{A_1,A_2,A_1\cup A_2\}}$ as in \eqref{eq:decsetup}, all $x\in{\partial B}$ and $v\in{(0,\infty)},$
\begin{enumerate}[(i)]
   \item \label{upperboundfb2}$\E\left[F_1^B(x)^2\right]\leq4\Cr{Charnack}\pi^B(x)^2,$
   \item \label{largedeviationF1B}$\P\left(F_1^B(x)\geq\pi^B(x)v\right)\leq \Cr{Clemmadecou}\exp\{-\Cr{clemmadecou}v\}.$
\end{enumerate}
\end{Lemme}

\begin{proof} We tacitly assume throughout the proof that $K \geq c$ so that Lemma \nolinebreak\ref{Lemmaharnack} applies. Theorem 4.8 in \nolinebreak\cite{MR3420516} asserts that for all $x\in{B}$
\begin{equation*}
    \E\left[F_1^B(x)^2\right]\leq4\pi^B(x)\sup_{y'\in{V}}\pi^B(y',x).
\end{equation*}
The function $y'\mapsto\pi^B(y',x)$ is $L$-harmonic on $B^{\mathsf{c}},$ and \eqref{upperboundfb2} follows from \eqref{defpib} and Lemma \nolinebreak\ref{Lemmaharnack}. We now turn to the proof of \eqref{largedeviationF1B}. Using \eqref{seconddefF1} and \eqref{eq:exctransition}, we have for all $x\in{\partial B}$ and $x'\in{\partial B\cup\{\Delta'\}},$ $\P$-a.s.,
\begin{align}
    F_1^B(x')=\sum_{k=1}^{T^B}\xi_kP_{\sigma_{k-1}^e}(Z_{H_B}=x')&\geq\inf_{y'\in{V}}\left\{ \frac{P_{y'}(Z_{H_B}=x')}{P_{y'}(Z_{H_B}=x)}\right\}\sum_{k=1}^{T^B}\xi_kP_{\sigma_{k-1}^e}(Z_{H_B}=x) \nonumber\\&\geq\frac{1}{\Cr{Charnack}} \frac{\inf_{y'\in{V}}P_{y'}(Z_{H_B}=x')}{\inf_{y'\in{V}}P_{y'}(Z_{H_B}=x)}F_1^B(x) \label{eq:Fineq},
\end{align}
where we used the fact that $y\mapsto P_y(Z_{H_B}=x)$ is harmonic on $B^{\mathsf{c}}$ and Lemma \ref{Lemmaharnack} in the last inequality. Slight care is needed above if $\sigma_{T^B-1}^e=\Delta'$, in which case $P_{\sigma_{T^B-1}^e}(Z_{H_B}=x')\geq P_{\sigma_{T^B-1}^e}(Z_{H_B}=x)=0$ for all $x\in{\partial B}$ and $x'\in{\partial B\cup\{\Delta'\}}$ so that \eqref{eq:Fineq} continues to hold. With \eqref{eq:Fineq}, we obtain for all $x\in{\partial{B}}$ and $v\in{(0,\infty)},$
\begin{equation}
\begin{split}
\label{eq:tailF1}
    &\P\left(F_1^B(x)\geq\pi^B(x)v\right)\\
    &\qquad \leq\P\left(\forall\,x'\in{\partial B\cup{\{\Delta'\}}}:\ F_1^B(x')\geq \frac{1}{\Cr{Charnack}}\frac{\inf_{y'\in{V}}P_{y'}(Z_{H_B}=x')}{\inf_{y'\in{V}}P_{y'}(Z_{H_B}=x)}\pi^B(x)v\right)\\& \qquad \leq\P\left(\forall\,x'\in{\partial B\cup{\{\Delta'\}}}:\ F_1^B(x')\geq \frac{1}{\Cr{Charnack}}\inf_{y'\in{V}}P_{y'}(Z_{H_B}=x')v\right),
    \end{split}
\end{equation}
since $\pi^B(x)\geq\inf_{y'\in{V}}P_{y'}(Z_{H_B}=x)$ by \eqref{pibyx} and \eqref{defpib}. By \eqref{vsmallerthanGn} and \eqref{defF1}, if $F_1^B(x')\geq u$ for some $u>0$ and $x'\in{\partial B\cup{\{\Delta'\}}},$ then for every $\sigma\in{\Xi_B(x')}$ and $v'\in{[0,u]}$ such that $(\sigma,v')\in{\text{supp}(\eta)},$ $(\sigma,v')\in{\{(\sigma_1,v_1),\dots,(\sigma_{T^B},v_{T^B})\}},$ and thus by \eqref{eq:tailF1},
for all $x\in{\partial B}$ and $v\in{(0,\infty)},$
\begin{align*}
    \P&\left(F_1^B(x)\geq\pi^B(x)v\right) \\&\leq\P\bigg[\eta\bigg(\bigcup_{x'\in{\partial B\cup{\{\Delta'\}}}} \hspace{-.59em}\{\Xi_B(x')\}\times\Big[0,\frac{1}{\Cr{Charnack}}\inf_{y'\in{V}}P_{y'}(Z_{H_B}=x')v\Big]\bigg)\leq T^B\bigg]\nonumber\\
    & \leq a_1+a_2,
    \end{align*}
    where
    \begin{align}
    & \label{firstsummand}  a_1= \P\bigg[\eta\bigg(\bigcup_{x'\in{\partial B\cup{\{\Delta'\}}}}\{\Xi_B(x')\}\times\Big[0,\frac{1}{\Cr{Charnack}}\inf_{y'\in{V}}P_{y'}(Z_{H_B}=x')v\Big]\bigg)\leq \frac{v}{2\Cr{Charnack}^2}\bigg],
    \\& \label{secondsummand}  a_2= \P\left(T^B\geq\frac{v}{2\Cr{Charnack}^2}\right).
\end{align}
We bound $a_1$ and $a_2$ separately.
For all $x'\in{\partial B\cup{\{\Delta'\}}},$ $\mu_B(\Xi_B(x'))=1,$ see \eqref{eq:mu_B}, so the parameter of the Poisson variable in \eqref{firstsummand} is
\begin{equation*}
    \frac{1}{\Cr{Charnack}}\sum_{x'\in{\partial B\cup{\{\Delta'\}}}}\inf_{y'\in{V}}P_{y'}(Z_{H_B}=x')v\geq\frac{v}{\Cr{Charnack}^2}
\end{equation*}
by Lemma \ref{Lemmaharnack}, and thus $a_1$ in \eqref{firstsummand} is indeed bounded by $C(K)\exp\{-c'(K)v\}$ by a standard concentration estimate for the Poisson distribution (recall that $\Cr{Charnack}=\Cr{Charnack}(K)$). We now seek an upper bound for \nolinebreak $a_2$. Assume for now that $B=A_1,$ whence $\{\Sigma_1=\Delta\}=\{H_{A_1}=\infty\}$ $P_y$-a.s.\ for all $y\in{V},$ and thus $T^B(= \inf\{n;\,\Sigma_n=\Delta\})$ is dominated by a geometric random variable with parameter 
  $  \inf_{y\in{V}}P_y(H_{A_1}=\infty)=1-\sup_{y\in{V}}P_y(H_{A_1}<\infty)$.
By \eqref{entrancegreenequi} and \eqref{eq:decsetup}, for all $y\in{V},$
\begin{equation}
\label{defC9}
\begin{split}
    P_y(H_{A_1}<\infty)=\sum_{x\in{A_1}}g(y,x)e_{A_1}(x)&\stackrel{\eqref{Green}}{\leq}\Cr{CGreen}\Big(\frac{s}{\sqrt{K}}-\Cr{Cdistance}\Big)^{-\nu}\mathrm{cap}(A_1)
    \\&\stackrel{\eqref{ballcapacity}}{\leq}2^{\nu}\Cr{CGreen}\Cr{Ccapacity}K^{-\nu/2},
\end{split}
\end{equation}
for all $y\in{V}$, where we used $s\geq(2\Cr{Cdistance}\sqrt{K})\vee(Kr)$ in the last inequality (this is guaranteed, cf.\ around \eqref{eq:decsetup}). By choosing $K_0$ large enough, we can ensure that the last constant in \eqref{defC9} is, say, at most $1/2$ for all $K \geq K_0$, so that $T^B$ is dominated by a geometric random variable with positive parameter and then $a_2$ in \eqref{secondsummand} is bounded by $C(K)\exp\{-c(K)v\},$ for all $K \geq 0$ and $v\in (0,\infty)$. The proof is essentially the same if $B=A_2$ or $B=A_1\cup A_2;$ the only point that requires slight care is that $T^B\geq2$ on account of \eqref{eq:decsetup}, and thus we use instead that $T^B-1$ is bounded by a suitable geometric random variable.
\end{proof}
With Lemma \ref{beforedecoupRI} at hand, we are now able to prove Lemma \ref{boundonsoftlocaltimes} using arguments similar to those appearing in the proof of Theorem 2.1 in \cite{MR3420516}.
\begin{proof}[Proof of Lemma \ref{boundonsoftlocaltimes}]
By \eqref{GubFkB}, \eqref{defpib} and Markov's inequality, we can write for all $a>0,$ $x\in{\partial B}$ and $\eps\in{(0,1)},$ recalling that $\Theta_u^V$ and the family $F$ are independent,
\begin{align}
\label{3.1.1}
   & \P\left(G_u^B(x)\geq\big(1+\frac{\eps}{3}\big)\E[G_u^B(x)]\right)\nonumber\\
   &\qquad\leq\E\left[\Big(\E\big[\exp\left\{aF_1^B(x)\right\}\big]\Big)^{\Theta_u^V}\right]\exp\left\{-a\big(1+\frac{\eps}{3}\big)u\mathrm{cap}(V)\pi^B(x)\right\}\nonumber
    \\&\qquad\leq\exp\left\{u\mathrm{cap}(V)\left(\E\big[\exp\left\{aF_1^B(x)\right\}\big]-1-a\big(1+\frac{\eps}{3}\big)\pi^B(x)\right)\right\}.
\end{align}
We now bound $\E\big[\exp\left\{aF_1^B(x)\right\}\big]$ for small enough $a$. If $t\in{[0,1]},$ $e^t\leq 1+t+t^2,$ so by \eqref{upperboundfb2} of Lemma \ref{beforedecoupRI}, for $K \geq K_0$, $x \in \partial B$ and $a>0$,
\begin{equation}
\label{3.1.2}
    \E\left[\exp\left\{aF_1^B(x)\right\}\1_{\{F_1^B(x)\leq a^{-1}\}}\right]\leq1+a\pi^B(x)+4a^2\Cr{Charnack}\pi^B(x)^2
\end{equation}
(recall for purposes to follow that $\Cr{Charnack}$ and also $\Cr{Clemmadecou}$, $\Cr{clemmadecou}$ all depend on $K$). Moreover, by \eqref{largedeviationF1B} of Lemma \ref{beforedecoupRI}, for all $K\geq K_0$, $x\in{\partial B}$ and $a\in{\big(0,\frac{\Cr{clemmadecou}}{2\pi^B(x)}\big]},$ 
\begin{equation}
\label{3.1.3}
\begin{split}
&\E\left[\exp\left\{aF_1^B(x)\right\}\1_{\{F_1^B(x)>a^{-1}\}}\right] \leq a\int_{a^{-1}}^{\infty}e^{at}\P(F_1^B(x)>t)\ \mathrm{d}t+e\P(F_1^B(x)>a^{-1})\\
& \quad\leq  a\pi^B(x) \Cr{Clemmadecou} \int_{(a\pi^B(x))^{-1}}^{\infty}e^{(a\pi^B(x)-\Cr{clemmadecou})t}\ \mathrm{d}t + e\times\Cr{Clemmadecou}e^{-\frac{\Cr{clemmadecou}}{a\pi_B(x)}}\\
&\quad \leq \Cr{Clemmadecou}(1+e)e^{-\frac{\Cr{clemmadecou}}{2a\pi^B(x)}}
\leq \Cr{Clemmadecou}(1+e)\left(\frac{2a\pi^B(x)}{\Cr{clemmadecou}}\right)^2,
\end{split}
\end{equation}
where we took advantage of the inequality $e^{-x}< \frac{1}{x^2}$ for $x>0$ in the last step.
Thus, combining \eqref{3.1.1}, \eqref{3.1.2} and \eqref{3.1.3} with the choice $a=\frac{c(K)\eps}{\pi^B(x)}$ for a small enough constant $c(K)>0$, we have for all $x\in{\partial B}$ and $\eps\in{(0,1)}$ and $K \geq K_0,$
\begin{equation*}
\P\left(G_u^B(x)\geq(1+\frac{\eps}{3})\E[G_u^B(x)]\right)\leq\exp\left\{-c'(K)u\eps^2\mathrm{cap}(V)\right\}\stackrel{\eqref{ballcapacity}}{\leq}\exp\left\{-c''(K)u\eps^2s^{\nu}\right\}.
\end{equation*}
In a similar way, one can bound $ \P(G_u^B(x)\leq(1-\frac{\eps}{3})\E[G_u^B(x)])$ from above. Indeed, using instead that for all $t>0,$ $e^{-t}\leq1-t+t^2,$ and so by \eqref{upperboundfb2} of Lemma \ref{beforedecoupRI}, one obtains for $a>0$, $x\in{\partial B}$ and $K \geq K_0$,
\begin{equation*}
    \E\left[\exp\left\{-aF_1^B(x)\right\}\right]\leq1-a\pi^B(x)+4a^2\Cr{Charnack}\pi^B(x)^2.
\end{equation*}
This completes the proof.
\end{proof}

We can now conclude.
\begin{proof}[Proof of \eqref{decouplingRI}]
Consider $\tilde{A}_1$ and $\tilde{A}_2$ as in the statement of Theorem \ref{decoup} and set
$\Cr{Ceta}=K_0$ with $K_0$ as appearing in Lemma \ref{boundonsoftlocaltimes}. This fits within the framework described above \eqref{eq:decsetup} with $K=K_0$, whence \eqref{decoupRI2} and Lemma \ref{boundonsoftlocaltimes} apply. Thus, \eqref{decouplingRI} follows upon using  \eqref{Ahlfors}, \eqref{lambda} and \eqref{eq:decsetup} to bound $|\partial B|$ for any $B\in{\{A_1,A_2,A_1\cup A_2\}}$.
\end{proof}

\section{General renormalization scheme}\label{S:renorm}

We now set up the framework for the multi-scale analysis that will lead to the proof of Theorems \ref{T:mainresultGFF} and \ref{T:mainresultRI} in Section \nolinebreak\ref{denouement}. This will bring together the coupling $\tilde{\mathbb{P}}^u$ from Section \nolinebreak\ref{seclevelsettilde}, see Corollary \nolinebreak\ref{phiisagff} and Remark \nolinebreak\ref{R:iso}, \ref{R:iso2}), and the decoupling inequalities of Theorem \nolinebreak\ref{decoup}, which have been proved in Section \nolinebreak\ref{secdecoup} and which will be used to propagate certain estimates from one scale to the next, see Proposition \nolinebreak\ref{Cordecoup} below, much in the spirit of \cite{MR2680403} and \cite{MR2891880}. Crucially, this renormalization scheme will be applied to a carefully chosen set of ``good'' local features indexed by points on the approximate lattice $\Lambda(L_0)$ (cf.\ Lemma \ref{Lambda}) at the lowest scale $L_0$, see Definition \ref{defgood}, which involve the fields $(\tilde{\gamma}_{\cdot}, \tilde{\ell}_{\cdot,u},\mathcal{B}_{\cdot}^p)$ from the coupling $\tilde{\mathbb{Q}}^{u,p}$, see \eqref{eq:Qup}. Importantly, good regions will allow for good local control on the set $\mathcal{C}_u^\infty$ which is defining for $\tilde{\varphi}_{\cdot}$, see \eqref{couplingbetweenGFFandRI}, and in particular of the $\tilde{\gamma}_{\cdot}$-sign clusters in the vicinity to the interlacement, cf.\ \eqref{eq:Cu}. This will for instance be key in obtaining the desired ubiquity of the two infinite sign clusters in \eqref{eq:MainGFFinformal}, see also \eqref{eqhbar1} and \eqref{eqhbar2}.

Following ideas of \cite{MR2680403}, improved in \cite{MR2891880}, \cite{MR3420516} for random interlacements and extended in \cite{MR3053773}, \cite{MR3325312} to the Gaussian free field, we first introduce an adequate renormalization scheme. As before, $G$ is any graph satisfying the assumptions \eqref{eq:Ass}. We introduce a triple $\mathcal{L}=(L_0,\overline{l},l_0)$ of parameters
\begin{equation}
\label{mathcalL}
L_0\geq\Cr{Cdistance}, \ \overline{l}\geq2\quad \text{ and } \quad l_0\geq8^{1/\nu}\vee\Cr{CLambda}^{-\frac{1}{2\alpha}}\vee(8+4\Cr{Ceta})\overline{l}
\end{equation}
(cf.\  \eqref{conditiondistance} for the definition of $\Cr{Cdistance},$ before \eqref{decouplingRI} for $\Cr{Ceta}$, \eqref{Lambda2} for $\Cr{CLambda},$ and recall $\nu$ from \eqref{eq:nu}), and define
\begin{equation}
\label{defrenor}
    L_n=l_0^nL_0\quad \text{ and } \quad \Lambda_n^{\mathcal{L}}=\Lambda(L_{n})\ \text{ for all }n\in\{0,1,2,\dots\}.
\end{equation}
Here, $\Lambda(L),$ $L\geq1$ is any  henceforth fixed sequence of subsets of $G$ as given by Lemma \nolinebreak\ref{Lambda}. For any family $B=\{B_x :  x\in\Lambda_0^{\mathcal{L}}\}$ of events defined on a common probability space, we introduce the events $G_{x,n}^{\mathcal{L}}(B)$ for all $x\in{\Lambda_n^{\mathcal{L}}}$ recursively in $n$ by setting 
\begin{align}
\label{Gxndef}
\begin{split}
G_{x,0}^{\mathcal{L}}(B)&=B_x \quad \quad \text{ for all $x\in{\Lambda_0^{\mathcal{L}}},$ and}\\
 G_{x,n}^{\mathcal{L}}(B)&=\bigcup_{\substack{y,y'\in{\Lambda_{n-1}^{\mathcal{L}}\cap B(x,\overline{l}L_{n})}\\d(y,y')\geq L_n}}G_{y,n-1}^{\mathcal{L}}(B)\cap G_{y',n-1}^{\mathcal{L}}(B)
\quad \text{ for all $n\geq1$ and $x\in{\Lambda_n^{\mathcal{L}}}$.}
\end{split}
\end{align}
We recall here that the distance $d$ in \eqref{Gxndef} and entering the definition of balls is the one from \eqref{eq:Ass} (consistent with the regularity assumptions \eqref{Ahlfors} and \eqref{Green}) and thus in general \textit{not} the graph distance, cf.\ Remark \ref{remarkendofsection3}. Note that since $L_0\geqslant\Cr{Cdistance}$ and $l_0\geqslant2\overline{l}\geqslant4,$ see \eqref{mathcalL}, then by \eqref{conditiondistance}, \eqref{Lambda1} and \eqref{defrenor} the union in \eqref{Gxndef} is not empty.
For $\tilde{A}$ any measurable subset of $\tilde{G}$ and $B$ a measurable subset of $C({\tilde{A}},\R),$ we say that $B$ is {\em increasing} if for all $f\in{B}$ and $f'\in{C({\tilde{A}},\R)}$ with $f\leq f',$ $f'\in{B},$ and $B$ is {\em decreasing} if $B^{{c}}$ is increasing. For $h\in{\R}$ and $u>0$, we define the events
\begin{equation}
\label{defevents}
    B^{G,h}=\{\tilde{\Phi}_{|\tilde{A}}+h\in{B}\}\quad \text{ and }\quad B^{I,u}=\{\tilde{\ell}_{\tilde{A},u}\in{B}\},
\end{equation}
and we add the convention $B^{I,u}=\emptyset$ for $u\leq0.$ If $B$ is increasing then \eqref{defevents} implies that $B^{G,h} \subset B^{G,h'}$ for $h<h'$ and $B^{I,u} \subset B^{I,u'}$ for $u < u'$. 
\begin{Prop} 
\label{Cordecoup}
For all graphs $G$ satisfying \eqref{eq:Ass}, there exist $\Cl[c]{ccor}>0$ and  $\Cl{Ccor}\geq1$ such that for all all $L_0,$ $\overline{l}$ and $l_0$ as in \eqref{mathcalL}, all $\eps>0$ and $h\in{\R}$ (resp. $u>0$) with
     \begin{equation}
     \label{decoupL0cond}
         \frac{\eps^2(\sqrt{l_0}L_0)^{\nu}}{\log(L_0+1)}\geq \Cr{Ccor}\qquad \Big(\text{resp. }\frac{u\eps^2(\sqrt{l_0}L_0)^{\nu}}{\log(L_0+1)}\geq \Cr{Ccor}\Big),
     \end{equation} 
     and all families $B= \{B_x : \, x\in{\Lambda_0^{\mathcal{L}}}\}$ such that the sets $B_x,$ $x\in{\Lambda_0^{\mathcal{L}}},$ are either all increasing or all decreasing measurable subsets of $C(\tilde{B}(x,\overline{l}L_0),\R)$ satisfying
    \begin{equation}
    \label{decoupatlevel0}
        \tilde{\P}^G(B_x^{G,h})\leq\frac{\Cr{ccor}}{l_0^{4\alpha}}\qquad \Big(\text{resp. }\tilde{\P}^I(B_x^{I,u})\leq\frac{\Cr{ccor}}{l_0^{4\alpha}}\Big) \qquad \text{ for all }x\in{\Lambda_0^{\mathcal{L}}},
    \end{equation}
    one has for all $n\in\{0,1,2,\dots\}$ and $x\in{\Lambda_n^{\mathcal{L}}},$ 
    \begin{equation}
    \label{decoupatleveln}
        \tilde{\P}^G\big(G_{x,n}^{\mathcal{L}}(B^{G,h\pm\eps})\big)\leq 2^{-2^n}\qquad \big(\text{resp. }\tilde{\P}^I\big(G_{x,n}^{\mathcal{L}}(B^{I,u(1\pm\eps)})\big)\leq 2^{-2^n}\big),
    \end{equation}
    where the plus sign corresponds to the case where the sets $B_x$ are all decreasing and the minus sign to the case where the sets $B_x$ are all increasing.
\end{Prop}

\begin{proof}
    We give the proof for the Gaussian free field in the case of decreasing events. The proof for increasing events and/or random interlacements is similar and relies in the latter case on \eqref{decouplingRI} rather than \eqref{decouplingGFF}, which will be used below. Thus, fix some $\eps>0,$ $h\in{\R},$ $\overline{l}$ and $l_0$ as in \eqref{mathcalL}, and assume $B= \{B_x : \, x\in{\Lambda_0^{\mathcal{L}}}\}$ is such that $B_x$ is a decreasing subset of $C({\tilde{B}(x,\overline{l}L_0)},\R)$ satisfying \eqref{decoupatlevel0}, for all $x\in{\Lambda_0^{\mathcal{L}}}$. The sequence $(h_n)_{n\geq 0}$ is defined by $h_0=h$ and for all $n\geq1$,
   $h_{n}=h+\sum_{k=1}^n\frac{\eps\wedge1}{2^k}$,
    whence $h_n\leq h + \eps$ for all $n$. 
    
    We now argue that there exists a constant $\Cr{Ccor}$ such that, if the first inequality in \eqref{decoupL0cond} holds, then for all $n\in \{0,1,2,\dots\}$,
    \begin{equation}
    \label{resttoprove}
        \tilde{\P}^G\big(G_{x,n}^{\mathcal{L}}(B^{G,h_n})\big)\leq \frac{2^{-2^n}}{2\Cr{CLambda}^2l_0^{4\alpha}} \text{ for all $x \in \Lambda_n^{\mathcal{L}}$,}
    \end{equation}
    with $\alpha$ as in \eqref{Ahlfors} and $\Cr{CLambda}$ defined by \eqref{Lambda2}. It is then clear that \eqref{decoupatleveln} follows from \eqref{resttoprove} since $l_0\geqslant\Cr{CLambda}^{-\frac{1}{2\alpha}}$ and the sets $B_x,$ $x\in{\Lambda_0^{\mathcal{L}}},$ are decreasing. We prove \eqref{resttoprove} by induction on $n$: for $n=0,$ \eqref{resttoprove} is just \eqref{decoupatlevel0} upon choosing
    \begin{equation*}
        \Cr{ccor}\stackrel{\text{def.}}{=}\frac{1}{4\Cr{CLambda}^2}.
    \end{equation*}
    Assume that \eqref{resttoprove} holds at level $n-1$ for some $n\geq1.$ Note that by \eqref{Gxndef} and \eqref{mathcalL}, for all $h'>0$ and $x\in \Lambda_{n-1}^{\mathcal{L}}$, $G_{x,n-1}^{\mathcal{L}}(B^{G,h'})\in{\sigma\big(\tilde{\Phi}_x , \, x \in \widetilde{B}(x,2\overline{l}L_{n-1})\big)}$. Let $r_n=2\overline{l}L_{n-1}$. Then, for all $x \in \Lambda_{n}^{\mathcal{L}}$ and $y,y'\in{\Lambda_{n-1}^{\mathcal{L}}\cap B(x,\overline{l}L_{n})}$ such that $d(y,y')\geq L_n$ (as appearing in the union in \eqref{Gxndef}),
    \begin{equation*}
        2\overline{l}L_n\geq d\big(B(y,r_n),B(y',r_n)\big)\geq\left(l_0-4\overline{l}\right)L_{n-1}\stackrel{\eqref{mathcalL}}{\geq}\frac{l_0}{2}L_{n-1}\stackrel{\eqref{mathcalL}}{\geq}\Cr{Ceta}r_n\stackrel{\text{def.}}{=}s_n.
    \end{equation*}
    Using \eqref{Lambda2}, \eqref{Gxndef}, \eqref{defrenor}, a union bound and the decoupling inequality \eqref{decouplingGFF}, we get
    \begin{align*}
        &\tilde{\P}^G\big(G_{x,n}^{\mathcal{L}}(B^{G,h_n})\big)\\
        &\leq \big(\Cr{CLambda}l_0^{2\alpha}\big)^2\bigg[\Big(\sup_{y}\tilde{\P}^G\big( G_{y,n-1}^{\mathcal{L}}(B^{G,h_{n-1}})\big)\Big)^2
        +\Cr{CdecouGFF}L_{n+1}^{\alpha}\exp\Big(-\Cr{cdecouGFF}\frac{\eps^2}{2^{2n}}s_n^{\nu}\Big)\bigg],
    \end{align*}
    where the supremum is over all $y\in\Lambda_{n-1}^{\mathcal{L}}\cap B(x,\overline{l}L_{n}).$
    Then \eqref{resttoprove} follows by the induction hypothesis upon choosing $\Cr{Ccor}$ large enough such that for all $\overline{l}$ and $l_0$ as in \eqref{mathcalL}, $\eps\in{(0,1)}$ and $L_0\geqslant1$  such that the first inequality in \eqref{decoupL0cond} holds, as well as all $n\geq1,$
    \begin{equation*}
        \Cr{CdecouGFF}\Cr{CLambda}^{2}l_0^{(5+n)\alpha}L_{0}^{\alpha}\exp\Big(-\Cr{cdecouGFF}\frac{\eps^2}{2^{2n}}s_n^{\nu}\Big)\leq \frac{2^{-2^n}}{4\Cr{CLambda}^2l_0^{4\alpha}},
    \end{equation*}
    which is possible since $\eps^2s_n^{\nu}\geqslant\eps^2(\Cr{Ceta}L_{0}l_0^{n-1})^{\nu}\geqslant\Cr{Ccor}\log(L_0+1)(\sqrt{l_0}l_0^{n-1}/2)^{\nu}$ (where the first inequality takes advantage of \eqref{defrenor}) and $l_0^{\nu}\geq8.$
\end{proof}

\begin{Rk}
	\leavevmode
\begin{enumerate}[1)]
		\label{R:applicor}
		\item\label{R:subcrit}(Existence of a subcritical regime) As a first consequence of the scheme put forth in \eqref{mathcalL}--\eqref{defevents} and noteworthily under the mere assumptions \eqref{eq:Ass}, Proposition \nolinebreak\ref{Cordecoup} can be readily applied to a suitable family of events $ B= \{B_x : \, x\in{\Lambda_0^{\mathcal{L}}}\}$ and of parameters $\mathcal{L}$ in \eqref{mathcalL} to obtain (stretched) exponential controls on the connectivity function above large levels. This complements results in  \cite{MR2891880}. The argument is classical, see e.g.\@ \cite{MR2891880}, so we collect this result and simply sketch its proof. Let 
\begin{equation}
\label{h_**}
    h_{**}\stackrel{\text{def.}}{=}\inf \big\{h\in{\R};\,\liminf_{L\rightarrow\infty}\sup_{x\in{G}}\P^G\big(B(x,L)\stackrel{E^{\geq h}}{\longleftrightarrow}\partial B(x,2L)\big)=0\big\},
\end{equation}
where the event under the probability refers to the existence of a nearest neighbor path of vertices from the ball $B(x,L)$ to the boundary of the ball $\partial B(x,2L)$ in $E^{\geq h}.$ The parameter $u_{**}$ is defined similarly, but with the infimum ranging over $u \geq 0$ in \eqref{h_**} and the probability under consideration replaced by $\P^I\big(B(x,L)\stackrel{\mathcal{V}^u}{\longleftrightarrow}\partial B(x,2L)\big)$. By definition, $h_* \leq h_{**}$ and $u_*\leq u_{**}$, cf.\ \eqref{h_*} and \eqref{u_*}.

\begin{Cor} 
\label{applidecoup}
For $G$ satisfying \eqref{eq:Ass}, there exists $\Cl[c]{ch**cor} > 0 $
such that
    \begin{equation}
    \label{applidecoupGFF}
        h_{**}=\inf\big\{h\in{\R};\,\liminf\limits_{L\rightarrow\infty}\sup_{x\in{G}}\P^G\big(B(x,L)\stackrel{E^{\geq h}}{\longleftrightarrow}\partial B(x,2L)\big)<\Cr{ch**cor}\big\} < \infty 
    \end{equation}
    and
    \begin{equation}
    \label{applidecoupRI}
        u_{**}=\inf\big\{u\geq0;\,\liminf\limits_{L\rightarrow\infty}\sup_{x\in{G}}\P^I\big(B(x,L)\stackrel{\mathcal{V}^u}{\longleftrightarrow}\partial B(x,2L)\big)<\Cr{ch**cor}\big\} < \infty.
    \end{equation}
    Moreover, for all $h>h_{**}$ and $u>u_{**},$ there exist constants $c>0$ and $C<\infty$ depending on $u$ and $h$
 such that for all $x\in{G}$ and $L\geq1,$
    \begin{equation}
        \label{stretchedexpo}
        \P^G\big(x\stackrel{E^{\geq h}}{\longleftrightarrow}\partial B(x,L)\big)\leq C\exp\{-L^c\}\quad\text{and}\quad\P^I\big(x\stackrel{\V^u}{\longleftrightarrow}\partial B(x,L)\big)\leq C\exp\{-L^c\}.
    \end{equation}
\end{Cor}
\noindent We now outline the proof, and focus on \eqref{applidecoupRI}. One chooses
$\overline{l}=4$ and $l_0=8^{1/\nu}\vee\Cr{CLambda}^{-\frac{1}{2\alpha}}\vee(8+4\Cr{Ceta})\overline{l}$ in \eqref{mathcalL}, takes $\eps=1$ and fixes some $L_0$ large enough so that the second condition in \eqref{decoupL0cond} holds for all $u \geq 1$. It is then clear from \eqref{Ahlfors}, \eqref{Green} and \eqref{lambda} that one can find $u\geq 1$ large enough such that
    $\P^I\big(B(x,2L_0)\stackrel{\mathcal{\V}^u}{\longleftrightarrow}\partial B(x,4L_0)\big)
    \leq\sum_{y\in{B(x,2L_0)}}e^{-\frac{u}{g(y,y)}}\leq\Cr{ch**cor}\stackrel{\mathrm{def.}}{=}\Cr{ccor}l_0^{-4\alpha}$, for all $x\in{G}$,
    and where we used \eqref{defIu} and a union bound to infer the first inequality. Having fixed such $u$, one first shows that $u_{**}\leqslant2u$ and hence $u_{**}$ is finite 
    as asserted by applying Proposition \nolinebreak\ref{Cordecoup} as follows:  for $x\in{G},$ one considers 
    \begin{equation*}
    B_{x}=\big\{f\in{C({\tilde{B}(x,4L_0)},\R)}:\, B(x,2L_0)\stackrel{\{x\in{{G}};\,f(x)\leq0\}}{\longleftrightarrow}\partial B(x,4L_0)\big\},
    \end{equation*} which are decreasing measurable subsets of $C({\tilde{B}(x,4L_0)},\R)$, and one proves by induction over $n$ with the help of \eqref{Lambda1} that for all $n\in\{0,1,2,\dots\}$ and $x\in{G},$
\begin{equation}
\label{incluGxn}
 \big(  \{0\stackrel{\mathcal{V}^{u(1+\eps)}}{\longleftrightarrow}\partial B(x,4L_n)\}\subset \big) \ \    \{B(x,2L_n)\stackrel{\mathcal{V}^{u(1+\eps)}}{\longleftrightarrow}\partial B(x,4L_n)\}\subset G_{x,n}^{\mathcal{L}}(B^{I,u(1+\eps)})
\end{equation}
(for now $\varepsilon=1$ but this is in fact true for any $\varepsilon, u >0$). By the above choices, Proposition \nolinebreak\ref{Cordecoup} applies, yielding for all $n\geq 0$ that
$    \tilde{\P}^I\big(G_{x,n}^{\mathcal{L}}(B^{I,2u})\big)\leq 2^{-2^n}\leq C\exp\{-L_n^c\}$, and in particular,
 $ \lim_n \tilde{\P}^I\big(B(x,2L_n)\stackrel{\mathcal{V}^{2u}}{\longleftrightarrow}\partial B(x,4L_n)\}\big) = 0$, as desired.

To prove the equality in \eqref{applidecoupRI}, one repeats the above argument but with different choices of $u$, $L_0$ and $\varepsilon$. Namely, one considers any $u>0$ for which 
\begin{equation}
\label{P<c17}
\liminf_{L_0\rightarrow\infty}\sup_{x\in{G}}\P^I\big(B(x,2L_0)\stackrel{\mathcal{V}^u}{\longleftrightarrow}\partial B(x,4L_0)\big)<\Cr{ch**cor}.
\end{equation}
It suffices to show that $u(1+\varepsilon) \geq u_{**}$, for then by letting $\varepsilon \downarrow 0$, it follows that $u_{**}$ is smaller or equal than the infimum in \eqref{applidecoupRI}, and the reverse inequality is obvious, as follows from \eqref{h_**}. With $u$ and $\varepsilon$ fixed, one selects $L_0 \geq 1$ large enough so as to ensure \eqref{decoupL0cond}, and such that the probabilities in \eqref{P<c17} are smaller than $\Cr{ch**cor}.$ Proposition \ref{Cordecoup} then implies as explained above that  $ \lim_n \tilde{\P}^I\big(B(x,2L_n)\stackrel{\mathcal{V}^{u(1+\eps)}}{\longleftrightarrow}\partial B(x,4L_n)\}\big) = 0 $ and $L\mapsto\P^I\big(x\stackrel{\mathcal{V}^{u(1+\varepsilon)}}{\longleftrightarrow}\partial B(x,L)\big)$ has stretched exponential decay in $L$ for all $x\in{G}$, thus yielding that $u(1+\varepsilon) \geq u_{**}$ and the interlacement part of \eqref{stretchedexpo} as a by-product. The proof of \eqref{applidecoupGFF} and the free field
  part of \eqref{stretchedexpo} follow similar lines. $\qquad$ 

\item \label{supernu>1}(Existence of a supercritical regime for $\nu>1$) Another simple consequence of Proposition 7.1 is that if $G$ is a graph satisfying \eqref{eq:Ass} with $\nu>1$ which contains a subgraph isomorphic to $\N^2,$ then, identifying with a slight abuse of notation this subgraph with $\N^2,$ there exists $u>0$ such that $\P^I$-a.s.,
\begin{equation}
\label{percoplanenu>1}
\V^u\cap\N^2\text{ contains an infinite connected component,}
\end{equation}
and in particular $u_*>0.$ In the proof of Theorem \ref{T:mainresultRI}, we only show that under the additional assumption \eqref{weakSecIso}, there exist $u>0$ and $L>0$ such that $\V^u\cap B(\N^2,L)$ contains an infinite connected component, see Theorem \ref{thmconclusion} and Remark \ref{lastremark}, \ref{percoplanes}). Thus, \eqref{percoplanenu>1} provides us with a stronger, and easier to prove, result for random interlacements when $\nu>1$. Examples of graphs for which we can prove \eqref{percoplanenu>1} are product graphs $G=G_1\times G_2$ as in Proposition \ref{P:product} with $\nu=\alpha-\beta>1$ since if $P_1$ and $P_2$ are two semi-infinite geodesics of $G_1$ and $G_2,$ which exist by Theorem 3.1 in \cite{MR864581}, then $P_1\times P_2$ is a subgraph of $G$ isomorphic to $\N^2.$ Also, finitely generated Cayley graphs verifying \eqref{Ahlfors} for some $\alpha>3$ which are not almost isomorphic to $\Z,$ see Theorem 7.18 in \cite{MR3616205}, are covered by this setting, as well as $d$-dimensional Sierpinski carpet with $d\geq4$, see Remark~\ref{R:extensions}, \ref{R:boundsierpinskicarpet}). 

Let us now sketch the proof of \eqref{percoplanenu>1}. Using the result from Exercise 1.16 in \cite{MR3308116}, which is given for $\Z^d$ but immediately transfers to our setting, we have for all positive integer $L,M$ and $N,$ since $\nu>1,$
\begin{align*}
\mathrm{cap}([M,M+L]\times\{N\})&\leq\frac{L+1}{\inf_{k\in{[M,M+L]}}\sum_{p=M}^{M+L}g((k,N),(p,N))}
\\&\hspace{-.4em} \stackrel{\eqref{Green}}{\leq}\frac{L+1}{\Cr{cGreen}\Cr{Cdistance}^{-\nu}\sum_{p=1}^Lp^{-\nu}}\leq CL.
\end{align*}
Here, we used that $d((k,N),(p,N))\leqslant\Cr{Cdistance}d_G((k,N),(p,N))\leqslant\Cr{Cdistance}|k-p|$ in the second inequality, see \eqref{conditiondistance}, and we also have a similar bound on the capacity of $\{M\}\times[N,N+L].$ For all positive integer $L$ and all $x\in{\{L+1,L+2,\dots\}^2},$ we write $\mathcal{S}(x,L)=x+\N^2\cap\partial_{\N^2}[-L,L]^2,$ where $\partial_{\N^2}A$ is the boundary of $A$ as a subset of $\N^2,$ and we thus get by a union bound
\begin{equation}
\label{cap0L}
\mathrm{cap}\left(\mathcal{S}(x,L)\right)\leq CL.
\end{equation}
Fix $\overline{l}=4$ and  $l_0=8^{1/\nu}\vee\Cr{CLambda}^{-\frac{1}{2\alpha}}\vee(8+4\Cr{Ceta})\overline{l}$ in \eqref{mathcalL}, take $\eps=1/2,$ and let $\Cl{Ch**cor}$ be such that for all $u>0$ and $L_0\geqslant\Cr{Cdistance}$ with $uL_0\leqslant\Cr{Ch**cor},$ and all $x\in{\{4L_0+1,4L_0+2,\dots\}^2},$
\begin{equation*}
\P^I\left(\mathcal{S}(x,2L_0)\stackrel{*\text{-}\I^u\cap\N^2}{\longleftrightarrow}\mathcal{S}(x,4L_0)\right)\stackrel{\eqref{cap0L}}{\leqslant}1-\exp\{-2CuL_0\}\leqslant\frac{\Cr{ccor}}{l_0^2},
\end{equation*} 
where $A\stackrel{*\text{-}B}{\longleftrightarrow}C$ means that there exists a $*$-path in $B\subset\N^2,$ as defined above Proposition \ref{generalstarconnected}, beginning in $A$ and ending in $C.$ Since $\nu>1$ one can find $L_0$ large enough so that \eqref{decoupL0cond} holds when $u=\Cr{Ch**cor}L_0^{-1},$ and, applying Proposition \ref{Cordecoup} and using a property similar to \eqref{incluGxn} for $*$-paths of $\I^u,$ we get that $L\mapsto\sup_x\P^I (x\stackrel{*\text{-}\I^{u/2}\cap\N^2}{\longleftrightarrow}\mathcal{S}(x,L))$ has stretched exponential decay, with the supremum ranging over all $x\in{\{L+1,L+2,\dots\}^2}.$ If $\V^u\cap\N^2$ has no infinite connected component, then for any positive integer $L$ the sphere $\partial_{\N^2}[0,L]^2$ is not connected to $\infty$ in $\V^u\cap\N^2.$ Thus, by planar duality, see for instance Proposition \ref{generalstarconnected}, there exists $L'\geqslant L-1$ and $x\in{\{L'+1,L'+2\}\times\{L'+1\}}$ which is connected to $\mathcal{S}(x,L')$ by a $*$-path in $\I^u\cap\N^2,$ which happens with probability $0.$

In order to prove $u_*>0$ for $\nu=1$ by the same method, one would need to remove the polynomial term $(r_n+s_n)^{\alpha}$ in the decoupling inequality \eqref{decouplingRI}, and it seems plausible that one could do that for a large class of graphs (including $\Z^3$), using arguments similar to \cite{MR3126579} or \cite{MR3602841}. This is proved in the case $G=G'\times\Z$ in \cite{MR2891880}. However, this method does not seem to work in the case $\nu<1.$ A (simpler) proof of $u_*>0$ is given for $G=\Z^d$ in \cite{MR3304409} without using decoupling inequalities, but it seems that one cannot adapt simply its proof to more general graphs if $\nu<1.$ Therefore, the result $u_*>0$ from Theorem \ref{T:mainresultRI} is particularly interesting when $\nu<1.$ 
 
\end{enumerate}
\end{Rk}

\medskip
We now introduce the families of events of the form \eqref{defevents} to which Proposition \ref{Cordecoup} will eventually be applied. The reason for the following choices will become apparent in the next section.
 The strategy developed in \cite{DrePreRod} to prove $h_*>0$ on $\Z^d,$ $d\geq3,$ serves as a starting point in the current setting, but the desired ubiquity result \eqref{eq:mainGFF2} requires a considerably finer analysis, which is more involved, see also Remark \ref{R:good} below. All our events will be defined under the probability $\tilde{\Q}^{u,p}$ from \eqref{eq:Qup}, under which the Gaussian free field $\tilde{\phi}_{\cdot}$ on $\tilde{G}$ is defined in terms of $(\tilde{\gamma}_{\cdot}$, $\tilde{\ell}_{\cdot,u})$ by means of \eqref{couplingbetweenGFFandRI}.

We now come to the central definition of good vertices. As usual, we denote by $(\ell_{x,u})_{x\in{G}}=(\tilde{\ell}_{x,u})_{x\in{G}},$ $\I^{u}=\tilde{\I}^{u}\cap G,$ $\gamma=(\tilde{\gamma}_x)_{x\in{G}}$ and $\phi=(\tilde{\phi}_x)_{x\in{G}}$ the projections of $\tilde{\ell},$ $\tilde{\I}^u,$ $\tilde{\gamma}$ and $\tilde{\phi}$ on the graph $G.$ For all $u>0,$ these fields have the same law as the occupation time field of random interlacements at level $u$, a random interlacement set at level $u$ and two Gaussian free fields on $G,$ respectively. We recall the definition of the constants $\Cr{Cstrong}$ from \eqref{ballsalmostconnected}, $\Cr{Cdistance}$ from \eqref{conditiondistance}, and $\Cr{ccapbig}$ from Proposition \ref{10}, the definition of $\mathcal{B}_y^p$ from \eqref{eq:Qup}, the definition of $\hat{\I}^u$ from above \eqref{strongonedges}, and that $C^u(x,L)$ is the set of vertices in $G$ connected to $x$ by a path of edges in $\hat{\I}^u\cap B_E(x,L)$, see below Lemma \ref{12}.

\begin{Def}[\textbf{Good vertex}]
	\label{defgood}
	For $u>0,$ $L_0\geq1,$ $K>0$, $p\in (0,1)$, $x\in{G},$ the event
	\begin{enumerate}[(i)]
		\item \label{defCx}$C_x^{L_0,K}$ occurs if and only if $\tilde{\gamma}_z\geq -\tfrac{K}{2}$ for all $z\in{\tilde{B}(x,3\Cr{Cstrong}(L_0+\Cr{Cdistance})+2L_0+\Cr{Cdistance})},$ 
		\item \label{defDx}$D_x^{L_0,u}$ occurs if and only if $\I^{u/4}\cap B(x,L_0)\neq\emptyset,$ 
		\item \label{defDxhat}$\hat{D}_x^{L_0,u}$ occurs if and only if $\mathrm{cap}\big({C}^{u/2}(y,2(L_0+\Cr{Cdistance}))\big)\geq\Cr{ccapbig}(L_0+\Cr{Cdistance})^{\frac{3\nu}4}(\frac{u}{8})^{\lfloor\gamma-1\rfloor}$ for all $y\in{\I^{u/4}\cap B(x,L_0+\Cr{Cdistance})},$
		\item \label{defDxbar}$\overline{D}_x^{L_0,u}$ occurs if and only if 
		\begin{equation}
		\label{strongconnectivityE}
		y\stackrel{\wedge}{\longleftrightarrow} y'\ \text{in} \ \hat{\I}^u\cap {B}_E(x,3\Cr{Cstrong}(L_0+\Cr{Cdistance}))
		\end{equation}
		for all $y,y'\in{\I}^{u/2}\cap {B}(x,L_0+\Cr{Cdistance})$ such that $\mathrm{cap}\big({C}^{u/2}(y,2(L_0+\Cr{Cdistance}))\big)\geq\Cr{ccapbig}(L_0+\Cr{Cdistance})^{3\nu/4}(u/8)^{\lfloor\gamma-1\rfloor}$ and $\mathrm{cap}\big({C}^{u/2}(y',2(L_0+\Cr{Cdistance}))\big)\geq\Cr{ccapbig}(L_0+\Cr{Cdistance})^{3\nu/4}(u/8)^{\lfloor\gamma-1\rfloor}$.
		\item \label{defEx}
		$E_x^{L_0,u}$ occurs if and only if every component of $\{y\in{G};\,\phi_y\geq-\sqrt{2u}\}\cap B(x,L_0/2)$ with diameter at least $L_0/4$ is connected to $\I^{u/4}$ in $\{y\in{G};\,\phi_y\geq-\sqrt{2u}\}\cap B(x,L_0),$
		\item \label{defFx}$F_x^{L_0,p}$ occurs if and only if $\mathcal{B}^p_y=1$ for all $y\in{B(x,3\Cr{Cstrong}(L_0+\Cr{Cdistance})+2L_0)}.$ 
	\end{enumerate}
	Moreover, a vertex $x\in{G}$ is said to be {\em $(L_0,u,K,p)$-good} if the event 
	\begin{equation}
	\label{eq:goodvertex}
	C_x^{L_0,K}\cap D_x^{L_0,u}\cap\hat{D}_x^{L_0,u}\cap\overline{D}_x^{L_0,u}\cap E_x^{L_0,u}\cap F_x^{L_0,p}
	\end{equation}
	occurs, and {\em $(L_0,u,K,p)$-bad} otherwise. 
\end{Def}

\begin{Rk}\label{R:good}
	The above definition of good vertices differs in a number of ways from a corresponding notion introduced in \cite{DrePreRod} (cf.\ Definition 4.2 therein) by the authors. This is due to the refined understanding of the isomorphism \eqref{Isomorphism} stemming from \eqref{eq:Cu} and \eqref{couplingbetweenGFFandRI}. Notably, property \eqref{defCx} above is new in dealing directly with $\tilde{\gamma}_{\cdot}$ (rather than $\tilde{\phi}_{\cdot}$). Observe that \eqref{defEx} involves both the field $\tilde{\phi}$ and the random interlacement set $\tilde{\I}^u$ simultaneously, coupled as in \eqref{couplingbetweenGFFandRI}. It will lead to a direct proof of the inequality $\overline{h}\geq0,$ see Corollary \ref{cor:hbargeq0}, without using our sign-flipping method, Proposition \ref{iuincluvu}. Properties \eqref{defDx}, \eqref{defDxhat} and \eqref{defDxbar} can be viewed as a more transparent substitute for the events involved in Lemma 3.3 and Definition 3.4 in \cite{DrePreRod} (see also (4.1) in \cite{MR3024098}), and have the advantage of preserving the local uniqueness of interlacements, at the cost of introducing a sprinkling between $u/4$ and $u.$ It would be possible to find sharp estimates on the `size' of the interlacement in a ball similar to Lemma 3.3 in \cite{DrePreRod} on the class of graphs considered here, but such bounds are in fact unnecessary once we have Lemma \ref{12} and Proposition \ref{10}. 
\end{Rk}

\medskip
We conclude this section by collecting the following result, which will be crucially used in the next section. It sheds some light on why good vertices may be useful. 

\begin{Lemme}
	\label{whygood}
	For all $u>0,$ $L_0\geq1,$ $K>0$, $p \in (0,1)$ and any connected set $A\subset G$ such that each $x\in{A}$ is an $(L_0,u,K,p)$-good vertex, there exists a connected set $\tilde{A}$ such that
	\begin{equation}
	\label{Atilde}
	\emptyset\neq{\I}^{u/4}\cap {B}(x,L_0)\subset\tilde{A}\text{ for all }x\in{A},\,\tilde{A}\subset\tilde{\I}^u\cap \tilde{B}(A,3\Cr{Cstrong}(L_0+\Cr{Cdistance})),
	\end{equation}
	as well as
	\begin{equation}
	\label{Atildephi}
	\begin{array}{c}
	\text{for all $x\in{A},$ $\tilde{A}\cap B(x,L_0)\neq\emptyset$ and every connected component}
	\\\text{of }\{y\in{G};\,\phi_y\geq-\sqrt{2u}\}\cap B(x,L_0/2)
	\text{ with diameter at least}
	\\L_0/4\text{ is connected to $\tilde{A}$ in $\{y\in{G};\,\phi_y\geq-\sqrt{2u}\}\cap B(x,L_0)$}
	\end{array}
	\end{equation}
	and
	\begin{equation}
	\label{Atilde2}
	\begin{split}
	&\tilde{\gamma}_z\geq-K/2\text{ for all }z\in{\tilde{B}(\tilde{A},2L_0+\Cr{Cdistance})}\text{ and }\mathcal{B}_y^p=1\text{ for all }y\in{B(\tilde{A}\cap G,2L_0)}.
	\end{split}
	\end{equation}
\end{Lemme}
\begin{proof}
	For all $x_1\sim x_2\in{A},$ by \eqref{defDx} of Definition \ref{defgood}, there exists $y_i\in{{\I}^{u/4}\cap \tilde{B}(x_i,L_0)}$ for each $i.$  By \eqref{conditiondistance}, $d(x_1,y_{2})\leq L_0+\Cr{Cdistance}$ and by \eqref{defDxhat} of Definition \ref{defgood} $\mathrm{cap}\big({C}^{u/2}(y_i,L_0+\Cr{Cdistance})\big)\geq\Cr{ccapbig}(L_0+\Cr{Cdistance})^{3\nu/4}(u/8)^{\lfloor\gamma-1\rfloor}$ for each $i\in{\{1,2\}}.$ Therefore, by \eqref{strongconnectivityE}, $y_1\stackrel{\wedge}{\longleftrightarrow} y_{2}$ in $\hat{\I}^u\cap{B}_E(x_1,3\Cr{Cstrong}(L_0+\Cr{Cdistance})),$ and since each edge traversed by a trajectory of the random interlacement process is included in $\tilde{\I}^u,$ we also have that $y_1\stackrel{\sim}{\longleftrightarrow} y_{2}$ in $\tilde{\I}^u\cap\tilde{B}(x_1,3\Cr{Cstrong}(L_0+\Cr{Cdistance})).$ We now define $\tilde{A}$ as the union of the connected paths in $\tilde{\I}^u\cap\tilde{B}(x,3\Cr{Cstrong}(L_0+\Cr{Cdistance}))$ between $y$ and $y'$ for all $x\in{A}$ and $y,y'\in{B(x,L_0+\Cr{Cdistance})\cap \I^{u/4}},$ which is thus connected and it is clear that \eqref{Atilde} holds.
	
	For all $x\in{A},$ we clearly have $\tilde{A}\cap B(x,L_0)\neq\emptyset$ by \eqref{Atilde}. Moreover, we have by \eqref{defEx} of Definition \ref{defgood} that every connected component of $\{y\in{G};\,\phi_y\geq-\sqrt{2u}\}\cap B(x,L_0/2)$  with diameter at least $L_0/2$ is connected to $\I^{u/4}$ in $\{y\in{G};\,\phi_y\geq-\sqrt{2u}\}\cap B(x,L_0),$ and thus is also connected to $\tilde{A}$ in $\{y\in{G};\,\phi_y\geq-\sqrt{2u}\}\cap B(x,L_0),$ and we obtain \eqref{Atildephi}. One infers from \eqref{defCx} and \eqref{defFx} of Definition \ref{defgood}  that \eqref{Atilde2} also hold.
\end{proof}

\section{Construction of a giant cluster}
\label{sectionpercsigncluster}
We are now going to use the general renormalization scheme from Proposition \ref{Cordecoup} to find a giant, or ubiquitous, cluster of $(L_0,u,K,p)$-good vertices, as defined in Definition \ref{defgood}, or of $\tilde{\I}^u$ with suitable properties. This comes in several steps. The first one is reached in Proposition \ref{Rpathofbad} below and yields under the mere assumptions \eqref{eq:Ass} that long good ($R$-)paths, cf.\ Definition \ref{defgood}, are very likely for suitable choices of the parameters. The second step is to prove the existence of a suitable infinite cluster $\tilde{A}$ of $\tilde{\I}^u$ and is presented in Lemma \ref{percolfora}, and the third step is to prove that this cluster is ubiquitous, see Lemma \ref{proofofhbartilde>0}. This giant cluster $\tilde{A}$ of $\tilde{\I}^u$ verifies \eqref{Atilde2} and is in the neighborhood of a cluster $A$ of good vertices, for which \eqref{Atildephi} hold. It can be seen as precursor of the giant cluster of $E^{\geq h},$ $h>0,$ that we will construct in Section \ref{denouement}, which will lead to \eqref{eqhbar1} and \eqref{eqhbar2} (for small $h>0$). In a sense, the resulting estimates \eqref{Atildeexpodecay} and \eqref{eqhbaru>0} provide a rough translation of the events appearing in \eqref{eqhbar1} and \eqref{eqhbar2} to the world of interlacements, and deliver directly \eqref{eqhbar1} and \eqref{eqhbar2} for any $h<0,$ see Corollary \ref{cor:hbargeq0}. Apart from the quantitative bounds leading to  Proposition \ref{Rpathofbad}, these two estimates crucially rely on the additional geometric information provided by \eqref{weakSecIso}, on all aspects of Definition \nolinebreak\ref{defgood} and on certain features of the renormalization scheme, in particular with regards to the desired ubiquity, gathered in Lemma \nolinebreak\ref{largecompareconnbygood} below.

We continue in the framework of the previous section and recall in particular the scheme \eqref{mathcalL}--\eqref{Gxndef}, the measure $\widetilde{\Q}^{u,p}$ from \eqref{eq:Qup} and Definition \ref{defgood}. We also keep our standing (but often implicit) assumption that $G$ satisfies \eqref{eq:Ass} and mention any other condition, such as \eqref{weakSecIso}, explicitly. Henceforth, we set
\begin{equation}
\label{lbarandl0}
    \overline{l}=22\Cr{cdiam}\Cr{Cstrong},\quad  l_0=8^{1/\nu}\vee\Cr{CLambda}^{-\frac{1}{2\alpha}}\vee(8+4\Cr{Ceta})\overline{l},
\end{equation}
where
\begin{equation}
\label{defcdiam}
\Cl[c]{cdiam}\stackrel{\mathrm{def.}}{=}7(1+7\Cr{cwsi2}^{-1})\text{ if $G$ satisfies \eqref{weakSecIso} and }\Cr{cdiam}\stackrel{\mathrm{def.}}{=}7\text{ otherwise.}
\end{equation}
Note that $\overline{l}$ and $l_0$ satisfy the conditions appearing in \eqref{mathcalL}. For all $L_0\geq\Cr{Cdistance},$ we write $\mathcal{L}_0=(L_0,\overline{l},l_0)$ rather than $\mathcal{L}$ to insist on the choice \eqref{lbarandl0}. Thus $L_0 \geq \Cr{Cdistance}$ remains a free parameter at this point. We now define bad vertices at all scales $L_n$, $n \geq 0$, cf.\ \nolinebreak\eqref{defrenor}. For all $L_0 \geq \Cr{Cdistance}$, $x\in{\Lambda(L_0)}= \Lambda_0^{\mathcal{L}_0}$, $u>0$, $K>0$ and $p\in (0,1)$, we introduce 
\begin{equation}
\label{eq:boldevents}
    \mathbf{C}_x^{L_0,K}=\bigcap_{y\in{B(x,20\Cr{cdiam}\Cr{Cstrong}L_0)}}C_y^{L_0,K},
\end{equation}
and similarly $\mathbf{D}_x^{L_0,u},$ $\hat{\mathbf{D}}_x^{L_0,u},$ $\overline{\mathbf{D}}_x^{L_0,u}$ $\mathbf{E}_x^{L_0,u}$ and $\mathbf{F}_x^{L_0,p}$ by replacing $C_y^{L_0,K}$ with the relevant events $D_y^{L_0,u}$, $\hat{D}_y^{L_0,u}$, $\overline{D}_y^{L_0,u}$ $E_y^{L_0,u}$ and $F_y^{L_0,p}$ in Definition \ref{defgood}, (ii)--(vi). We introduce the family $(\mathbf{C}^{L_0,K})^{\mathsf{c}}= \{( \mathbf{C}_x^{L_0,K})^{\mathsf{c}}:  x\in  \Lambda_0^{\mathcal{L}_0}\},$ and the families  $(\mathbf{D}^{L_0,u})^{\mathsf{c}}$, $(\hat{\mathbf{D}}^{L_0,u}))^{\mathsf{c}},$ $(\overline{\mathbf{D}}^{L_0,u}))^{\mathsf{c}}$ $(\mathbf{E}^{L_0,u})^{\mathsf{c}}$ and $(\mathbf{F}^{L_0,p})^{\mathsf{c}}$ are defined correspondingly.
For $n \geq 0$ and $x\in{\Lambda_n^{\mathcal{L}_0}}$ (cf.\ \eqref{defrenor}), we then say that the vertex $x$ is 
 {\em $n-(L_0,u,K,p)$ bad} if (recall \eqref{Gxndef})
\begin{equation}
\label{defnbad}
\begin{array}{c}
    G_{x,n}^{\mathcal{L}_0}\big((\mathbf{C}^{L_0,K})^{\mathsf{c}}\big)\cup G_{x,n}^{\mathcal{L}_0}\big((\mathbf{D}^{L_0,K})^{\mathsf{c}}\big)\cup G_{x,n}^{\mathcal{L}_0}\big((\hat{\mathbf{D}}^{L_0,u})^{\mathsf{c}}\big) \\\cup
    G_{x,n}^{\mathcal{L}_0}\big((\overline{\mathbf{D}}^{L_0,u})^{\mathsf{c}}\big)\cup G_{x,n}^{\mathcal{L}_0}\big((\mathbf{E}^{L_0,u})^{\mathsf{c}}\big) \cup
    G_{x,n}^{\mathcal{L}_0}\big((\mathbf{F}^{L_0,p})^{\mathsf{c}}\big)
\end{array}
\end{equation}
occurs (under $\widetilde{\Q}^{u,p}$), and $x$ is {\em $n-(L_0,u,K,p)$ good} otherwise. In view of \eqref{eq:goodvertex} and the first line of \eqref{Gxndef}, an $(L_0,u,K,p)$-bad vertex in ${\Lambda_0^{\mathcal{L}_0}}$ is always a $0-(L_0,u,K,p)$ bad vertex, but not vice versa. A key to Proposition \ref{Rpathofbad}, see \eqref{lemmaatln} below, is to prove that the probability of having an $n-(L_0,u,K,p)$ bad vertex decays rapidly in $n$ for a suitable range of parameters $(L_0,u,K,p)$. This relies on individual bounds for each of the events in \eqref{defnbad}, which are the objects of Lemmas \ref{badGFF} and \ref{badinter} as well as \eqref{badbern} below. Due to the presence of long-range correlations, the decoupling estimates from Proposition \nolinebreak\ref{Cordecoup} will be crucially needed.

\begin{Lemme}
\label{badGFF}
There exist constants $\Cl{Cgff}<\infty$ and $\Cr{Cgff}'<\infty$ such that for all $L_0\geq\Cr{Cgff},$  $K\geq \Cr{Cgff}'\sqrt{\log(L_0)}$, $n\in \{0,1,2,\dots\}$ and  $x\in{\Lambda_n^{\mathcal{L}_0}}$, and all $u>0$, $p\in (0,1),$
\begin{equation}
\label{badgffeq}
    \widetilde{\Q}^{u,p}\big(G_{x,n}^{\mathcal{L}_0}\big((\mathbf{C}^{L_0,K})^{\mathsf{c}}\big)\big)\leq2^{-2^n}.
\end{equation}
\end{Lemme}
\begin{proof}
In view of \eqref{eq:boldevents}, Definition \ref{defgood} \eqref{defCx}, and \eqref{lbarandl0}, if $L_0\geq\Cr{Cdistance},$ the event $(\mathbf{C}^{L_0,K}_x)^{\mathsf{c}}$ is measurable with respect to the $\sigma$-algebra generated by $ \tilde{\gamma}_{|\tilde{B}(x,\overline{l}L_0)}$, and $(\mathbf{C}^{L_0,K}_x)^{\mathsf{c}}$ is of the form $\{ \tilde{\gamma}_{|\tilde{B}(x,\overline{l}L_0)}+K \in{B}_x\}$, cf.\ \eqref{defevents}, for a suitable decreasing 
 subset ${B}_x$ of $C({\tilde{B}(x,\overline{l}L_0)},\R)$. With this observation, and since $\tilde{\gamma}$ has the same law under $\widetilde{\Q}^{u,p}$ as $\tilde{\Phi}$ under $\tilde{\P}^G$, in order to show \nolinebreak\eqref{badgffeq}, it is enough by Proposition \ref{Cordecoup} to prove that there exists $\Cr{Cgff}'$ such that 
\begin{equation}
\label{resttoprovegff}
    \text{for all }L_0\geq\Cr{Cgff},\ K\geq \Cr{Cgff}'\sqrt{\log(L_0)}-1\text{ and } x\in \Lambda_0^{\mathcal{L}_0}:\ \widetilde{\Q}^{u,p}\left((\mathbf{C}_x^{L_0,K})^{\mathsf{c}}\right)<\frac{\Cr{ccor}}{l_0^{4\alpha}},
\end{equation}
where $\Cr{Cgff}\geq\Cr{Cdistance}\vee2$ is chosen so that the first inequality in \eqref{decoupL0cond} holds for all $L_0\geq\Cr{Cgff},$ with $l_0$ as in \eqref{lbarandl0} and $\eps=1.$ Conditionally on the field $\gamma= \tilde{\gamma}_{|G}$, and for each edge $e=\{y,y'\},$ the process $(\tilde{\gamma}_{y+te})_{t\in{[0,\rho_{y,y'}]}}$ on $I_e$ has the same law as a Brownian bridge of length $\rho_{y,y'} =1/(2\lambda_{y,y'})$ (the length of $I_e$, cf.\ below \eqref{eq:u_*alternative}) between $\gamma_y$ and $\gamma_{y'}$ of a Brownian motion with variance $2$ at time $1,$ as defined in Section 2 of \cite{DrePreRod}. This fact has already appeared in the literature, see Section 2 of \cite{MR3502602}, Section 1 of \cite{LuWe} or Section 2 of \cite{LuSaTa} for example. We refer to Section 2 of \cite{DrePreRod} for a proof of this result when $G= \Z^d,$ which can be easily adapted to a general graph satisfying \eqref{eq:Ass}. Let us denote by $(W_t^{y,y'})_{t\in{[0,\rho_{y,y'}]}}$ defined as $W_t^{y,y'}= \tilde{\gamma}_{y+te} - 2\lambda_{y,y'}t \tilde{\gamma}_{y'} - (1-2\lambda_{y,y'}t)\tilde{\gamma}_{y}$ the Brownian bridge of length $\rho_{y,y'}$ between $0$ and $0$ of a Brownian motion with variance $2$ at time $1$ associated with $(\tilde{\gamma}_{y+te})_{t\in{[0,\rho_{y,y'}]}}.$  For all $L\geq1,$ $K>0$ and $x\in{G},$ we thus have
\begin{equation}
\label{eq:supedges}
\begin{split}
   & \widetilde{\Q}^{u,p}\Big(\sup_{z\in{\tilde{B}(x,L)}}\tilde{\gamma}_z\geq \frac{K}{2}\Big)\\
   &\qquad \qquad \leq\widetilde{\Q}^{u,p}\Big(\sup_{y\in{{B}(x,L)}}\gamma_y\geq\frac{K}{4}\Big)+\sum_{\{y,y'\}\in{B_E(x,L)}}\widetilde{\Q}^{u,p}\Big(\sup_{t\in{[0,\rho_{y,y'}]}}W_t^{y,y'}\geq\frac{K}{4}\Big).
    \end{split}
\end{equation}
We consider both terms in \eqref{eq:supedges} separately. For all $y\in{B(x,L)},$ $\gamma_y$ is a centered Gaussian variable with variance $g(y,y),$ thus by \eqref{Ahlfors} and \eqref{Green}
\begin{align*}
    \widetilde{\Q}^{u,p}\Big(\sup_{y\in{{B}(x,L)}}\gamma_y\geq\frac{K}{4}\Big)&\leq \sum_{y\in{B(x,L)}}C\sqrt{\frac{g(y,y)}{K^2}}\exp\Big\{-\frac{K^2}{32g(y,y)}\Big\}\\&\leq \frac{CL^{\alpha}}{K}\exp\{-cK^2\}.
\end{align*}
The law of the maximum of a Brownian bridge is well-known, see for instance \cite{MR1912205}, Chapter IV.26, and so for all $y\sim y'$ in $G,$ by \eqref{lambda},
\begin{equation*}
    \widetilde{\Q}^{u,p}\Big(\sup_{t\in{[0,\rho_{y,y'}]}}W_t^{y,y'}\geq\frac{K}{4}\Big)=\exp\Big\{-\frac{K^2}{16\rho_{y,y'}}\Big\}\leq\exp\{-cK^2\},
\end{equation*}
where to obtain the inequality we took advantage of the fact that $\frac{1}{\rho_{y,y'}}=2\lambda_{y,y'} \geq c$,
cf.\@  \eqref{lambda}.
Therefore, returning to \eqref{eq:supedges}, using \eqref{Ahlfors}, \eqref{lambda} and the fact that $G$ has uniformly bounded degree, we obtain that for all $L\geq1$ and $K\geq1,$
$    \widetilde{\Q}^{u,p}(\sup_{z\in{\tilde{B}(x,L)}}\tilde{\gamma}_z\geq K)\leq CL^{\alpha}\exp\{-cK^2\}$.
Choosing $L=\overline{l}L_0$ and using the symmetry of $\tilde{\gamma}_{\cdot}$, we can finally bound for all $L_0\geq\Cr{Cgff}$ and $K\geq1,$
\begin{equation*}
    \widetilde{\Q}^{u,p}\big((\mathbf{C}_x^{L_0,K})^{\mathsf{c}}\big)\leq \widetilde{\Q}^{u,p}\Big(\sup_{z\in{\tilde{B}(x,\overline{l}L_0)}}\tilde{\gamma}_z\geq\frac{K}{2}\Big)\leq CL_0^{\alpha}\exp\{-cK^2\},
\end{equation*}
from which \eqref{resttoprovegff} readily follows for a suitable choice of $\Cr{Cgff}'$. 
\end{proof}

The next lemma deals with the events involving the families $\mathbf{D}_x^{L_0,u},$ $\hat{\mathbf{D}}_x^{L_0,u},$ $\overline{\mathbf{D}}_x^{L_0,u}$ and $\mathbf{E}_x^{L_0,u}$ in \eqref{defnbad}, which all involve the interlacement parameter $u>0$.  For the first three events, this will bring into play the connectivity estimates from Section \ref{secconnec} in order to initiate the decoupling.

\begin{Lemme}
\label{badinter}
For all $u_0>0,$ there exist constants $\Cl[c]{cinter}$ and $\Cl{Cinter}$ depending on $u_0$ such that for all $u\in{(0,u_0)},$ $L_0\geq\Cr{Cdistance}$ with $L_0u^{\Cr{cinter}}\geq \Cr{Cinter}$, $n\in \{0,1,2,\dots\}$,  $x\in{\Lambda_n^{\mathcal{L}_0}}$, and $p\in (0,1)$,
\begin{equation}
\label{badintereq}
\begin{split}
    &\widetilde{\Q}^{u,p}\big(G_{x,n}^{\mathcal{L}_0}\big((\mathbf{D}^{L_0,u})^{\mathsf{c}}\big)\big)\leq2^{-2^n},\phantom{ and } \widetilde{\Q}^{u,p}\big(G_{x,n}^{\mathcal{L}_0}\big((\mathbf{\hat{D}}^{L_0,u})^{\mathsf{c}}\big)\big)\leq2^{-2^n},
    \\&\widetilde{\Q}^{u,p}\big(G_{x,n}^{\mathcal{L}_0}\big((\mathbf{\overline{D}}^{L_0,u})^{\mathsf{c}}\big)\big)\leq2^{-2^n}\text{ and } \widetilde{\Q}^{u,p}\big(G_{x,n}^{\mathcal{L}_0}\big((\mathbf{E}^{L_0,u})^{\mathsf{c}}\big)\big)\leq2^{-2^n}.
    \end{split}
\end{equation}
\end{Lemme}
\begin{proof}
We start with the estimate involving the family $(\mathbf{D}^{L_0,u})^{\mathsf{c}}.$ By \eqref{defIu} and \eqref{ballcapacity} we have
\begin{equation*}
\widetilde{\Q}^{u,p}\left((D_x^{L_0,{u/2}})^{\mathsf{c}}\right)\leq\exp(-\Cr{ccapacity}(u/8)L_0^\nu).
\end{equation*}
By \eqref{eq:boldevents} and a union bound, this readily implies that both \eqref{decoupL0cond}, for $l_0$ as in \eqref{lbarandl0} and $\eps=1,$ and $\widetilde{\Q}^{u,p}\big((\mathbf{D}_x^{L_0,u/2})^{\mathsf{c}}\big)\leq \Cr{ccor}l_0^{-4\alpha}$ hold for all $u\in (0,u_0)$ and $L_0 \geq \Cr{Cdistance} \vee Cu^{-c}$ (and all $x\in{\Lambda_0^{\mathcal{L}_0}}$). For all $L_0\geq\Cr{Cdistance},$ $v>0$ and $x\in{G}$ the events $(\mathbf{D}_x^{L_0,u})^{\mathsf{c}}$ are measurable with respect to the $\sigma$-algebra generated by $\tilde{\ell}_{\tilde{B}(x,\overline{l}L_0),u}$ and decreasing in $u.$ Therefore, Proposition \ref{Cordecoup} with $\eps=1$ applies and \eqref{decoupatleveln} yields the first part of \eqref{badintereq}.

Let us now turn to the events $(\mathbf{\hat{D}}^{L_0,u})^{\mathsf{c}}.$ For all $L_0>0,$ $v\geq u/8$ and $x\in{G},$ we say that the event $\hat{D}_x^{L_0,v,u}$ occurs if and only if $\mathrm{cap}\big({C}^{u/4+v}(y,2(L_0+\Cr{Cdistance}))\big)\geq\Cr{ccapbig}(L_0+\Cr{Cdistance})^{3\nu/4}(u/8)^{\lfloor\gamma-1\rfloor}$ for all $y\in{\I^{u/4}\cap B(x,L_0+\Cr{Cdistance})},$ and we define $\mathbf{\hat{D}}_x^{L_0,v,u}$ similarly as in \eqref{eq:boldevents}, replacing $C_y^{L_0,u}$ by $\hat{D}_y^{L_0,v,u}.$ Consider a fixed value of $u_0>0.$ Note that the law of $\hat{\I}^{u/4+v}\setminus\hat{\I}^{u/4}$ conditionally on $\hat{\I}^{u/4}$ is the same as the law of $\hat{\I}^{v}.$ By \eqref{conditiondistance} the set ${C}^{u/4}(y,L_0+\Cr{Cdistance})$ has diameter at least $L_0$ for all $y\in{\I^{u/4}},$ and thus by \eqref{defIu} and \eqref{eqcapline}, we have for all $v\geq u/8$ and $y\in{\I^{u/4}\cap B(x,L_0+\Cr{Cdistance})}$ that
\begin{equation*}
\tilde{\Q}^{u,p}\big((\I^{u/4+v}\setminus\I^{u/4})\cap C^{u/4}(y,L_0+\Cr{Cdistance})=\emptyset\,|\,\hat{\I}^{u/4}\big)\leq \exp\big(-cuL_0^{\frac\nu{2}\wedge1}\big).
\end{equation*}
Moreover, if on the other hand $(\I^{u/4+v}\setminus\I^{u/4})\cap C^{u/4}(y,L_0+\Cr{Cdistance})\neq\emptyset$ for some $y\in{\I^{u/4}\cap B(x,L_0+\Cr{Cdistance})},$ then $C^{u/4+v}(y,2(L_0+\Cr{Cdistance}))$ contains the cluster of edges in $B(y',L_0+\Cr{Cdistance})$ traversed by at least one of the trajectories of $\hat{\I}^{u/4+v}\setminus\hat{\I}^{u/4}$ for some $y'\in(\I^{u/4+v}\setminus\I^{u/4})\cap B(x,L_0+\Cr{Cdistance}).$ By Proposition \ref{10} applied to $\hat{\I}^{u/4+v}\setminus\hat{\I}^{u/4},$ \eqref{Ahlfors} and a union bound, we thus have for all  $u<u_0$ and $v\in{[u/8,u_0]}$ that
\begin{align*}
\widetilde{\Q}^{u,p}\left((\hat{D}_x^{L_0,v,u})^{\mathsf{c}}\,\middle|\,\hat{\I}^{u/4}\right)\leq C(u_0)(L_0+\Cr{Cdistance})^{\alpha}\Big(&\exp\big(-c(u_0)u(L_0+\Cr{Cdistance})^{\Cr{Ccapbig}}\big)
\\&+\exp\big(-cuL_0^{\frac\nu{2}\wedge1}\big)\Big).
\end{align*}
Moreover, conditionally on $\hat{\I}^{u/4},$ the events $(\mathbf{\hat{D}}_x^{L_0,v,u})^{\mathsf{c}}$ are decreasing in $v,$ i.e., there exists a decreasing subset $B_x$ of $C(\tilde{B}(x,\overline{l}L_0),\R)$ (depending on $L_0$ and $\hat{\I}^{u/4}$) such that $(\mathbf{\hat{D}}_x^{L_0,v,u})^{\mathsf{c}}$ has the same law as $B^{I,v}_{x}$ for all $u>0$ and $v\geq u/8,$ see \eqref{defevents}. By a union bound, we have that $\widetilde{\Q}^{u,p}\big((\mathbf{\hat{D}}_x^{L_0,u/8,u})^{\mathsf{c}}\big)\leq \Cr{ccor}l_0^{-4\alpha}$ and the second part of \eqref{decoupL0cond} with $l_0$ as in \eqref{lbarandl0} and $\eps=1$ simultaneously hold for all $u\in{(0,u_0)},$ and $L_0\geq\Cr{Cdistance}\vee C(u_0)u^{-c(u_0)},$ and by another application of Proposition \ref{Cordecoup} with $\eps=1$ we obtain that for all $u\in{(0,u_0)},$
\begin{equation*}
\widetilde{\Q}^{u,p}\big(G_{x,n}^{\mathcal{L}_0}\big((\mathbf{\hat{D}}^{L_0,u/4,u})^{\mathsf{c}}\big)\,|\,\hat{\I}^{u/4}\big)\leq2^{-2^n}.
\end{equation*}
Since $\mathbf{\hat{D}}^{L_0,u/4,u}=\mathbf{\hat{D}}^{L_0,u},$ we obtain directly the second part of \eqref{badintereq} by integrating over $\hat{\I}^{u/4}.$ 

We now consider the events $(\mathbf{\overline{D}}^{L_0,u})^{\mathsf{c}}.$ For all $L_0>0,$ $u>0,$ $v>0$ and $x\in{G},$ we say that the event $\overline{D}_x^{L_0,v,u}$ occurs if and only if
\begin{equation*}
y\stackrel{\wedge}{\longleftrightarrow} y'\ \text{in} \ \hat{\I}^{u/2+v}\cap {B}_E(x,3\Cr{Cstrong}(L_0+\Cr{Cdistance})),
\end{equation*}
for all $y,y'\in{\I}^{u/2}\cap {B}(x,L_0+\Cr{Cdistance})$ such that $\mathrm{cap}\big({C}^{u/2}(y,2(L_0+\Cr{Cdistance}))\big)\geq\Cr{ccapbig}(L_0+\Cr{Cdistance})^{3\nu/4}(u/8)^{\lfloor\gamma-1\rfloor}$ and $\mathrm{cap}\big({C}^{u/2}(y',2(L_0+\Cr{Cdistance}))\big)\geq\Cr{ccapbig}(L_0+\Cr{Cdistance})^{3\nu/4}(u/8)^{\lfloor\gamma-1\rfloor}$,
 and we define $\mathbf{\overline{D}}_x^{L_0,v,u}$ similarly as in \eqref{eq:boldevents}, replacing $C_y^{L_0,u}$ by $\overline{D}_y^{L_0,v,u}.$ Note that $C^{u/2}(y,2(L_0+\Cr{Cdistance}))\subset B(x,3(L_0+\Cr{Cdistance}))$ for all $y\in{B(x,L_0+\Cr{Cdistance})}.$ By \eqref{Ahlfors}, Lemma \ref{12} and a union bound, we have for all $u\in{(0,u_0)},$ $v\in{[u/4,u/2]},$ $x\in{G}$ and $L_0 \geq\Cr{Cdistance},$
\begin{equation*}
    \widetilde{\Q}^{u,p}\left((\overline{D}_x^{L_0,v,u})^{\mathsf{c}}\,|\,\hat{\I}^{u/2}\right)\leq C(L_0+\Cr{Cdistance})^\alpha\exp\big(-cu^{2\lfloor\gamma-1\rfloor+1}(L_0+\Cr{Cdistance})^{\nu/2}\big).
\end{equation*}
Conditionally on $\hat{\I}^{u/2},$ the events $(\overline{\mathbf{D}}_x^{L_0,v,u})^{\mathsf{c}}$ are decreasing in $v,$ and similarly as before we can apply Proposition \ref{Cordecoup} with $\eps=1$ to obtain the third bound of \eqref{badintereq} for all $u\in{(0,u_0)}$ and $L_0\geq\Cr{Cdistance}\vee C(u_0)u^{-c(u_0)}$ since $\overline{\mathbf{D}}_x^{L_0,u/2,u}=\overline{\mathbf{D}}_x^{L_0,u}.$
 
Regarding $(\mathbf{E}^{L_0,u})^{\mathsf{c}}$, under $\widetilde{\Q}^{u,p},$ note that by \eqref{couplingbetweenGFFandRI}, the clusters of  $\{y\in{G};\,\phi_y>-\sqrt{2u}\}$ are the same as the clusters of $\{y\in{G};\,y\in{\mathcal{C}_u^\infty}\text{ or }\gamma_y>0\}.$ Therefore if the cluster $\mathcal{U}_x$ of $x$ in $\{y\in{G};\,\phi_y>-\sqrt{2u}\}\cap B(x,L_0/2)$ has diameter at least $L_0/4$ and is not connected to $\I^{u/4}$ in $\{y\in{G};\,\phi_y>-\sqrt{2u}\}\cap B(x,L_0),$ then either $\mathcal{U}_x$ is a cluster of $\{y\in{G};\,y\in{\mathcal{C}_u^\infty\setminus\mathcal{C}_{u/4}^\infty}\text{ or }\gamma_y>0\}\cap B(x,L_0/2)$ of diameter at least $L_0/4,$ or $\mathcal{U}_x$ contains a vertex $y$ in $\mathcal{C}_{u/4}^\infty\cap B(x,L_0/2)$ not connected to $\I^{u/4}$ in $\{y\in{G};\,\phi_y>-\sqrt{2u}\}\cap B(x,L_0),$ and then by \eqref{eq:Cu} and \eqref{couplingbetweenGFFandRI}, $y$ is in a connected component of $\{z\in{\tilde{G}};\,|\tilde{\gamma}_z|>0\}\cap \tilde{B}(x,L_0)$ of diameter $\geq L_0/4$ not intersecting $\I^{u/4}.$ Therefore, defining the event
\begin{equation*}
    E^{L_0,v,u}_x=\left\{\begin{array}{c}\text{ all the connected components of }\\\{y\in{G};\,y\in{\mathcal{C}_u^\infty\setminus\mathcal{C}_{u/4}^\infty}\text{ or }\gamma_y>0\}\cap B(x,L_0/2) \\\text{ or of }\{z\in{\tilde{G}};\,|\tilde{\gamma}_z|>0\}\cap \tilde{B}(x,L_0)\\\text{ with diameter }\geq L_0/4\text{ intersect }\I^{v}\end{array}\right\}
\end{equation*}
for all $v\leq u/4,$ we have $E^{L_0,v,u}_x\subset E^{L_0,u}_x$ by Definition \ref{defgood} \eqref{defEx}. We also define $\mathbf{E}_x^{L_0,v,u}$ similarly as in \eqref{eq:boldevents}, replacing $C_y^{L_0,u}$ by $E_y^{L_0,v,u}.$ Let $\tilde{\I}^{3u/4}_2=\tilde{\I}^u\setminus\tilde{\I}^{u/4},$ then $\mathcal{C}_u^\infty\setminus\mathcal{C}_{u/4}^\infty$ is $\tilde{\I}^{3u/4}_2$ measurable. Moreover
$\tilde{\gamma}$ is independent from the random interlacement set ${\I}^{u/4}$, see \eqref{eq:Qup},  $\tilde{\I}^{3u/4}_2$ is also independent from $\I^{u/4},$ and there are at most $2|B(x,L_0)|$ connected components of either $(\{y\in{\mathcal{C}_u^\infty\setminus\mathcal{C}_{u/4}^\infty\}\cup\{y\in{G};\,\gamma_y>0\}})\cap B(x,L_0/2)$ or $\{z\in{\tilde{G}};\,|\tilde{\gamma}_z|>0\}\cap \tilde{B}(x,L_0)$  with diameter at least $\frac{L_0}{4}.$ 
Thus, by \eqref{Ahlfors}, Lemma \ref{Lemmacapline}, and \eqref{defIu}, $\widetilde{\Q}^{u,p}$-a.s., for all $u>0$, $v\in{[u/8,u/4]}$ and $p\in (0,1)$,
\begin{equation}
\label{eq:initiateD}
\tilde{\Q}^{u,p}\left((E_x^{L_0,v,u})^{\mathsf{c}}\,\middle|\,\tilde{\gamma},\tilde{\I}^{3u/4}_2\right)\leq 2\Cr{CAhlfors}L_0^{\alpha}\exp\big\{-cuL_0^{\frac{\nu}{2}\wedge 1}\big\}.
\end{equation}
The fourth bound in \eqref{badintereq} is then obtained by virtue of another application of Proposition \nolinebreak\ref{Cordecoup} under the conditional measure $\tilde{\Q}^{u,p}(\cdot  \, | \, \tilde{\gamma},\tilde{\I}^{3u/4}_2)$, using \eqref{eq:initiateD} and a union bound to deduce that $\tilde{\Q}^{p}((\mathbf{E}_x^{L_0,u/8,u})^{\mathsf{c}}\,|\,\tilde{\gamma},\tilde{\I}^{3u/4}_2)\leq\Cr{ccor}l_0^{-4\alpha};$ the second part of \eqref{decoupL0cond} with $l_0$ as in \eqref{lbarandl0} and $\eps=1$ simultaneously holds true whenever $L_0u^{c}\geq C'$. Noting that, for all $v\leq u/4,$ conditionally on $\tilde{\gamma}$ and $\tilde{\I}^{3u/4}_2$, $(E_x^{L_0,v,u})^{\mathsf{c}}$ is a decreasing $\sigma(\tilde{\ell}_{B(x,\overline{l}L_0),v})$-measurable event in $v,$ Proposition \ref{Cordecoup} yields an upper bound similar to \eqref{badintereq} but for $G_{x,n}^{\mathcal{L}_0}\big((\mathbf{E}_x^{L_0,u/4,u})^{\mathsf{c}}\big)$ under $\tilde{\Q}^{u,p}(\cdot \, | \,\tilde{\gamma},\tilde{\I}^{3u/4}_2)$. The desired bound  \eqref{badintereq} then follows by integrating over \nolinebreak$\tilde{\gamma}$ and $\tilde{\I}^{3u/4}_2$ since $G_{x,n}^{\mathcal{L}_0}\big((\mathbf{E}_x^{L_0,u})^{\mathsf{c}}\big)\subset G_{x,n}^{\mathcal{L}_0}\big((\mathbf{E}_x^{L_0,u/4,u})^{\mathsf{c}}\big).$
\end{proof}

Finally for the events involving the family $(\mathbf{F}^{L_0,p})^{\mathsf{c}}$ in \eqref{defnbad}, by a similar reasoning as in Lemma 4.7 of \cite{MR3024098} and using \eqref{Ahlfors}, there exists a constant $\Cl{Bernoulli}$ such that for all $p\in{(0,1)}$ such that $p\geq\exp\{-\Cr{Bernoulli}L_0^{-\alpha}\},$ all $u>0$, $n \geq 0$ and $x\in{\Lambda_n^{\mathcal{L}_0}}$,\vphantom{$\Cl[c]{cBernoulli}$}
\begin{equation}
\label{badbern}
\widetilde{\Q}^{u,p}\big(G_{x,n}^{\mathcal{L}_0}\big((\mathbf{F}^{L_0,p})^{\mathsf{c}}\big)\big)\leq2^{-2^n}.
\end{equation}
For all $u_0>0$ and $R\geq1$ we define
\begin{equation}
\label{eq:choiceL_0}
L_0(u)=R\vee\Cr{Cdistance}\vee\Cr{Cgff}\vee\Cr{Cinter}u^{-\Cr{cinter}},
\end{equation}
where we keep the dependence of various constants and of $L_0(u)$ on $u_0$ and $R$ implicit.
Furthermore, we choose constants $\Cl{Ctruncated}$ and $\Cl[c]{ctruncated}$ such that $\sqrt{\log(\Cr{Ctruncated}u^{-\Cr{ctruncated}})}\geq\Cr{Cgff}'\sqrt{\log(l_0L_0(u))}$ for all $u\in{(0,u_0)},$ and constants $\Cl{CBern}$ and $\Cl[c]{cBern}$ such that $1-\Cr{CBern}u^{\Cr{cBern}}\geq\exp\big\{-\Cr{Bernoulli}(l_0L_0(u))^{-\alpha}\big\}$ for all $u\in{(0,u_0)},$ which can both be achieved on account of \eqref{eq:choiceL_0}. Then, by \eqref{defnbad}, Lemmas \ref{badGFF} and \ref{badinter} and \eqref{badbern}, for all $n\in{\N}$ and $u\in{(0,u_0)}$
\begin{equation}
\label{boundbadnvertex}
\left.\begin{array}{c}L_0\in{[L_0(u),l_0L_0(u)]},\\K\geq\sqrt{\log(\Cr{Ctruncated}u^{-\Cr{ctruncated}})}\\\text{and }p\geq1-\Cr{CBern}u^{\Cr{cBern}}\end{array}\right\}\quad\text{ imply }\quad\tilde{\Q}^{u,p}(x\text{ is }n-(L_0,u,K,p)\text{ bad})\leq 6\times2^{-2^n}.
\end{equation} 
Relying on \eqref{boundbadnvertex}, we now deduce a strong bound on the probability to see long $R$-paths of $(L_0,u,K,p)$-bad vertices (see above \eqref{weakSecIso} for a definition of $R$-paths). We emphasize that the following result holds for all graphs satisfying \eqref{eq:Ass}. In particular, \eqref{weakSecIso} is not required for \eqref{eq:longbad} below to hold.

\begin{Prop}
\label{Rpathofbad} For $G$ satisfying \eqref{eq:Ass} and each $u_0>0,$
there exist constants $c(u_0), C(u_0) \in (0,\infty)$ such that for all
 $R\geq1,$ $x\in{G},$  $u\in{(0,u_0)},$ $K>0$ with $K\geq\sqrt{\log(\Cr{Ctruncated}u^{-\Cr{ctruncated}})},$ $p\in{(0,1)}$ with $p\geq 1-\Cr{CBern}u^{\Cr{cBern}},$ and $N>0,$ 
\begin{equation}
\label{eq:longbad}
    \widetilde{\Q}^{u,p}\Big(\begin{array}{c}\text{there exists an $R$-path of $(L_0,u,K,p)$}\\\text{-bad vertices from $x$ to $B(x,N)^{\mathsf{c}}$}\end{array}\Big)\leq C(u_0)\exp\left\{-(N/{L_0(u)})^{c(u_0)}\right\}.
\end{equation}
\end{Prop}
\begin{proof}
We will show by induction that for all $n\in\{ 0,1,2,\dots\},$ $L_0\geq R\vee\Cr{Cdistance},$ and $x\in{\Lambda_n^{\mathcal{L}_0}}$,
\begin{equation}
    \label{lemmaatln}
    \Big\{\begin{array}{l}\text{there exists an $R$-path of $(L_0,u,K,p)$-bad}\\\text{vertices from $B(x,L_n)$ to $B(x,\overline{l}L_n)^{\mathsf{c}}$}\end{array}\Big\}\subset\{x\text{ is }n-(L_0,u,K,p)\text{ bad}\}.
\end{equation}
If \eqref{lemmaatln} holds, then Proposition \ref{Rpathofbad} directly follows from \eqref{eq:choiceL_0} and \eqref{boundbadnvertex} by taking $n\in\N$ and $L_0\in{[L_0(u),l_0L_0(u))}$ such that $\overline{l}l_0^nL_0=N.$ Let us fix some $L_0\geq R\vee\Cr{Cdistance}.$ For $n=0,$ if there exists a bad vertex in $B(x,L_0),$ then, see below \eqref{defnbad}, $x$ is $0-(L_0,u,K,p)$ bad. Suppose now that \eqref{lemmaatln} holds at level $n-1$ for all $x\in{\Lambda_{n-1}^{\mathcal{L}_0}}$ for some $n\geq1.$ Then, since $L_0\geq R\vee\Cr{Cdistance}$ and $\overline{l}\geq22,$ if there exists an $R$-path $\pi$ of $(L_0,u,K,p)$-bad vertices from $B(x,L_n)$ to $B(x,\overline{l}L_n)^{\mathsf{c}},$ one can find for each $k\in{\{1,\dots,7\}}$ a vertex 
\begin{equation*}
    y_k\in{\pi\cap\big( B(x,3kL_n)\setminus B(x,(3k-1)L_n)\big)}.
\end{equation*}
Using \eqref{Lambda1}, one then picks for each $k\in{\{1,\dots,7\}}$ a vertex $z_k\in{\Lambda_{n-1}^{\mathcal{L}_0}}$ such that $y_k\in{B(z_k,L_{n-1})}.$ One then easily checks that with the choice of $\overline{l}$ and $l_0$ in \eqref{lbarandl0}, for all $k\neq k'$ in ${\{1,\dots,7\}},$ $d(z_k,z_{k'})\geq L_n,$ and $B(z_k,\overline{l}L_{n-1})\subset B(x,\overline{l}L_n)\setminus B(x,L_n).$ In particular, for each $k\in{\{1,\dots, 7\}},$ $\pi$ yields an $R$-path of $(L_0,u,K,p)$-bad vertices from $B(z_k,L_{n-1})$ to $B(z_k,\overline{l}L_{n-1})^{\mathsf{c}},$  and the induction hypothesis implies that $z_k$ is $(n-1)-(L_0,u,K,p)$ bad. Among these seven $(n-1)-(L_0,u,K,p)$ bad vertices, there exist $i\neq j\in{\{1,\dots,7\}}$ and $A\in{\{(\mathbf{C}^{L_0,K})^{\mathsf{c}},(\mathbf{D}^{L_0,K})^{\mathsf{c}},(\hat{\mathbf{D}}^{L_0,K})^{\mathsf{c}},(\overline{\mathbf{D}}^{L_0,K})^{\mathsf{c}},(\mathbf{E}^{L_0,u})^{\mathsf{c}},(\mathbf{F}^{L_0,p})^{\mathsf{c}}\}}$ such that $G_{z_i,n-1}^{\mathcal{L}_0}(A)$ and $G_{z_j,n-1}^{\mathcal{L}_0}(A)$ both occur, whence $z_i$ and $z_j$ appear in the union for $G_{x,n}^{\mathcal{L}_0}(A),$ see \eqref{Gxndef}. By definition \eqref{defnbad}, $x$ is $n-(L_0,u,K,p)$ bad and \eqref{lemmaatln} follows.
\end{proof}

Using the additional condition \eqref{weakSecIso}, Proposition \ref{Rpathofbad} together with Lemma \ref{whygood} can be used to show the existence of a certain set $\tilde{A}$, see Lemma \ref{percolfora} below, from which 
the prevalence of the infinite cluster of $E^{\ge h}$, $h >0$ small,
will eventually be deduced. The bound obtained in \eqref{Atildeexpodecay} will later lead to
\nolinebreak\eqref{eqhbar1}.

\begin{Lemme}
\label{percolfora}
Assume $G$ satisfies \eqref{weakSecIso} (in addition to \eqref{eq:Ass}), and let $R=R_0$ as in \eqref{weakSecIso}. Furthermore, let $u_0>0,$ $u\in{(0,u_0)},$ $K>0$ with $K\geq\sqrt{\log(\Cr{Ctruncated}u^{-\Cr{ctruncated}})},$  and $p\in{(0,1)}$ with $p\geq 1-\Cr{CBern}u^{\Cr{cBern}}.$ Then $\widetilde{\Q}^{u,p}$-a.s.\ there exists $L_0\geq1$ and a connected and unbounded set $\tilde{A}_\infty^u\subset\tilde{\I}^u$ such that \eqref{Atilde2} holds and there exist constants $c>0$ and $C<\infty$ depending on $u$ and $u_0$ such that for all $x_0 \in G$ and $L>0,$
\begin{equation}
\label{Atildeexpodecay}
\widetilde{\Q}^{u,p} \big(\tilde{A}_\infty^u\cap \tilde{B}(x_0,L)=\emptyset\big)\leq C\exp\{-L^c\}.
\end{equation}
\end{Lemme}
\begin{proof} Fix a vertex $x_0 \in G$. By \eqref{weakSecIso}, there exists $R_0\geq1$ such that,
for all finite connected subsets $A$ of $G$ with $x_0 \in A$ and $\delta(A) \geq  \Cr{Cdistance}$, noting that $d(x,x_0) \leq \delta(A) + \Cr{Cdistance} \leq 2\delta(A)$ for all $x\in{\partial_{ext}A}$ by \eqref{conditiondistance},
\begin{equation}
\label{Peierls2}
\text{for all $x\in{\partial_{ext}{A}}$, } \exists \text{ an $R_0$-path from $x$ to $ B(x,\Cr{cwsi2}d(x,x_0)/2)^{\mathsf{c}}$ in $\partial_{ext}A$.} 
\end{equation}
It is then enough to prove that for  $L_0$ as in \eqref{eq:choiceL_0}, for all $u\in{(0,u_0)},$ $K\geq\sqrt{\log(\Cr{Ctruncated}u^{-\Cr{ctruncated}})}$ and $p\geq 1-\Cr{CBern}u^{\Cr{cBern}},$ the probability under $\widetilde{\Q}^{u,p}$ of the event 
\begin{equation}
\label{resteaprouver}
   \Big\{\begin{array}{c}\text{there does not exist an unbounded nearest neighbor path in }G
    \\\text{of }(L_0,u,K,p)\text{-good vertices starting in }B(x_0,L)\end{array}\Big\}
\end{equation}
has stretched-exponential decay in $L$ for some $L_0\geq1$ (with constants depending on $u$ and $u_0$). Indeed by the Borel-Cantelli lemma one easily deduces that there is a.s.\ an unbounded nearest neighbor path $\pi$ in $G$ of $(L_0,u,K,p)$-good vertices, and by Lemma \nolinebreak\ref{whygood} there exists an unbounded connected component $\tilde{A}_\infty^u\subset\tilde{\I}^u$ such that \eqref{Atilde2} holds and $\tilde{A}_{\infty}\cap B(x,L_0)\neq\emptyset$ for all $x$ in $\pi$. Moreover if \eqref{resteaprouver} does not occur, then $\tilde{A}_\infty^u$ intersect $\tilde{B}(x_0,L+L_0)$, and the bound \eqref{Atildeexpodecay} follows after a change of variable for $L$. 

Thus, in order to establish the desired decay, assume that \eqref{resteaprouver} occurs for some $u\in{(0,u_0)},$ $K\geq\sqrt{\log(\Cr{Ctruncated}u^{-\Cr{ctruncated}})},$ $p\geq1-\Cr{CBern}u^{ \Cr{cBern}},$ a positive integer $L$ and $L_0$ as in \eqref{eq:choiceL_0}. We may assume that $L \geq  \Cr{Cdistance}$. We now use Proposition \nolinebreak\ref{Rpathofbad} and a contour argument involving \eqref{Peierls2} to bound its probability. Note that the assumptions of Proposition \ref{Rpathofbad} on the set of parameters $(L_0, u, K,p)$ are met for all $u\in (0,u_0)$ by our choice of constants. Define
\begin{equation*}
    A_L=B(x_0,L)\cup \big\{x\in{G};\,x\leftrightarrow B(x_0,L)\text{ in  the set of }(L_0,u,K,p)\text{-good vertices} \big\},
\end{equation*}
which is the set of vertices in $G$ either in, or connected to $B(x_0,L)$ by a nearest neighbor path of $(L_0,u,K,p)$-good vertices in $G$. Since \eqref{resteaprouver} occurs, $A_L$ is finite. It is also connected, and $\delta(A_L) \geq \Cr{Cdistance}$. Hence, since every vertex in $\partial_{ext}{A_L}$ is $(L_0,u,K,p)$-bad, by \eqref{Peierls2} there exists $x\in{\partial_{ext}{A_L}}$ and an $R_0$-path of $(L_0,u,K,p)$-bad vertices from $x$ to $B(x,\Cr{cwsi2}d(x,x_0)/2)^{\mathsf{c}}.$ Let $N=\lfloor d(x,x_0)\rfloor,$ then $N\geq L,$ and thus by a union bound the probability that the event \eqref{resteaprouver} occurs is smaller than
\begin{equation*}
\sum_{N=L}^{\infty}\sum_{x\in{B(x_0,N+1)}}\widetilde{\Q}^{u,p}\left(\begin{array}{c}\text{there exists an $R_0$-path of }(L_0,u,K,p)\\\text{-bad }\text{vertices from $x$ to }B(x,cN)^{\mathsf{c}}\end{array}\right),
\end{equation*}
which has stretched-exponential decay in $L$ by \eqref{Ahlfors}, \eqref{lambda} and  Proposition \ref{Rpathofbad}.
\end{proof}

\begin{Rk}\label{R:wsi'1} One can replace \eqref{weakSecIso} by the following (weaker) condition \eqref{weakSecIso'} and still retain a statement similar to Lemma \ref{percolfora}. This is of interest in order to determine how little space (in $G$) one can afford to use in order for various sets, in particular $\mathcal{V}^u$ at small $u>0$ in Theorem \ref{T:mainresultRI}, to retain an unbounded component; see Theorem \ref{thmconclusion} and Remark \ref{lastremark}, \ref{percoplanes}) below. We first introduce \eqref{weakSecIso'}. Suppose that
there exists an infinite connected subgraph $G_p$ of $G,$ $\zeta>0,$ $R_0\geq1,$ a vertex $x_0\in{G_p}$ and $\Cl[c]{cwsi1}>0$ such that\vphantom{$\Cl{Cwsi1}$}
\begin{gather}\tag{$\widetilde{\text{WSI}}$}\label{weakSecIso'}
\begin{split}
&\text{for all finite connected $A \subset G_p$ with $x_0\in{A},$ there exists $x\in{(\partial_{ext}{A})\cap G_p}$}\nonumber
\\
&\text{and an $R_0$-path from $x$ to $ B(x, \Cr{cwsi1}d(x,x_0)^{\zeta})^{\mathsf{c}}$ in $(\partial_{ext}A)\cap G_p,$}
\end{split}
\end{gather}
i.e., all the vertices of this path are in $(\partial_{ext}A)\cap G_p.$ It is easy to see that \eqref{weakSecIso} implies \eqref{weakSecIso'} with $\zeta=1.$ 
Suppose now that instead of \eqref{weakSecIso}, condition \eqref{weakSecIso'} hold for some subgraph $G_p$ of $G.$ Then the conclusions of Lemma \ref{percolfora} leading to \eqref{Atilde} still hold and the set $\tilde{A}_\infty^u$ thereby constructed satisfies $\tilde{A}_\infty^u\subset\tilde{B}(G_p,3\Cr{Cstrong}(L_0(u)+\Cr{Cdistance}))$. To see this, one replaces \eqref{Peierls2} by the following consequence of \eqref{weakSecIso'}: there exists $R_0\geq1,$ $x_0\in{G_p}$ and $c>0$ such that for all finite connected subsets $A$ of $G_p$ with $x_0\in{A},$
\begin{equation}
\label{P2'}
\tag{\ref{Peierls2}'}
\exists\,x\in{(\partial_{ext}{A})\cap G_p\text{ and a $R_0$-path from $x$ to }B(x,cd(x,x_0)^{\zeta})^{\mathsf{c}}\text{ in }(\partial_{ext}A)\cap G_p.}
\end{equation}
One then argues as above, with small modifications due to \eqref{P2'}, whence, in particular, the set $A_L$  needs to be replaced by $  A_L(G_p)\stackrel{\text{def.}}{=}\big(B(x_0,L)\cap G_p\big)\cup\{x\in{G_p};\,x\leftrightarrow B(x_0,L)\cap G_p\text{ in  the set of }(L_0,u,K,p)\text{-good vertices in }G_p\}$, so that $A_L=A_L(G)$.  
\end{Rk}
\bigskip
The bound \eqref{Atildeexpodecay} will be useful to prove that \eqref{eqhbar1} holds, and we seek a similar  result which roughly translates \eqref{eqhbar2} to the world of random interlacements. This appears in Lemma \ref{proofofhbartilde>0} below. 
Its proof rests on the following technical result, which is a feature of the renormalization scheme. 
\begin{Lemme}
	\label{largecompareconnbygood}
	Assume G satisfies \eqref{weakSecIso}, and recall the definition of $\Cr{cdiam}$ from \eqref{defcdiam}. For any $L_0\geq\Cr{Cdistance},$ $K>0$, $u>0$ and $n \in \{0,1,2,\dots \}$, if there exists a vertex $x\in{\Lambda_{n}^{\mathcal{L}_0}}$ which is $n-(L_0,u,K,p)$ good, then every two connected components of $B(x,20\Cr{cdiam}L_{n})$ with diameter at least $\Cr{cdiam}L_n$ are connected via a path of $(L_0,u,K,p)$-good vertices in $B(x,30\Cr{cdiam}\Cr{Cstrong}L_{n}).$
\end{Lemme}
\begin{proof}
	We use induction on $n.$ For $n=0,$ if $x$ is $0-(L_0,u,K,p)$ good, then in view of \eqref{eq:boldevents}, \eqref{defnbad} and Definition \ref{defgood}, every path in $B(x,20\Cr{cdiam}L_0)$ is a path of $(L_0,u,K,p)$-good vertices and all the vertices in $B(x,20\Cr{cdiam}\Cr{Cstrong}L_0)$ are $(L_0,u,K,p)$-good, so the result follows directly from \eqref{ballsalmostconnected}. Let us now assume that the conclusion of the lemma holds at level $n-1$ for some $n\geq1$ and let 
	\begin{equation}
	\label{eq:xisgood}
	\text{$x$ be an $n-(L_0,u,K,p)$ good vertex.}
	\end{equation}
	 Let $\mathcal{U}_1$ and $\mathcal{U}_2$ be any two connected components of $B(x,20\Cr{cdiam}L_{n})$ with diameter at least $\Cr{cdiam}L_n.$ We are first going to show that 
	\begin{equation}
	\label{eq:U_ilink}
	\begin{gathered}
	\text{$\mathcal{U}_1$ and $\mathcal{U}_2$ are linked via $(n-1)-(L_0,u,K,p)$}\\\text{-good vertices in $B(x,22\Cr{cdiam}\Cr{Cstrong}L_{n})$},
	\end{gathered}
	\end{equation}
	by which we mean that there exists a subset $S$ of $\Lambda_{n-1}^{\mathcal{L}_0} \cap B(x,22\Cr{cdiam}\Cr{Cstrong}L_{n})$ containing only $(n-1)-(L_0,u,K,p)$ good vertices and such that $\bigcup_{y\in S} B(y,L_{n-1})$ contains a connected component intersecting both $\mathcal{U}_1$ and $\mathcal{U}_2$.
To see that \eqref{eq:U_ilink} holds, for each $i\in{\{1,2\}}$ choose seven connected subsets $(\mathcal{U}_i^k)_{k\in{\{1,\dots,7\}}}$ of $\mathcal{U}_i$ such that for all $k\neq k'\in{\{1,\dots,7\}},$
	\begin{equation*}
	d(\mathcal{U}_i^k,\mathcal{U}_i^{k'})\geq L_n+2L_{n-1}\quad \text{ and }\quad \delta(\mathcal{U}_i^k)\geq 7L_n\Cr{cwsi2}^{-1};
	\end{equation*}
	such a choice is possible since $L_0\geq\Cr{Cdistance},$ $l_0\geq\overline{l}\geq22$ and $\Cr{cdiam}=7(1+7\Cr{cwsi2}^{-1}).$ If for each $k\in{\{1,\dots,7\}}$ there exists an $(n-1)-(L_0,u,K,p)$ bad vertex $y_k\in{\Lambda_{n-1}^{\mathcal{L}_0}}$ such that $B(y_k,L_{n-1})\cap\mathcal{U}_i^k\neq\emptyset,$ then there are at least seven $(n-1)-(L_0,u,K,p)$ bad vertices  in $B(x,20\Cr{cdiam}L_{n}+L_{n-1})\subset B(x,\overline{l}L_n)$ with mutual distance at least $L_n,$ which contradicts \eqref{eq:xisgood} by \eqref{defnbad} and the definition of the renormalization scheme, see \eqref{Gxndef}. For each $i\in{\{1,2\}}$ we can thus find $k_i$ such that each $y\in{\Lambda_{n-1}^{\mathcal{L}_0}}$ with $B(y,L_{n-1})\cap\mathcal{U}_i^{k_i}\neq \emptyset$ is $(n-1)-(L_0,u,K,p)$ good. 
 Recalling that $\mathcal{U}_i^{k_i}$ is connected,
 	we can define for each $i\in{\{1,2\}}$ the set $\mathrm{comp}_{n-1}(\mathcal{U}_i^{k_i}) \subset G$ as the connected component in 
 	\begin{equation}
	\label{eq:defcomp}
	\bigcup_{\substack{y\in{\Lambda_{n-1}^{\mathcal{L}_0}\cap B(x,22\Cr{cdiam}\Cr{Cstrong}L_{n}}),\\y\text{ is } (n-1)-(L_0,u,K,p)\text{ good}}}B(y,L_{n-1})
	\end{equation}
	 containing $\mathcal{U}_i^{k_i}.$
	
The claim \eqref{eq:U_ilink} amounts to showing that $\mathrm{comp}_{n-1}(\mathcal{U}_1^{k_1})=\mathrm{comp}_{n-1}(\mathcal{U}_2^{k_2}).$ Suppose on the contrary that $\mathrm{comp}_{n-1}(\mathcal{U}_1^{k_1})$ and $\mathrm{comp}_{n-1}(\mathcal{U}_2^{k_2})$ are not equal. By \eqref{ballsalmostconnected}, there is a nearest neighbor path $(x_1,\dots,x_p)$ in $B(x,20\Cr{cdiam}\Cr{Cstrong}L_{n})$ connecting $\mathcal{U}_1^{k_1}$ and $\mathcal{U}_2^{k_2}.$ Recalling the notion of external boundary from \eqref{eq:bext}, since $x_1\in{\mathcal{U}_1^{k_1}},$ either there exists $m\in{\{1,\dots,p\}}$ such that $x_m\in{\partial_{ext}\mathrm{comp}_{n-1}(\mathcal{U}_1^{k_1})},$ or every unbounded nearest neighbor path beginning in $x_p$ intersects $\mathrm{comp}_{n-1}(\mathcal{U}_1^{k_1}),$ and likewise for $\mathrm{comp}_{n-1}(\mathcal{U}_2^{k_2}).$ If every unbounded path beginning in $x_p$ hits $\mathrm{comp}_{n-1}(\mathcal{U}_1^{k_1})$ and every unbounded path beginning in $x_1$ hits $\mathrm{comp}_{n-1}(\mathcal{U}_2^{k_2}),$ then by connectivity every unbounded path beginning in $\mathrm{comp}_{n-1}(\mathcal{U}_1^{k_1})$ hits $\mathrm{comp}_{n-1}(\mathcal{U}_2^{k_2})$ and every unbounded path beginning in $\mathrm{comp}_{n-1}(\mathcal{U}_2^{k_2})$ hits $\mathrm{comp}_{n-1}(\mathcal{U}_1^{k_1}),$ which is impossible since $\mathrm{comp}_{n-1}(\mathcal{U}_1^{k_1})\neq\mathrm{comp}_{n-1}(\mathcal{U}_2^{k_2})$ (indeed, unless $\mathrm{comp}_{n-1}(\mathcal{U}_1^{k_1})=\mathrm{comp}_{n-1}(\mathcal{U}_2^{k_2})$, these conditions would require any such path to `oscillate' between $\mathrm{comp}_{n-1}(\mathcal{U}_1^{k_1})$ and $\mathrm{comp}_{n-1}(\mathcal{U}_2^{k_2})$ infinitely often and thus it remains bounded). Therefore, we may assume that $\partial_{ext}\mathrm{comp}_{n-1}(\mathcal{U}_1^{k_1})\cap B(x,20\Cr{cdiam}\Cr{Cstrong}L_{n})\neq\emptyset$ (otherwise exchange the roles of $\mathcal{U}_1$ and $\mathcal{U}_2$), and by \eqref{weakSecIso}, there exists an $R_0$-path in $\partial_{ext}\mathrm{comp}_{n-1}(\mathcal{U}_1^{k_1})$ of diameter between $7L_n$ and $8L_n$ beginning in $B(x,20\Cr{cdiam}\Cr{Cstrong}L_{n}).$ By definition of $\mathrm{comp}_{n-1}(\mathcal{U}_1^{k_1})$, see \eqref{eq:defcomp}, every vertex of this $R_0$-path is contained in $B(y,L_{n-1})$ for some $(n-1)-(L_0,u,K,p)$ bad vertex $y$ in $\Lambda_{n-1}^{\mathcal{L}_0}\cap B(x,(20\Cr{cdiam}\Cr{Cstrong}+8+l_0^{-1})L_{n})\subset B(x,22\Cr{cdiam}L_n),$ and, since $L_0\geq\Cr{Cdistance}$ and $l_0\geq\overline{l}\geq22,$ there are at least $7$ $(n-1)-(L_0,u,K,p)$ bad vertices in $B(x,22\Cr{cdiam}\Cr{Cstrong}L_{n})=B(x,\overline{l}L_n)$ with mutual distance at least $L_n.$  By \eqref{Gxndef} and \eqref{defnbad}, $x$ is $n-(L_0,u,K,p)$ bad, which is a contradiction. 
	
Therefore, we have $\mathrm{comp}_{n-1}(\mathcal{U}_1^{k_1})=\mathrm{comp}_{n-1}(\mathcal{U}_2^{k_2}),$ i.e., \eqref{eq:U_ilink} holds. Thus, by \eqref{Lambda1} there exists $y_0\in{\mathcal{U}_1^{k_1}},$ $y_{m+1}\in{\mathcal{U}_2^{k_2}}$ and a sequence of vertices $y_1,\dots,y_m\in{\Lambda_{n-1}^{\mathcal{L}_0}\cap B(x,22\Cr{cdiam}\Cr{Cstrong}L_{n})}$ of good $(n-1)-(L_0,u,K,p)$ vertices such that 
\begin{equation}\label{eq:ind1}
	5\Cr{cdiam}L_{n-1}\leq d(y_{j-1},y_{j})\leq6\Cr{cdiam}L_{n-1}\forall\,j\in{\{1,\dots,m\}}\text{ and }d(y_{m},y_{m+1})\leq6\Cr{cdiam}L_{n-1}.
\end{equation}
We now construct the desired nearest neighbor path of $(L_0,u,K,p)$-good vertices connecting $\mathcal{U}_1$ and $\mathcal{U}_2$. To this end, we fix a nearest neighbor path $\pi_0$ in $\mathcal{U}_1^{k_1}$ beginning in $y_0,$ a nearest neighbor path $\pi_{m+1}$ in $\mathcal{U}_2^{k_2}$ beginning in $y_{m+1},$ and, for each $j\in{\{1,\dots,m\}}$ a nearest neighbor path $\pi_j$ beginning in $y_j$ such that for all $j\in{\{0,\dots,m+1\}},$ 
$\Cr{cdiam}L_{n-1}\leq\delta(\pi_j)\leq2\Cr{cdiam}L_{n-1}$, which is always possible since $7l_0\Cr{cwsi2}^{-1}\geq\Cr{cdiam},$ see \eqref{lbarandl0}. Note that, using \nolinebreak \eqref{eq:ind1},
	\begin{equation}
	\label{eq:ind2}
	\pi_0, \pi_1\subset {B(y_1,20\Cr{cdiam}L_{n-1})}\text{ and }d(\pi_0,\pi_1)\geq\Cr{cdiam}L_{n-1}.
	\end{equation}
	Due to \eqref{eq:ind2}, applying the induction hypothesis to $\pi_0$ and $\pi_1,$ we can construct a nearest neighbor path $\overbar{\pi}_1$ of  $(L_0,u,K,p)$-good vertices in $B(y_1,30\Cr{cdiam}\Cr{Cstrong}L_{n-1})\subset B(x,30\Cr{cdiam}\Cr{Cstrong}L_n)$ with diameter at least $\Cr{cdiam}L_{n-1}$ connecting $\pi_0$ and $\pi_1.$ Moreover, we can further extract from $\overbar{\pi}_1$ a nearest neighbor path $\overbar{\pi}'_1$ included in $B(y_1,2\Cr{cdiam}L_{n-1})$ and with diameter at least $\Cr{cdiam}L_{n-1},$ and so we have $\overbar{\pi}'_1\subset B(y_2,20\Cr{cdiam}L_{n-1})$ and $d(\overbar{\pi}'_1,\pi_{2})\geq\Cr{cdiam}L_{n-1}$. By the induction hypothesis, we can thus find a nearest neighbor path $\overbar{\pi}_2$ of $(L_0,u,K,p)$-good vertices in $B(y_2,30\Cr{cdiam}\Cr{Cstrong}L_{n-1})\subset B(x,30\Cr{cdiam}\Cr{Cstrong}L_n)$ with diameter at least $\Cr{cdiam}L_{n-1}$ between $\overbar{\pi}_1$ and $\pi_2.$ Iterating this construction, we find a sequence of $(\overbar{\pi}_{j})_{j\in{\{1,\dots,m+1\}}}$ of nearest neighbors paths of $(L_0,u,K,p)$-good vertices in $B(x,30\Cr{cdiam}\Cr{Cstrong}L_n)$ such that $\pi_0\cap\overbar{\pi}_1\neq\emptyset,$ $\overbar{\pi}_j\cap\overbar{\pi}_{j+1}\neq\emptyset$ for all $j\in{\{1,\dots,m\}}$ and $\overbar{\pi}_{m+1}\cap\pi_{m+1}\neq\emptyset.$ Concatenating the paths $\overbar{\pi}_0,\dots,\overbar{\pi}_{m+1}$ provides a path of $(L_0,u,K,p)$-good vertices in $B(x,30\Cr{cdiam}\Cr{Cstrong}L_{n})$ connecting $\mathcal{U}_1$ and $\mathcal{U}_2$, as desired. 
\end{proof}

Using Lemma \ref{whygood} and the quantitative bounds derived earlier in this section, we infer from Lemma \ref{largecompareconnbygood} the following estimate tailored to our later purposes. Let us define
\begin{equation*}
    E_\phi^{\geq-\sqrt{2u}}=\{y\in{G};\,\phi_y\geq-\sqrt{2u}\}.
\end{equation*}

\begin{Lemme}
	\label{proofofhbartilde>0}
	Assume $G$ satisfies \eqref{weakSecIso} (in addition to \eqref{eq:Ass}), and take $R=R_0$ from \eqref{weakSecIso}. Then for all $u_0>0,$ $u\in{(0,u_0)}$,  $x \in G$, $K>0$ with $K\geq\sqrt{\log(\Cr{Ctruncated}u^{-\Cr{ctruncated}})},$ $p\in{(0,1)}$ with $p\geq 1-\Cr{CBern}u^{\Cr{cBern}}$ and $L>0,$ there exists $L_0=L_0(L)\in{\big[L_0(u),l_0L_0(u)\big)},$ $C<\infty$ and $c>0$ depending on $u$ and $u_0$ such that
	\begin{equation}
	\label{eqhbaru>0}\nonumber
	\widetilde{\Q}^{u,p}\left(\mathcal{E}_{x,L}^u\right)\geq 1-C(u,u_0)\exp\{-L(u,u_0)^c\},
	\end{equation}
	where $\mathcal{E}_{x,L}^u$ is the event
	\begin{equation}
	\label{EuxL}
	    \left\{\begin{array}{c}\exists\,\text{a connected set $A_{x,L}^u\subset B(x,2\Cr{Cstrong}L)$ which intersects every cluster}
	    \\\text{of $B(x,L)$ with diameter $\geq\sqrt{L},$ and a connected set }\tilde{A}_{x,L}^u\subset\tilde{\I}^u\cap
	\\\tilde{B}(x,2\Cr{Cstrong}L)\text{ verifying \eqref{Atilde2}, such that }B(y,L_0)\cap \tilde{A}_{x,L}^u\neq\emptyset\\\text{ for all }y\in{A_{x,L}^u}\text{ and every cluster of $E_\phi^{\geq-\sqrt{2u}}\cap B(x,L)$ with}
	\\\text{diameter $\geq{L/10}$ is connected to $\tilde{A}_{x,L}^u\cap G$ in $E_\phi^{\geq-\sqrt{2u}}\cap B(x,2L)$}
	\end{array}\right\}.
	\end{equation}

\end{Lemme}
\begin{proof}
As a direct consequence of Lemma \ref{largecompareconnbygood} and \eqref{boundbadnvertex} with $R=R_0$ from \eqref{weakSecIso}, we obtain that for all $u_0>0,$ $u\in{(0,u_0]},$ $K\geq\sqrt{\log(\Cr{Ctruncated}u^{-\Cr{ctruncated}})},$  $p\geq 1-\Cr{CBern}u^{\Cr{cBern}}$ $n\in{\N},$ $x\in{\Lambda_n^{\mathcal{L}_0}},$ and $L_0\in{[L_0(u),l_0L_0(u)]},$ see \eqref{eq:choiceL_0},
\begin{equation}
\label{hbarunderpu}
\widetilde{\Q}^{u,p}\left(\begin{array}{c}\text{there exist connected components of }B(x,20\Cr{cdiam}L_{n})
\\\text{ with diameter }\geq\Cr{cdiam}L_n\text{ which are not connected by}
\\\text{a path of $(L_0,u,K,p)$-good vertices in }B(x,30\Cr{cdiam}\Cr{Cstrong}L_{n})\end{array}\right)\leq6\times2^{-2^{n}}.
\end{equation}
Therefore, for all $L$ large enough, taking ${L}_0={L}_0(L)\in{\big[L_0(u),l_0L_0(u)\big)}$ and $n\in{\N}$ such that $L=20\Cr{cdiam}l_0^n{L}_0,$ we have
\begin{equation}
\label{hbarunderpu2}
\begin{split}
\widetilde{\Q}^{u,p}\left(\begin{array}{c}\text{there exist connected components of }B(x,L)
\\\text{with diameter }\geq\frac{L}{10}\text{ which are not connected by a}
\\\text{path of $(L_0,u,K,p)$-good vertices in }B(x,2\Cr{Cstrong}L)\end{array}\right)\leq C\exp\{-L^{c}\},
\end{split}
\end{equation}
for some constants $C=C(u,u_0)$ and $c=c(u,u_0).$ Let us call $\overline{\mathcal{E}}_{x,L}^u$ the complement of the event on the left-hand side of \eqref{hbarunderpu2}. On the event $\overline{\mathcal{E}}_{x,L}^u,$ there exists a connected set $\mathcal{A}_{x,L}^u\subset B(x,2\Cr{Cstrong}L)$ of $(L_0,u,K,p)$-good vertices which intersects every connected component of $B(x,L)$ with diameter $\geq\frac{L}{10}.$ One can construct such a set by starting with a path $\pi$ of $(L_0,u,K,p)$-good vertices in $B(x,L)$ with diameter $\geq\frac{L}{10},$ and taking $\mathcal{A}_{x,L}^u$ as the union of all the paths of $(L_0,u,K,p)$-good vertices between $\pi$ and every other connected component of $B(x,L)$ with diameter $\geq\frac{L}{10}.$

By Lemma \ref{whygood}, for $L$ large enough, this implies the existence of a connected set $\tilde{\mathcal{A}}_{x,L}^u\subset\tilde{B}(x,2\Cr{Cstrong}L+3\Cr{Cstrong}(L_0+\Cr{Cdistance}))\subset \tilde{B}(x,3\Cr{Cstrong}L)$ such that \eqref{Atilde}, \eqref{Atildephi} and \eqref{Atilde2} hold when replacing $A$ by $\mathcal{A}_{x,L}^u$ and $\tilde{A}$ by $\tilde{\mathcal{A}}_{x,L}^u.$ Moreover, if $\mathcal{C}$ is a cluster of $E_\phi^{\geq-\sqrt{2u}}\cap B(x,L)$ with diameter at least $L/10,$ then there exists $z\in{\mathcal{C}\cap{\mathcal{A}}_{x,L}^u},$ and thus $\mathcal{C}$ contains a cluster of $E_\phi^{\geq-\sqrt{2u}}\cap B(z,L_0/2)$ with diameter at least $L_0/4.$ By \eqref{Atildephi} we obtain that $\mathcal{C}$ is connected to $\tilde{A}_{x,L}^u$ in $E_\phi^{\geq-\sqrt{2u}}\cap B(z,L_0)\subset B(x,3\Cr{Cstrong}L).$

If $\overline{\mathcal{E}}_{y,L}^u$ and $\overline{\mathcal{E}}_{y',L}^u$ happen for $y$ and $y'$ in $G$ with $y\sim y',$ then $\delta(\mathcal{A}_{y,L}^u\cap B(y',L))\geq\frac{L}{10},$ and so there exists $z\in{\mathcal{A}_{y,L}^u\cap \mathcal{A}_{y',L}^u}.$ By \eqref{Atilde}, $\emptyset\neq B(z,L_0)\cap\I^{u/4}\subset \tilde{\mathcal{A}}_{y,L}^u\cap\tilde{\mathcal{A}}_{y',L}^u.$  If $\overline{\mathcal{E}}_{y,10\sqrt{L}}^u$ happens for all $y\in{B(x,\Cr{Cstrong}L)},$ let us define $B_L\subset B(x,\Cr{Cstrong}L)$ a connected set containing $B(x,L),$ which exists by \eqref{ballsalmostconnected}, and 
\begin{equation*}
    {A}_{x,L}^u=\bigcup_{y\in{B_L}}{\mathcal{A}}_{y,10\sqrt{L}}^u\qquad\text{and}\qquad\tilde{A}_{x,L}^u=\bigcup_{y\in{B_L}}\tilde{\mathcal{A}}_{y,10\sqrt{L}}^u.
\end{equation*}
Then $A_{x,L}^u$ is a connected subset of $B(x,\Cr{Cstrong}(L+20\sqrt{L}))\subset B(x,2\Cr{Cstrong}L)$ and $\tilde{A}_{x,L}^u$ is a connected subset of $B(x,\Cr{Cstrong}(L+30\sqrt{L}))\subset B(x,2\Cr{Cstrong}L)$ for $L$ large enough. We clearly have that \eqref{Atilde2} still holds, that $B(y,L_0)\cap \tilde{A}_{x,L}^u\neq\emptyset$ for all $y\in{A_{x,L}^u},$ that every cluster of $E_\phi^{\geq-\sqrt{2u}}\cap B(x,L)$ with diameter at least $\sqrt{L}$ is connected to $\tilde{{A}}_{x,L}$ in $E_\phi^{\geq-\sqrt{2u}}\cap B(x,L+30\Cr{Cstrong}\sqrt{L})\subset E_\phi^{\geq-\sqrt{2u}}\cap B(x,2L)$, and that ${A}_{x,L}$ intersects every connected component of $B(x,L)$ with diameter at least $\sqrt{L}.$ Therefore by \eqref{Ahlfors} and \eqref{hbarunderpu2}, we have
\begin{equation*}
     \tilde{\Q}^{u,p}\left(\mathcal{E}_{x,L}^u\right)\geq\tilde{\Q}^{u,p}\left(\bigcap_{y\in{B(x,\Cr{Cstrong}L)}}\overline{\mathcal{E}}_{y,10\sqrt{L}}^u\right)\geq 1-CL^{\alpha}\exp\left\{-(10\sqrt{L})^c\right\}.
\end{equation*}
\end{proof}

Under $\mathcal{E}_{x,L}^u,$ we have constructed by \eqref{Iuincluded} a giant cluster $\tilde{A}_{x,L}^u\cap G$ intersecting ${B}(x,L/2),$ with $\tilde{A}_{x,L}^u\cap G\subset E_\phi^{\geq-\sqrt{2u}}\cap B(x,2\Cr{Cstrong}L)$ and such that $\tilde{A}_{x,L}^u\cap G$ is connected in $E_\phi^{\geq-\sqrt{2u}}\cap B(x,2L)$ to every cluster of $E_\phi^{\geq-\sqrt{2u}}\cap B(x,L)$ with diameter at least $L/10.$ Combining this with Lemma~\ref{percolfora}, we readily obtain by \eqref{eqhbaru>0} that:

\begin{Cor}
\label{cor:hbargeq0}
    For all $h<0,$ there exists constants $c(h)>0$ and $C(h)<\infty$ such that \eqref{eqhbar1} and \eqref{eqhbar2} hold, and thus $\overline{h}\geq0.$
\end{Cor}

\begin{Rk}
    As the perceptive reader will have already noticed, one does not need to use our ``sign flipping'' result, Proposition \ref{iuincluvu}, to prove $\overline{h}\geq0.$ One also does not need local uniqueness for random interlacements on the cable system, see Proposition \ref{connectivity}, but only on the discrete graph. We need to use percolation results for random interlacements on the cable system and Proposition \ref{iuincluvu} only to prove $\overline{h}>0,$ which is the content of the next section. This is similar to the case of $h_*$ on $\Z^d,$ $d\geq3,$ where one can prove $h_*\geq0$ without using Proposition \ref{iuincluvu}, see for instance \cite{MR914444} or \eqref{Iuincluded}, but an equivalent of Proposition \ref{iuincluvu} is used to prove $h_*>0,$ see Lemma 5.1 in \cite{DrePreRod}. We also note that in the next section to prove $\overline{h}>0$ we will never use the events $E_x^{L_0,u}$ from Definition~\ref{defgood}, which we only introduced to prove Corollary~\ref{cor:hbargeq0}.
\end{Rk}

\section{Denouement}
\label{denouement}
We proceed to the proof of our main results, Theorems \ref{T:mainresultGFF} and \ref{T:mainresultRI}. In Lemma \ref{giant}, we first use Proposition \ref{iuincluvu} to translate the result of Lemma \ref{proofofhbartilde>0}, which is stated in terms of $\I^u$ and $E_\phi^{\geq-\sqrt{2u}},$ to a similar result in terms of $\overline{E}^{\geq \sqrt{2v}},$  $0\leq v<u,$ which correspond to level sets of a Gaussian free field, see \eqref{eq:couplinglaw1}. This gives us directly, with overwhelming probability as $L\rightarrow\infty$, that there exists a giant cluster of $\overline{E}^{\geq\sqrt{2u}}$ in $B(x,L)$ which is at distance at most $\ell_0L_0(u)$, see \eqref{lbarandl0} and \eqref{eq:choiceL_0}, of any connected component in $B(x,L)$ with diameter $\sqrt{L}$, see Lemma \ref{giant}. The sets $H_{u,v,K,p}$ from Proposition \ref{iuincluvu} provide us with additional randomness, and we will take advantage of it to finish the connection of the giant cluster of  $\overline{E}^{\geq\sqrt{2u}}$ to any connected component of $\overline{E}^{\geq\sqrt{2u}}$ with diameter at most $\sqrt{L}$, and together with Lemma \ref{percolfora} this delivers Theorem \ref{T:mainresultGFF}. We then use the couplings from \eqref{vuincluphistrong} and Proposition \ref{iuincluvu} to also obtain Theorem \ref{T:mainresultRI}.
As a by-product of our methods, Theorem \ref{thmconclusion} asserts the existence of infinite sign clusters (in slabs) without any statements regarding their local structural properties under the slightly weaker assumption \eqref{weakSecIso'}, introduced in Remark \ref{R:wsi'1} above. We then conclude with some final remarks.

Let us first choose the parameters $u>0,$ $K<\infty$ and $p\in{(0,1)}$ in such a way that the conclusions of Proposition \ref{iuincluvu} and Lemmas \ref{percolfora} and \ref{proofofhbartilde>0} simultaneously hold. Recall that $\Cr{cweight}\leq \lambda_x\leq\Cr{Cweight}$ for all $x\in{G},$ see \eqref{lambda}. Fix an arbitrary reference level $u_0>0$, say $u_0=1$, and choose $u_1 \in (0, u_0)$ such that, for all $0< u\leq u_1,$ 
\begin{align}
\label{eq:defu_1}
\begin{split}
&\text{$ \exists \, K\geq2\sqrt{2u}$ with }
\sqrt{\log(\Cr{Ctruncated}u^{-\Cr{ctruncated}})} \vee \frac{\Cr{Kflip}}{2\sqrt{2u}\, \Cr{Cweight}} \leq K\leq\frac{\Cr{Kflip}}{\sqrt{2u}\,\Cr{Cweight}},\\
&\text{and $ \exists \, p\in{\big[\tfrac12,1\big)}$ such that } 1- \Cr{CBern}u^{\Cr{cBern}} \leq p\leq F\Big(\frac{\sqrt{ \Cr{cweight}} \Cr{Kflip}}{8\sqrt{u}\, \Cr{Cweight}}\Big),
\end{split}
\end{align}
where we recall that $F$ denotes the cumulative distribution function of a standard normal distribution. Also, note that $u_1$ with the desired properties exists by considering the limit as $u \downarrow 0$ and using the standard bound $F(x)\geq1-\frac{1}{\sqrt{2\pi}x}\exp\{-\frac{x^2}{2}\}$ for all $x>0$ in the second line. For a given $u \in (0,u_1]$, we then select any specific value of $K=K(u)$ and $p=p(u)$ satisfying the constraints in \eqref{eq:defu_1}, and henceforth refer to these values when writing $K$ and $p,$ and in particular we take the probability $\tilde{\Q}^{u,p}$, cf.\ \eqref{eq:Qup}, and $\mathcal{Q}^{u,K,p},$ cf.\ Proposition \ref{iuincluvu}, for this particular value of $K$ and $p.$ Then $K$ satisfies the constraint in \eqref{eq:smallnessCond} and $p$ satisfies the constraint in \eqref{eq:pSmallnesCond} on account of \eqref{eq:defu_1} and \eqref{lambda}, and noting that $K/2\leq K-\sqrt{2u}$. Therefore Proposition \ref{iuincluvu} applies for $u \in (0,u_1].$ Recalling $S_K$ from \eqref{eq:S} and \eqref{eq:allrandomsets}, taking $X_{u,K,p}$ as in \eqref{pxsmaller} and \eqref{eq:allrandomsets} and using \eqref{conditiondistance}, we have for all sets $\tilde{A}$ such that \eqref{Atilde2} holds that $B(\tilde{A}\cap G,L_0)\subset S_K\cap X_{u,K,p}.$ Moreover, recalling $R_u$ from \eqref{eq:eventsRSbar} and \eqref{eq:allrandomsets}, and using \eqref{Iuincluded}, we have that $\I^u\subset R_u.$ We thus obtain by \eqref{eq:couplinglaw3} that for all $u\in{(0,u_1]}$ and $v\leq u,$ under $\mathcal{Q}^{u,K,p},$
\begin{equation}
\label{ifAtildethenEbar}
\begin{array}{c}\text{if }\tilde{A}\subset\tilde{\I}^u\text{ is a connected set such that }\eqref{Atilde2}\text{ holds for some }L_0\geq1,\\\text{then }
\tilde{A}\cap G\subset \I^u\cap S_K\cap X_{u,K,p}\subset\overline{E}^{\geq\sqrt{2v}}\text{ and }B(\tilde{A}\cap G,L_0)\cap H_{u,v,K,p}\subset\overline{E}^{\geq\sqrt{2v}}.
\end{array}
\end{equation}
\begin{Def}
	\label{defEbar}
	For all $x\in{G},$ $L>0,$ $L_0=L_0(L)$ as in Lemma \ref{proofofhbartilde>0}, $u\in{(0,u_1)}$ and $0\leq v<u$ let us define $\overline{\mathcal{E}}_{x,L}^{u,v}$ as the event that
	\begin{enumerate}[i)]
		\item \label{i}there exists a $\sigma\big(\tilde{\I}^u,\tilde{\gamma},(\mathcal{B}_x^p)_{x\in{G}}\big)$-measurable and connected set $A_{x,L}^{u,v}\subset B(x,2\Cr{Cstrong}L)$ such that $A_{x,L}^{u,v}$ intersects every connected component of $B(x,L)$ with diameter at least $\sqrt{L},$
		\item \label{ii}there exists a connected set $\mathcal{C}_{x,L}^{u,v}\subset\overline{E}^{\geq \sqrt{2v}}\cap{B}(x,2\Cr{Cstrong}L)$ such that $B(\mathcal{C}_{x,L}^{u,v},L_0)\cap H_{u,v,K,p}\subset \overline{E}^{\geq\sqrt{2v}},$
		\item \label{iii} for all $y\in{A_{x,L}^{u,v}},$ $B(y,L_0)\cap\mathcal{C}_{x,L}^{u,v}\neq\emptyset.$

	\end{enumerate}
\end{Def}
Applying \eqref{ifAtildethenEbar} to the set $\tilde{A}_{x,L}^u$ from \eqref{EuxL} and taking $A_{x,L}^{u,v}=A_{x,L}^u$ and $\mathcal{C}_{x,L}^{u,v}=\tilde{A}_{x,L}^u\cap G,$ it is clear that $\mathcal{E}_{x,L}^u\subset \overline{\mathcal{E}}_{x,L}^{u,v},$ see \eqref{EuxL} for the definition of $\mathcal{E}_{x,L}^u.$ Moreover, it is clear that Lemma \ref{proofofhbartilde>0} holds for any $0<u\leq u_1,$ and $K$ and $p$ as in \eqref{eq:defu_1}, and we obtain:

\begin{Lemme}
	\label{giant}
	For all $x\in{G},$ $L>0,$ $L_0=L_0(L)$ as in Lemma \ref{proofofhbartilde>0}, $u\in{(0,u_1)},$ $K$ and $p$ as in \eqref{eq:defu_1}, and $0\leq v<u\leq u_1,$ there exist constants $C<\infty$ and $c>0$ depending on $u$ such that
	\begin{equation*}
	{\mathcal{Q}}^{u,K,p}\left(\overline{\mathcal{E}}_{x,L}^{u,v}\right)\geq 1-C\exp\{-L^c\},
	\end{equation*}
\end{Lemme}

Under $\overline{\mathcal{E}}_{x,L}^{u,v},$  we have thus constructed a giant component $\mathcal{C}_{x,L}^{u,v}\subset \overline{E}^{\geq\sqrt{2u}}\cap B(x,2\Cr{Cstrong}L)$ such that, by \ref{i}), any cluster of $\overline{E}^{\geq\sqrt{2v}}\cap B(x,L)$ with diameter at least $\sqrt{L}$ intersect the set $A_{x,L}^{u,v},$ and, by \ref{iii}), it also intersects $B(y,L_0)$ for some $y\in{\mathcal{C}_{x,L}^{u,v}}.$ Therefore, any cluster of $\overline{E}^{\geq\sqrt{2v}}\cap B(x,L)$ with diameter at least $L/10$ is connected to $B(y,L_0)$ for many vertices $y\in{\mathcal{C}_{x,L}^{u,v}},$ and if $B(y,L_0)\subset H_{u,v,K,p}$ for one of these $y,$ by \ref{ii}), this cluster would be connected to the giant component $\mathcal{C}_{x,L}^{u,v}$ in $\overline{E}^{\geq\sqrt{2v}}\cap B(y,L_0).$ We use this remark and the independence of $H$ from $A^{u,v}_{x,L}$ to deduce Theorem \ref{T:mainresultGFF} from \eqref{eq:couplinglaw1} and Lemma \ref{giant}.

\begin{proof}[Proof of Theorem \ref{T:mainresultGFF}]
	Let $h\leq h_1\stackrel{\text{def.}}{=}\sqrt{2u_1}$. The set $\tilde{A}_{\infty}^{u_1}$ from Lemma~\ref{percolfora} is included in $\overline{E}^{\geq h_1}\subset\overline{E}^{\geq h}$ by \eqref{ifAtildethenEbar} and monotonicity, and \eqref{eqhbar1} follows readily. Let us now prove that \eqref{eqhbar2} hold for all $h\leq h_1.$ By Corollary \ref{cor:hbargeq0}, it is enough to prove that \eqref{eqhbar2} hold for all $0\leq h\leq h_1,$ and let us fix $u=u_1$ and $v=h^2/2.$ We will simply denote by $H$ the event $H_{u,v,K,p}$ from Proposition \ref{iuincluvu}. Let us define for all $x\in{G},$ $L$ large enough, $L_0$ as in Lemma \ref{proofofhbartilde>0}, $k\in{\{2,\dots,\lfloor\frac{\sqrt{L}}{20}\rfloor\}}$ and $y\in{B(x,L)}$
	\begin{equation*}
	\hat{\mathcal{E}}_{x,L}^{y,k}=\overline{\mathcal{E}}_{x,L}^{u,v}\cap\left\{\begin{array}{c}\text{the cluster of $y$ in $\overline{E}^{\geq\sqrt{2v}}\cap B(y,2k\sqrt{L})\cap B(x,L)$}
	\\\text{intersects }\partial B(y,2k\sqrt{L})\text{ but does not intersect }\mathcal{C}_{x,L}^{u,v}\end{array}\right\}.
	\end{equation*}
	Let also $\mathcal{Z}_{x,L}^{y,k}={A}_{x,L}^{u,v}\cap B(y,2k\sqrt{L}-L_0-\Cr{Cdistance})\cap B(x,L)\setminus B(y,2(k-1)\sqrt{L}+L_0),$ and $Z_k$ be the smallest $z\in{\mathcal{Z}_{x,L}^{y,k}}$ (in some deterministic fixed order on the vertices of $G$) such that
	\begin{equation}
	\label{zk}
	y\longleftrightarrow \partial_{ext}B(z,L_0)\\\text{ in }\overline{E}^{\geq\sqrt{2v}}\cap B(y,2k\sqrt{L})\cap B(x,L)\setminus \bigcup_{z'\in{\mathcal{Z}_{x,L}^{y,k}}}B(z',L_0).
	\end{equation}
	We fix arbitrarily $Z_k=y$ if \eqref{zk} never happens. By \eqref{conditiondistance}, if $\hat{\mathcal{E}}_{x,L}^{y,k}$ happens and $L$ is large enough, since the set of vertices in $B(y,2k\sqrt{L}-L_0-\Cr{Cdistance})\cap B(x,L)\setminus B(y,2(k-1)\sqrt{L}+L_0)$ connected to $y$ in $\overline{E}^{\geq\sqrt{2v}}\cap B(y,2k\sqrt{L})\cap B(x,L)$ contains a connected component with diameter $\geq2\sqrt{L}-2L_0-3\Cr{Cdistance}\geq\sqrt{L},$ by \ref{i}) of Definition \ref{defEbar} it must intersect some $z\in{A_{x,L}^{u,v}},$ and so $Z_k\neq y.$  Since under $\hat{\mathcal{E}}_{x,L}^{y,k}$ the cluster of $y$ in $\overline{E}^{\geq\sqrt{2v}}\cap B(y,2k\sqrt{L})\cap B(x,L)$ does not intersect $\mathcal{C}_{x,L}^{u,v},$ we obtain by \ref{ii}) and \ref{iii}) of Definition \ref{defEbar} that $H^c\cap B(Z_k,L_0)\neq\emptyset.$  Therefore
	\begin{equation}
	\label{Zknonempty}
	\hat{\mathcal{E}}_{x,L}^{y,k}\subset\{Z_k\neq y,\,H^c\cap B(Z_k,L_0)\neq\emptyset\}.
	\end{equation}
	Since $A_{x,L}^{u,v}$ is $\sigma(\tilde{\I}^u,\tilde{\gamma},(\mathcal{B}^p_x)_{x\in{G}})$ measurable, we have that the events $\{Z_k=z\}$ are $\mathcal{F}_z$ measurable for all $z\in{B(y,2(k-1)\sqrt{L}+L_0)^c},$ where
	\begin{equation*}
	\mathcal{F}_z=\sigma\big(\tilde{\I}^u,\tilde{\gamma},(\mathcal{B}^p_x)_{x\in{G}},\{x'\in{\overline{E}^{\geq\sqrt{2v}}\}_{x'\in{B(z,L_0)^c}}}\big).
	\end{equation*}
	Moreover by Lemma \ref{iuincluvu} the event $\{x'\in{H}\}$ is independent of $\mathcal{F}_z$ for all $z\in{G}$ and $x'\in{B(z,L_0)}$ and so, under $\mathcal{Q}^{u,K,p}(\cdot\,|\,\mathcal{F}_z),$  $\{x'\in{H}\}_{x'\in{B(z,L_0)}}$ is an i.i.d.\ sequence of events with common probability $\mathcal{Q}^{u,K,p}(x\in{H})>0.$ Since for all $k\in{\N}$ we have $\hat{\mathcal{E}}_{x,L}^{y,k}\subset\hat{\mathcal{E}}_{x,L}^{y,k-1}$ and $\hat{\mathcal{E}}_{x,L}^{y,k-1}$ is $\mathcal{F}_z$ measurable for all $z\in{B(y,2(k-1)\sqrt{L}+L_0)^c},$ with the convention  $\hat{\mathcal{E}}_{x,L}^{y,0}=\overline{\mathcal{E}}_{x,L}^{u,v},$ we obtain by \eqref{Ahlfors} and \eqref{Zknonempty} that
	\begin{align*}
	&\mathcal{Q}^{u,K,p}(\hat{\mathcal{E}}_{x,L}^{y,k})\\&\leq\sum_{z\in{B(y,2(k-1)\sqrt{L}+L_0)^c}}
	\E_{\mathcal{Q}^{u,K,p}}\left[\1_{\hat{\mathcal{E}}_{x,L}^{y,k-1}\cap\{{Z}_k=z\}}\mathcal{Q}^{u,K,p}(H^c\cap B(z,L_0)\neq\emptyset\,|\,\mathcal{F}_{z})\right]
	\\&\leq \mathcal{Q}^{u,K,p}(\hat{\mathcal{E}}_{x,L}^{y,k-1})\left(1-\mathcal{Q}^{u,K,p}(x\in{H^c})^{\Cr{CAhlfors}L_0^{\alpha}}\right),
	\end{align*}
	
	Iterating, we obtain that there exist constants $c=c(u,v)>0$ and $C=c(u,v)<\infty$ such that for all $k\in{\{2,\dots,\lfloor\frac{\sqrt{L}}{20}\rfloor\}},$
	\begin{equation}
	\label{QupEyksmall}
	\mathcal{Q}^{u,K,p}(\hat{\mathcal{E}}_{x,L}^{y,k})\leq C\exp(-ck).
	\end{equation}
	By \ref{ii}) of Definition \ref{defEbar}, we have moreover under $\overline{\mathcal{E}}_{x,L}^{u,v}$ that $\mathcal{C}_{x,L}^{u,v}\subset\overline{E}^{\geq\sqrt{2v}}\cap B(x,2\Cr{Cstrong}L)$ and is connected. Now the event in \eqref{eqhbar2} for $h=\sqrt{2v}$ and $\overline{E}^{\geq\sqrt{2v}}$ instead of $E^{\geq h}$ implies that either $\overline{\mathcal{E}}_{x,L}^{u,v}$ does not happen, or it happens and there exists $y\in{B(x,L)}$ such that the component of $y$ in $\overline{E}^{\geq\sqrt{2v}}\cap B(x,L)$ has diameter at least $L/10$ and is not connected to $\mathcal{C}_{x,L}^{u,v}$ in $\overline{E}^{\geq\sqrt{2v}}\cap B(x,2\Cr{Cstrong}L),$ and then there exists $y\in{B(x,L)}$ such that $\hat{\mathcal{E}}_{x,L}^{y,\lfloor\frac{\sqrt{L}}{20}\rfloor}$ happens. By \eqref{eq:couplinglaw1} $\overline{E}^{\geq\sqrt{2v}}$ has the same law under $\mathcal{Q}^{u,K,p}$ as $E^{\geq h}$ under $\P^G,$ and thus  by \eqref{Ahlfors}, Lemma \ref{giant} and \eqref{QupEyksmall}, we obtain that the probability in \eqref{eqhbar2} is smaller than 
	\begin{equation*}
	\tilde{\Q}^{u,p}\big((\overline{\mathcal{E}}_{x,L}^{u,v})^{{c}}\big)+\mathcal{Q}^{u,K,p}\Big(\bigcup_{y\in{B(x,L)}}{\hat{\mathcal{E}}_{x,L}^{y,\lfloor\frac{\sqrt{L}}{20}\rfloor}}\Big)\leq C\exp(-L^c)+CL^\alpha\exp(-c\sqrt{L}).
	\end{equation*}
\end{proof}

We now continue with the proof of Theorem \ref{T:mainresultRI}.
\begin{proof}[Proof of Theorem \ref{T:mainresultRI}]
	We continue with the setup of \eqref{eq:defu_1}, and fix some $u\leq \tilde{u}\stackrel{\text{def.}}{=}u_1.$ We now define the probability $\nu_1$ on $(\{0,1\}^G)^2\times(\{0,1\}^G)$ as the (joint) law of 
	\begin{equation*}
	\big(\big ( \1_{\{x\in{\I^u}\}},\1_{\{S_K^x\cap\{X^x_{u,K,p}=1\}\}} \big )_{x\in{G}}, \big(\1_{\{x\in{\overline{E}^{\geq{\sqrt{2u}}}}\}}\big )_{x\in{G}} \big)
	\end{equation*}
	under $\mathcal{Q}^{u,K,p},$ and the probability $\nu_2$ on $(\{0,1\}^G)\times(\{0,1\}^G)^2$ as the law of
	\begin{equation*}
	\big( \big(\1_{\{-\phi_x\geq\sqrt{2u}\}} \big)_{x\in{G}} , \big(\1_{\{x\in{\V^u}\}},\1_{\{-\gamma_x\geq0\}}\big)_{x\in{G}}\big)
	\end{equation*} 
	under $\tilde{\Q}^{u,p}.$ We concatenate these probabilities by defining the probability $Q^{u}$ on the product space $(\{0,1\}^G)^{2}\times(\{0,1\}^G)\times(\{0,1\}^G)^{2}$ such that for all measurable sets $A_1\subset(\{0,1\}^G)^2,$ $A_2\subset\{0,1\}^G$ and  $A_3\subset(\{0,1\}^G)^2$
	\begin{equation*}
	Q^{u}\big(A_1\times A_2\times A_3\big)=\E_{\nu_1}\Big[\1_{\{\eta_1^1\in{A_1},\eta_2^1\in{A_2}\}}\nu_2\big(\eta_2^2\in{A_3}\,\big|\,\eta_1^2=\eta_2^1\big)\Big],
	\end{equation*}
	where we wrote the coordinates under $\nu_i$ as $(\eta^i_1,\eta^i_2)$ for all $i\in{\{1,2\}},$ and furthermore $\nu_2(\eta_2^2\in \cdot \,|\,\eta_1^2= \cdot)$ is a regular conditional probability distribution on $\{0,1\}^G$ for $\eta_2^2$ given $\sigma(\eta_1^2).$ One then defines the three random sets from the statement of the theorem under $Q^{u}$ as follows: the sets $\I$ and $\mathcal{K}$ are defined by the marginals of $\eta^1_1$ and the set $\V$ as the first marginal of $\eta_2^2.$ With these choices, part $i)$ and $ii)$ of \eqref{eq:mainresultRI} are clear by definition, noting that $\I^u$ and $S_K\cap X_{u,K,p}$ with $X_{u,K,p}$ coming from \eqref{pxsmaller} are independent under $\tilde{\Q}^{u,p}$, which follows from \eqref{eq:Qup} on account of \eqref{eq:S}.  Since $\I^u\cap S_K\cap X_{u,K,p}\subset \overline{E}^{\geq\sqrt{2u}}$ by \eqref{Iuincluded}, \eqref{eq:eventsRSbar} and \eqref{eq:couplinglaw3}, $\overline{E}^{\geq\sqrt{2u}}$ has the same law as $\{x\in{G};-\phi_x\geq \sqrt{2u}\}$ by \eqref{eq:couplinglaw1} and symmetry of $\phi,$ and $\{x\in{G};-\phi_x\geq\sqrt{2u}\}\subset\V^u$ by \eqref{vuincluphistrong}, one can easily check that  the inclusion $ \I \cap \mathcal{K} \subset \V$ holds under $Q^u.$  Finally, $\I^u\cap S_K\cap X_{u,K,p}$ contains $\tilde{\Q}^{u,p}$-a.s.\ an infinite cluster by Lemma \ref{percolfora} and \eqref{ifAtildethenEbar}, and thus $\I \cap \mathcal{K}$ under $Q^u$ too. This completes the proof.
\end{proof}

As the perceptive reader will have noticed, the inclusion in part $iii)$ of Theorem \ref{T:mainresultRI} can be somewhat strengthened to a statement of the form $(\mathcal I \cap \mathcal{K}) \subset (\mathcal V\cap \mathcal K') $ with $ \mathcal K'$ independent of $\mathcal V$ and with the same law as $\{x\in{G};\,\Phi_x>0\}$ under $\P^G$ by taking into account the effect of $\tilde{\gamma}$ in \eqref{vuincluphistrong}, cf.\ \eqref{eq:Qup} regarding the asserted independence.

The sole existence of an infinite cluster without the local connectivity picture entailed in \eqref{eqhbar2} can be obtained under the slightly weaker geometric assumption \eqref{weakSecIso'} from Remark \ref{R:wsi'1}. We record this in the following

\begin{The}
	\label{thmconclusion}
	Under the assumptions \eqref{eq:Ass} and \eqref{weakSecIso'} on $G$, there exists $h_1>0$ such that for all $h\leq h_1$, \eqref{eqhbar1} holds for some $x\in G$ and there exists a.s.\ an infinite connected component in $E^{\geq h}\cap B(G_p,CL_0(h^2/2))$ and in $\V^{h^2/2}\cap B(G_p,CL_0(h^2/2))$ with $L_0(\cdot)$ given by \eqref{eq:choiceL_0}. In particular $h_*>0$ and $u_*>0.$ 
\end{The}
\begin{proof}
	One adapts the argument leading to \eqref{eqhbar1} in the proof of Theorem \ref{T:mainresultGFF}, replacing Lemma \ref{percolfora} by the corresponding result obtained under the weaker assumption \eqref{weakSecIso'} described in Remark \ref{R:wsi'1}. We omit further details.
\end{proof}

We conclude with several comments.

\begin{Rk}
	\leavevmode
	\label{lastremark}
	\begin{enumerate}[1)]
		\item \label{lastremark1}In \cite{MR3390739}, on $\Z^d,$ $d\geq3,$ a slightly different parameter $\overbar{h}_1$ is introduced since only a super-polynomial decay in $L$ is required in the conditions corresponding to \eqref{eqhbar1} and \eqref{eqhbar2}, and in \cite{MR3417515} yet another parameter $\overbar{h}_2$ is introduced by allowing the addition of a small sprinkling parameter $h'$ to connect together the large paths of $E^{\geq h}.$ However, it is clear that $\overbar{h}\leq\overbar{h}_1\leq\overbar{h}_2,$ and so the parameters $\overbar{h}_1$ and $\overbar{h}_2$ are also positive as a consequence of Theorem \ref{T:mainresultGFF}. 
		
		\item \label{remarkchoicemathcalK}Looking at the proof of Theorem 1.2, one sees that for $u$ small enough, the set $\mathcal{K}$ can be taken with the same law under $Q^u$ as $S_K\cap X_{u,K,p}$ under $\tilde{\Q}^{u,p},$ for some $K>0$ and $p\in{(0,1)}$ as in \eqref{eq:defu_1}, where $S_K$ is defined in \eqref{eq:S} and \eqref{eq:allrandomsets}, and $X_{u,K,p}$ in \eqref{pxsmaller} and \eqref{eq:allrandomsets}. Changing the event $C_{x}^{L_0,p}$ in Definition \ref{defgood} by the increasing event $\overbar{C}_x^{L_0,p}$ which occurs if and only if for all $z\in{\tilde{B}(x,2\Cr{Cstrong}(L_0+\Cr{Cdistance})+\Cr{Cdistance})},$ $\tilde{\phi}_z\geq-K,$ and the event $F_{x}^{L_0,p}$ by the decreasing event $\overbar{F}_x^{L_0,p}$ which occurs if and only if for all $z\in{\tilde{B}(x,2\Cr{Cstrong}(L_0+\Cr{Cdistance})+\Cr{Cdistance})},$ $\tilde{\phi}_z\leq K,$ one can show as in Lemma \ref{percolfora} that there exists a connected and unbounded set $\tilde{A}\subset\tilde{G}$ such that
		\begin{equation*}
		\tilde{A}\subset\tilde{\I}^u,\text{ and } |\tilde{\phi}_z|\leq K\text{ for all }z\in{\tilde{B}(\tilde{A},2L_0+\Cr{Cdistance})}.
		\end{equation*}
		Therefore, adapting the proof of Theorem \ref{T:mainresultRI}, one can take $\mathcal{K}$ with the same law under $Q^u$ as $\overbar{S}_K\cap X_{u,K,p}$ under $\tilde{\Q}^{u,p},$ for some $K>0$ and $p\in{(0,1)}$ as in \eqref{eq:defu_1}, where $\overbar{S}_K$ is defined in \eqref{eq:eventsRSbar} and \eqref{eq:allrandomsets}, and $X_{u,K,p}$ in \eqref{Xisphismall} and \eqref{eq:allrandomsets}, or with the same law as $\{x\in{G};\,|\tilde{\phi}_z|\leq K\text{ for all }z\in{U^x}\},$ and i) and iii) in \eqref{eq:mainresultRI} still hold. This choice for $\mathcal{K}$ has a simple expression and would be enough for the purpose of proving $\overbar{h}>0$ and $u_*>0,$ but has the disadvantage of not being independent from $\I.$
		Independence, however, is expected to be useful for future applications.
		
		\item \label{complement}Taking complements in the inclusion $\I\cap\mathcal{K}\subset\V,$ see Theorem \ref{T:mainresultRI}, and intersecting with $\mathcal{K},$ we obtain that $\V^c\cap \mathcal{K}\subset\I^c.$ Taking $\I'=\V^c$ and $\V'=\I^c,$ we obtain the inclusion $\I'\cap \mathcal{K}\subset \V',$ and $\mathcal{K}$ is independent of $\I,$ and thus of $\V'.$ Therefore, we could have chosen $\mathcal{K}$ independent of $\V$ in ii) of  \eqref{eq:mainresultRI} instead of $\mathcal{K}$ independent of $\I.$
		
		\item \label{R:cablesfinal}Using a similar reasoning as the one leading to Corollary \ref{cor:hbargeq0}, one can prove strong percolation, as in \eqref{hbardef}, for the level sets $\tilde{E}^{>h},$ see \eqref{Etildedef}, for all $h<0,$ in the sense that \eqref{eqhbar1} and \eqref{eqhbar2} hold but for the level sets $\tilde{E}^{>h}$ of the Gaussian free field on the cable system $\tilde{G}$ instead of the graph $G.$ Moreover, the critical parameter $\tilde{h}_*$ for percolation of the continuous level sets $\tilde{E}^{>h}$ is exactly equal to $0$ by Proposition \ref{nonpercolationsignclusterscable}, and thus the strongly percolative phase consists of the \textit{entire} supercritical phase for the Gaussian free field on the cable system, i.e.\ if one introduces $\tilde{\overbar h}$ as in \eqref{hbardef}, but putting ``tildes everywhere'' in \eqref{eqhbar1} and \eqref{eqhbar2}, one arrives at the following 
		\begin{The}
			\label{T:cablessharp}
			If $G$ satisfies \eqref{eq:Ass} and \eqref{weakSecIso}, then $\tilde{\overbar h}= \tilde{h}_*=0$.
		\end{The}
		This result can also be proved without condition \eqref{weakSecIso}. Indeed, by \eqref{Iuincluded}, \eqref{ballcapacity} and the definition of random interlacements, the probability that $\tilde{E}^{>-\sqrt{2u}}$ does not contain a connected component of diameter at least $L/10$ has stretched exponential decay in $L$ for any $u>0.$ Moreover, by Corollary \ref{phiisagff}, any connected component of $\{z\in{\tilde{G}};\,\tilde{\phi}_z>-\sqrt{2u}\}\cap B(x,L)$ either intersects $\tilde{\I}^u$ or is a connected component of $\{z\in{\tilde{G}};\,\tilde{\gamma}_z>0\}$ not intersecting $\tilde{\I}^u.$ Since $\tilde{\I}^u$ and $\tilde{\gamma}$ are independent under $\tilde{\Q}^{u,p},$ the probability that $\tilde{\I}^u$ does not intersect a component of $\{z\in{\tilde{G}};\,\tilde{\gamma}_z>0\}$ with diameter at least $L/10$ has stretched exponential decay by Lemma \ref{Lemmacapline} and \eqref{defIu}. Therefore, with high enough probability, any connected component of $\{z\in{\tilde{G}};\,\tilde{\phi}_z>-\sqrt{2u}\}\cap B(x,L)$ with diameter at least $L/10$ intersects $\tilde{\I}^u,$ and strong connectivity of $\tilde{E}^{>-\sqrt{2u}}$ then readily follows from Proposition \nolinebreak\ref{connectivity}. 
		
		\item \label{percoplanes}
		Looking at Theorem \ref{thmconclusion}, we have in fact proved that if \eqref{weakSecIso'} holds for some subgraph in $G_p$ of $G$, then there exists $0<h_1\leq h_*$ such that for all $h<h_1,$ there exists $L>0$ with
		\begin{equation*}
		\P^G\big(\text{there exists an infinite connected components in }E^{\geq h}\cap B(G_p,L)\big)=1.
		\end{equation*}
		It then follows by \eqref{incluvufaible}, that the same is true for $\V^{u}$ i.e., there exists $0<u_1\leq u_*$ such that for all $u<u_1,$ and some $L>0$,
		\begin{equation*}
		\P^I\big(\text{there exists an infinite connected components in }\V^u\cap B(G_p,L)\big)=1.
		\end{equation*}
		If $G=G_1\times G_2,$ we may choose $G_p=P_1\times P_2$ a half-plane, where $P_1$ and $P_2$ are two semi-infinite geodesics in $G_1$ and $G_2.$ Hence, we obtain that $E^{\geq h}$ and $\V^u$ percolate in thick planes $B(G_p,L)$ for $h>0$ and $u>0$ small enough. If $\nu>1,$ then $\V^u$ actually percolates in the plane $G_p$ for $u$ small enough, see Remark \ref{R:applicor}, \ref{supernu>1}), and in Theorem 5.1 of \cite{MR2891880}, it is shown that this is also true if $\nu=1$ and $G_1=\Z.$ It is still unclear, and an interesting open question, whether this holds true for $\nu<1$ or \nolinebreak not. 
		\item The existence of a non-trivial supercritical phase for Bernoulli percolation (and other models) is proved in \cite{MR3520023} if $G$ satisfies the volume upper bound of \eqref{Ahlfors} and a local isoperimetric inequality. The proof involves events similar to those considered in \eqref{eqhbar2}, and it is possible that our condition \eqref{weakSecIso} could be replaced by this local isoperimetric inequality, which would for example cover the case of the Menger sponge, see Remark \ref{remarkconnectivity}, \ref{Menger}). However, one would then need to take a super-geometric scale in our renormalization scheme \eqref{defrenor}, and then lose the stretched exponential decay in \eqref{eqhbar1} and \eqref{eqhbar2}. 
		
		\item One may also inquire whether a phase coexistence regime for percolation of $\{ |\phi|>h\}$ and $\{|\phi|<h\}$ exists, or similarly for the level sets of local times $\{x\in{G};\,\ell_{x,u}>\alpha\}$ of random interlacements, with $u>0$, $\alpha \geq 0$, considered in \cite{MR3163210}. For instance, regarding the latter, is it possible for all $\alpha>0$ to find $u\geq0$ such that percolation for the local times at level $u$ above and below $\alpha$ occur simultaneously?
		
		\item \label{betterunderstanding}Finally, it would be desirable to have a conceptual understanding of the mechanism that lurks behind the percolation above small enough levels $h\geq 0$ for the discrete level sets $E^{\geq h}$ (as opposed to their continuous counterparts $\widetilde{E}^{\geq h}$, cf.\ \ref{R:cablesfinal}) above). Our current techniques are based on stochastic comparison, see Lemma \ref{lemmaflip} and Proposition \nolinebreak\ref{iuincluvu}, but the induced couplings suggest that one should be able to exhibit these features as a property of $\tilde{\varphi}$ \textit{itself}, without resorting to additional randomness.
	\end{enumerate} 
\end{Rk}

\noindent\textbf{Acknowledgments.}
We would like to thank A.-S.\ Sznitman for pointing out an error in an earlier version of the paper, as well as several reviewers for careful reading and beneficial suggestions.
AD gratefully acknowledges support of the 
UoC Forum \lq Classical and quantum dynamics of interacting particle systems\rq, \lq QM2 -- Quantum Matter and Materials\rq\ as well as Deutsche Forschungsgemeinschaft (DFG) grants DR 1096/1-1 and DR 1096/2-1, and AP gratefully acknowledges support of the \lq IPaK\rq-program at the University of Cologne. Part of this work was completed while the research of PFR was supported by the grant ERC-2017-STG 757296-CriBLaM. 

\appendix
\setcounter{secnumdepth}{0}
\section{Appendix: Proof of Proposition \ref{someproperties}}
\renewcommand*{\theThe}{A.\arabic{The}}
\setcounter{The}{0}
\setcounter{equation}{0}
\renewcommand{\theequation}{A.\arabic{equation}}
Proposition \ref{someproperties} is proved in \cite{MR1853353} when $d$ is the graph distance, and we are going to adapt its proof for a general distance $d.$ Let us begin with the
\begin{proof}[Proof of Proposition \ref{someproperties}, $i)$]
Using \eqref{Green} and \eqref{Ahlfors}, we have for all $x\in{G}$ and $t\leq\Cr{CGreen},$
    \begin{align*}
        \lambda(\{y\in{G}:\ g(x,y)>t\})&\stackrel{\eqref{Green}}{\leq}\lambda(\{y\in{G}:\ \Cr{CGreen}d(x,y)^{-\nu}>t\})
        \\&\stackrel{\phantom{\eqref{Green}}}{\leq} \lambda\Big(B\big(x,\Big(\frac{t}{\Cr{CGreen}}\Big)^{-\frac{1}{\nu}}\big)\Big)
        \stackrel{\eqref{Ahlfors}}{\leq} Ct^{-\frac{\alpha/\beta}{\alpha/\beta-1}}
    \end{align*}
    Moreover, by \eqref{Green}, $\lambda(\{y\in{G}:\ g(x,y)>t\})=0$ for all $x\in{G}$ and $t>\Cr{CGreen},$ and \eqref{due} follows directly from Proposition 5.1 in \cite{MR1853353}.
\end{proof}
In order to prove Proposition \ref{someproperties}, $ii)$ we first need the following bounds on the expected time at which the random walk $Z$ on $G$ leaves a ball.
\begin{Lemme}
\label{expectedexittime}
There exist constants $0<\Cl[c]{cexittime}\leq\Cl{Cexittime}<\infty$ only depending on $G$ such that for all $x\in{G}$ and $R\geq1,$
\begin{equation}
\label{boundexittime}
    \Cr{cexittime}R^{\beta}\leq E_x[T_{B(x,R)}]=\sum_{y\in{B(x,R)}}\lambda_yg_{B(x,R)}(x,y)\leq\sum_{y\in{B(x,R)}}\lambda_yg(x,y)\leq\Cr{Cexittime}R^{\beta}
\end{equation}
\end{Lemme}
\begin{proof}
    Let us fix some $x\in{G}$ and $R\geq1.$ The equality in \eqref{boundexittime}
    is true by definition of the stopped Green function \eqref{Greenstopped}. Partitioning $B(x,R)\setminus B(x,1)$ into $B_k=B(x,2^{-k}R)\setminus B(x,2^{-k-1}R)$ for $k\in{\{0,\dots,\lfloor\log_2R\rfloor\}},$ we have
    \begin{align*}
    \sum_{y\in{B(x,R)\setminus B(x,1)}}\lambda_yg(x,y)&\leq \sum_{k=0}^{\lfloor\log_2R\rfloor}\sum_{y\in{B_k}}\lambda_yg(x,y)
    \stackrel{\eqref{Green}}{\leq}\Cr{CGreen}\sum_{k=0}^{\lfloor\log_2R\rfloor}\lambda(B_k)\big(2^{-k-1}R\big)^{-\nu}
    \\&\hspace{-.42em}\stackrel{\eqref{Ahlfors}}{\leq}CR^{\alpha-\nu}\sum_{k=0}^{\infty}2^{-k(\alpha-\nu)},
    \end{align*}
    and the upper bound in \eqref{boundexittime} follows since $\alpha-\nu=\beta>0$ and 
    \begin{equation*}
        \sum_{y\in{B(x,1)}}\lambda_yg_{B(x,R)}(x,y)\leq \Cr{CAhlfors}\Cr{CGreen}.
    \end{equation*}
    For the lower bound, we can assume w.l.o.g.\ that $R$ is large, and we write
    \begin{align*}
        \sum_{y\in{B(x,R)}}\lambda_yg_{B(x,R)}(x,y)&\geq \sum_{y\in{B(x,\frac{R}{1+2\Cr{CHarnack}})}}\lambda_yg_{B(x,R)}(x,y) \\&\hspace{-.42em} \stackrel{\eqref{Greenstoppedbound}}{\geq}\frac{\Cr{cGreen}}{2}\sum_{y\in{B(x,\frac{R}{1+2\Cr{CHarnack}})\setminus\{x\}}}\lambda_yd(x,y)^{-\nu}\stackrel{\eqref{Ahlfors}}{\geq} cR^{\alpha-\nu}.
\end{align*}
\end{proof}
We now follow the proof of Proposition 4.33 in \cite{MR3616731}. The bounds in Lemma \ref{expectedexittime} on the expected exit time of a ball give us the following lemma as a first step in the proof of Proposition \ref{someproperties}, $ii)$.

\begin{Lemme}
\label{A.2}
There exist constants $\Cl{Cappendix}>0$ and $\Cl[c]{cappendix}>0$ only depending on $G$ such that for all $x\in{G}$ and $R>0,$
\begin{equation*}
     P_x\big(T_{B(x,R)}>\Cr{Cappendix}R^{\beta}\big)\geq\Cr{cappendix}
\end{equation*}
\end{Lemme}
\begin{proof}
    Take $\Cr{Cappendix}=(\Cr{cexittime}\wedge1)/4.$ Let us fix $x\in{G}$ and $R>0,$ and we can assume w.l.o.g.\ that $\Cr{Cappendix}R^\beta\geq1/2$ (and then $R\geq1$). We first need to remark that, by Lemma \ref{expectedexittime}, for all $y\in{B(x,R)},$
    \begin{equation*}
        E_y\left[T_{B(x,R)}\right]\leq E_y\left[T_{B(y,2R)}\right]\leq \Cr{Cexittime}(2R)^{\beta}.
    \end{equation*}
    Let us write $n=\left\lceil\Cr{Cappendix}R^{\beta}\right\rceil.$ An application of the Markov property of $Z$ at time $n$ gives us 
    \begin{equation}
    \label{LemmaA.31}
        E_x\big[T_{B(x,R)}\1_{T_{B(x,R)}> n}\big]=E_x\big[E_{X_n}[T_{B(x,R)}+n]\1_{T_{B(x,R)}>n}\big]\leq CR^\beta P_x(T_{B(x,R)}> n).
    \end{equation}
    On the other hand, by Lemma \ref{expectedexittime},
    \begin{equation}
        \label{LemmaA.32}
        E_x\big[T_{B(x,R)}\1_{T_{B(x,R)}> n}\big]\geq\Cr{cexittime}R^\beta-n\geq\Cr{Cappendix}R^{\beta},
    \end{equation}
    and combining \eqref{LemmaA.31} and \eqref{LemmaA.32} let us conclude. 
\end{proof}

It is interesting to note that Lemma \ref{A.2} is analogue to Proposition \ref{someproperties}, $ii)$ for $n=\left\lfloor\Cr{Cappendix}R^{\beta}\right\rfloor,$ and we are going to use it iteratively with the help of \eqref{conditiondistance} to finish the proof of Proposition \ref{someproperties}.

\begin{proof}[Proof of Proposition \ref{someproperties}, ii)]
Let us fix $x\in{G},$ $r>0$ and a positive integer $m.$ We define recursively the sequence of stopping times $S_p,$ $p\in{\N}$ by 
\begin{equation*}
   S_0=0, \quad \text{ and for all }p\geq1,\, S_p=T_{B(X_{S_{p-1}},r)}.
\end{equation*}
For all $p\in{\N},$ $d(Z_{S_{p-1}},Z_{S_p-1})\leq r$ and by \eqref{conditiondistance}, $d(Z_{S_{p-1}},Z_{S_p})\leq r+\Cr{Cdistance}.$ In particular, $d(x,Z_k)\leq (r+\Cr{Cdistance})m$ for all $0\leq k\leq S_m$ and thus $S_m\leq T_{B(x,(r+\Cr{Cdistance})m)}.$ Let us define
\begin{equation*}
    \xi_p=\1_{S_p-S_{p-1}\geq \Cr{Cappendix}r^{\beta}}\qquad\text{ and }\qquad N=\sum_{p=0}^m\xi_p.
\end{equation*}
By definition, $T_{B(x,(r+\Cr{Cdistance})m)}\geq S_m\geq \Cr{Cappendix}r^{\beta}N.$ Moreover, by the strong Markov property and Lemma \ref{A.2}, $E_x[\xi_p\,|\,\mathcal{F}_{S_{p-1}}]\geq \Cr{cappendix},$ where $\mathcal{F}_i=\sigma(Z_0,\dots,Z_i)$ for all $i\geq0.$ Using a martingale inequality, Lemma A.8 in \cite{MR3616731}, we thus get
\begin{equation}
\label{last}
     P_x\Big(T_{B(x,(r+\Cr{Cdistance})m)}<\frac{\Cr{Cappendix}\Cr{cappendix}}{2}r^{\beta}m\Big)\leq P_x\Big(N<\frac{m\Cr{cappendix}}{2}\Big)\leq \exp\{-cm\}.
\end{equation}
Let us now fix a constant $\Cl[c]{clast}$ small enough so that, if $\Cr{Cdistance}^{-1}R\leq n\leq\Cr{clast}R^{\beta},$ then 
    \begin{equation*}
        m\stackrel{\mathrm{def.}}{=}\left\lceil\Big(\frac{\Cr{clast}R^\beta}{n}\Big)^{\frac{1}{\beta-1}}\right\rceil\leq2\Big(\frac{\Cr{clast}R^\beta}{n}\Big)^{\frac{1}{\beta-1}},\quad  r\stackrel{\mathrm{def.}}{=}\frac{R}{m}-\Cr{Cdistance}\geq\frac{1}{4}\Big(\frac{n}{\Cr{clast}R}\Big)^{\frac{1}{\beta-1}},
    \end{equation*}
    and
    \begin{equation*}
        \frac{\Cr{Cappendix}\Cr{cappendix}}{2}r^\beta m\geq\frac{\Cr{Cappendix}\Cr{cappendix}}{2\times4^\beta}\times\frac{n}{\Cr{clast}} \geq n,
    \end{equation*}
    and \eqref{exittime} (with $C=1$) then readily follows from \eqref{last} as long as $\Cr{Cdistance}^{-1}R<n<\Cr{clast}R^{\beta}.$ Finally, if $n< \Cr{Cdistance}^{-1}R,$ then by \eqref{conditiondistance} $B_G(x,n)\subset B(x,R)$ and the left-hand side of \eqref{exittime} is always $0,$ and it is easy to find a constant $C$ large enough so that the right-hand side of \eqref{exittime} is always larger than $1$ whenever $n>\Cr{clast}R^{\beta}.$ 
\end{proof}

\bibliography{bibliography}
\bibliographystyle{abbrv}

\end{document}